%% file: intrinsic_dimension_theory.tex
\documentclass[twoside,11pt]{article}
\usepackage[T1]{fontenc}
\usepackage{amsmath,amsthm}
\usepackage{float}
\usepackage[preprint]{jmlr2e}
\usepackage{lmodern}
\usepackage{mathptmx,booktabs,longtable,microtype}
\usepackage{placeins}
\hypersetup{hidelinks,pdftitle={Manifold Hypothesis: Conditional Certificates and Consistent Intrinsic Dimension Estimation},pdfauthor={U Jin Choi},pdfsubject={Theory-aligned evidence and master concept diagram with preserved proofs}}
\newtheorem{algorithm}{Algorithm}
\newcommand{\R}{\mathbb R}
\newcommand{\Prob}{\mathbb P}
\newcommand{\E}{\mathbb E}
\newcommand{\dist}{\operatorname{dist}}
\newcommand{\reach}{\operatorname{reach}}
\newcommand{\supp}{\operatorname{supp}}
\newcommand{\Cov}{\operatorname{Cov}}
\newcommand{\Tan}{\operatorname{Tan}}

\newcommand{\B}{\overline B}
\newcommand{\EMIT}{\mathsf{EMIT}}
\newcommand{\ABSTAIN}{\mathsf{ABSTAIN}}
\newcommand{\norm}[1]{\lVert #1\rVert}
\ShortHeadings{Manifold Hypothesis and Intrinsic Dimension}{Choi}
\firstpageno{1}
\begin{document}
\title{Manifold Hypothesis: Conditional Certificates\\
and Consistent Intrinsic Dimension Estimation}
\author{\name U Jin Choi\\
\addr Department of Mathematical Science\\
Korea Advanced Institute of Science and Technology\\
Daejeon, Republic of Korea}
\maketitle
\thispagestyle{jmlrps}

\begin{abstract}
We study what noisy data can justify before a manifold and its intrinsic
dimension are assumed. For repeated observations with independent isotropic
Gaussian errors of unknown variance, we compose a support confidence event
with a sound geometric verifier to certify a nearby smooth witness of
arbitrary compact latent support. Whole-support inference requires an
external lower mass condition; independent validation instead certifies
proximity of probability mass. Separate all-scale hypotheses yield
true-support structure and dimension recovery: a credited Reifenberg--Dini
criterion and positive continuous density give local $C^1$ structure and
consistent dimension estimation, while two-sided Ahlfors regularity and
positive reach give quantitative covariance separation under the stated
sampling and observation-transfer conditions.
An exact-arithmetic implementation verifies restricted witness families.
A further matrix check certifies linear coordinates on a geometrically
verified affine span containing the witness and propagates approximation
and decoding errors. This restricted construction neither implements the
general minimal-dimension Whitney theorem nor infers exact manifold support
from finite approximation. Executed torus and Julia-set experiments test
the observation-to-witness-to-coordinate route, including singular-support
controls. Dynamical systems and chaos are the primary application context;
state-identification and finite-horizon prediction results are conditional
on observation embeddings and transition validation. Takens reconstruction is an optional scalar-delay
application requiring additional smoothness, dynamics and observable
conditions. Image comparisons are auxiliary reconstruction diagnostics.
The finite-data verifier establishes neither those dynamical conditions
nor the all-scale geometric premises for true-dimension recovery.
\end{abstract}


\section{Introduction}\label{sec:intro}
Whitney's embedding theorems \citep{whitney1936,whitney1944} motivate the
data-analysis objective
\[
\text{Dataset}\ \longrightarrow\ \text{certified smooth witness}\
\longrightarrow\ \text{verified Whitney coordinates}.
\]
We start with an arbitrary nonempty compact support $E=\supp\mu\subset\R^q$.
The statistical target is an explicit smooth \emph{witness} $M$ whose
distance from $E$ is controlled at a reported confidence level. Smoothness,
reach and dimension are verified for $M$; they are not assumed for $E$.
Thus branches, boundaries and sets with interior remain admissible latent
supports. A nearby smooth witness does not establish that $E$ itself is a
manifold. Whole-support approximation, concentration of probability mass
near a witness, and exact support regularity are separate inferential
targets, each with its own premises.

The coordinate step acts on the accepted witness. Our restricted routine
checks a linear map on a rational affine span that is verified to contain
all of $M$, rather than merely the fitted sample. Theorem~\ref{thm:coordinatecomposition}
propagates the approximation error and gives decoding and finite-resolution
separation bounds; Appendix~\ref{app:directembedding} supplies the exact
matrix implementation. A generic embedding theorem does not guarantee
injectivity of an arbitrarily chosen projection, and this routine is not a
general construction in Whitney's minimal embedding dimension.

\paragraph{Dynamical systems and observation routes.}
For vector observations, the direct route certifies geometry in observation
space. Interpreting its coordinates as state coordinates additionally
requires full-state measurements or a separately justified state-separating
observation embedding. Section~\ref{sec:directexperiment} executes the route
in $\R^{16}$ on a data-fitted torus with a specified quasiperiodic flow.
The Lorenz, R\"ossler and KS studies in Section~\ref{sec:arxivvisuals} and
Appendix~\ref{app:arxivbench} are reconstruction diagnostics without issued
geometric certificates. Their endpoint sampling is independent; applying
iid guarantees to a dependent trajectory requires a separate sampling
argument. Chaotic attractors may be fractal, and an orbit, its closure,
an attractor and the full state space need not have the same dimension.

Scalar observations admit an optional delay route. For a compact $C^2$
state manifold $S$ without boundary, a $C^2$ diffeomorphism $F:S\to S$ and
a scalar $C^2$ observable $h:S\to\R$, write
\[
\Phi_{F,h}^{(m)}(x)
=\bigl(h(x),h(Fx),\ldots,h(F^{m-1}x)\bigr).
\]
Takens's theorem gives an embedding for $m=2d_S+1$ and an open dense set of
pairs $(F,h)$ in the stated $C^2$ topology, where $d_S=\dim S$
\citep[Theorem~1]{takens1981}. Eventual recovery of the \emph{true} dimension
$d_S$ therefore gives eventual correctness of that coordinate count.
Genericity does not establish the conditions for a fixed physical sensor
and sampling interval, nor does it provide finite-noise stability. The
$C^1$ conclusion of Theorem~\ref{thm:localexistence} alone is insufficient
for this classical $C^2$ result. Proposition~\ref{prop:observationinterface}
gives an independent-group scalar protocol, and
Algorithm~\ref{alg:observationroute} prepares its observations without
claiming a scalar-delay certificate. Takens is not required for the direct
vector route, and fewer sensor channels do not by themselves ensure lower
cost or noise sensitivity. Recovering an evolution law further requires
temporal information, as the nonidentifiability example in
Appendix~\ref{app:dynamicdistinction} makes explicit.

\paragraph{Statistical verification and its contribution.}
Approximate manifold testing for arbitrary distributions already appears
in \citet{fmn2016}; \citet{fikln2020} give geometric conditions for a nearby
smooth submanifold. Building on these precedents, we specify how an
unknown-variance repeated-observation confidence event composes with a
sound verifier. Numerical SVD, fitting and optimization may propose a
candidate; acceptance uses verified geometry and distance bounds.
Our rational comparisons and outward square-root bounds transport the
confidence event through data-dependent proposal selection, propagate
declared recording errors, and check both approximation directions
without trusting numerical neighbor distances. Independent validation
handles the alternative probability-mass target when lower mass constants
are unavailable. The contribution is this composed contract and its
restricted implementation. The geometric existence results, concentration
inequalities and local PCA remain credited ingredients.
Section~\ref{sec:contributions} states the distinct outputs, and
Section~\ref{sec:experiments} measures calibration cost against matched
numerical references while retaining failures and adverse outcomes.

\section{Related work}\label{sec:related}
\paragraph{Dimension in dynamical systems and chaos.}
Mathematicians and physicists studied dimension as a tool for identifying
the effective degrees of freedom of observed dynamics well before modern
manifold learning. \citet{kaplanyorke1979} connect Lyapunov exponents with
attractor dimension through the Kaplan--Yorke proposal; \citet{mori1980}
studies fractal dimensions of autonomous dissipative flows using their
Lyapunov characteristics. \citet{farmerottyorke1983} distinguish geometric
dimensions of an attractor from dimensions of its natural measure and
analyze their relationships. \citet{grassbergerprocaccia1983} develop
correlation-exponent calculations from a scalar time series. Lyapunov,
correlation, information, box-counting, and manifold dimensions have distinct
definitions. Their agreement requires hypotheses; a numerical value of one
does not by itself certify a smooth manifold or its integer dimension.
These studies motivate our dimension question and retain credit for their
dynamical and computational foundations.

\paragraph{Delay reconstruction and Euclidean embedding.}
\citet{whitney1936} establishes embedding in $2d+1$ Euclidean coordinates;
\citet{whitney1944} establishes the stronger $2d$ bound for smooth manifolds.
These general existence results do not require dynamics or a time series. \citet{takens1981} connects state-space
reconstruction to a single observable through delay coordinates.
\citet{saueryorkecasdagli1991} extend the reconstruction framework to
compact sets using box-counting dimension and prevalence, with additional
periodic-orbit conditions for delay maps. Their fractal results are
particularly relevant when an attractor is not a smooth manifold.
Our integer manifold-dimension estimators do not certify those
box-counting hypotheses or the required periodic-orbit conditions.
The present contribution supplies an observation-to-certificate interface,
conditional dimension recovery and a restricted verified coordinate step.
The classical embedding theorems remain credited inputs; Takens gives an
additional reconstruction application.
Recent work by \citet{botvinick2025} distinguishes invariant-measure
agreement from dynamical identification and establishes identification
results using time-delay invariant measures under additional conditions.
This addresses a stronger target than a nearby-witness certificate:
our geometric and coordinate checks do not establish conjugacy of an
unknown dynamical system or identify its evolution law.

\paragraph{Mass regularity and support geometry.}
Ahlfors's work in geometric function theory \citep{ahlfors1935} belongs
to the historical background of quantitative regularity. The precise
two-sided Ahlfors--David condition used here is taken from
\citet[Section~1, Eq.~(1.5)]{davidsemmes1991};
Appendix~\ref{app:termsmass} records the distinction. Reach and
positive-reach structure \citep{federer1959,rataj2017} connect this mass
control to local support geometry. They form a geometric strand of the
bridge to dimension estimation, rather than a claim that all chaotic
attractors have smooth manifold support.

\paragraph{Existence and testing.}
\citet{fmn2016} are direct prior work on approximate manifold testing without a generating manifold. \citet{fikln2020} supply the geometry used in Corollary~\ref{cor:whitney}. \citet{oneill2025} study smooth manifold extension through Whitney-type conditions; containing a compact set and equaling its support are different requirements. Theorem~\ref{thm:localexistence} specializes the local regularity results of \citet[including Remark~4.1]{kafer2019}; its all-scale hypotheses cannot be inferred from a finite prefix of scale tests without an additional argument. Our statistical composition and restricted implementation do not supersede these geometric results.

\paragraph{Local dimension and fitting.}
Multiscale SVD \citep{little2017}, tangent and curvature estimation \citep{aamari2019}, and Wasserstein-based local dimension analysis \citep{lim2024} provide relevant foundations. \citet{yaoxia2025} study manifold fitting under unbounded Gaussian noise,
including reach-controlled geometry, local PCA, and a local lower eigenvalue
bound (Lemma~10). \citet{yao2023manifoldfitting} analyze smooth manifold
estimators with geometric approximation and reach guarantees.
\citet{liyao2026curvature} study curvature-dependent fitting bias; their
Theorems~5.1 and~5.5 report pointwise bias and geometric output guarantees,
respectively, under their stated smoothness and sampling conditions.
\citet{yaosu2026review} review the broader fitting literature.
These are direct predecessors for combining geometry, local statistics,
algorithms, and numerical illustrations. We do not claim that this combination
in general, covariance-based dimension estimation, or positive-reach regularity
originates here.

\paragraph{Assumptions and the certification target.}
The preceding fitting works start from a latent manifold model. Conditional
inference under such a model is mathematically legitimate; it does not itself
test that model's existence premise. Our nearby-witness contract begins with
arbitrary compact support but retains its declared observation and coverage
assumptions. Exact dimension recovery uses separate geometric conditions.
The 2023 fitting model assumes known intrinsic dimension and noise level;
the 2026 curvature experiments use the true noise level. Our repeated-observation
model calibrates unknown Gaussian variance and estimates dimension. This
difference in information and observation budget is not a claim of uniform
superiority. Absence of the term Ahlfors regularity in another paper does not
establish disjointness: smooth positive-reach geometry and positive bounded
sampling densities can already give local two-sided mass bounds.

\paragraph{Positive-reach structure.}
\citet{rataj2017} establish the regular-part structure used in
Theorem~\ref{ar:thm:geometry}. Section~\ref{sec:ahlforsreach} combines it with
a slab-covering calculation and matrix Bernstein. In that model a threshold
count is justified, whereas the largest-gap rule additionally needs the
tangent-disk approximation of Lemma~\ref{lem:growingspectral}.

\paragraph{HADES.}
\citet{hades2025}, \emph{HADES: Fast singularity detection with local measure comparison}, combine local PCA with kernel maximum mean discrepancy (MMD). Appendix~\ref{app:termshades} writes its score and our certificate target explicitly. HADES's conditional singularity-recovery theorem is valid within its assumptions. We claim neither an experimental comparison nor superiority.

\section{Contributions and logical scope}
\label{sec:contributions}
The results separate four outputs and their assumptions.
\begin{enumerate}
\item \textbf{Finite-data witness certificates.}
Theorem~\ref{thm:main} certifies a verified smooth witness on a common
support-confidence event, allowing data-dependent proposals.
Theorem~\ref{thm:tube} instead certifies held-out probability mass.
Neither concludes exact manifold support.
\item \textbf{Conditional true-dimension recovery.}
The credited geometric criterion in Theorem~\ref{thm:localexistence},
the covariance limit in Proposition~\ref{prop:covariance}, and the
observed-repeat consistency theorem form one route. The alternative
Ahlfors--reach route in Section~\ref{sec:ahlforsreach} provides quantitative
spectral bounds and a threshold estimator. The finite-data verifier does
not establish either route's all-scale premises.
\item \textbf{Verified coordinates and reconstruction interfaces.}
Appendix~\ref{app:directembedding} checks affine-span coordinates for the
accepted witness and propagates its error. The construction uses $2d+1$
output coordinates in the restricted family; it is distinct from the
geometric Whitney approximation of Corollary~\ref{cor:whitney} and from a
general optimal-dimension embedding algorithm.
Section~\ref{sec:observationbridges} supplies fixed-sensor mass transfer,
independent delay groups, decoding, conditional true-state coordinates,
finite-horizon prediction and fixed-bank protocol selection.
Algorithms~\ref{alg:observationroute} and~\ref{alg:transitionvalidation}
record the additional premises. These are compositions of credited
ingredients, not general dynamics-identification theorems.
\item \textbf{Executed, auditable experiments.}
Section~\ref{sec:directexperiment} fits center and radii in a supplied
rational sensor span, verifies support error and geometry, and then checks
coordinates on the same observed input. Julia-set and boundary controls
retain the distinction between witness and latent dimensions.
The code separates numerical proposals, rational certificate decisions
and floating-point dimension estimates. Image comparisons in
Section~\ref{sec:imagesummary} are auxiliary diagnostics with credited
baselines and retained adverse results.
\end{enumerate}
All formal proof calculations are collected in the appendices, with
imported results attributed where used. Unresolved uncertainty bounds and
abstentions remain part of the reported evidence.

\section{Setting and certificate targets}\label{sec:setting}
For nonempty compact sets, write
\begin{equation}
d_H(A,B)=\max\{\sup_{a\in A}\dist(a,B),\sup_{b\in B}\dist(b,A)\}.
\label{eq:hausdorff}
\end{equation}
The reach of a closed set is the largest radius within which nearest-point projection is unique \citep{federer1959}. The base witness class consists of nonempty compact embedded $C^2$ submanifolds without boundary, of dimension $1\le d<q$. Corollary~\ref{cor:boundary} explicitly extends the statistical contract to verified manifolds with boundary.

We observe $k\ge2$ replicates of each of $N$ independent latent points:
\begin{equation}
X_{i\ell}=Z_i+\varepsilon_{i\ell},\quad
Z_i\stackrel{\mathrm{iid}}\sim\mu,\quad
\varepsilon_{i\ell}\stackrel{\mathrm{iid}}\sim N(0,\sigma^2 I_q).
\label{eq:model}
\end{equation}
The two collections are independent and $\sigma^2\ge0$ is unknown. Repeated measurements for unknown-noise inference have a long history, including \citet{dhm2008}. Here their Gaussian structure supplies a simple finite-sample confidence radius. The error model itself is an assumption, not a conclusion of variance calibration.

A \emph{support certificate} reports $(M,d,\widehat\tau,\widehat\epsilon)$ and asserts
\begin{equation}
\reach(M)\ge\widehat\tau>0,\qquad d_H(E,M)\le\widehat\epsilon.
\label{eq:target}
\end{equation}
Here $\EMIT$ denotes the event that the algorithm issues a support certificate. The error criterion is
\[
\Prob\{\EMIT\text{ and \eqref{eq:target} is false}\}\le\delta.
\]
This bounds the joint probability of issuance and error, not error conditional on issuance. Fix application tolerances $\epsilon_*,\tau_*$ in advance; issue only when $\widehat\epsilon\le\epsilon_*$ and $\widehat\tau\ge\tau_*$. Failed checks return $\ABSTAIN$ (\emph{withhold certification}), which does not imply that no suitable manifold exists.

For whole-support inference we use an external lower mass condition:
\begin{equation}
\mu(\overline B(x,t))\ge c_0t^s\quad
(x\in E,\ 0<t\le t_0).
\label{eq:mass}
\end{equation}
The supplied constants $c_0,s,t_0>0$ are external, not verified by finite data. By comparison, \emph{Ahlfors--David $s$-regularity} requires two-sided bounds
\[
c_{\rm A}t^s\le\mu(\overline B(x,t))\le C_{\rm A}t^s
\]
uniformly in $x\in E$ and $0<t\le t_0$ \citep{davidsemmes1991}. Equation~\eqref{eq:mass} uses only the lower bound; $s$ need not equal the witness dimension $d$, and neither condition implies manifold structure. Appendix~\ref{app:termsmass} specifies the attribution and measure convention. Without justified lower mass constants, we report only a \emph{mass certificate}, bounding $\mu\{z:\dist(z,M)>\epsilon\}$ on independent validation data.

\section{Theory and certificate guarantees}\label{sec:main}
Let $Y_i=k^{-1}\sum_\ell X_{i\ell}$, $Y=\{Y_i\}_{i=1}^N$, and
\begin{equation}
S=\sum_{i,\ell}\norm{X_{i\ell}-Y_i}^2,\quad \nu=Nq(k-1).
\label{eq:residual}
\end{equation}
Choose positive budgets $\delta_C,\delta_L,\delta_E$ whose sum is $\delta<1$. Let $q_{\nu,\delta_C}$ be the lower $\delta_C$ quantile of $\chi^2_\nu$, and put
\begin{equation}
U=\frac{S}{q_{\nu,\delta_C}},\qquad
a=\sqrt{\frac{U}{k}}\left(\sqrt q+\sqrt{2\log(N/\delta_L)}\right).
\label{eq:noise}
\end{equation}
Choose a deterministic $0<h\le2t_0$ satisfying
\begin{equation}
\frac{4^s}{c_0h^s}\exp\{-Nc_0(h/2)^s\}\le\delta_E,
\qquad \eta=h+a.
\label{eq:coverage}
\end{equation}
Section~\ref{sec:arithmetic} gives conservative rational replacements actually used in the code, including recording errors.

\begin{theorem}[Support confidence and verified witnesses]\label{thm:main}
Under \eqref{eq:model} and \eqref{eq:mass}, the construction above satisfies
\begin{equation}
\Prob\{d_H(E,Y)\le\eta\}\ge1-\delta.
\label{eq:supportconfidence}
\end{equation}
Suppose a deterministic sound verifier accepts only candidates $M$ that are witnesses and for which
\begin{equation}
\reach(M)\ge\widehat\tau>0,\qquad d_H(Y,M)\le e.
\label{eq:verify}
\end{equation}
For any measurable, data-dependent choice of an accepted candidate, including its dimension, the reported value $\widehat\epsilon=\eta+e$ satisfies
\begin{equation}
\Prob\{\EMIT,\ d_H(E,M)>\widehat\epsilon\}\le\delta.
\label{eq:honesty}
\end{equation}
The same confidence event applies simultaneously to all candidates verified against this $Y$ and $\eta$.
\end{theorem}
\par\noindent\textit{Detailed proof.} Appendix~\ref{app:main} (p.~\pageref{app:main}).
The proof separates Gaussian calibration, support coverage and simultaneous candidate selection.

The composition in Appendix~\ref{app:main} uses the Hausdorff triangle inequality, a classical deterministic fact. Its practical value is a contract that keeps a flexible proposal procedure separate from a sound acceptance rule. No additional candidate-selection penalty is necessary for this common support event. Optional stopping over new data or an uncalibrated change in the mass assumptions is not covered.

\subsection{Connection to geometric Whitney approximation}
Here the geometric Whitney problem concerns construction of a smooth set
near prescribed data. It supplies the witness-existence part of our
workflow; the classical Whitney embedding principle and Takens delay
reconstruction discussed in Section~\ref{sec:intro} concern coordinate maps
on an established state space.
Fix $r>0$ and $1\le d<q$. \emph{Local SVD} uses one neighborhood $I(y,r)=\{i:\norm{Y_i-y}\le r\}$, its size $m\ge2$, and mean $\bar Y_{y,r}$:
\[
A_{y,r}=m^{-1/2}[Y_i-\bar Y_{y,r}]_{i\in I(y,r)}
       =U_{y,r}\Sigma_{y,r}V_{y,r}^{\mathsf T}.
\]
The first $d$ left singular vectors propose the anchored plane $L_y=y+\operatorname{span}(u_1,\ldots,u_d)$. \emph{Multiscale SVD} repeats this calculation over radii $r_j$, comparing $\{\Sigma_{y,r_j},U_{y,r_j}\}_j$ \citep{little2017}; Appendices~\ref{app:termssvd}--\ref{app:termsmultiscale} give the covariance, residual and scale formulas. Neither an SVD fit nor its scale stability verifies the following \emph{bilateral} condition:
\begin{equation}
\max_{y\in Y}d_H\bigl(Y\cap\overline B(y,r),L_y\cap\overline B(y,r)\bigr)\le b.
\label{eq:bilateral}
\end{equation}

\begin{corollary}[Classical geometric bridge with statistical error]\label{cor:whitney}
Let $\sigma_d>0$ be an admissible constant in Theorem~2 of \citet{fikln2020}. Under Theorem~\ref{thm:main}'s statistical assumptions, every verified choice satisfying \eqref{eq:bilateral} and $0<b<\sigma_dr$ admits, on the common confidence event, a compact smooth $d$-manifold $M$ with
\begin{equation}
d_H(E,M)\le\eta+5b,\qquad \reach(M)\ge r/3.
\label{eq:whitneyconclusion}
\end{equation}
\end{corollary}
\par\noindent\textit{Detailed proof.} Appendix~\ref{app:whitneyproof} (p.~\pageref{app:whitneyproof}).
The smooth existence theorem and its constants belong to \citet{fikln2020}. We apply it to the observed finite set, then use \eqref{eq:supportconfidence}. This is an existence certificate. To claim that a particular numerical atlas realizes it requires verification of that atlas. We have not extracted a usable certified numerical value of $\sigma_d$ or implemented the general Whitney construction. No experiment below invokes an arbitrary surrogate value. Instead, explicit witness families admit direct verification of \eqref{eq:verify}.

\subsection{Local manifold existence from geometric conditions}\label{sec:localexistence}
The following local dyadic consequence of \citet[Theorems~1.2 and~1.4, Remark~4.1]{kafer2019} derives manifold structure from bilateral flatness of a closed set. No chart, tangent space, or reach bound is assumed.

\begin{theorem}[Local $C^1$ regularity: a dyadic Reifenberg--Dini criterion]\label{thm:localexistence}
Let $E\subset\R^q$ be closed, $x_0\in E$, $R>0$, and $1\le d<q$. Put $r_j=2^{-j}R$ for $j\ge0$. Suppose nonnegative numbers $\epsilon_j$ satisfy $\sum_{j=0}^\infty\epsilon_j<\infty$ and that, for every $x\in E\cap B(x_0,2R)$ and every $j$, there is a linear $d$-plane $V_{x,j}$ with
\begin{equation}
\begin{split}
d_H\bigl(E\cap\overline B(x,r_j),\,
(x+V_{x,j})\cap\overline B(x,r_j)\bigr)\\
\le\epsilon_jr_j.
\end{split}
\label{eq:dyadicregularity}
\end{equation}
Then some open neighborhood $U$ of $x_0$ has $E\cap U$ an embedded $C^1$ $d$-manifold without boundary. Its tangent projection at $x$ is the limit of the orthogonal projections onto $V_{x,j}$.
\end{theorem}
\par\noindent\textit{Detailed proof.} Appendix~\ref{app:localexistence} (p.~\pageref{app:localexistence}).
The underlying existence and regularity results are credited to classical Reifenberg--Dini theory, specifically \citet[Remark~4.1]{kafer2019}, which already treats summable flatness along controlled scale sequences; we claim no new regularity principle. K\"afer also credits an independent Dini-condition proof to Ranjbar-Motlagh (1998); we acknowledge that attribution without claiming to have inspected the original preprint. Appendix~\ref{app:localexistence} supplies a self-contained proof of this specialization, deriving the limiting planes, their continuity, and a local graph. The uniform condition at every support point and arbitrarily small scales is substantive: the current finite-sample procedure does not verify it. The conclusion is $C^1$ regularity, not positive reach or $C^2$ regularity.

\subsection{Independent validation without a mass lower bound}
Let a training sample select a geometrically verified witness $M$ before observing independent validation latent points $Z^v_1,\ldots,Z^v_n\sim\mu$. Their replicate means are $V_i$. Let $a_v\ge0$ be a possibly data-dependent radius and suppose
$\Prob\{\max_i\norm{V_i-Z_i^v}>a_v\}\le\delta_N$ and verified distances satisfy $\bar d_i\ge\dist(V_i,M)$. Define
\begin{equation}
\widehat p(t)=\frac1n\sum_i\mathbf1\{\bar d_i>t\},\quad
u=\sqrt{\frac{\log(2/\delta_V)}{2n}}.
\label{eq:massstats}
\end{equation}

\begin{theorem}[Tube-mass certificate]\label{thm:tube}
With probability at least $1-\delta_N-\delta_V$, simultaneously for all $t\ge0$,
\begin{equation}
\mu\{z:\dist(z,M)>t+a_v\}\le\min\{1,\widehat p(t)+u\}.
\label{eq:tube}
\end{equation}
The threshold $t$ may be selected from validation data. If selection is among $J$ candidates fixed by training, replace $u$ by $\sqrt{\log(2J/\delta_V)/(2n)}$.
\end{theorem}
\par\noindent\textit{Detailed proof.} Appendix~\ref{app:mass} (p.~\pageref{app:mass}).
Conditionally on training, the latent distances are iid. The DKW--Massart bound \citep{massart1990} holds at all thresholds. On the noise event, $\dist(Z_i^v,M)>t+a_v$ implies $\bar d_i>t$. Their combination proves the result. This is a learn-then-validate construction related to \citet{learnthentest}; refitting $M$ on the same validation sample requires additional control. The result certifies probability mass, not coverage of every part of the support.

\paragraph{Fixed latent radius.}
For an application radius $r_*$ fixed in advance, set $t=r_*-a_v$ if $a_v\le r_*$ and otherwise abstain. The uniform statement of Theorem~\ref{thm:tube} then controls
$\mu\{\dist(z,M)>r_*\}$ with the same error budget, although $t$ depends on the data. The public request \texttt{mass\_\allowbreak radius\_\allowbreak star} implements this substitution with exact rational arithmetic. Its alternative \texttt{mass\_\allowbreak threshold} fixes $t$ and certifies the variable radius $t+a_v$; the API rejects requests specifying both. Appendix~\ref{app:round2} gives the execution details.

\subsection{Limits of exact support certification}\label{sec:limits}
There is an elementary obstruction to unrestricted exact support certification, even with arbitrarily many repeats at each of finitely many latent points.
\begin{proposition}[A missed short branch]\label{prop:spur}
Let $P_0$ be normalized length on the unit circle, and let $P_\epsilon$ be normalized length on its union with $[1,1+\epsilon]\times\{0\}$, $0<\epsilon\le1/4$. For any event $A$ based on $N$ latent samples and their repeats under the same noise mechanism,
\begin{equation}
\Prob_{P_\epsilon}(A)\ge
\left(\frac{2\pi}{2\pi+\epsilon}\right)^N\Prob_{P_0}(A).
\label{eq:spur}
\end{equation}
Consequently, a rule with false exact-regularity issuance probability at most $\delta$ for every $P_\epsilon$ has issuance probability at most $\delta$ under $P_0$. These alternatives have uniform positive one-dimensional lower and upper mass constants at small scales.
\end{proposition}
\par\noindent\textit{Detailed proof.} Appendix~\ref{app:spurproof} (p.~\pageref{app:spurproof}).
The proof conditions on missing the branch and lets $\epsilon\downarrow0$. It is an elementary indistinguishability argument, not a claimed new minimax theorem. It explains why increasing replication cannot compensate for an unobserved feature of arbitrarily small mass.

Dimension needs the same care. The annulus $\{x:1-\epsilon\le\norm{x}\le1+\epsilon\}$ is within Hausdorff distance $\epsilon$ of a circle but has Hausdorff dimension two. Thus the dimension of a certified witness is not necessarily the intrinsic dimension of $E$. Appendix~\ref{app:dimension} derives the geometric--covariance bridge from Theorem~\ref{thm:localexistence}'s geometric hypotheses, a local density condition, and a population covariance limit. Section~\ref{sec:dimconsistency} constructs a statistic from observed repeats and proves its required covariance consistency. Local manifold existence is derived from the geometric conditions. This separate all-scale result is not used to establish Theorem~\ref{thm:main}. Covariance eigenvalue counting is also distinct from spectral dimension defined through a heat kernel or a Laplacian.

\begin{corollary}[Verified witnesses with boundary]\label{cor:boundary}
The support and tube-mass conclusions of Theorems~\ref{thm:main} and~\ref{thm:tube} remain valid when the verifier instead establishes that $M$ is a compact embedded $C^2$ manifold possibly with boundary, with the same verified reach and distance bounds. No boundary regularity of the latent support is assumed. The local boundary-free conclusion of Theorem~\ref{thm:localexistence} is unchanged.
\end{corollary}
\par\noindent\textit{Detailed proof.} Appendix~\ref{app:boundaryproof} (p.~\pageref{app:boundaryproof}).
The support proof uses only the Hausdorff triangle inequality; the mass proof uses the 1-Lipschitz distance function. Appendix~\ref{app:boundaryproof} proves the extension; Appendix~\ref{app:boundaryfractal} gives the explicit arc/segment checks. This is a contract extension, not a new general manifold-existence theorem.

\section{Observed-data intrinsic dimension and consistency}
\label{sec:dimconsistency}
This section closes the estimator-error premise of Proposition~\ref{prop:covariance} for a specified, computable statistic. It does not turn a nearby-witness certificate into an exact-support certificate. Local PCA, dimension estimation, and their consistency have substantial precedents \citep{little2017,kaslovsky2013,hein2005,aamari2019,lim2024}. In particular, \citet[Proposition~4.4 and Theorem~5.3]{lim2024} compare local covariance to a tangent-disk covariance. The related preprint \citet[Appendix~E.3]{choi2026} already uses accurately averaged observations and gives the same sufficient almost-sure sampling orders used below. We claim no new local-PCA principle or first consistency theorem. Here we give the complete observation-to-limit argument in the locally derived $C^1$ setting, with a noisy observed anchor and an explicit observable membership-error bound.

\subsection{Geometric target before statistical estimation}

\begin{proposition}[Local geometric--covariance bridge]\label{prop:covariance}
For fixed ambient dimension $q$, let $E=\supp\mu$ satisfy the geometric hypotheses of Theorem~\ref{thm:localexistence}, and let $U$ be the neighborhood it produces. Fix $x\in E\cap U$. Assume near $x$ that $\mu=f\,\mathcal H^d|_E$, where $f$ is continuous and strictly positive and $\mathcal H^d$ is the normalized $d$-dimensional Hausdorff measure. Let
$C_x(r)=\Cov(Z\mid Z\in\overline B(x,r))$. Then, in operator norm,
\[
r^{-2}C_x(r)\longrightarrow\frac{1}{d+2}P_x\quad\text{as }r\downarrow0,
\]
where $P_x$ is the orthogonal projection onto $T_x(E\cap U)$, the tangent space established by that theorem. If $r_n\downarrow0$ and a symmetric estimator satisfies $\norm{\widehat C_n-C_x(r_n)}_{\rm op}/r_n^2\to0$ in probability (respectively almost surely), then, for fixed $0<\gamma<1/(q+2)$,
\[
\widehat d_n=\#\{j:\lambda_j(\widehat C_n)/r_n^2>\gamma\}
\]
equals $d$ with probability tending to one (respectively eventually almost surely).
\end{proposition}
\par\noindent\textit{Detailed proof.} Appendix~\ref{app:dimension} (p.~\pageref{app:dimension}).


The proof in Appendix~\ref{app:dimension} retains the centering term and uses the actual ambient ball, rather than silently substituting a tangent disk. The constant $1/(d+2)$ comes from the second moment of the uniform unit $d$-ball. The statistic below estimates the geometric dimension identified by this limit. It does not estimate a heat-kernel spectral dimension.

\subsection{A formula using observed repeats only}
At stage $n$, observe $k_n\ge2$ repeats of each of $n+1$ groups, indexed by $i=0,\ldots,n$, under \eqref{eq:model}. Reserve group $0$ as the anchor and do not include it in its own neighborhood. We use one infinite iid latent sequence $(Z_i)_{i\ge0}$ and an independent infinite array $(\varepsilon_{i\ell})_{i\ge0,\ell\ge1}$ of iid Gaussian errors, with $X_{i\ell}=Z_i+\varepsilon_{i\ell}$. Stage $n$ uses the first $k_n$ repeats of groups $0,\ldots,n$, reusing previously observed entries. The latent anchor $Z_0$ is fixed throughout; when $k_n\equiv k$, its observed mean $Y_0$ is also fixed. Write
\begin{align}
Y_{i,n}&=\frac1{k_n}\sum_{\ell=1}^{k_n}X_{i\ell},\nonumber\\
I_n&=\{1\le i\le n:\norm{Y_{i,n}-Y_{0,n}}\le r_n\},\quad m_n=|I_n|,\label{eq:dimneighbor}\\
\bar Y_{I_n}&=\frac1{m_n}\sum_{i\in I_n}Y_{i,n},\nonumber\\
\widehat C_n&=\frac1{m_n}\sum_{i\in I_n}
(Y_{i,n}-\bar Y_{I_n})(Y_{i,n}-\bar Y_{I_n})^{\mathsf T}.
\label{eq:dimcov}
\end{align}
Equivalently, let $\mathsf X_n$ have the selected $Y_{i,n}$ as its $q\times m_n$ columns. For $m_n>0$,
\begin{align}
\mathsf S_n&=m_n^{-1}\mathsf X_n\mathsf X_n^{\mathsf T},\quad
H_{m_n}=I_{m_n}-\mathbf1\mathbf1^{\mathsf T}/m_n,\nonumber\\
\widehat C_n&=m_n^{-1}\mathsf X_nH_{m_n}\mathsf X_n^{\mathsf T}
=\mathsf S_n-\bar Y_{I_n}\bar Y_{I_n}^{\mathsf T}.
\label{eq:dimgram}
\end{align}
Thus $XX^{\mathsf T}/N$ is a covariance when the columns have been centered; otherwise it is a second moment. Here the normalization is the local count $m_n$, not the total number of groups.
The radius $r_n>0$ is deterministic, fixed before inspecting stage-$n$ observations. If $m_n<2$, set $\widehat d_n=0$ and record insufficient local data. Otherwise define
\begin{equation}
\quad\widehat d_n
=\sum_{j=1}^{q}\mathbf1\left\{
\lambda_j(\widehat C_n)>\frac{r_n^2}{2(q+2)}\right\}.\quad
\label{eq:actualdimension}
\end{equation}
For the centered $m_n\times q$ matrix $A_n$ with rows
$(Y_{i,n}-\bar Y_{I_n})^{\mathsf T}$, multiplication gives, for $m_n\ge2$,
\[
A_n^{\mathsf T}A_n
=\sum_{i\in I_n}(Y_{i,n}-\bar Y_{I_n})(Y_{i,n}-\bar Y_{I_n})^{\mathsf T}
=m_n\widehat C_n.
\]
If $A_n=U\Sigma V^{\mathsf T}$, then
$\widehat C_n=V(\Sigma^{\mathsf T}\Sigma/m_n)V^{\mathsf T}$.
Thus $\lambda_j(\widehat C_n)=s_j(A_n)^2/m_n$, with zero padding.
Multiplying the positive threshold by $m_n$ gives exactly
\[
\widehat d_n=\sum_{j=1}^{\min(m_n,q)}
\mathbf1\left\{s_j(A_n)^2>\frac{m_nr_n^2}{2(q+2)}\right\}.
\]
Centering gives $A_n^{\mathsf T}\mathbf1=0$ and rank at most $\min(q,m_n-1)$; included zero singular values fail the positive threshold. This is a samplewise identity. Neither $d$, the tangent plane, the density, reach, nor $\sigma$ is an estimator input. In particular the output is allowed to be $q$; artificially clipping it to $q-1$ would conceal model failures.

\paragraph{Unknown noise and selection.}
Replicate averaging is the noise treatment in \eqref{eq:dimcov}. We do not subtract $\widehat\sigma^2 I_q/k_n$ and then claim conditional unbiasedness: selection using the same averages generally destroys that identity. Instead a simultaneous deterministic perturbation bound controls every selected group. The observed pooled residual still provides a computable uncertainty bound. Set
\begin{align}
S_n&=\sum_{i=0}^n\sum_{\ell=1}^{k_n}\norm{X_{i\ell}-Y_{i,n}}^2,
\quad\nu_n=(n+1)q(k_n-1),\nonumber\\
\widehat\sigma_n^2&=S_n/\nu_n,\quad U_n=S_n/q_{\nu_n,\delta_C},\nonumber\\
a_n&=\sqrt{U_n/k_n}
\left(\sqrt q+\sqrt{2\log((n+1)/\delta_L)}\right).
\label{eq:dimcalibration}
\end{align}
As in Theorem~\ref{thm:main}, with probability at least $1-\delta_C-\delta_L$, all $\norm{Y_{i,n}-Z_i}\le a_n$, including the anchor. No true variance is supplied to the program.
The pooled Gaussian residual matrix has an exact Wishart law \citep{wishart1928}; its trace gives this chi-square calibration. Appendix~\ref{app:dimfinite} derives that representation. The covariance after noisy neighborhood selection is not assigned a Wishart law.

\subsection{An observable covariance-error bound}
For one fixed stage suppress $n$, let $x=Z_0$, and use indices $1\le i\le n$ in all sets below. Put
\begin{align}
I^-&=\{i:\norm{Y_i-Y_0}\le r-2a\},\nonumber\\
I^+&=\{i:\norm{Y_i-Y_0}\le r+2a\},\nonumber\\
m_-&=|I^-|,\quad B=|I^+\setminus I^-|,\quad R=r+2a.
\label{eq:dimobservablecore}
\end{align}
Define $I^-=\varnothing$ if $r-2a<0$. For $m\ge1$ and $0<\delta<1$ define
\begin{equation}
\mathcal B_q(m,\delta)=3\sqrt{\frac{2t}{m}}+\frac{4t}{m},
\qquad t=\log\frac{2(q+1)}{\delta}.
\label{eq:dimbernsteinradius}
\end{equation}
For $m_->0$ let
\begin{align}
E_n={}&2Ra+a^2+6R^2\min\{1,B/m_-\}\nonumber\\
&+r^2\mathcal B_q(m_-,\delta_P).
\label{eq:dimfinitebound}
\end{align}
Set $E_n=+\infty$ when $m_-=0$. The three terms control location error, changed membership, and population sampling, respectively. The last term applies matrix Bernstein \citep[Theorem~6.1]{tropp2012} to augmented first and second moments; Appendix~\ref{app:dimbounded} checks its variance parameter and the centering step explicitly. No proportional-dimensional random-matrix limit is required.

\begin{proposition}[Observable local covariance uncertainty]
\label{prop:dimfinite}
For deterministic $r>0$ and positive budgets with $\delta_C+\delta_L+\delta_P<1$, under the repeated Gaussian observation model,
\[
\Prob\{\norm{\widehat C_n-C_{Z_0}(r)}_{\rm op}\le E_n\}
\ge1-\delta_C-\delta_L-\delta_P.
\]
The covariance is defined as zero for an empty selected set. The bound is vacuous if $E_n=+\infty$. No manifold or lower mass premise is needed for this proposition, provided the true ball has positive probability. Its proof is in Appendix~\ref{app:dimfinite}.
\end{proposition}
\par\noindent\textit{Detailed proof.} Appendix~\ref{app:dimfinite} (p.~\pageref{app:dimfinite}).
The observable radius bounds estimation of a population ball covariance. It is not by itself a finite-sample confidence interval for geometric dimension. The geometric remainder
$\theta_x(r)=\norm{r^{-2}C_x(r)-P_x/(d+2)}_{\rm op}$
is not bounded from finite observations here.

\subsection{Consistency with the geometric intrinsic dimension}
\begin{theorem}[Consistent intrinsic-dimension estimation from repeats]
\label{thm:dimconsistent}
Fix $q\ge2$. Let $E=\supp\mu$ be nonempty and compact. Assume that at every $x\in E$ the geometric hypotheses of Theorem~\ref{thm:localexistence} hold locally with the same unknown integer $1\le d<q$. Assume $\mu=f\mathcal H^d|_E$ with $f$ continuous and strictly positive. Thus local manifold existence is a conclusion of the geometric theorem. Let the errors be independent $N(0,\sigma^2I_q)$ with fixed unknown $0\le\sigma^2<\infty$, independent of the iid latent groups. Use \eqref{eq:dimneighbor}--\eqref{eq:actualdimension}.

If
\begin{equation}
r_n\downarrow0,\quad nr_n^d\longrightarrow\infty,
\quad\frac{k_nr_n^2}{\log(n+1)}\longrightarrow\infty,
\label{eq:dimweakrates}
\end{equation}
then
\begin{equation}
\frac{\norm{\widehat C_n-C_{Z_0}(r_n)}_{\rm op}}{r_n^2}
\xrightarrow{\Prob}0,
\qquad\Prob(\widehat d_n=d)\longrightarrow1.
\label{eq:dimweakconclusion}
\end{equation}
If the point-sample condition is strengthened to
\begin{equation}
\frac{nr_n^d}{\log(n+1)}\longrightarrow\infty,
\label{eq:dimstrongrates}
\end{equation}
the covariance convergence holds almost surely and
\begin{equation}
\ \Prob\{\exists n_0<\infty\ \forall n\ge n_0:
\widehat d_n=d\}=1.\ 
\label{eq:dimstrongconclusion}
\end{equation}
Moreover $\dim_H(E)=d$. The pointwise conclusion also holds at a fixed unknown anchor $x$ with its own repeats, using only the local hypotheses near $x$; that local version makes no assertion about other components of the support.
\end{theorem}
\par\noindent\textit{Detailed proof.} Appendix~\ref{app:dimconsistent} (p.~\pageref{app:dimconsistent}).
We prove every observation-error step in Appendix~\ref{app:dimproof}. No consistency of $\widehat C_n$ is assumed in this theorem. The ambient dimension is fixed; increasing-$q$ asymptotics would require additional control of the shrinking spectral margin. No Hölder rate, $C^2$ regularity, or positive reach of $E$ is required for this qualitative result. Its all-scale geometric and density assumptions remain substantive external model conditions.

\begin{corollary}[Schedules requiring no knowledge of $d$]
\label{cor:dimschedule}
For any fixed $r_*>0$, the choices
\[
r_n=\frac{r_*}{\log(n+e)},\qquad
k_n=\max\{2,\lceil\log^4(n+e)\rceil\}
\]
satisfy the almost-sure conditions for every fixed $1\le d<q$ and every fixed finite $\sigma$. Alternatively, $r_n=r_*n^{-\alpha}$ and $k_n=\max\{2,\lceil n^\beta\rceil\}$ work whenever $0<\alpha<1/q$ and $\beta>2\alpha$.
\end{corollary}
\par\noindent\textit{Detailed proof.} Appendix~\ref{app:dimschedule} (p.~\pageref{app:dimschedule}).
These are sufficient schedules, not optimal sample-complexity or finite-sample tuning recommendations. The first uses $(n+1)k_n=O(n\log^4n)$ replicate vectors, each with $q$ coordinates; its geometric bias may decay very slowly. The proof of Theorem~\ref{thm:dimconsistent} needs no unknown $d$ as an input to these schedules.

\subsection{Why fixed replication is a different problem}
\begin{proposition}[Fixed-noise local covariance targets the observed law]
\label{prop:dimfixednoise}
Under this nested observation convention, fix $k<\infty$ and $\sigma>0$, and use the same raw local estimator with $r_n\downarrow0$. If $nr_n^q\to\infty$, then $\widehat d_n\to q$ in probability; if $nr_n^q/\log(n+1)\to\infty$, it equals $q$ eventually almost surely. In contrast, subtracting the true fixed noise floor $\sigma^2 I_q/k$ from this selected covariance makes every counted eigenvalue fall below the positive threshold for all sufficiently small $r_n$ whenever the covariance is defined.
\end{proposition}
\par\noindent\textit{Detailed proof.} Appendix~\ref{app:dimfixed} (p.~\pageref{app:dimfixed}).
The proof in Appendix~\ref{app:dimfixed} uses the strictly positive smooth density of the Gaussian-convolved observation law. This is a failure statement about the displayed shrinking-ball estimators at fixed replication, not an impossibility theorem for every deconvolution estimator. Ahlfors regularity alone, a nearby witness, and a low rank at one scale do not supply the assumptions of Theorem~\ref{thm:dimconsistent}.

\paragraph{Implementation and numerical precision.}
\texttt{code/\allowbreak intrinsic\_\allowbreak dimension.py} implements the observed neighborhood, centered SVD, pooled variance, and \eqref{eq:dimfinitebound}. It reports a point estimate separately from its covariance uncertainty and never emits an exact-support certificate. The mathematics above concerns exact observations and arithmetic. If stored means differ by at most $e_n$ in Euclidean norm, consistency persists when $e_n/r_n\to0$; this follows by replacing the uniform noise bound by its sum with $e_n$. A symmetric covariance computation error of operator norm $v_n$ is harmless when $v_n/r_n^2\to0$. A fixed nonvanishing recording-error bound does not satisfy these shrinking-scale requirements. The supplied floating-point implementation is a reproducible numerical implementation, not a rational-arithmetic proof of its eigenvalue comparisons.

\subsection{A spectral bridge with growing ambient dimension}
\label{sec:growingdimension}
In this subsection only, write $d=d_N$ for ambient dimension and $p$ for intrinsic dimension; these correspond to $q$ and $d$, respectively, in the preceding sections. Here $N$ counts local observations, not all groups before neighborhood selection. We allow $d_N/N=O(1)$ and keep $p$ fixed.

The exact tangent-disk identity is already \citet[Lemma~13]{arias2017}, restated with attribution in \citet[Lemma~6.1]{lim2024}:
\begin{equation}
\Cov\{\operatorname{Unif}(T\cap\overline B(0,r))\}
=\frac{r^2}{p+2}P_T,\qquad \dim T=p.
\label{eq:tangentdiskcredited}
\end{equation}
The covariance--transport comparison and its local geometric application are developed in \citet[Propositions~3.3, 4.4 and Theorem~5.3]{lim2024}. The following lemma combines that established mechanism with our stated matrix Bernstein specialization. It is a conditional integration lemma, not a new manifold-existence principle.

\begin{lemma}[Geometric dimension from a local covariance spectrum]
\label{lem:growingspectral}
Fix $p\ge1$ and let $d_N>p$. At $x_N\in E_N\subset\R^{d_N}$, assume the geometric hypotheses of Theorem~\ref{thm:localexistence} with dimension $p$. They yield a local tangent plane $T_N$ and projection $P_N$. Let $\mu_N=f_N\mathcal H^p|_{E_N}$ locally near $x_N$, with $f_N$ positive and continuous, and let $r_N\downarrow0$. Define
\[
\nu_N=\operatorname{Law}\left(\frac{Z-x_N}{r_N}\ \middle|\ \norm{Z-x_N}\le r_N\right),\quad Z\sim\mu_N.
\]
Assume that $\omega_N:=W_1(\nu_N,\mathsf U_N)\to0$, where $\mathsf U_N$ is uniform on $T_N\cap\overline B(0,1)$. The distance $W_1$ is the infimum of $\E\norm{U-V}$ over couplings of the two laws.

For $N$ iid draws $Z_{Ni}$ from that conditional ball law, center the columns and put
\begin{equation}
X_N=[Z_{Ni}-\bar Z_N]_{i=1}^N\in\R^{d_N\times N},
\qquad S_N=\frac1N X_NX_N^{\mathsf T}.
\label{eq:growingsamplecov}
\end{equation}
For every $0<\delta<1$, with probability at least $1-\delta$,
\begin{align}
\left\|r_N^{-2}S_N-\frac{P_N}{p+2}\right\|_{\rm op}
&\le\epsilon_N(\delta),\nonumber\\
\epsilon_N(\delta)&:=2\omega_N+\omega_N^2+\mathcal B_{d_N}(N,\delta).
\label{eq:growingbridge}
\end{align}
Order the eigenvalues decreasingly, append $\lambda_{d_N+1}(S_N)=0$, and define
\begin{equation}
\widehat p_N=\min\mathop{\rm arg\,max}_{1\le j\le d_N}
\{\lambda_j(S_N)-\lambda_{j+1}(S_N)\}.
\label{eq:growinggap}
\end{equation}
Set $\widehat p_N=0$ if $S_N=0$. If $\epsilon_N(\delta)<1/[4(p+2)]$, then $\Prob(\widehat p_N=p)\ge1-\delta$.
If $d_N/N=O(1)$ and $\omega_N\to0$, then $\widehat p_N=p$ eventually almost surely, and hence in probability, on any common probability space carrying these row-wise iid samples.
\end{lemma}
\par\noindent\textit{Detailed proof.} Appendix~\ref{app:growingspectral} (p.~\pageref{app:growingspectral}).

\paragraph{Why the geometric approximation is an additional condition.}
For a fixed ambient dimension and model, the positive continuous density and the derived $C^1$ chart imply $\omega(r)\to0$. A sequence of models requires uniform control. In tangent coordinates, write the charts supplied by Theorem~\ref{thm:localexistence} as $u\mapsto x_N+(u,g_N(u))$, with $g_N(0)=Dg_N(0)=0$. If the whole local ball lies in this chart, sufficient conditions are
\begin{align}
\eta_N&:=\sup_{\norm u\le r_N}\norm{Dg_N(u)}_{\rm op}\to0,\nonumber\\
\zeta_N&:=\sup_{\norm u\le r_N}
\left|\frac{f_N(x_N+(u,g_N(u)))}{f_N(x_N)}-1\right|\to0.
\label{eq:growinggeomconditions}
\end{align}
Appendix~\ref{app:growingspectral} proves a bound for $\omega_N$ independent of ambient dimension. These are external model conditions, not conclusions of $d_N/N=O(1)$ or tests performed on finite observations.

\paragraph{Observed noise and the meaning of spectral dimension.}
If an observed covariance $\widetilde S_N$ satisfies
$r_N^{-2}\norm{\widetilde S_N-S_N}_{\rm op}\le\tau_N$ outside probability $\alpha_N$, replace $\epsilon_N$ by $\epsilon_N+\tau_N$ and $\delta$ by $\delta+\alpha_N$ in the finite conclusion. The strong conclusion requires $\tau_N\to0$ and summable exceptional probabilities, or an almost-sure error bound. The location and membership calculations above are one route to such control; it is not automatic. In growing dimension the Gaussian simultaneous mean-error envelope contains $\sqrt{d_N}$, so the earlier fixed-$q$ replication schedule cannot simply be reused.

Here ``spectral dimension'' means the dimension inferred by \eqref{eq:growinggap}. It is not the algebraic rank of a finite covariance matrix or a heat-kernel dimension. The lemma does not extend the shrinking threshold in \eqref{eq:actualdimension} to growing ambient dimension. For example, a global uniform circle has intrinsic dimension one but covariance $I_2/2$, showing why localization is essential. The aspect ratio alone also permits full-dimensional Gaussian noise, the classical covariance setting of \citet{marchenko1967}; a random-matrix bulk limit supplies no local geometric premise. The code reports the gap estimate separately from the original threshold estimate and from certificate issuance.




\section{Two-sided Ahlfors regularity and positive reach}
\label{sec:ahlforsreach}
This section supplies a second geometric route. In this section only, $d$ denotes
ambient dimension and $p$ denotes intrinsic dimension, as in
Section~\ref{sec:growingdimension}. The hypotheses concern the true support,
not the fitted witness. No manifold chart or continuous sampling density is
assumed. The structural input is due to Federer and
\citet[Theorem~7.5]{rataj2017}; modern two-sided regularity is used in the sense of
\citet{davidsemmes1991}. The covariance argument complements established local
PCA analyses, including \citet{arias2017,lim2024} and
\citet[Lemma~10]{yaoxia2025}. The latter gives a lower eigenvalue bound for a
tangent-projected local sample moment matrix; it is not identical to the
centered population covariance below. Our role is to state an explicit
conditional combination, with a complete calculation in
Appendix~\ref{app:ahlforsreach}. We make no priority claim for its ingredients.

Let $\mu$ be a probability measure with nonempty compact support $E\subset\R^d$. Assume that for some $0<p<d$, $0<c\le C<\infty$, and $r_0>0$,
\begin{equation}
 c s^p\le\mu(\B(x,s))\le C s^p
 \quad(x\in E,\ 0<s\le r_0),\qquad \reach(E)\ge\tau>0.
 \label{ar:eq:model}
\end{equation}
Here $\tau$ is a finite positive lower bound, even when the actual reach is infinite. The mass bounds concern the \emph{sampling measure}. Regularity of $\mathcal H^p|_E$ alone would not prevent the sampling distribution from concentrating on a smaller subset. No smooth support, tangent plane, density continuity, or density Lipschitz constant is assumed in \eqref{ar:eq:model}.

For $x\in E$ and $0<r\le r_0$, let
\begin{equation}
 C_x(r)=\Cov(Z\mid Z\in\B(x,r)),\qquad Z\sim\mu.
 \label{ar:eq:population}
\end{equation}
All eigenvalues are in decreasing order. ``Spectral dimension'' below means a local covariance threshold count; it is neither algebraic rank nor heat-kernel spectral dimension.

\subsection{Geometric and spectral consequences}
\begin{theorem}[Geometric dimension and almost-everywhere local existence]
\label{ar:thm:geometry}
Under \eqref{ar:eq:model}, $p$ is an integer and $\dim_H E=p$. There is a relatively open subset $E_{\rm reg}\subset E$ that is an embedded $p$-dimensional $C^{1,1}$ manifold without boundary, with
\[
 \mu(E\setminus E_{\rm reg})=0,\qquad
 \dim_H(E\setminus E_{\rm reg})\le p-1.
\]
The structural assertion is a corollary of \citep[Theorem 7.5]{rataj2017}, with the dimension identities credited there to Federer. At every $x\in E$, including possible singular points, the linear span
\[
 T_x=\operatorname{span}\Tan(E,x)
\]
has dimension $p$. This span need not be a manifold tangent space at a singular point.
\end{theorem}
\par\noindent\textit{Detailed proof.} Appendix~\ref{app:argeometryproof} (p.~\pageref{app:argeometryproof}).

\begin{theorem}[A two-sided covariance bridge with explicit constants]
\label{ar:thm:bridge}
Under \eqref{ar:eq:model}, write
\begin{equation}
 \kappa=C/c,\quad K_p=4^p3^{p-1},\quad
 t_p=\frac{1}{2\kappa K_p},\quad a_p=\frac{t_p^2}{2}.
 \label{ar:eq:constants}
\end{equation}
For every $x\in E$ and $0<r\le\min(r_0,2t_p\tau)$,
\begin{align}
 a_p r^2&\le\lambda_p(C_x(r))\le\lambda_1(C_x(r))\le r^2,
 \label{ar:eq:lower}\\
 \sum_{j=p+1}^{d}\lambda_j(C_x(r))&\le\frac{r^4}{4\tau^2}.
 \label{ar:eq:tail}
\end{align}
In particular $\lambda_{p+1}(C_x(r))\le r^4/(4\tau^2)$. The tail bound holds for every $0<r\le r_0$, even without the restriction $r\le2t_p\tau$.
\end{theorem}
\par\noindent\textit{Detailed proof.} Appendix~\ref{app:arbridgeproof} (p.~\pageref{app:arbridgeproof}).

The constants are independent of ambient dimension, but are deliberately conservative and deteriorate rapidly with $p$ and $\kappa$. They establish a rigorous bridge, not a claim of improved empirical sample efficiency. Theorem~\ref{ar:thm:bridge} avoids the stronger approximation $r^{-2}C_x(r)\approx P_{T_x}/(p+2)$, which need not hold at every anchor under merely two-sided mass bounds.

\subsection{The estimator for \texorpdfstring{$S_N=X_NX_N^{\mathsf T}/N$}{S-N = X-N X-N-transpose / N}}
Take $N\ge2$ independent draws $Z_1,\ldots,Z_N$ from $\mu(\cdot\mid\B(x,r))$ and define
\begin{equation}
 X_N=[Z_1-\bar Z,\ldots,Z_N-\bar Z]\in\R^{d\times N},\qquad
 S_N=\frac{1}{N}X_NX_N^{\mathsf T}.
 \label{ar:eq:sample}
\end{equation}
Thus $N$ is the \emph{local} sample size and the data are centered. For $0<\delta<1$, set
\begin{align}
 u_N&=\log\frac{2(d+1)}{\delta},\qquad
 B_d(N,\delta)=3\sqrt{\frac{2u_N}{N}}+\frac{4u_N}{N},\label{ar:eq:bernstein}\\
 b_N&=\left(\frac{r}{2\tau}\right)^2+B_d(N,\delta),\qquad
 \widehat p_N=\#\{j:\lambda_j(S_N)>2b_Nr^2\}.
 \label{ar:eq:estimator}
\end{align}
This formula uses a justified reach lower bound, but not the unknown $p$ or the mass ratio $\kappa$.

\begin{corollary}[Finite-sample recovery and strong consistency]
\label{ar:cor:statistics}
Under Theorem~\ref{ar:thm:bridge}, if $3b_N<a_p$, then
\begin{equation}
 \Prob(\widehat p_N=p)\ge1-\delta.
 \label{ar:eq:finite}
\end{equation}
More generally, consider stage-wise models $E_N\subset\R^{d_N}$ satisfying \eqref{ar:eq:model} with the same fixed $p$, uniformly bounded $C_N/c_N$, and possibly varying $\tau_N,r_{0,N}$. At each stage take independent conditional draws as in \eqref{ar:eq:sample}, at any deterministic anchor in $E_N$ or at an independently chosen anchor. If
\[
 r_N\le r_{0,N},\qquad r_N\downarrow0,\qquad
 r_N/\tau_N\to0,\qquad d_N/N=O(1),
\]
then the estimator with $\delta_N=(N+1)^{-2}$ equals $p$ eventually almost surely on any common probability space carrying the samples. Consequently it converges to $p$ in probability. Independence across stages is unnecessary. The ratio condition may be weakened to $\log(d_N+1)=o(N)$.
\end{corollary}
\par\noindent\textit{Detailed proof.} Appendix~\ref{app:arstatisticsproof} (p.~\pageref{app:arstatisticsproof}).

\paragraph{An observable issuance rule.}
The separation condition in \eqref{ar:eq:finite} involves the unknown $p$. For a sufficient rule that does not insert the answer, supply external bounds $p\le p_{\max}<d$ and $\kappa\le\kappa_{\max}$, and use $t_*=1/(2\kappa_{\max}K_{p_{\max}})$ and $a_*=t_*^2/2$. The checks
\[
 r\le\min(r_0,2t_*\tau),\qquad 3b_N<a_*
\]
suffice for exact dimension issuance, conditional on the supplied model class. When they fail, the estimator may still be useful, but this guarantee is unresolved. Setting $p_{\max}=d-1$ is valid and can be prohibitively conservative.

\paragraph{Connection to all $n$ latent groups.}
For an independent anchor and $n$ iid latent observations, the oracle count has law $N\sim\mathrm{Bin}(n,\pi_x(r))$, with
\[
 cr^p\le\pi_x(r)\le Cr^p,\qquad
 \Prob\{N<\tfrac12ncr^p\}\le e^{-ncr^p/8}.
\]
The lower mass bound therefore controls whether local PCA has enough observations. It also gives support coverage. The upper bound, combined with reach, prevents the local conditional law from being concentrated in a thin slab inside a $p$-plane. Bounding thin annuli by $C((r+h)^p-(r-h)^p)$ does \emph{not} follow by subtracting two-sided Ahlfors bounds.

\paragraph{Repeated noisy observations.}
Suppose the existing procedure supplies a valid event of probability $1-\alpha$ on which
\[
 \norm{\widehat C-C_x(r)}_{\rm op}/r^2\le e.
\]
Replace $B_d(N,\delta)$ in \eqref{ar:eq:estimator} by $e$. The same separation argument gives exact recovery on that event whenever $r\le2t_p\tau$ and $3[(r/(2\tau))^2+e]<a_p$. If the available bound is infinite, the certificate remains unresolved. Alternatively a transfer bound from an observed covariance to $S_N$ can be added to $B_d$, with its exceptional probability added to $\delta$. These are conditional transfer statements, not a proof that noisy neighborhood selection satisfies the premises. In growing ambient dimension, repeated Gaussian means require control of a location envelope of order $\sigma(\sqrt{d_N}+\sqrt{\log n})/\sqrt{k}$; membership and counts require additional control. Fixed replication with positive Gaussian noise does not give a shrinking-ball latent-dimension guarantee.

\subsection{Consequences and scope}
\begin{enumerate}
\item The combination yields local $C^{1,1}$ existence at $\mu$-almost every point \emph{without an assumed manifold chart}. This is a credited structural consequence, not an empirical discovery.
\item It yields explicit finite-scale separation and a consistent covariance estimator without a continuous-density assumption or a uniform tangent-disk transport assumption.
\item It does not imply that the entire support is a boundaryless manifold. A uniform line segment is one-Ahlfors regular and has infinite reach; its endpoints are not boundaryless-manifold points.
\item It is an alternative model class to K\"afer's all-scale $C^1$ criterion \citep{kafer2019}. Positive reach restricts geometry more strongly in some respects, while allowing boundary points excluded by the preceding boundaryless local conclusion.
\item The reach in \eqref{ar:eq:model} is that of $E$, not of a fitted witness $M$. A circle with an arbitrarily short attached segment has zero reach, can satisfy uniform one-dimensional mass bounds, and is arbitrarily close to a positive-reach circle. Witness reach cannot be transferred to its support by Hausdorff closeness alone.
\item Finite observations do not verify the all-center, all-small-scale conditions in \eqref{ar:eq:model}. The resulting dimension certificate is conditional. These results do not retroactively certify the historical experiments or make their unresolved covariance bounds finite.
\end{enumerate}

\paragraph{A failure of the largest-gap shortcut.}
In a planar unit disk embedded in $\R^3$, use polar density
\[
 f(\rho,\theta)=\frac{1+(9/10)\cos(2\theta)}{\pi}
\]
with respect to area; its value at the origin is immaterial. The support is convex, so reach is infinite, and the density is bounded above and below by positive constants. For $s\le1$, a ball centered in the disk contains a disk of radius $s/2$, giving Ahlfors constants $c=1/40$ and $C=19/10$. At the origin, for every $r\le1$,
\[
 r^{-2}C_0(r)=\operatorname{diag}(29/80,11/80,0).
\]
The largest adjacent gap is $18/80$ at index one, not $11/80$ at index two. Thus replacing the preceding tangent-disk approximation by two-sided Ahlfors regularity requires a change of estimator: one cannot silently keep its largest-gap proof. Equation~\eqref{ar:eq:estimator} supplies a suitable threshold alternative. This does not contradict the preceding existing largest-gap lemma, which assumes convergence to an isotropic tangent-disk law.

\paragraph{Implementation.}
The numerical routine \texttt{code/\allowbreak reach\_\allowbreak dimension.py} implements the threshold
count from supplied eigenvalues, a reach lower bound, and a normalized covariance
error bound. Optional external bounds on dimension and mass ratio permit a
conditional separation check. It neither estimates true-support reach from
the witness nor verifies the Ahlfors assumptions. Floating-point decisions are
reported as numerical, not as rational-arithmetic certificates.


\section{Observation routes and end-to-end reconstruction bounds}
\label{sec:observationbridges}
The following five results make the interfaces in Section~\ref{sec:intro}
explicit. They are elementary consequences of pushforward measures,
Theorem~\ref{thm:main}, the matrix calculation in
Appendix~\ref{app:directembedding}, and a finite-step error recurrence.
They are not new Whitney or Takens embedding theorems. All displayed
constants concern the declared maps and domains. The true state support,
the observed support, and a nearby witness remain distinct objects.

\begin{proposition}[Fixed observation maps and independent delay groups]
\label{prop:observationinterface}
Let $S=\supp\nu$ be nonempty and compact in a Euclidean space, with
$\nu(\overline B(x,t))\ge c_0t^s$ for $x\in S$, $0<t\le t_0$.
Let $H:S\to\R^q$ be fixed and $C_H$-Lipschitz, $C_H>0$.
Then $E_H=\supp(H_\#\nu)=H(S)$, and
\[
 (H_\#\nu)(\overline B(y,t))\ge c_0C_H^{-s}t^s,
 \qquad y\in E_H,\quad 0<t\le C_Ht_0.
\]
This statement requires neither injectivity nor manifold structure.
For a deterministic $\Lambda$-Lipschitz map $F:S\to S$ and an
$L_h$-Lipschitz scalar observable $h$, the fixed delay map
$H(x)=(h(x),\ldots,h(F^{m-1}x))$ has the admissible upper bound
\[
 C_H=L_h\left(\sum_{j=0}^{m-1}\Lambda^{2j}\right)^{1/2}>0.
\]
Suppose $x_i\stackrel{\rm iid}{\sim}\nu$ and each group has $k\ge2$
independent measurements of the same deterministic latent path:
\[
 V_{i\ell,j}=h(F^jx_i)+\xi_{i\ell,j},\qquad
 \xi_{i\ell,j}\stackrel{\rm iid}{\sim}N(0,\sigma^2),
 \quad 0\le j<m.
\]
The noise is independent of all $x_i$. Then the grouped vectors $V_{i\ell}$
satisfy \eqref{eq:model} with $q=m$ and latent support $E_H$.
For $\sigma>0$, their pooled residual divided by $\sigma^2$ has law
$\chi^2_{Nm(k-1)}$ and is independent of all group means.
At $\sigma=0$ the noise and residual vanish. Theorem~\ref{thm:main}
therefore applies with the transported lower mass constants.
\end{proposition}
\par\noindent\textit{Detailed proof.}
Appendix~\ref{app:observationinterfaceproof} (p.~\pageref{app:observationinterfaceproof}).
The map, delay length and sampling interval must be fixed independently
of the certification data, or covered by a separate simultaneous argument.
Noise here is added after observation. Noisy dynamics, observation noise
shared across times, and overlapping windows from one orbit do not satisfy
the stated independent-group model automatically. Takens's state
identification remains an additional condition \citep{takens1981}.

\begin{theorem}[Verified witness coordinates and approximate decoding]
\label{thm:coordinatecomposition}
Under Theorem~\ref{thm:main}, suppose a support-emitting witness $M$ is
also verified to satisfy $M\subset c+\operatorname{ran}B$, with $B$
full column rank. For $A\in\R^{m\times q}$ and $\kappa,L>0$, verify
\[
 (AB)^{\mathsf T}AB-\kappa^2B^{\mathsf T}B\succ0,
 \qquad L^2I_m-AA^{\mathsf T}\succeq0.
\]
Write $f(x)=A(x-c)$ and $\epsilon=\widehat\epsilon$. The map $f|_M$
is an embedding with chord bounds $\kappa,L$. With probability at least
$1-\delta$, every end-to-end output on the common support event obeys
\[
 d_H(f(E),f(M))\le L\epsilon.
\]
For every $z\in E$, every recorded coordinate $u$ with
$\|u-f(z)\|\le b$, and every verified $p\in M$ with
$\|f(p)-u\|\le\dist(u,f(M))+\rho$, one has
\[
 \|p-z\|\le\epsilon+\frac{2(L\epsilon+b)+\rho}{\kappa}.
\]
For every $z,z'\in E$ one also has the finite-resolution bound
\[
 \|f(z)-f(z')\|\ge
 \kappa\|z-z'\|-2(\kappa+L)\epsilon.
\]
These statements hold simultaneously for data-dependent verified
candidates and maps on that event. Equivalently, the probability of
issuance with any of these geometric conclusions false is at most $\delta$.
\end{theorem}
\par\noindent\textit{Detailed proof.}
Appendix~\ref{app:coordinatecompositionproof} (p.~\pageref{app:coordinatecompositionproof}).
Exact closest-point computation is unnecessary: a verified $\gamma$-net
$G\subset M$ and minimization of $\|f(g)-u\|$ over $G$ give $\rho=L\gamma$.
An additional recording error $e_p$ for $p$ adds $e_p$ to its error bound.
The last inequality separates sufficiently distant support points;
it does not prove injectivity on $E$ when $\epsilon>0$.
The implemented $m=2d+1$ selector still requires
$\operatorname{rank}B\le2d+1$; failure of this restricted check is not
a failure of Whitney's general existence theorem \citep{whitney1936,whitney1944}.

\begin{corollary}[When witness coordinates are true state coordinates]
\label{cor:statecoordinates}
Let $S\subset\R^p$ be a compact $C^1$ state manifold, and let $H:S\to\R^q$ be a
$C^1$ embedding with externally justified chord bounds
\[
 a_H\|x-x'\|\le\|H(x)-H(x')\|\le C_H\|x-x'\|,
 \qquad a_H>0.
\]
If additionally $H(S)=M$ exactly and the matrix checks of
Theorem~\ref{thm:coordinatecomposition} hold, then
$\Psi=f\circ H:S\to K=f(M)$ is a $C^1$ embedding of dimension
$\dim S=\dim M$, with chord bounds $a=\kappa a_H$ and $b_\Psi=LC_H$.
For any continuous evolution map $F:S\to S$,
\[
 G=\Psi\circ F\circ\Psi^{-1},\qquad G^n\Psi=\Psi F^n.
\]
If $u$ is within $b$ of $\Psi(x)$ and a witness decoder has coordinate
residual slack $\rho$, then
\[
 \|H^{-1}(p)-x\|\le\frac{2b+\rho}{\kappa a_H}.
\]
\end{corollary}
\par\noindent\textit{Detailed proof.}
Appendix~\ref{app:statecoordinatesproof} (p.~\pageref{app:statecoordinatesproof}).
The equality $H(S)=M$, smooth state embedding and lower chord constant
are additional hypotheses; finite nearby-witness issuance establishes
none of them. For scalar inputs Takens may justify an embedding, but
neither its genericity nor dimension counting supplies a numerical
lower stability constant for a chosen sensor. When $\epsilon>0$,
Theorem~\ref{thm:coordinatecomposition} supplies approximation, not this
exact conjugacy or an implementable state inverse.

\begin{proposition}[Transition validation and finite-horizon prediction]
\label{prop:transitionbounds}
Let $\Psi:S\to K$ be a true state embedding with chord bounds
$a,b_\Psi>0$. Suppose $F:S\to S$ is $L_F$-Lipschitz. Then
$G=\Psi F\Psi^{-1}$ is $L_G$-Lipschitz for
$L_G=b_\Psi L_F/a$. Let $\widehat G$ be defined on a set containing
$K$ and all validation input records, with verified Lipschitz bound
$L_{\widehat G}$ on that set and $\widehat G(K)\subset K$.
Suppose the true validation inputs $v_i=\Psi(x_i)$ form a $\zeta$-net of
$K$, and recorded ordered pairs $\widetilde v_i,\widetilde w_i$ satisfy
\[
 \|\widetilde v_i-v_i\|\le\eta_-,\qquad
 \|\widetilde w_i-G(v_i)\|\le\eta_+,
 \qquad
 r_v\ge\max_i\|\widehat G(\widetilde v_i)-\widetilde w_i\|.
\]
Then the uniform one-step error is bounded by
\[
 \sup_{v\in K}\|\widehat G(v)-G(v)\|
 \le r:=(L_{\widehat G}+L_G)\zeta
          +L_{\widehat G}\eta_-+r_v+\eta_+.
\]
For $u_0\in K$, $v_0=\Psi(x_0)$ with $\|u_0-v_0\|\le e_0$, put
$u_{t+1}=\widehat G(u_t)$ and $v_t=G^tv_0$. For every integer $t\ge0$,
\[
 \|u_t-v_t\|\le L_G^te_0+r\sum_{j=0}^{t-1}L_G^j,
 \qquad
 \|\Psi^{-1}(u_t)-F^tx_0\|
 \le\frac{L_G^te_0+r\sum_{j=0}^{t-1}L_G^j}{a}.
\]
The empty sum is zero and $L_G^0=1$, including when $L_G=0$.
\end{proposition}
\par\noindent\textit{Detailed proof.}
Appendix~\ref{app:transitionboundsproof} (p.~\pageref{app:transitionboundsproof}).
The net coverage, observation-error bounds, true embedding constants,
evolution regularity and invariant-domain check are substantive inputs.
If the coverage and pair-error event has failure probability at most
$\delta_T$, the transition conclusions inherit that budget. Its union
with the geometric support event has failure probability at most
$\delta+\delta_T$, without requiring event independence. A small training
residual alone supplies neither net coverage nor a uniform forecast bound.
For $L_G>1$ the bound grows geometrically: coordinates can preserve states
while long-horizon trajectory prediction remains difficult.

\begin{corollary}[Selection from a fixed finite observation bank]
\label{cor:protocolbank}
Fix $J$ observation protocols before the certification data are used.
For protocol $j$, suppose the assumptions of Theorem~\ref{thm:main}
hold with its own observation dimension, mass constants and budget
$\delta_j>0$. Its sound geometric and matrix verifiers use its common
support event as in Theorem~\ref{thm:coordinatecomposition}.
Let $\mathcal B_j$ denote issuance with a false geometric coordinate
or conditional decoding conclusion for protocol $j$.
If $\sum_{j=1}^J\delta_j\le\delta_*$, then
\[
 \Prob\left(\bigcup_{j=1}^J\mathcal B_j\right)
 \le\sum_{j=1}^J\delta_j\le\delta_*.
\]
Consequently any measurable data-dependent selection among these
end-to-end outputs obeys the same budget. Protocols may share raw
observations; independence between their confidence events is unnecessary.
\end{corollary}
\par\noindent\textit{Detailed proof.}
Appendix~\ref{app:protocolbankproof} (p.~\pageref{app:protocolbankproof}).
This is a union-bound composition, not a new selection principle. A finite
bank can cover data-selected delay lengths and lags when each fixed
protocol meets Proposition~\ref{prop:observationinterface}. It does not
cover an unregistered new sensor, unrestricted stopping over fresh data,
or state-embedding hypotheses inferred from witness dimension.

\section{Certification algorithm and Python implementation}\label{sec:algorithm}
The implementation turns the preceding guarantees into an observation-to-certificate procedure. The input is an NPZ file of grouped repeats $X\in\R^{N\times k\times q}$, or recorded means and residual statistics, together with a JSON request specifying the family, tolerances, recording-error contract, and optional lower mass model. The command-line program \texttt{mh\_\allowbreak certificate.py} selects exactly one candidate using the training prefix and emits a machine-readable decision record. Appendix~\ref{app:codes} gives the core routines, their theoretical roles, and replay commands.

\paragraph{Observation interface and certified target.}
For vector inputs the certified target is a witness near the observed
support; identifying that support with a state space requires a justified
observation embedding. For scalar inputs, forming delay vectors first
defines an observation-space dataset, but does not verify that its delay
map embeds the true state space. Takens or another observability argument
must supply that separate justification. In either route the grouped
repeats, noise law, independence and coverage assumptions must hold for
the actual vector inputs. Overlapping delay windows from one noisy orbit
are dependent and cannot be passed through the present iid calibration
without a new sampling and noise argument. Independent repeated
experiments may provide suitable inputs when these assumptions are
justified. No certified scalar-delay experiment is implemented here.

\begin{algorithm}[Verified witness]\label{alg:pipeline}
\normalfont
\begin{enumerate}
\setlength{\itemsep}{0pt}
\item Validate finite observations and the fixed request. For raw repeats, compute exact recorded moments and outward recording-error bounds.
\item Fit one rational-parameter candidate $M$ from training means only. Freeze it before validation.
\item Calibrate $a$ from the pooled residual. If calibration is unavailable, return $\ABSTAIN$ with a reason.
\item Verify smooth candidate geometry, reach $\widehat\tau$, and both fit directions using rational nets. Set $e=\max(e_+,e_-)$.
\item If a lower mass model is supplied, compute $h$ and $\widehat\epsilon=h+a+e$. Issue the support certificate only if its error and reach tolerances pass.
\item On held-out means, compute the tube-mass bound. Issue this separate certificate if its mass tolerance passes; otherwise return $\ABSTAIN$. Save exact bounds, assumptions, and decision reasons.
\item For the restricted direct route, after support issuance and verified affine containment, construct and verify a rational $2d+1$ coordinate map. Failed separation or span checks withhold the end-to-end output.
\item Encode rational witness points and decode with the verified span inverse, including recording-error bounds. The executable \texttt{rebuild\_direct\_whitney.py} implements these final stages; reconstructions after support abstention are labeled diagnostic only.
\end{enumerate}
\end{algorithm}

\subsection{Observation preparation and optional transition verification}
\label{sec:observationroutealgorithm}
Algorithm~\ref{alg:observationroute} adds the observation interface to
Algorithm~\ref{alg:pipeline}. Its executable preparation stage is
\texttt{code/\allowbreak prepare\_\allowbreak observation\_\allowbreak route.py}; it selects entries
and preserves their dtype without whitening, interpolation or denoising.
It checks array structure and the declared protocol, not physical
independence or the truth of the Gaussian observation model.

\begin{algorithm}[Two observation routes]\label{alg:observationroute}
\normalfont
\begin{enumerate}\setlength{\itemsep}{0pt}
\item Fix the route, sensor map, delay length and sampling interval before
using the certification data, or predeclare a finite bank and allocate
budgets $\delta_j$ with $\sum_j\delta_j\le\delta_*$.
Declare the independent-group and repeated
measurement model, external mass constants and recording-error contract.
\item For vector observations, retain the grouped $N\times k\times q$
array. For scalar measurements from independent initial-state groups,
select the same $m$ distinct times in each of the $k$ measured copies of
each latent path, obtaining $N\times k\times m$ delay vectors.
\item Reject single-series inputs for this iid route. Overlapping or
disjoint windows of one orbit require a separately proved dependence
calibration; merely choosing disjoint windows does not establish independence.
\item Apply Proposition~\ref{prop:observationinterface} when its premises
are justified. Run Algorithm~\ref{alg:pipeline} with a sound verifier;
abstain if support, witness geometry, continuum span containment or
coordinate checks fail. Training can propose several witnesses against
one common support event. Across fixed observation protocols use
Corollary~\ref{cor:protocolbank}; an unregistered sensor remains uncalibrated.
\item Return the verified witness and coordinate bounds of
Theorem~\ref{thm:coordinatecomposition}. For nearest-net decoding include
the net radius and all recording errors in the stated reconstruction bound.
\item Report true-state coordinates or conjugacy only when the additional
premises of Corollary~\ref{cor:statecoordinates} have been justified.
Otherwise report the observation-space witness and its coordinates.
\end{enumerate}
\end{algorithm}

The scalar preparation code is executable, but its structural output is
not an issued geometric or Takens certificate. The executed torus route
uses \texttt{rebuild\_direct\_whitney.py}; arbitrary scalar-delay supports
may require witness families beyond those implemented here. Existing
independent-snapshot experiments are not relabeled as ordered transitions.
Without justified lower mass constants, Theorem~\ref{thm:tube} supplies
only a probability-mass assertion, whose coordinate transport is given
in Appendix~\ref{app:directembedding}. It does not supply the uniform
support radius needed for Theorem~\ref{thm:coordinatecomposition}'s
worst-case decoding bound.

\begin{algorithm}[Conditional transition and forecast bounds]
\label{alg:transitionvalidation}
\normalfont
\begin{enumerate}\setlength{\itemsep}{0pt}
\item Establish a true state embedding $\Psi$, lower and upper chord
bounds $a,b_\Psi$, a bound $L_F$ for the actual evolution, and the
coordinate domain $K$. Nearby-witness issuance alone cannot supply them.
\item Preserve ordered transition pairs. Justify their input coverage
$\zeta$ and simultaneous record errors $\eta_-,\eta_+$ under an explicit
sampling model with its own probability budget.
\item Fit a candidate $\widehat G$. Verify its Lipschitz bound on the
required domain, its residual bound $r_v$, and $\widehat G(K)\subset K$.
If a prerequisite is unavailable, withhold the uniform dynamical claim.
\item Compute $L_G=b_\Psi L_F/a$ and $r$ from
Proposition~\ref{prop:transitionbounds}. Propagate the error by
$e_{t+1}=L_Ge_t+r$ for the declared finite horizon; return $e_t$ and $e_t/a$.
\item Keep each statistical budget and deterministic verification input
in the record. The exact arithmetic utility
\texttt{code/\allowbreak reconstruction\_bounds.py} computes conditional
radii from supplied rational bounds; it does not verify these premises
or learn an unknown evolution map.
\end{enumerate}
\end{algorithm}

For the general geometric route, a finite grid can upper-bound \eqref{eq:bilateral}. Write $Y_y=Y\cap\overline B(y,r)$ and $D_y=L_y\cap\overline B(y,r)$. For a certified $\kappa$-net $G_y$ of $D_y$, set
\begin{align}
b_+(y)&=\max_{p\in Y_y}\dist(p,L_y),\nonumber\\
b_-(y)&=\max_{g\in G_y}\dist(g,Y_y)+\kappa.
\label{eq:netcheck}
\end{align}
Then $d_H(Y_y,D_y)\le\max\{b_+(y),b_-(y)\}$. The orthogonal projection of a point in $Y_y$ remains in $D_y$, since $L_y$ passes through $y$. The reverse direction follows from the 1-Lipschitz distance function. A grid on $[-r,r]^d$ of maximum spacing $\ell$, projected to the closed ball, gives $\kappa=\sqrt d\ell/2$. This local-disk construction is theoretical here; the executable verifier below instead checks global nets of explicit witnesses. All local anchors would need checking unless a transfer argument covers omitted anchors. Exact membership and orthogonality, or certified enclosures for them, are necessary.

\subsection{Geometric verification and implemented families}
\label{sec:families}
The new planar proposals are a union of two circular arcs and a segment, each with boundary. Single linkage and least squares propose two arcs; training SVD proposes a segment. Their center, direction, size and endpoints are fitted from the training prefix, without true labels or pose. The verifier accepts an arc only in the rational parameterization
\begin{equation}
F(p)=c+RD\bigl((1-p^2)/(1+p^2),2p/(1+p^2)\bigr),
\end{equation}
with $R>0$, exact rational rotation $D$, and $-1\le p_0<p_1\le1$. Its reach is $R$. A finite net verifies positive component separation $\underline\Delta$, giving union reach at least $\min(R_1,R_2,\underline\Delta/2)$. Segment projection is exact and its reach is infinite. The old fixed-pose circle and ring-torus families remain as appendix controls. Detailed proofs, distance formulas and net radii appear in Appendices~\ref{app:boundaryfractal} and~\ref{app:geometry}.

With a global witness net $G$ of radius $\kappa$, we verify
\begin{equation}
e=\max\left\{\max_i\overline d(Y_i,M),
\max_{g\in G}\overline d(g,Y)+\kappa\right\}.
\label{eq:globalfit}
\end{equation}
The distance to any selected sample point is an upper bound on distance to $Y$. Therefore a floating-point nearest-neighbor routine may suggest an index without compromising soundness: its selected distance is recomputed with exact rational arithmetic. A poor suggestion can only make the bound more conservative.

\subsection{Observation processing and decision records}
The raw-input path computes exact integer moments of the recorded binary coordinates. It converts the resulting means and residual sum to stored floating-point values while tracking the conversion error exactly. Given an external per-coordinate input-error bound, a residual-projection inequality propagates it to $e_Y,e_S$; Appendix~\ref{app:rawprocessing} gives the derivation. The program does not infer whether a sensor satisfies that contract or whether the Gaussian model is correct. Training records and preprocessing must not use validation latent points or observations; input ordering, the split, and requested tolerances are fixed before validation. A bound on recording-error magnitude alone does not imply this independence.

The request may omit the lower mass model. In that case only support issuance is disabled; the independent tube-mass test remains available. Unavailable calibration, invalid candidate geometry, failed coverage, and exceeded tolerances have distinct abstention reasons. Malformed input is rejected as an input error. Records contain exact rational bounds, probability budgets, the candidate, the training split, input/code hashes, and the explicit value \texttt{exact\_\allowbreak support\_\allowbreak manifold\_\allowbreak certified=false}. Hashes support replay, not a statistical guarantee. Plotting is a separate program and cannot feed reference geometry into verification.

\subsection{Rational calibration and cost}\label{sec:arithmetic}
Our implementation avoids uncertified quantile and exponential evaluations. For $\chi^2_\nu$ with mean $\nu$ and variance $2\nu$, Cantelli's inequality gives the admissible lower threshold
\begin{equation}
\ell_C=\nu-\left\lceil\sqrt{2\nu(1-\delta_C)/\delta_C}\right\rceil.
\label{eq:cantelli}
\end{equation}
If $\ell_C\le0$, the routine abstains. Given recorded $\widetilde S$ with known absolute error $e_S$, use $U=(\max\{0,\widetilde S+e_S\})/\ell_C$. Choose an integer $m$ with $N2^{\lceil q/2\rceil-m}\le\delta_L$ and use
\begin{equation}
a=\overline{\sqrt{4mU/k}}+qe_Y,
\label{eq:rationalnoise}
\end{equation}
where $e_Y$ bounds the coordinate recording error and the bar denotes an outward dyadic square root. The Gaussian moment-generating function proves this bound. For integer $s$, replace the exponential in \eqref{eq:coverage} by $2^{-\lfloor Nc_0(h/2)^s\rfloor}$ and search a rational grid. A similarly conservative rational DKW slack is used. Proofs and all constants are in Appendix~\ref{app:arithmetic}.

For $G$ net points, the distance checks require $O((N+|G|)q)$ rational arithmetic operations after neighbor indices are proposed; brute-force index selection costs $O(N|G|q)$ and has a simple worst-case bound. Storage is $O((N+|G|)q)$. The experiment uses a KD tree for speed without assuming a better worst-case bound. For two arcs, the extra separation check uses $O(|G_1||G_2|q)$ exact operations. The training single-linkage proposal stores quadratically many pair distances. These are arithmetic-operation counts; rational bit costs also depend on input precision. General local disk grids grow exponentially with $d$, and we make no high-dimensional scalability claim.

\section{Executed experiments and visualization}\label{sec:experiments}
\paragraph{Experimental target and prior work.}
\citet[Section~5]{yaoxia2025} already demonstrate circle, sphere and torus fitting and facial-image denoising. Visual resemblance to a smooth shape or recognizable image is therefore insufficient to distinguish the present contribution. We evaluate the issued assertion, abstention, reported error bound and calibration cost. The geometric examples below exercise specific support, tube-mass and coordinate assertions; none replaces the assumptions of the intrinsic-dimension consistency theorem. Section~\ref{sec:mnistvisual} retains a separate reconstruction diagnostic, with the full auxiliary image audits in Appendices~\ref{app:matchedmnist} and~\ref{app:faceaudit}. These finite comparisons do not establish general performance superiority or algorithmic equivalence.

The accompanying numerical supplement retains the four exact main-text tables replaced or consolidated into Figures~\ref{fig:directbounds}, \ref{fig:mfcurves}, \ref{fig:juliadirect} and~\ref{fig:dimensionchecks}. Its values and trial denominators are unchanged.

\subsection{Dataset to certification to Whitney: an executed direct route}
\label{sec:directexperiment}
The revised executable connects proposal, certificate issuance, coordinate
construction and witness reconstruction. We rerun the same 60 archived
observation files: $N=16{,}384$ independent snapshots in $\R^{16}$,
$k=4,64,256$ repeats represented by their sufficient-statistic law, and
unknown isotropic Gaussian observation variance. The supplied rational
sensor basis has $B^{\mathsf T}B=I_3$. The verifier receives recorded
means, a residual sum, the external lower mass constants
$(c_0,s,t_0)=(1/100,2,1)$ and the primary request
$(\epsilon_*,\tau_*)=(1.2,0.25)$. Latent references and generator
variance remain outside the execution input.

\paragraph{The complete implemented sequence.}
Training means propose a circle and a ring torus. A robust geometric
least-squares proposal estimates their centers and radii, rounded to
rational denominator $65536$. Optimization is a proposal, not a proof
of a global optimum. Rational parameterizations verify smooth geometry,
reach, dimension and affine containment; a $J=32$ stereographic net
verifies both approximation directions. This doubles the original
angular resolution. The calibrated confidence event is unchanged.
Among passing candidates, choose the smallest verified witness dimension
$d$. If none passes, withhold the end-to-end output. This finite family
selection does not certify the true support dimension.

For an accepted witness $M\subset\widehat c+\operatorname{im}(B)$,
Appendix~\ref{app:directembedding} constructs and verifies a rational
selector $A:\R^{16}\to\R^{2d+1}$, with $L=1$ and a strictly positive
lower Gram margin. For the executed torus, write $A=(P;0)$, where $P$
selects ambient rows $0,1,2$; two zero rows are appended and
$\kappa^2=1/144$. With $S_B=PB$, the exact decoder
$D=BS_B^{-1}[I_3\;0]$ satisfies $DAB=B$.
Proposed nearest-point phases are replaced by exact rational circle
parameters, giving a point $p_i\in M$. Encode it as
$f(p_i)=A(p_i-\widehat c)$ and decode with $\widehat c+Df(p_i)$.
The saved arrays track the final grid recording error. Consequently
\[
 d_H(A(E-\widehat c),f(M))\le\widehat\epsilon,
 \qquad \dist(\widetilde p_i,E)\le\widehat\epsilon+e_{\rm enc}
\]
on the same support-confidence event. Appendix~\ref{app:rebuiltcalc}
gives the exact decoder, rounding bounds and a simultaneous
sample-reconstruction bound. Floating-point nearest-point optimality
is not part of the issued certificate.

\begin{figure}[p]\centering
\includegraphics[width=.98\textwidth,height=.72\textheight,keepaspectratio]{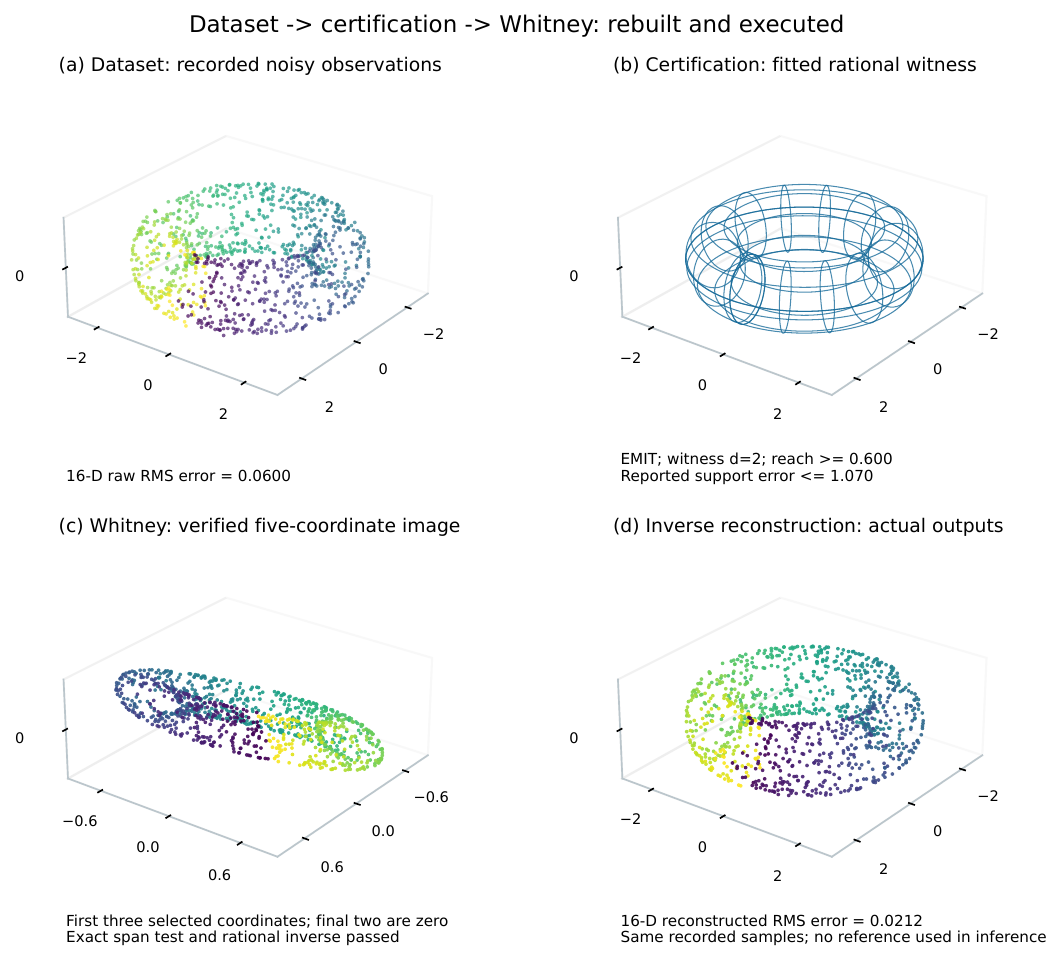}
\caption{Rebuilt Dataset $\to$ certification $\to$ Whitney algorithm, executed on the first fixed torus trial at $k=64$. (a) Recorded observations displayed in the three supplied sensor coordinates. (b) The fitted rational witness and its verified bounds. (c) The actual first three selected Whitney coordinates; the last two are zero. (d) Actual witness reconstruction outputs, decoded through the verified coordinate map. Panels (a), (c) and (d) use the same fixed sample indices and evaluation-only phase colors. References do not enter inference. The display uses vector graphics and the executed arrays, not an idealized reconstruction substituted for them.}
\label{fig:directpipeline}
\end{figure}

\paragraph{Paired results and a necessary reconstruction control.}
Figure~\ref{fig:directbounds} compares the first implementation with the
revision on identical observations and requests. The revision was designed
after inspecting the first results; it is a paired development experiment,
not an independent confirmation set. Both implementations abstain in all
$k=4$ trials. The torus emits in all ten trials at $k=64,256$.
Its mean reported bounds decrease from $1.1704$ to $1.0711$ and from
$1.0772$ to $0.9804$, respectively. Witness reconstruction decreases the
full-state RMS error relative to the earlier span inverse by $18.4\%$
and $18.5\%$. The difference comes from fitting and projection onto the
witness, not from increasing the image resolution or adding zero
coordinates. The $k=4$ reconstruction numbers are diagnostics after
abstention.

The second support is a $0.9/0.1$ mixture of a torus and an attached
segment of length $0.15$. It is not a manifold at the attachment, but
satisfies the same proved lower mass condition. The sharper approximation
bounds now permit all ten $k=64,256$ trials to emit at tolerance $1.2$.
This certifies a nearby smooth witness, including for a singular support.
It does not certify that the branch is absent. Projection onto the torus
erases this component: RMS error increases from $0.0260$ to $0.0314$
and from $0.0130$ to $0.0266$. Thus a clearer picture need not represent
a more accurate reconstruction. The first implementation's additional,
post-inspection tolerance-$1.5$ request and all original decisions remain
in the archived records.

\begin{figure}[!t]\centering
\includegraphics[width=.98\textwidth]{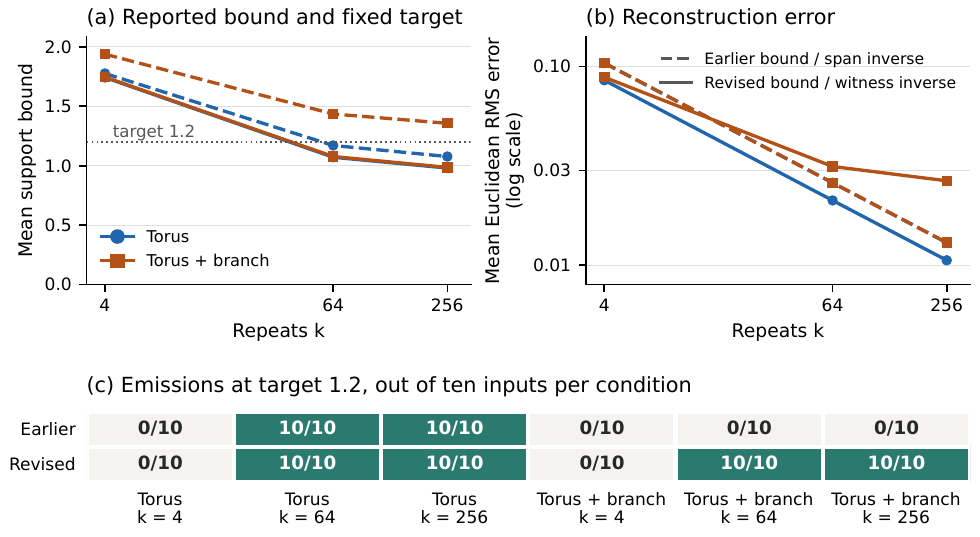}
\caption{Paired results for all 60 unchanged observation files. (a) Mean reported bounds and the fixed support target $1.2$; the vertical axis starts at zero. (b) Mean Euclidean RMS errors on an explicitly logarithmic scale: the earlier span inverse (dashed) and witness/Whitney inverse (solid). Color and marker identify each support. Branch projection worsens reconstruction at $k=64,256$. (c) Every earlier and revised emission count, including all $k=4$ abstentions; those inverses are diagnostics after abstention. These are rounded table means, without inferred uncertainty bars. Exact values are retained in the numerical supplement, Table~S1.}
\label{fig:directbounds}
\end{figure}

The earlier exact map controls are preserved: the three-row isometry and
its five- or seven-row padded versions pass; dropping a separating row
fails even when zero padding raises the declared output count to five.
All fitted circle candidates in the revision fail the primary support
request. The restricted construction still uses a supplied rational
three-dimensional span; it is not a general implementation of Whitney's
optimal $2d$ embedding theorem. It uses no Takens premise or learned
dynamics. Dynamical systems remain the separate application below.
All 60 reference Hausdorff enclosures lie below their reported bounds,
and replays check both decision records and reconstructed arrays.
These finite comparisons do not estimate zero population risk.
\FloatBarrier

\subsection{Dynamical systems and chaos: reconstruction diagnostics}\label{sec:arxivvisuals}
The imported Lorenz, R\"ossler, Swiss-roll and KS studies of
\citet{choi2026} use the mean-crossing point rank in
\eqref{eq:termsdiagnosticrank} followed by local-PCA projection
\eqref{eq:termsprojection}. Figure~\ref{fig:odeapplication} restores the
observable input/output comparison for the two named ODE systems.
These are independently initialized finite-time endpoints, not a
single time-ordered trajectory. The recognizable clouds help locate
the reconstruction task, but visual similarity is insufficient:
all 24 archived point/ODE output-error medians exceed their reported
input-error medians. Figure~\ref{fig:pointodeerrors} displays every
condition and its output count; numerical supplement Table~S5 keeps
the exact scores and rank counts. They issue no geometric certificate,
and the finite population-spectral gate remains unresolved.
Appendix~\ref{app:arxivbench} gives the sampling protocols and replay;
the complete original figures and arrays remain in the separate full
source/data archive.
These full-coordinate, independently initialized endpoint studies
neither execute a verified Whitney map nor establish invariant-attractor
dimension or the assumptions of scalar-delay reconstruction.

\begin{figure}[p]\centering
\includegraphics[width=.98\textwidth]{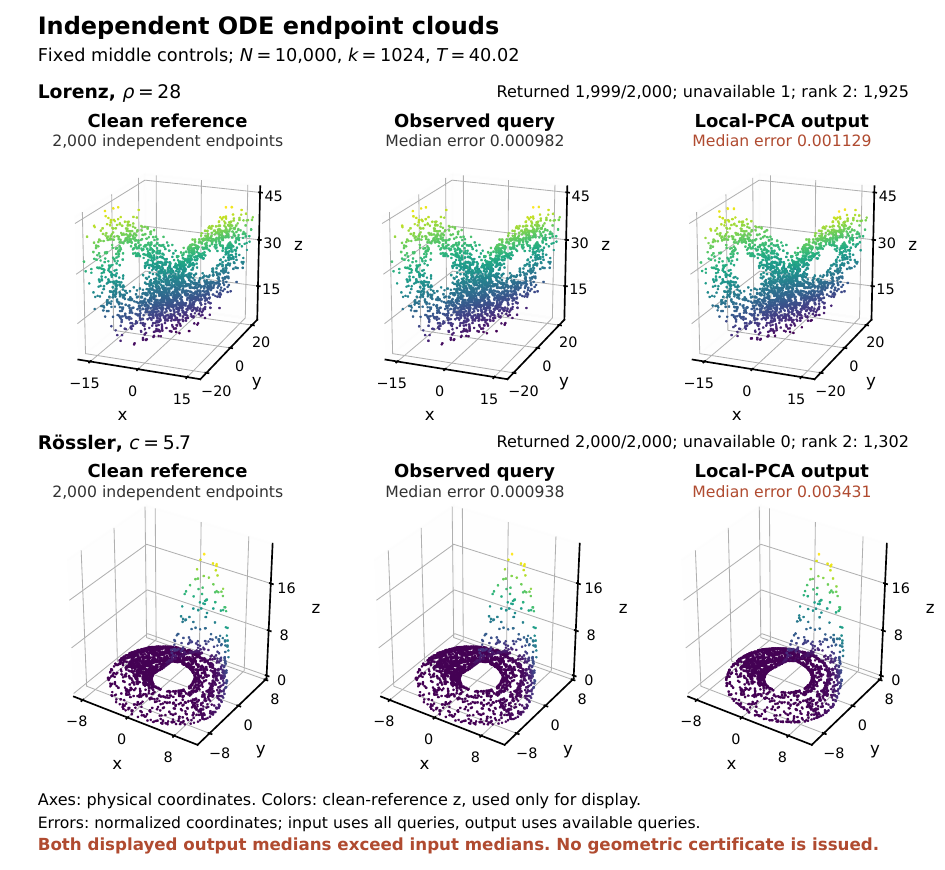}
\caption{Lorenz and R\"ossler endpoint reconstruction diagnostics from the archived inputs. The middle controls, $\rho=28$ and $c=5.7$, and the largest stored pilot prefix $N=10{,}000$ are a deterministic display choice, not a new test or a selection by error. Left: all 2,000 clean reference endpoints. Middle: their noisy observed averages. Right: available local-PCA projections of the same queries, with unavailable outputs explicitly counted. Each endpoint has an independent initialization and final time $40.02$; no lines assert time ordering. Axes are physical coordinates; error scores divide Lorenz coordinates by 30 and R\"ossler coordinates by 15. Clean coordinates are used only for evaluation and color. The projection is less accurate than the observed input on the reported median-error comparison in both displayed cases and every archived condition (Figure~\ref{fig:pointodeerrors}). These outputs are auxiliary reconstructions, not manifold, attractor-dimension, Whitney or Takens certificates.}
\label{fig:odeapplication}
\end{figure}
\FloatBarrier

\subsection{Certificate issuance and dimension point estimates}\label{sec:certificateexamples}
\paragraph{Fitted witnesses and nonsmooth support.}
Two Moons exercises the support and tube-mass guarantees of Theorems~\ref{thm:main} and~\ref{thm:tube} through the boundary extension in Corollary~\ref{cor:boundary}: the fitted witness consists of two disjoint arcs with unknown pose and endpoints. The independent-repeat model, supplied lower-mass contract $(c_0,s,t_0)=(1/16,1,1/4)$, and rationally verified reach are the relevant assumptions; they are not inferred from visual appearance. Figure~\ref{fig:moonsapplication} displays the saved observations and fitted witnesses, including the high-noise abstention. Figure~\ref{fig:mfcurves} reports every condition's issuance rates.

The Julia and Cantor descriptions refer to the same experiment on $J(z^2-3)$, a disconnected subset of the real axis. Its fitted segment tests a different consequence: a valid coarse support certificate can fill persistent gaps and need not identify exact support regularity or dimension. Figures~\ref{fig:juliaactualcontext} and~\ref{fig:juliadirect} make this distinction explicit. Reference coordinates and labels enter only evaluation and color. The singular circle-with-branch study below separately isolates calibration at matched targets.

\paragraph{Design and comparisons.}
Four Two Moons noise levels and three Julia replication levels each have 30 independent trials: 210 saved inputs and 840 method comparisons. We simulate the joint Gaussian sufficient-statistic law, not individually stored raw repeats. On the same candidate, split, verified geometry and targets we compare uncalibrated plug-in (P), known-variance numerical reference (O), unknown-variance numerical reference (Q), and the actual rational verifier (R). Only O receives the generator variance. P/O/Q are numerical diagnostics, not machine-verified certificates. All use a fixed latent tube radius $0.12$ and outside-mass tolerance $0.08$. The support tolerances are $0.3$ for Moons and $1.3$ for Julia, with reach threshold $0.12$.

\begin{figure}[p]\centering
\includegraphics[width=.98\textwidth]{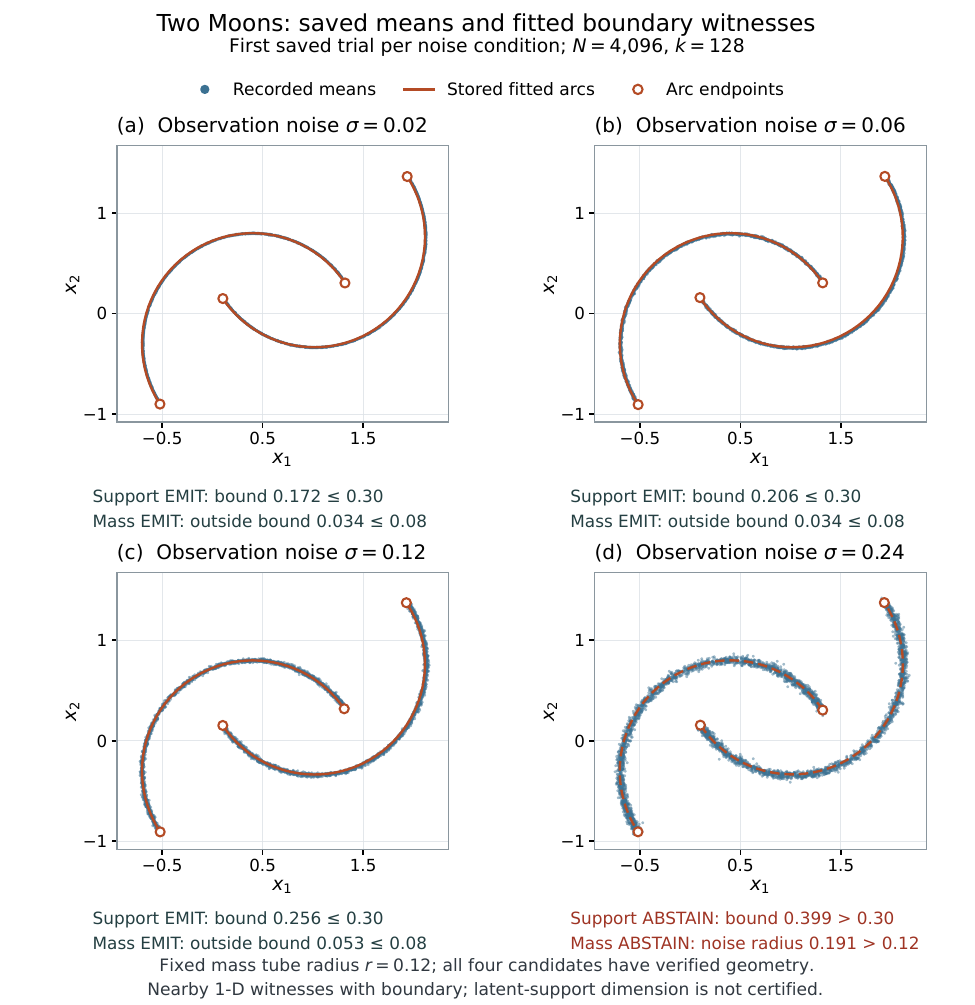}
\caption{Two Moons as a verified boundary-witness example. Each panel uses all 4,096 recorded means and the rationally stored fitted arcs from trial 0 of one predeclared noise condition, with $k=128$ repeats. The four panels preserve the condition order and use common axes; no trial was selected by fit quality. The curves are the actual training-fitted witnesses, not the clean generator arcs. Solid curves accompany support and tube-mass issuance; dashed curves show the candidate retained after abstention at the highest noise level. All-condition counts remain in Figure~\ref{fig:mfcurves}. Independent sampling, the supplied lower-mass bound and the exact geometry checks justify applying Corollary~\ref{cor:boundary}; visual resemblance alone does not certify latent dimension.}
\label{fig:moonsapplication}
\end{figure}



\paragraph{Executed outcomes and limitations.}
For Moons, R emits support and mass certificates in all 30 trials at each of $\sigma=.02,.06,.12$, and abstains in all trials at $.24$. At the last level Q emits 23/30 support decisions, so conservatism remains material. For Julia, R emits the coarse support certificate in all trials, but the same verified bounds pass neither of the predeclared stricter tolerances $0.6$ and $0.9$. At tolerance $0.9$, the reference bounds already establish acceptable proximity in all trials, so abstention reflects conservatism. Tube-mass issuance is 0/30 at $k=16$ and 30/30 at $k=64,256$.
The uncalibrated Julia rule underestimates its reported support error in 30/30, 30/30, and at least 6/30 trials, respectively; the other 24 cases at $k=256$ remain unresolved by the evaluation enclosure. These are errors in the reported bound, not false assertions that the loose target $1.3$ is satisfied. No emitted R bound or target fails or remains unresolved in the rational reference assessment. Zero failures among 30 trials has one-sided 95\% upper limit $9.50\%$ per condition, not zero population risk or a pooled guarantee. Pointwise binomial intervals follow \citet{clopper1934}.

\begin{figure}[!t]
\centering
\includegraphics[width=.98\textwidth]{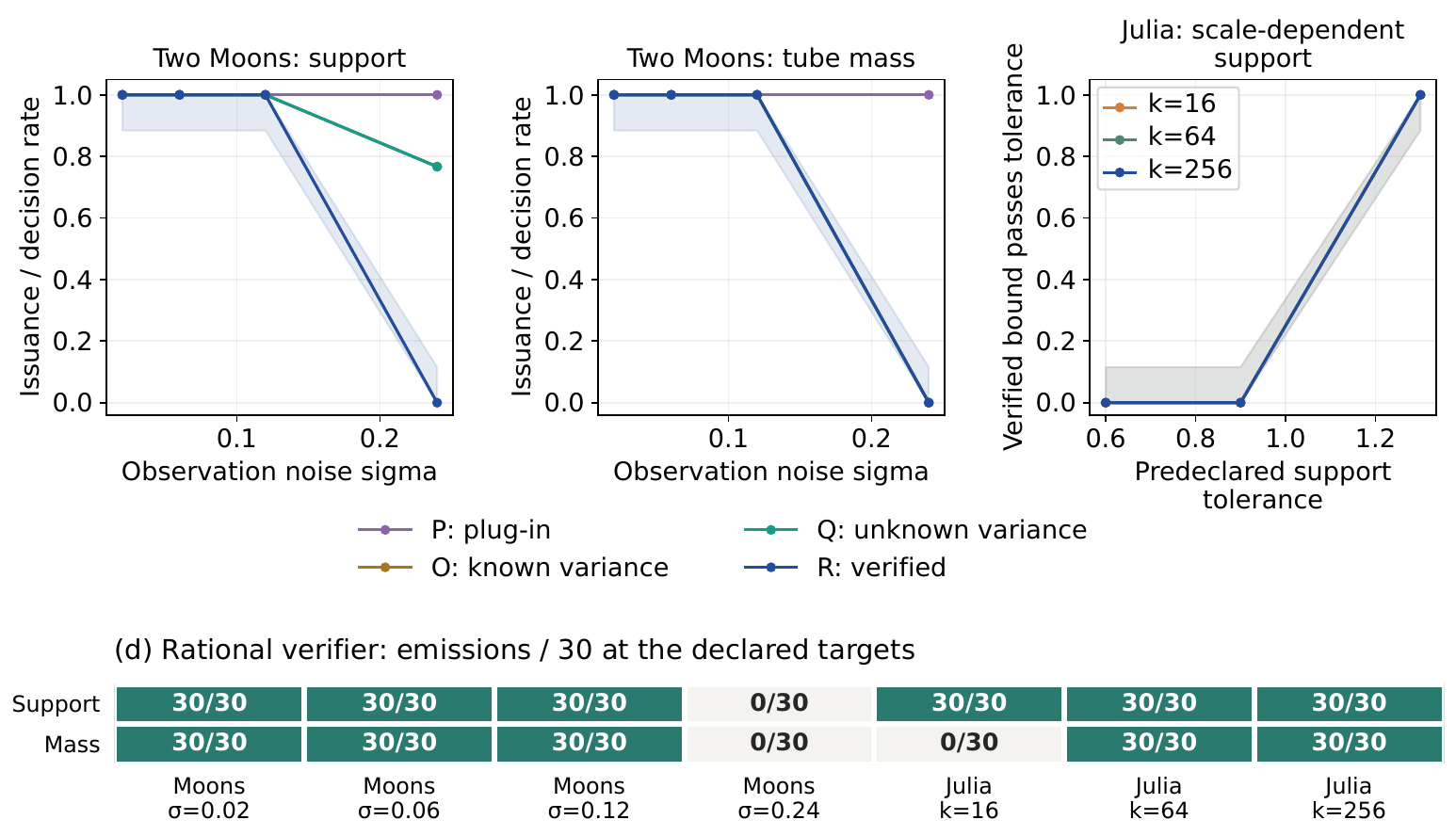}
\caption{Upper panels: the original matched support and mass decisions for Two Moons and paired Julia support-target checks, with their unchanged pointwise 95\% intervals. Lower strip: all rational support and mass counts out of 30, including zero emissions. The fixed tube radius is $0.12$; support tolerances are $0.3$ (Moons) and $1.3$ (Julia), and outside-mass tolerance is $0.08$. Julia support targets use the same 30 inputs at each $k$. A coarse witness certificate need not imply manifold support; abstention is not nonexistence. Exact counts are in the numerical supplement, Table~S2.}
\label{fig:mfcurves}
\end{figure}

\FloatBarrier
\paragraph{Julia and Mandelbrot sets in the complex plane.}
Figure~\ref{fig:fractalcontext} provides visual orientation for the two familiar fractal objects. For $f_c(z)=z^2+c$, the filled Julia set consists of initial states $z$ with bounded forward orbit at a fixed parameter $c$; its boundary is the Julia set. The Mandelbrot set instead consists of parameters $c$ for which the critical orbit starting at $z=0$ remains bounded. Thus the panels use different planes. The connected Basilica example $c=-1$ is context only. The certification experiment below uses the real Cantor Julia set at $c=-3$, while the Mandelbrot control uses a finite boundary-pixel sampling law (Appendix~\ref{app:mandelbrot}). Neither orientation image supplies a certificate for an infinite fractal.

\begin{figure}[!t]\centering
\includegraphics[width=.98\textwidth]{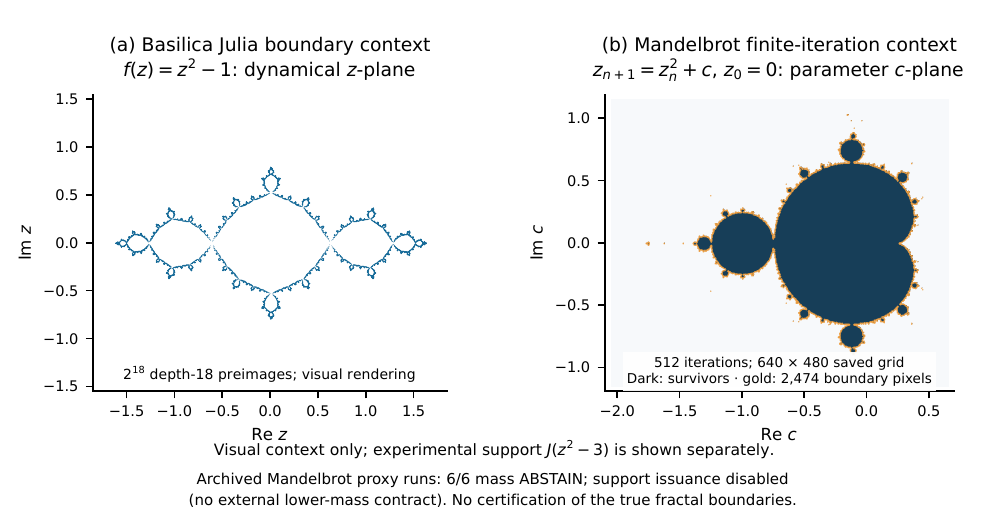}
\caption{Complex-plane orientation, distinct from the certified real-Cantor experiment. (a) Finite backward-iteration rendering of the Basilica Julia set $J(z^2-1)$ in the dynamical $z$-plane: all $2^{18}$ numerically computed preimages after 18 floating-point inverse-branch steps from $0.25+0.35i$. This finite cloud is visual context, not a rigorous boundary enclosure or experimental output. (b) The archived Mandelbrot parameter grid: $640\times480$ points, 512 iterations of $z\leftarrow z^2+c$ from $z=0$, with the saved survivor fill and 2,474 boundary-proxy points. Finite survival does not establish membership in the infinite Mandelbrot set or certify its boundary. Across all six archived 128/512-iteration proxy runs, mass issuance is $0/6$; whole-support issuance is disabled because no lower-mass contract is supplied. The actual $J(z^2-3)$ support and its certificate diagnostics follow in Figures~\ref{fig:juliaactualcontext} and~\ref{fig:juliadirect}.}
\label{fig:fractalcontext}
\end{figure}
\FloatBarrier
\paragraph{Cantor Julia set: support certification and verified coordinates.}
The dedicated direct-coordinate study uses 30 additional inputs for the same $J(z^2-3)$, ten at each
$k=16,64,256$. Each input has $N=16{,}384$ independent latent addresses,
an isometric rational sensor map into $\mathbb R^{16}$, and Gaussian
measurement variance unknown to the algorithm. A segment is fitted from
the first 4,096 recorded means, without reference coordinates, polynomial
membership tests or the generator variance. Its axis, center and endpoints
are frozen as rational numbers. Exact integer distance checks against a
1,025-point segment net certify the data-to-witness error, and the support
confidence bound transfers it to the Julia support. The supplied lower-mass
contract is the one proved in Appendix~\ref{app:juliaproof}; it is not
inferred from finite data. Corollary~\ref{cor:boundary} applies because
the witness is a segment with boundary and infinite reach.

For every coarse support emission, the exact Gram verifier applies
Theorem~\ref{thm:coordinatecomposition}, constructing
$m=2d+1=3$ coordinates on the full affine span of this $d=1$ witness.
Nearest-segment outputs are rationally represented, mapped to these
coordinates and inverted exactly. The selected map has one nonzero
output coordinate and two zero coordinates; this is a restricted
constructive embedding, not a general implementation of Whitney's
existence theorem. Appendix~\ref{app:juliaproof} proves that this Julia set is real and totally disconnected, and that its distance to its convex hull $[-\beta,\beta]$ is $\sqrt{3-\beta}\simeq0.835$, where $\beta=(1+\sqrt{13})/2$. This reference value explains why a segment can be a coarse approximation while filling a persistent central gap; the actual fitted-witness errors are assessed separately by deterministic enclosures. Figure~\ref{fig:juliaactualcontext} shows the finite outer cylinders and evaluation samples at two scales. Figure~\ref{fig:juliadirect} shows the observed data, fitted witness, decoded outputs and verified bounds. The cylinders and samples are not themselves the infinite support. Appendix~\ref{app:juliadirect} gives the complete calculation; these runs are distinct from the historical trial counts.

\begin{figure}[!t]\centering
\includegraphics[width=.98\textwidth]{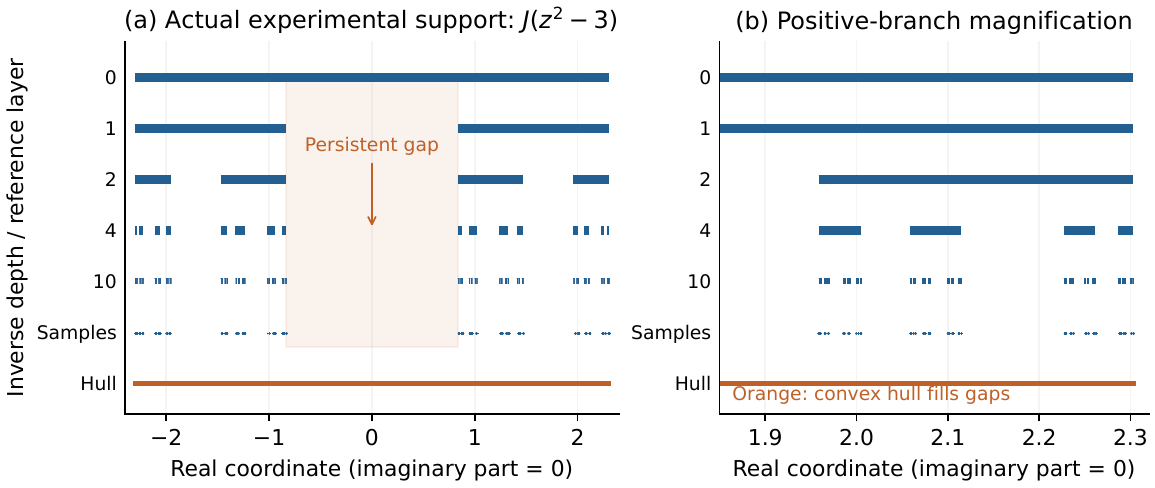}
\caption{The actual experimental Julia set $J(z^2-3)$ is a disconnected subset of the real axis. Blue rows are the saved finite rational outer enclosures (inverse cylinders) at the stated depths, followed by all 16,384 frozen clean evaluation samples from $k=16$, trial 0. The right panel magnifies the positive branch; it is not a different Julia set. Finite cylinder unions enclose the infinite set and display its persistent gaps. Orange is its convex hull, shown only to explain how a smooth segment can fill those gaps; it is not a new fitted or certified result. The fitted witness, inverse outputs, error bounds and strict-target abstentions remain in Figure~\ref{fig:juliadirect}. No connected planar Julia image or fractal-manifold claim is substituted for these data.}
\label{fig:juliaactualcontext}
\end{figure}

All 30 runs issued the support certificate at tolerance $1.3$ and passed
the exact coordinate and inverse checks. All abstained at the paired
targets $0.6$ and $0.9$. Independent deterministic reference enclosures
placed each whole-support error below its reported bound. This checks
the realized geometric assertions; 30 successful trials do not establish
a zero failure probability or a simultaneous confidence claim across
all runs. The confidence error is $0.05$ per input under the declared
sampling contracts. The segment reconstruction reduces the displayed
RMS error, but its full image fills the Cantor gaps. Neither a covariance
rank of one nor the certified witness dimension $d=1$ identifies the
Julia set's Hausdorff dimension or certifies exact manifold regularity.
Takens is not invoked: these are independent vector observations, not
delay windows from a single orbit.

\begin{figure}[p]\centering
\includegraphics[width=.98\textwidth]{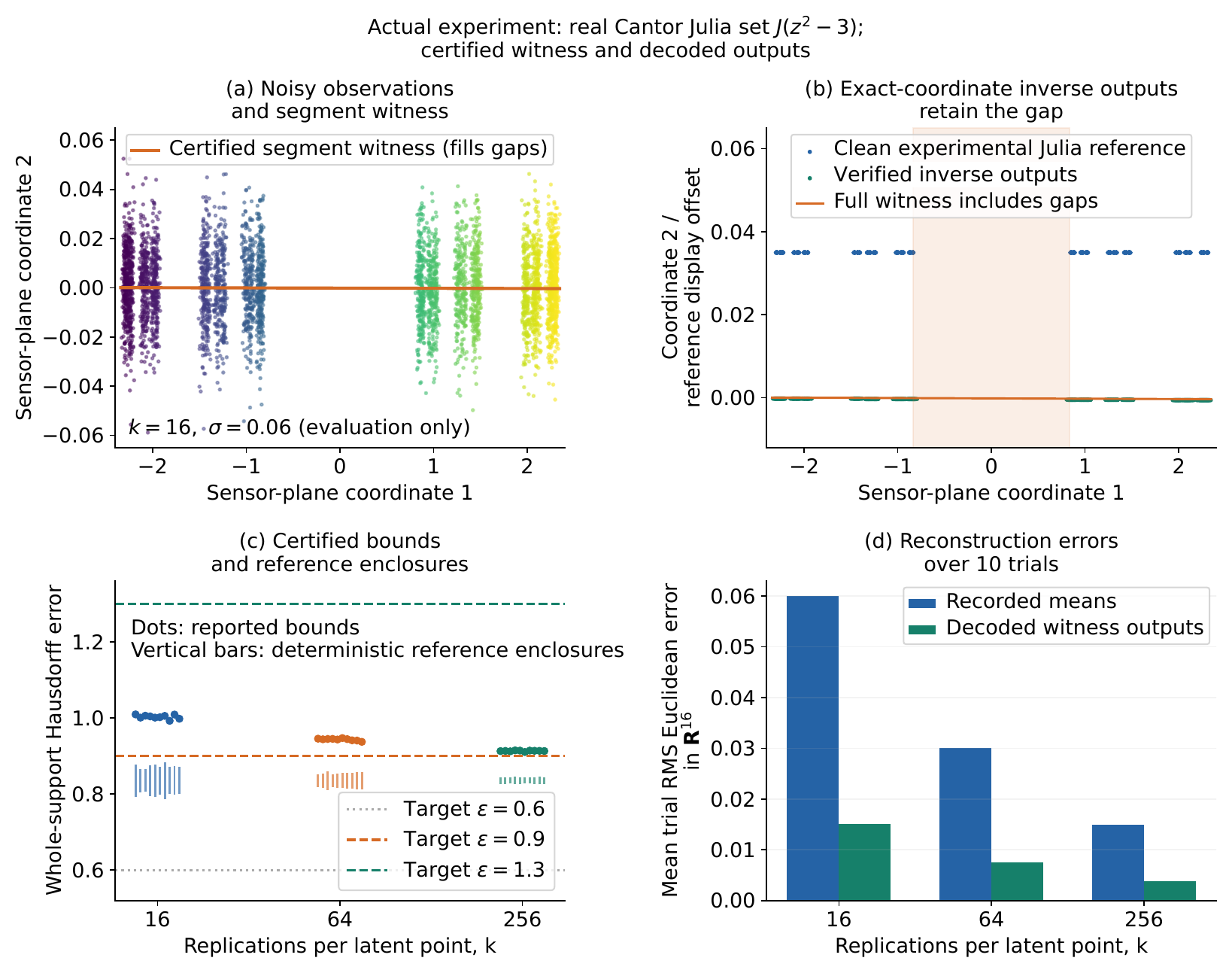}
\caption{Dedicated direct-coordinate experiment on the actual $J(z^2-3)$ data. (a) Every fourth recorded mean, in stored index order, and the training-fitted segment for the first predeclared $k=16$ input. The orange segment is the certified nearby witness; it fills Cantor gaps and is not the Julia set. The sensor plane is an evaluation-only projection of the $\mathbb R^{16}$ arrays. (b) Every eighth paired exact-coordinate inverse output and clean sample, in stored index order, with the full segment witness; clean samples are shifted vertically by $0.035$ only for visibility. (c) Reported support bounds for all ten trials at each $k$ (dots) and deterministic reference enclosures (vertical bars). At each $k=16,64,256$, support emissions are $10/10$ at $1.3$ and $0/10$ at each strict target $0.6,0.9$. These paired target checks add no independent trials. (d) Mean RMS errors of recorded means and decoded witness outputs. Colors and references never enter fitting or verification. Exact bound ranges and RMS values are in the numerical supplement, Table~S3.}
\label{fig:juliadirect}
\end{figure}

\FloatBarrier
\paragraph{Matched calibration and its power cost.}
The circle-with-branch study uses 17 conditions with 60 trials each, with identical observations, fitted candidate and requested targets across P/O/Q/R. Figure~\ref{fig:matched} separates branch mass, branch length and allocation of a fixed observation budget. At branch mass $0.10$, R issues 3/60 mass certificates while Q issues 55/60 numerical decisions. At mass $0.165$, the uncalibrated P rule makes 15/60 false target decisions; R abstains. These unfavorable power comparisons remain part of the evidence. Appendix~\ref{app:round2} reports every condition, its population evaluation and pointwise intervals. The circle is a controlled calibration family, not a new fitting benchmark.

\begin{figure}[!t]
\centering
\includegraphics[width=.98\textwidth]{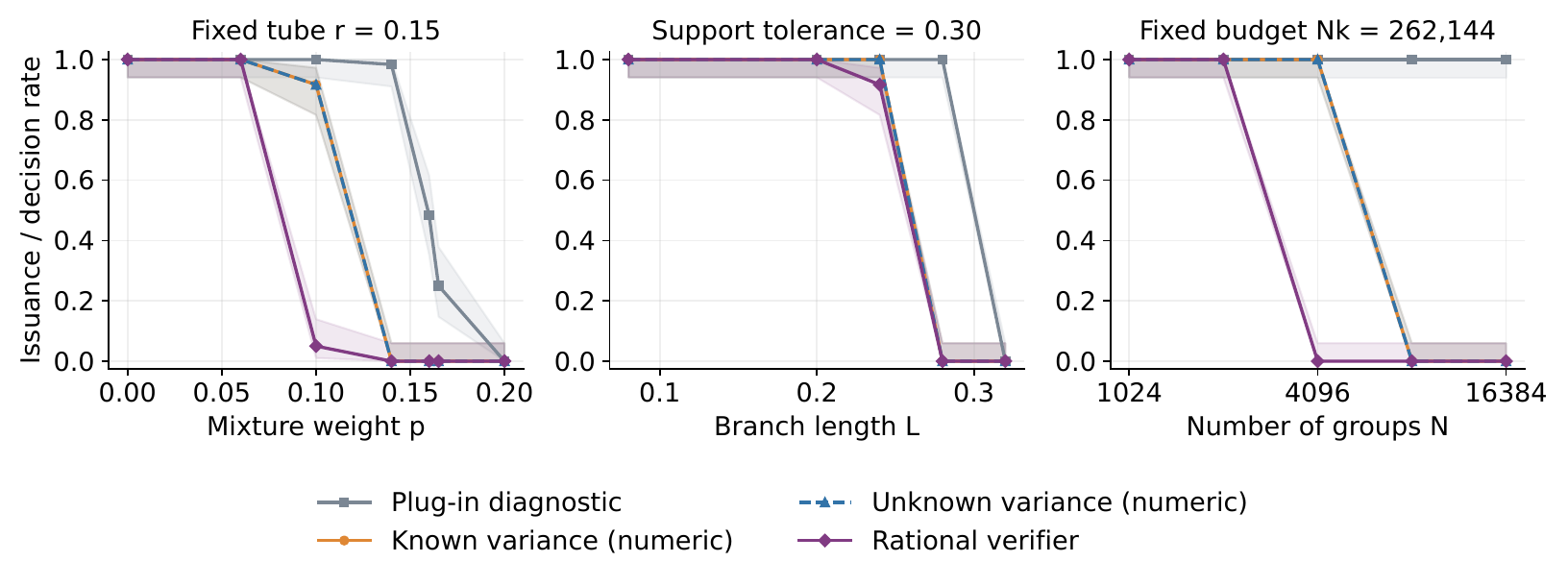}
\caption{Same-target comparisons from executed observations. Left: tube-mass issuance against mixture weight; the latent radius is fixed at $0.15$. Center: support issuance against branch length, with tolerance $0.3$. Right: tube-mass issuance at fixed observation budget $Nk=262144$ and radius $0.08$. Bands are pointwise 95\% Clopper--Pearson intervals from 60 independent trials per condition. Plug-in and quantile curves are numerical references; only the rational method emits machine-verified certificates. Distinct conditions are not pooled into an error-rate claim.}
\label{fig:matched}
\end{figure}

\paragraph{Dimension estimation and an unresolved finite-bound regime.}
The separate observed-data estimator is evaluated on the 200 saved sphere inputs in Figure~\ref{fig:dimensionchecks}. It receives neither true dimension nor true variance. Accuracy improves across the displayed designs, but all 200 observable covariance bounds are vacuous: this study supplies no nontrivial finite-sample dimension certificate and does not prove consistency. Appendix~\ref{app:dimimplementation} gives the analytic covariance, complete records and implementation checks. This numerical limitation is distinct from the proved asymptotic result.

\begin{figure}[!t]
\centering
\includegraphics[width=.98\textwidth]{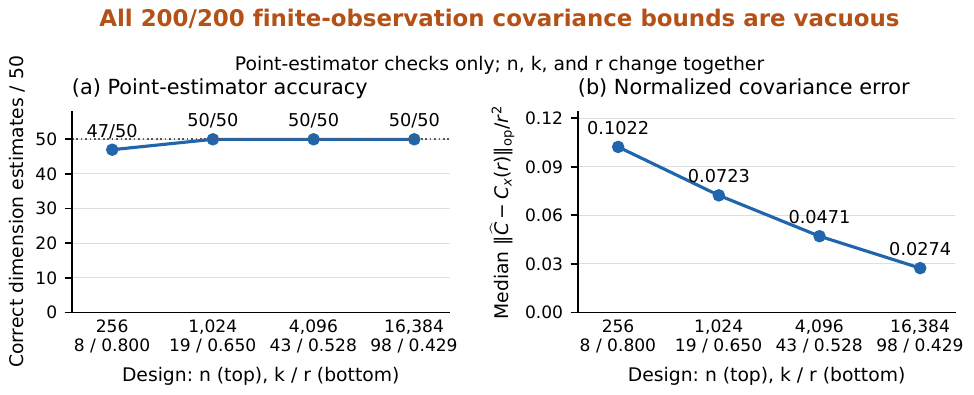}
\caption{Observed-repeat point-estimator checks on $S^2\subset\mathbb R^8$: 50 independent simulated inputs per design. Left: correct dimension counts. Right: median normalized operator error. The categorical designs vary $n,k,r$ together; the connecting lines only guide the eye and do not isolate a sample-size effect. All 200 finite-observation covariance bounds are vacuous. No uncertainty intervals are inferred from these summaries. Exact values are retained in the accompanying numerical supplement, Table~S4.}
\label{fig:dimensionchecks}
\end{figure}

\paragraph{Scope and retained adverse results.}
Appendix~\ref{app:boundaryfractal} gives the Two Moons and Julia generator bounds, reference enclosures and complete outcomes. Six finite Mandelbrot-proxy runs issue no mass certificate; they do not concern the entire infinite boundary. The SL/CGL experiments are summarized together because their matched equilibrium sensors have the same circular observation law (Appendix~\ref{app:landau}). The generic point, ODE and KS reconstruction studies, including adverse outcomes, are retained in the separate reproducibility archive. A compact corrected MNIST input/output comparison follows below. Its complete scores are preserved in numerical supplement Table~S6, with failure and computational details in Appendices~\ref{app:restored}--\ref{app:faceaudit}; supplementary certificate controls remain adjacent to their detailed analyses.

\subsection{Auxiliary image diagnostics}\label{sec:mnistvisual}\label{sec:imagesummary}
The historical MNIST run uses 8000 selected training images and 40
predeclared test queries, yielding 39 reconstructions; the first digit~5
abstains. Its observed-image compression errors, diagnostic ranks and
lack of natural repeats do not test the manifold-certification or
intrinsic-dimension consistency theorems. Appendix~\ref{app:restored}
retains the complete quantitative audit and the distinct corrected
synthetic-noise comparison.

Figure~\ref{fig:mnistapplication} shows that distinct corrected comparison
on fixed digit queries, with the original images, observed repeat means,
PCA-only baseline, local PCA64 and the credited Yao--Xia objective.
It makes preprocessing loss and residual noise visible without claiming
that recognizable digits are manifold certificates. The all-condition
MSE and numerical failure counts are in Figure~\ref{fig:mnistscores};
PSNR, SSIM and exact values remain in numerical supplement Table~S6.
The local-PCA64 image row uses the common observed-training PCA map and
is not the separate native-784 algorithm.

\begin{figure}[!t]\centering
\includegraphics[width=.98\textwidth]{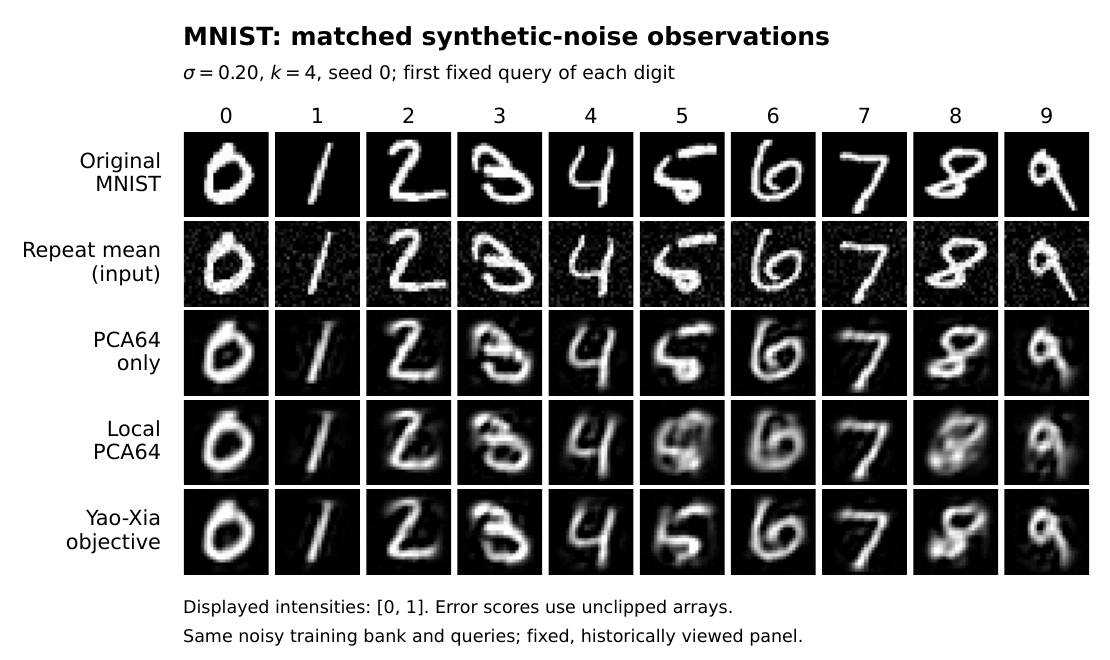}
\caption{Corrected synthetic-noise MNIST comparison on the frozen test panel. At $\sigma=0.20$, $k=4$, and noise seed 0, columns use the first query of each digit in the frozen ordering, following the archived plotting rule. Rows compare the original targets, observed repeat means, PCA64-only reconstructions, local-PCA64 outputs and the independently implemented Yao--Xia objective on those same queries. Images are clipped to $[0,1]$ only for display; all scores use unclipped outputs. All displayed corrected outputs are defined; other conditions retain their explicitly flagged PCA64 fallbacks in the aggregate comparison. This comparison uses all 8,000 observed training averages and validation-selected settings; it is distinct from the historical no-added-noise 39/40 audit and from Native LPCA784. Figure~\ref{fig:mnistscores} retains every noise/repeat condition, including failures and unfavorable baselines. These are fixed-panel pixel-reconstruction diagnostics, not geometric certificates or evidence of general superiority.}
\label{fig:mnistapplication}
\end{figure}

The face controls test only pixel reconstruction; Appendix~\ref{app:faceaudit} retains their conditions, failures and unfavorable comparisons, including the lack of a uniform denoising advantage for the geometric methods.
\FloatBarrier

\section{Discussion and limitations}\label{sec:discussion}
\paragraph{Coverage and implementation.}
The certificate controls a nearby smooth explanation under its stated
observation and coverage assumptions. Lower mass constants are external.
Dependence, anisotropic or non-Gaussian errors, and unobserved corruption
of replicate groups are outside the guarantees. Arc and segment families
allow fitted pose and boundary; retained circle/torus controls have fixed
pose, and the direct torus experiment fits within a supplied sensor span.
These remain restricted families. Flexible globally checked atlases and a
general geometric Whitney construction are substantive future tasks.
True regularity and intrinsic dimension require the additional conditions
of Theorem~\ref{thm:dimconsistent} or Section~\ref{sec:ahlforsreach}.
The coordinate routine requires verified affine containment; it does not
cover arbitrary witness atlases, minimal $2d$ constructions, or embedding
of an arbitrary singular support.

\paragraph{What the experiments establish.}
Finite Swiss-roll patches have boundary. Independent rollout endpoints
target finite-time laws rather than invariant attractor measures, and
dependent trajectories require a separate sampling analysis. Historical
MNIST observations have no repeats and do not instantiate the Gaussian
model. The synthetic-corruption follow-up in
Appendix~\ref{app:matchedmnist} evaluates denoising against original images;
the face comparisons in Appendix~\ref{app:faceaudit} assess separate
reconstruction tasks. None supplies manifold-certification evidence, and
the geometric image tracks do not uniformly improve on observation means
or PCA-only baselines. The torus-with-branch and Julia controls likewise
show why successful witness certification and a clearer reconstruction
need not preserve singular latent structure.

\paragraph{Geometry and dynamical reconstruction.}
For a smooth flow $\varphi_t$ and initial state $x_0$, the finite orbit
trace $\{\varphi_t(x_0):0\le t\le T\}$ is locally a one-dimensional
immersed curve wherever its vector field is nonzero. Its closure may fill
a torus or lie in a fractal invariant set; neither is automatically the
full state manifold. A witness dimension cannot identify all these objects.
Even exact state-support geometry and its invariant probability measure
need not determine direction, speed or the evolution map
(Appendix~\ref{app:dynamicdistinction}). If a true-state embedding
$\Psi:S\to\Psi(S)$ and evolution map $F$ are established, reconstructed
dynamics are $G=\Psi\circ F\circ\Psi^{-1}$. Learning $G$ from noisy ordered
transitions is an additional inference problem, not an output of the
geometric certificate. Corollary~\ref{cor:statecoordinates} records the
identification conditions, and Proposition~\ref{prop:transitionbounds}
quantifies the uniform-transition and forecast requirements. Its geometric
error sum explains the limitation of long-horizon prediction under an
expanding evolution. The optional scalar-delay route additionally requires
the $C^2$ state-space, dynamics, observability and sampling conditions
described in Section~\ref{sec:intro}; they are not premises of the direct
witness-coordinate calculation.

\section*{AI Use Statement}
AI assisted with research planning, literature retrieval, mathematical formulation and proof checking, code writing and debugging, execution and interpretation of synthetic experiments, figure creation, and manuscript drafting and editing. Numerical results in this manuscript were obtained by executing the supplied programs. Automated checks and the recorded arithmetic verification do not constitute an independent human review of every mathematical claim. The author has reviewed and approved the AI-assisted content and takes responsibility for the final manuscript. No AI system is listed as an author.

\clearpage
\raggedbottom
\appendix
\section*{Master Concept Diagram}
\begin{figure}[H]
\centering
\includegraphics[width=\textwidth]{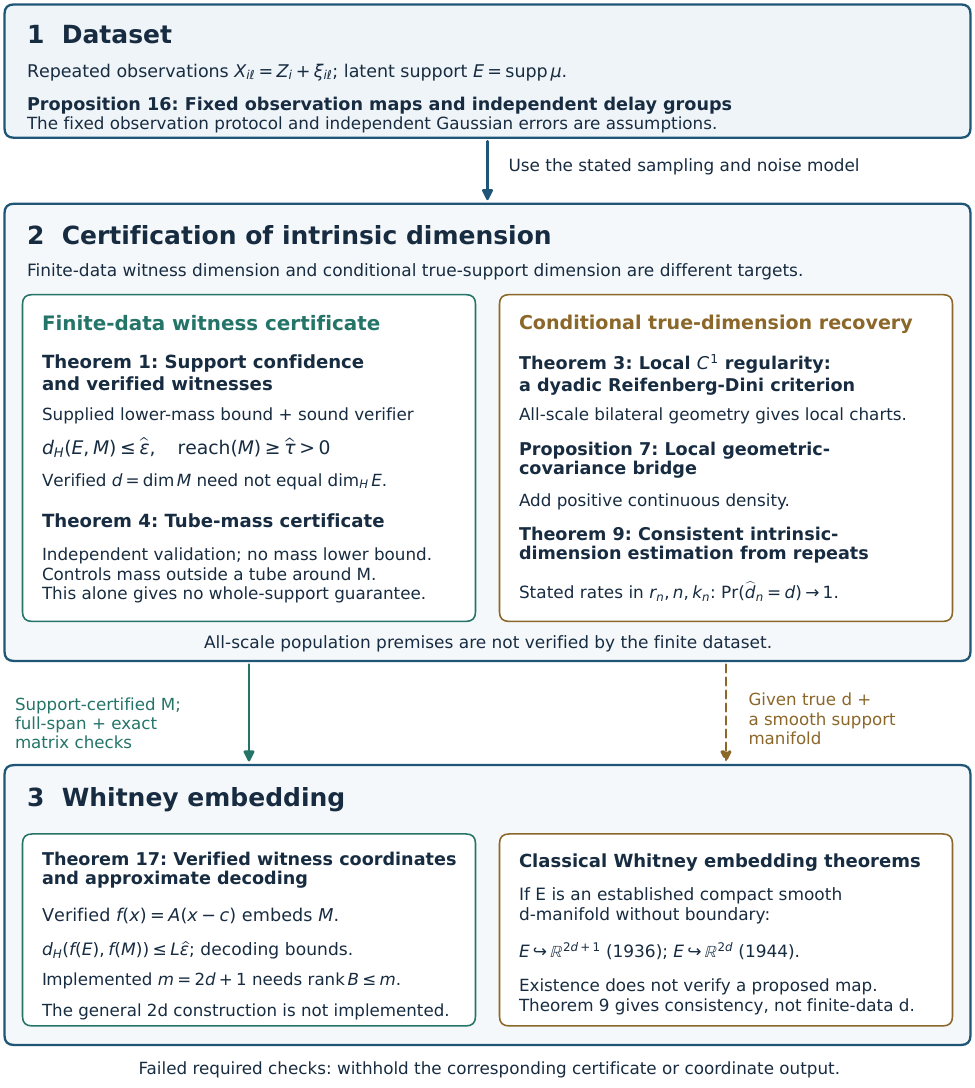}
\caption{Master concept pipeline. The solid coordinate route requires Theorem~\ref{thm:main}'s whole-support certificate and Theorem~\ref{thm:coordinatecomposition}'s matrix checks. The dashed route additionally assumes an established smooth support manifold of true dimension $d$ for Whitney's existence theorems \citep{whitney1936,whitney1944}. Theorem~\ref{thm:dimconsistent} supplies asymptotic consistency, not finite-data identification of $d$.}
\label{fig:masterconcept}
\end{figure}
\clearpage

\section{Proof of Theorem~\ref{thm:main}}\label{app:main}
\paragraph{Attribution and proof convention.}
The Gaussian tail calculation is credited to \citet{laurent2000}; the packing and occupancy argument is the classical support-estimation mechanism used by \citet{chazal2015}. For completeness, we prove Theorem~\ref{thm:main} with all intermediate calculations below. The later proofs identify every imported theorem at its point of use. In particular, the geometric Whitney theorem and the sharp DKW inequality are cited inputs, not results claimed as new or re-proved in their full generality here.

\subsection{Orthogonal decomposition and the unknown variance}
\begin{proof}
\paragraph{Step 1: identities before a distributional claim.}
For each group, $\sum_\ell(X_{i\ell}-Y_i)=0$. Expanding the squares gives
\begin{align*}
\sum_{\ell=1}^k\|X_{i\ell}-Y_i\|^2
&=\sum_\ell\|X_{i\ell}\|^2
 -2Y_i^{\mathsf T}\sum_\ell X_{i\ell}+k\|Y_i\|^2\\
&=\sum_\ell\|X_{i\ell}\|^2-k\|Y_i\|^2.
\end{align*}
Also
\begin{align*}
\sum_{\ell<m}\|X_{i\ell}-X_{im}\|^2
&=\tfrac12\sum_{\ell,m}
  (\|X_{i\ell}\|^2+\|X_{im}\|^2-2X_{i\ell}^{\mathsf T}X_{im})\\
&=k\sum_\ell\|X_{i\ell}\|^2-\Big\|\sum_\ell X_{i\ell}\Big\|^2\\
&=k\sum_\ell\|X_{i\ell}-Y_i\|^2.
\end{align*}
These equalities also hold for recorded rational coordinates. With
$H_k=I_k-\mathbf1\mathbf1^{\mathsf T}/k$, multiplication yields
$H_k^{\mathsf T}=H_k$, $H_k^2=H_k$ and $\operatorname{tr}H_k=k-1$.
Thus centering removes exactly one replicate coordinate per group and
ambient coordinate; it does not remove an additional latent degree of freedom.
\paragraph{Step 2: orthogonal Gaussian coordinates.}
Fix a group $i$ and coordinate $a\in\{1,\ldots,q\}$. For $\sigma>0$, let
\[
 e_{ia}=\sigma^{-1}(\varepsilon_{i1,a},\ldots,\varepsilon_{ik,a})^{\mathsf T},
 \quad h_0=k^{-1/2}(1,\ldots,1)^{\mathsf T},
 \quad H=(h_0,h_1,\ldots,h_{k-1})\in O(k).
\]
The density of $H^{\mathsf T}e_{ia}$ is
$(2\pi)^{-k/2}\exp(-\sum_{b=0}^{k-1}v_b^2/2)$: the change of variables has absolute determinant one and preserves squared norm. Thus its coordinates $\xi_{ia,b}=h_b^{\mathsf T}e_{ia}$ are independent standard normals. Moreover,
\begin{align*}
 Y_{i,a}-Z_{i,a}&=\frac{\sigma}{\sqrt k}\xi_{ia,0},\\
 \sum_{\ell=1}^k(X_{i\ell,a}-Y_{i,a})^2
 &=\sigma^2\norm{(I_k-h_0h_0^{\mathsf T})e_{ia}}^2
 =\sigma^2\sum_{b=1}^{k-1}\xi_{ia,b}^2.
\end{align*}
Summing independent terms over $i,a$ proves
\[
 Q:=S/\sigma^2=\sum_{i,a}\sum_{b=1}^{k-1}\xi_{ia,b}^2\sim\chi^2_\nu,
 \qquad \nu=Nq(k-1).
\]
The residual coordinates are independent of every mean coordinate and of the latent sample, so $S$ is independent of all $Y_i$.
The joint orthogonal Gaussian density factors as
\[
\prod_{i,a}\left[(2\pi)^{-1/2}e^{-\xi_{ia,0}^2/2}
 \prod_{b=1}^{k-1}(2\pi)^{-1/2}e^{-\xi_{ia,b}^2/2}\right].
\]
The first factors determine the mean errors and the remaining factors
determine $S$, all independently of $(Z_i)_i$. In particular
\[
\E(Y_i-Z_i)=0,\qquad\Cov(Y_i-Z_i)=\frac{\sigma^2}{k}I_q.
\]
This verifies the asserted independence with random latent points. The chi-square density is positive on $(0,\infty)$; hence its lower $\delta_C$ quantile $q_{\nu,\delta_C}$ is positive and satisfies $\Prob(Q<q_{\nu,\delta_C})=\delta_C$. Therefore
\[
 \Prob(\sigma^2>U)=\Prob\left(S/\sigma^2<q_{\nu,\delta_C}\right)=\delta_C.
\]
When $\sigma=0$, every error, $S$, and $U$ is zero and all mean errors vanish. No division by $\sigma$ is used in this case.

\subsection{Gaussian mean-error bound, including its Chernoff calculation}
We give the special chi-square calculation of \citet{laurent2000} for completeness. If $G\sim N(0,I_q)$ and $W=\norm G^2$, Gaussian integration gives, for $0<\lambda<1/2$,
\[
 \E e^{\lambda W}=\prod_{a=1}^q\frac1{\sqrt{2\pi}}
 \int_{\R}e^{-(1-2\lambda)t^2/2}\,dt=(1-2\lambda)^{-q/2}.
\]
For $0\le z<1$,
\[
 -\log(1-z)-z=\sum_{m\ge2}\frac{z^m}{m}
 \le\frac12\sum_{m\ge2}z^m=\frac{z^2}{2(1-z)}.
\]
Substituting $z=2\lambda$ into this log-series bound gives
\begin{align*}
 \log\E e^{\lambda(W-q)}
 &=-q\lambda-\frac q2\log(1-2\lambda)\\
 &=\frac q2\{-\log(1-2\lambda)-2\lambda\}
 \le\frac q2\frac{(2\lambda)^2}{2(1-2\lambda)}
 =\frac{q\lambda^2}{1-2\lambda}.
\end{align*}
This is an upper bound from the logarithmic series. To pass from it
to a tail probability, use the exponential form of Markov's inequality.
For any $t>0$ and $0<\lambda<1/2$, the underlying
pointwise inequality is
\[
 \mathbf1\{W-q\ge t\}\le e^{\lambda(W-q-t)}.
\]
Indeed the right side is at least one on the event and is positive
on its complement. Taking expectations, and then using the preceding
moment-generating-function bound, yields
\begin{align*}
 \Prob\{W-q\ge t\}
 &=\E\mathbf1\{W-q\ge t\}\\
 &\le\E e^{\lambda(W-q-t)}
 =e^{-\lambda t}\E e^{\lambda(W-q)}\\
 &\le\exp\left\{-\lambda t+
                   \frac{q\lambda^2}{1-2\lambda}\right\}.
\end{align*}
Now fix $x>0$, put $b=\sqrt{x/q}$, and choose
$\lambda=b/(1+2b)$. Since $b>0$, this choice lies in $(0,1/2)$.
For $t=2\sqrt{qx}+2x$ we have
\[
 x=qb^2,\qquad t=2qb+2qb^2,\qquad
 1-2\lambda=\frac1{1+2b},\qquad
 \frac{q\lambda^2}{1-2\lambda}=\frac{qb^2}{1+2b}.
\]
Thus the exponent is exactly
\begin{align*}
 -\lambda t+\frac{q\lambda^2}{1-2\lambda}
 &=-\frac{2qb^2+2qb^3}{1+2b}+\frac{qb^2}{1+2b}\\
 &=-\frac{qb^2(1+2b)}{1+2b}=-qb^2=-x,
\end{align*}
and $\Prob\{W-q\ge2\sqrt{qx}+2x\}\le e^{-x}$.
Finally,
\[
 (\sqrt q+\sqrt{2x})^2
 =q+2\sqrt{2qx}+2x\ge q+2\sqrt{qx}+2x.
\]
Consequently $\{\norm G>\sqrt q+\sqrt{2x}\}$ is contained in
$\{W-q>2\sqrt{qx}+2x\}$ and hence in the event just bounded.
This proves
\[
 \Prob\{\norm G>\sqrt q+\sqrt{2x}\}\le e^{-x}.
\]
Apply this with $G_i=\sqrt k(Y_i-Z_i)/\sigma$ and $x=\log(N/\delta_L)$. A union bound gives
\[
 \Prob\left\{\max_i\norm{Y_i-Z_i}>
 \frac{\sigma}{\sqrt k}(\sqrt q+\sqrt{2\log(N/\delta_L)})\right\}
 \le Ne^{-\log(N/\delta_L)}=\delta_L.
\]
On the intersection of this event's complement with $\{\sigma^2\le U\}$,
$\max_i\norm{Y_i-Z_i}\le a$. Its failure probability is at most $\delta_C+\delta_L$; independence of the two events is not required.

\subsection{Packing, occupancy, and both support-distance directions}
Fix the deterministic $h$ of \eqref{eq:coverage}. Compactness supplies a finite maximal set $x_1,\ldots,x_J\in E$ with pairwise distances strictly greater than $h/2$. One construction selects a new uncovered point until none remains; total boundedness precludes infinitely many $h/2$-separated points. Maximality gives $E\subset\bigcup_j\overline B(x_j,h/2)$. The closed $h/4$-balls are disjoint, since a point in two such balls would imply $\norm{x_j-x_l}\le h/2$. Because $h/2\le t_0$,
\[
 1\ge\sum_{j=1}^J\mu(\overline B(x_j,h/4))
 \ge Jc_0(h/4)^s,\qquad J\le4^s/(c_0h^s).
\]
Write $p_j=\mu(\overline B(x_j,h/2))\ge c_0(h/2)^s$. Independence of the latent draws gives
\[
 \Prob\{\{Z_1,\ldots,Z_N\}\cap\overline B(x_j,h/2)=\varnothing\}
 =(1-p_j)^N\le e^{-Np_j}\le e^{-Nc_0(h/2)^s}.
\]
The inequality $1-u\le e^{-u}$ follows by differentiating $e^{-u}-(1-u)$ on $[0,1]$. The probability of at least one empty ball is thus at most
$4^s(c_0h^s)^{-1}e^{-Nc_0(h/2)^s}\le\delta_E$. Constants that force $c_0(h/2)^s>1$ are inconsistent with the population hypothesis, not an alternative calibration.

Let $Z=\{Z_i\}_{i=1}^N$. On the occupancy event, for any $x\in E$, choose a center within $h/2$ and a sample within $h/2$ of that center; then $\dist(x,Z)\le h$. Conversely, $\sup_i\dist(Z_i,E)=0$ almost surely by the definition of support. Hence $d_H(E,Z)\le h$. On the simultaneous mean-error event,
\[
 \sup_i\dist(Y_i,Z)\le\max_i\norm{Y_i-Z_i}\le a,
 \qquad \sup_i\dist(Z_i,Y)\le a,
\]
so $d_H(Z,Y)\le a$. More explicitly, for each $x\in E$ choose an occupied covering
ball with center $x_j$ and a latent point $Z_i$ in that ball. Then
\[
 \dist(x,Y)\le\norm{x-Y_i}
 \le\norm{x-x_j}+\norm{x_j-Z_i}+\norm{Z_i-Y_i}
 \le h/2+h/2+a.
\]
In the other direction, $Z_i\in E$ gives
$\dist(Y_i,E)\le\norm{Y_i-Z_i}\le a$.
Let $\mathcal H$ be the occupancy event and $\mathcal N$ the event
$\max_i\norm{Y_i-Z_i}\le a$. On the probability-one event $Z_i\in E$ for all $i$, these
inequalities prove the pointwise inclusion
\[
 \{d_H(E,Y)>h+a\}\subseteq\mathcal H^c\cup\mathcal N^c.
\]
The covering centers and $h$ are fixed before the latent draws, so the
occupancy bound applies to these same balls. A union bound now gives
\[
 \Prob\{d_H(E,Y)>h+a\}
 \le\Prob(\mathcal H^c)+\Prob(\mathcal N^c)
 \le\delta_E+\delta_C+\delta_L=\delta.
\]
This is the classical coverage mechanism, with attribution to \citet{chazal2015}.

For reference, set $v=Nc_0(h/2)^s$. Then
\[
 \frac{4^s}{c_0h^s}e^{-Nc_0(h/2)^s}
 =\frac{2^sN}{v}e^{-v}\le\delta_E
 \quad\Longleftrightarrow\quad ve^v\ge2^sN/\delta_E.
\]
Since $v\mapsto ve^v$ is strictly increasing on $(0,\infty)$, equality gives
$h=2[W(2^sN/\delta_E)/(Nc_0)]^{1/s}$, usable only if $h\le2t_0$. For large $a$, $W(a)\le\log a$ because $W(a)+\log W(a)=\log a$ and $W(a)\ge1$. Thus the displayed radius is $O((\log N/N)^{1/s})$ for fixed constants.
\paragraph{Admissibility of a deterministic coverage radius.}
Let $B_E(h)$ denote the left side of \eqref{eq:coverage}. Then
\[
\frac{d}{dh}\log B_E(h)
=-\frac{s}{h}-\frac{Nc_0s}{2^s}h^{s-1}<0,\qquad h>0.
\]
It diverges as $h\downarrow0$. A permissible $h\in(0,2t_0]$ exists
precisely when
\[
B_E(2t_0)=\frac{2^s}{c_0t_0^s}e^{-Nc_0t_0^s}\le\delta_E,
\quad\text{equivalently}\quad
N\ge\frac{\log(2^s/(c_0t_0^s\delta_E))}{c_0t_0^s}.
\]
If this design condition fails, this calibration is unavailable.
The model implies $c_0t_0^s\le1$, so the logarithm is positive.
The Lambert-$W$ expression characterizes the ideal threshold; the
implementation uses its separately verified rational grid test.

\subsection{Simultaneous composition and the false-issuance event}
For nonempty compact $A,B,C$ and $a\in A$, choose a nearest $b\in B$. Then
\[
 \dist(a,C)\le\norm{a-b}+\dist(b,C)
 \le\overrightarrow d(A,B)+\overrightarrow d(B,C).
\]
Taking the supremum and repeating in the reverse direction proves the Hausdorff triangle inequality. On $\mathcal G=\{d_H(E,Y)\le\eta\}$, every accepted witness satisfies
\[
 d_H(E,M)\le d_H(E,Y)+d_H(Y,M)\le\eta+e=\widehat\epsilon.
\]
The verifier supplies its smoothness, dimension, and reach deterministically. Consequently
\[
 \{\EMIT,\ d_H(E,M)>\widehat\epsilon\}\subseteq\mathcal G^c,
 \qquad \Prob(\mathcal G^c)\le\delta.
\]
This inclusion is pointwise for every accepted candidate, including data-dependent $e$ and dimension. It requires no union bound over candidates. Measurability of the chosen output is part of the theorem; repeated acquisition or optional stopping would require another argument. This completes the proof of Theorem~\ref{thm:main}.

\end{proof}

\subsection{Proof of the geometric Whitney consequence}\label{app:whitneyproof}
\begin{proof}
For completeness, we prove Corollary~\ref{cor:whitney} by checking the hypotheses of the imported result. In \citet[Definition~4 and Theorem~2, published version]{fikln2020}, substitute the ambient Hilbert space $\R^q$, the nonempty finite set $X=Y$, the flat dimension $n=d$, and the approximation parameter $\delta=b$. The affine flat $A_y=L_y$ passes through its center $y$. Equation~\eqref{eq:bilateral} is exactly the required two-sided distance bound at every center, and $0<b<\sigma_dr$ supplies the strict smallness requirement. The derivative order in that theorem can be fixed at one; the smallness constant used for the present conclusions depends only on $d$.

That theorem gives a closed smooth $d$-submanifold $M$ with $d_H(Y,M)\le5b$ and $\reach(M)\ge r/3$. For any $m\in M$, a closest $y\in Y$ obeys
\[
 \norm m\le\norm y+\norm{m-y}\le\max_{y\in Y}\norm y+5b.
\]
Thus $M$ is bounded as well as closed in finite-dimensional Euclidean space, and hence compact. On $\mathcal G$, $d_H(E,M)\le\eta+5b$, as asserted. This proves the stated statistical consequence; the smooth construction and its constants remain the credited theorem of \citet{fikln2020}. No numerical value for its admissible smallness constant, or verification of a numerical atlas, is supplied by this proof.
On $\mathcal G$, a point $z\in E$ is within $\eta$ of $Y$,
and then within $\eta+5b$ of $M$. In the reverse direction a point
of $M$ is within $5b$ of $Y$ and then within $5b+\eta$ of $E$. Thus
\[
\overrightarrow d(E,M)\le\eta+5b,\qquad
\overrightarrow d(M,E)\le5b+\eta.
\]
For every realized bilateral choice satisfying the verified premises,
the imported theorem supplies such an $M$ on the same event.
Its statistical failure is contained in $\mathcal G^c$, independently
of the number of plane choices examined. This asserts existence
and does not identify a fitted numerical atlas with that manifold.

\end{proof}

\section{Proof of Theorem~\ref{thm:localexistence}: local manifold existence}\label{app:localexistence}
\begin{proof}
For completeness, we prove Theorem~\ref{thm:localexistence} as a local dyadic specialization of \citet[Theorems~1.2 and~1.4]{kafer2019}. The existence and regularity principle is credited to that prior work. Throughout, projections are orthogonal and operator norms Euclidean; the $V_{x,j}$ are ordinary linear subspaces, with no tangent plane or chart assumed.

\paragraph{Step 1: construct the limiting planes.}
Write $P_{x,j}$ for projection onto $V_{x,j}$ and $T_j=\sum_{k\ge j}\epsilon_k$. Discard finitely many initial scales so that $\epsilon_j\le1/4$ thereafter. For a unit vector $v\in V_{x,j}$, set $z=x+(r_j/4)v$. The reverse direction of \eqref{eq:dyadicregularity} gives $p\in E\cap\overline B(x,r_j)$ with $\norm{p-z}\le\epsilon_jr_j$. In fact $\norm{p-x}\le r_j/2=r_{j+1}$, so the forward direction at scale $j+1$ gives
\[
\dist(z,x+V_{x,j+1})\le\epsilon_jr_j+\epsilon_{j+1}r_{j+1}.
\]
Dividing by $r_j/4$ bounds $\norm{(I-P_{x,j+1})P_{x,j}}$ by $4\epsilon_j+2\epsilon_{j+1}$. For the reverse comparison, take a unit vector $w\in V_{x,j+1}$
and put $z'=x+(r_j/4)w$. The scale-$j+1$ reverse bound supplies
$p'\in E\cap\overline B(x,r_{j+1})$ such that
$\norm{p'-z'}\le\epsilon_{j+1}r_{j+1}$. Because
$p'\in\overline B(x,r_j)$, the scale-$j$ forward bound applies to
$p'$. Therefore
\[
 \frac{r_j}{4}\norm{(I-P_{x,j})w}
 =\dist(z',x+V_{x,j})
 \le\epsilon_{j+1}r_{j+1}+\epsilon_jr_j.
\]
Dividing by $r_j/4$ gives
$\norm{(I-P_{x,j})P_{x,j+1}}\le4\epsilon_j+2\epsilon_{j+1}$,
the same directed bound with the planes interchanged. The identity
\[
P-Q=P(I-Q)-(I-P)Q
\]
and invariance of the operator norm under transpose therefore imply
\begin{equation}
\norm{P_{x,j}-P_{x,j+1}}\le8\epsilon_j+4\epsilon_{j+1},\qquad
\norm{P_x-P_{x,j}}\le12T_j,
\label{eq:planelimit}
\end{equation}
For clarity, for $m>j$ the telescoping estimate is
\[
 \norm{P_{x,j}-P_{x,m}}\le\sum_{k=j}^{m-1}(8\epsilon_k+4\epsilon_{k+1})
 \le12T_j.
\]
Thus the sequence is uniformly Cauchy in the finite-dimensional space of matrices, and $P_x=\lim_{j\to\infty}P_{x,j}$ exists uniformly over the specified centers.
The tail arithmetic is
\[
\sum_{k\ge j}(8\epsilon_k+4\epsilon_{k+1})
=12T_j-4\epsilon_j\le12T_j.
\]
Matrix multiplication is continuous:
$\|P_{x,j}^2-P_x^2\|\le
(\|P_{x,j}\|+\|P_x\|)\|P_{x,j}-P_x\|\to0$.
The limiting symmetric idempotent has only zero and one as
eigenvalues; its trace $d$ is exactly its rank. Taking limits in $P_{x,j}^2=P_{x,j}$, $P_{x,j}^{\mathsf T}=P_{x,j}$, and $\operatorname{tr}P_{x,j}=d$ shows that $P_x$ is an orthogonal rank-$d$ projection. Denote its range by $V_x$.

\paragraph{Step 2: prove continuity, rather than assume it.}
Let $x,y\in E\cap B(x_0,2R)$ satisfy $\norm{x-y}\le r_j/4$. For unit $v\in V_{x,j}$ and $z=x+(r_j/4)v$, the point $p$ from Step~1 satisfies $\norm{p-y}\le3r_j/4$. Therefore
\[
\dist(z,y+V_{y,j})\le2\epsilon_jr_j,\qquad
\dist(x,y+V_{y,j})\le\epsilon_jr_j.
\]
Subtracting the projected displacement of $x-y$ gives
$\norm{(I-P_{y,j})(z-x)}\le3\epsilon_jr_j$, hence
$\norm{(I-P_{y,j})P_{x,j}}\le12\epsilon_j$. Interchanging $x,y$ and using the same projection identity gives
\begin{equation}
\norm{P_{x,j}-P_{y,j}}\le24\epsilon_j,\qquad
\norm{P_x-P_y}\le24T_j+24\epsilon_j\le48T_j.
\label{eq:planecontinuity}
\end{equation}
Given $\zeta>0$, choose $j$ with $48T_j<\zeta$. Then $\norm{x-y}\le r_j/4$ implies $\norm{P_x-P_y}<\zeta$, proving continuity of $x\mapsto P_x$ in the relative topology. No regularity of the initial plane-selection rule was needed.

\paragraph{Step 3: obtain all-radius bilateral approximation to the limiting planes.}
Orthogonal projection of a point in a centered ball to a plane through its center stays in the ball. Using \eqref{eq:planelimit} in both directions of \eqref{eq:dyadicregularity} consequently gives
\begin{equation}
d_H\bigl(E\cap\overline B(x,r_j),(x+V_x)\cap\overline B(x,r_j)\bigr)
\le\beta_jr_j,\qquad \beta_j=13T_j.
\label{eq:fixedplanedyadic}
\end{equation}
For the forward direction and $p\in E\cap\overline B(x,r_j)$,
\[
 \norm{(I-P_x)(p-x)}\le\norm{(I-P_{x,j})(p-x)}
      +\norm{P_{x,j}-P_x}\norm{p-x}
 \le(\epsilon_j+12T_j)r_j\le13T_jr_j.
\]
For the reverse direction, let $z\in(x+V_x)\cap\overline B(x,r_j)$
and put $z_j=x+P_{x,j}(z-x)$. Orthogonal projection is a contraction,
so $z_j\in(x+V_{x,j})\cap\overline B(x,r_j)$. Also
\[
 \norm{z-z_j}=\norm{(P_x-P_{x,j})(z-x)}\le12T_jr_j.
\]
The original reverse bound supplies
$p\in E\cap\overline B(x,r_j)$ with
$\norm{p-z_j}\le\epsilon_jr_j$. Hence
$\norm{p-z}\le(12T_j+\epsilon_j)r_j\le13T_jr_j$.
Both the selected point and the projected point stay in the required
ball, which verifies both directions of the displayed Hausdorff bound.

Choose a further starting index $J$ such that $\beta_j\le1/4$ for $j\ge J$. If $r_{j+1}<r\le r_j$, the forward error at radius $r$ is at most $\beta_jr_j\le2\beta_jr$. For a point $z$ in the radius-$r$ disk of $x+V_x$, first contract it to
\[
z'=x+\left(1-\frac{\beta_jr_j}{r}\right)(z-x).
\]
Here $\beta_jr_j/r\le2\beta_j\le1/2$, so the contraction
factor belongs to $[1/2,1]$. In particular
\[
 \norm{z'-x}\le r-\beta_jr_j,\qquad
 \norm{z-z'}=\frac{\beta_jr_j}{r}\norm{z-x}\le\beta_jr_j.
\]
By \eqref{eq:fixedplanedyadic}, there is
$p\in E\cap\overline B(x,r_j)$ with
$\norm{p-z'}\le\beta_jr_j$. The available radial margin gives
\[
 \norm{p-x}\le\norm{p-z'}+\norm{z'-x}\le r,
 \qquad
 \norm{p-z}\le\norm{p-z'}+\norm{z'-z}
 \le2\beta_jr_j\le4\beta_jr.
\] Thus, for all $0<r\le r_J$,
\begin{equation}
d_H\bigl(E\cap\overline B(x,r),(x+V_x)\cap\overline B(x,r)\bigr)
\le r\omega(r),\qquad
\omega(r)=52T_j\quad(r_{j+1}<r\le r_j).
\label{eq:allradiilimit}
\end{equation}
The function $\omega$ is nondecreasing as a function of $r$ and tends to zero at zero. Contracting the disk before finding $p$ is essential: simply restricting a larger-ball Hausdorff bound would not guarantee that the selected point stays in the smaller ball.

\paragraph{Step 4: prove injectivity and existence of a local graph.}
Translate $x_0$ to zero. Let $P_0=P_{x_0}$ and $Q_0=I-P_0$. By continuity and \eqref{eq:allradiilimit}, choose $\rho>0$ so that $\overline B(0,2\rho)\subset B(0,2R)$, $4\rho\le r_J$, and
\[
\norm{P_x-P_0}\le1/16\quad(x\in E\cap\overline B(0,2\rho)),
\qquad \omega(t)\le1/16\quad(0<t\le4\rho).
\]
Put $K=E\cap\overline B(0,2\rho)$, which is compact. For distinct $x,y\in K$, use \eqref{eq:allradiilimit} at radius $\norm{y-x}$ to obtain
\begin{equation}
\norm{Q_0(y-x)}\le\tfrac18\norm{y-x},\qquad
\norm{P_0(y-x)}\ge c\norm{y-x},\quad c=\sqrt{63}/8.
\label{eq:chordcone}
\end{equation}
Indeed, with $t=\norm{y-x}$, the normal component is bounded by
\[
 \norm{(I-P_0)(y-x)}\le\norm{(I-P_x)(y-x)}+\norm{P_x-P_0}\,t
 \le(\omega(t)+1/16)t\le t/8.
\]
Orthogonality then gives $\norm{P_0(y-x)}^2=t^2-\norm{Q_0(y-x)}^2\ge63t^2/64$. In particular $P_0$ is injective on $K$.

We now prove that its image contains the open $d$-ball
$D=\{u\in V_{x_0}:\norm u<\rho/2\}$. Fix $u\in D$ and minimize $\norm{P_0y-u}$ over $y\in K$. At $y=0$ its value is $\norm u<\rho/2$. If $\norm y=2\rho$, \eqref{eq:chordcone} gives a value at least $2c\rho-\norm u>\rho/2$. Every minimizer is therefore in the interior of the ambient ball.

Suppose its residual $\delta=\norm{P_0y-u}$ is positive. Since $\norm{P_y-P_0}<1$, the map $P_0:V_y\to V_{x_0}$ is a linear isomorphism: for $v\in V_y$, $\norm{Q_0v}\le\norm{P_y-P_0}\norm v$, so $\norm{P_0v}\ge\sqrt{1-1/256}\norm v$. Choose $v\in V_y$ with $P_0v=u-P_0y$; then $\norm v\le2\delta$. Choose $0<\alpha\le1$ sufficiently small that $s=\alpha\norm v<2\rho-\norm y$ and $s\le4\rho$. The reverse bound \eqref{eq:allradiilimit}, applied to $y+\alpha v$, supplies $p\in E\cap\overline B(y,s)\subset K$ with
\[
\norm{p-(y+\alpha v)}\le s/16.
\]
The exact projected displacement is
\[
 P_0p-u=(1-\alpha)(P_0y-u)
             +P_0\{p-(y+\alpha v)\},
\]
because $P_0v=u-P_0y$. Since $\alpha\le1$ and $\norm{P_0}=1$,
\begin{align*}
 \norm{P_0p-u}
 &\le(1-\alpha)\delta+\norm{p-(y+\alpha v)}\\
 &\le(1-\alpha)\delta+\frac{\alpha\norm v}{16}
 \le(1-\alpha)\delta+\frac{\alpha\delta}{8}
 =(1-7\alpha/8)\delta<\delta,
\end{align*}
contradicting minimality. Therefore the minimum is zero. Injectivity gives a unique $F(u)\in K$ with $P_0F(u)=u$. By \eqref{eq:chordcone}, $F$ is $c^{-1}$-Lipschitz and
\[
F(u)=u+g(u),\qquad g(u)=Q_0F(u),\qquad
\norm{g(u)-g(v)}\le\frac{1}{\sqrt{63}}\norm{u-v}.
\]
For the open cylinder $U=\{z:\norm{P_0z}<\rho/2,\ \norm{Q_0z}<\rho/2\}$, every point of $E\cap U$ belongs to $K$ and must be $F(P_0z)$. Conversely, the displayed Lipschitz bound and $g(0)=0$ place every $F(u)$, $u\in D$, in $U$. Hence
\begin{equation}
E\cap U=\{u+g(u):u\in D\}.
\label{eq:localgraphconstructed}
\end{equation}
This equality, including surjectivity over an open disk, has been proved; it was not a support-model assumption.

\paragraph{Step 5: prove continuous differentiability.}
Each $V_x$, $x\in E\cap U$, is the graph of a linear map $A_x:V_{x_0}\to V_{x_0}^{\perp}$, since $P_0|_{V_x}$ is invertible. The maps $A_x$ are continuous in $x$ and have norm less than one. For example,
\[
A_x=Q_0P_x|_{V_{x_0}}\bigl(P_0P_x|_{V_{x_0}}\bigr)^{-1},
\]
Writing $B_x=P_0P_x|_{V_{x_0}}$, we have $\norm{B_x-I}\le1/16$, so the convergent Neumann series gives $B_x^{-1}=\sum_{m\ge0}(I-B_x)^m$ and $\norm{B_x^{-1}}\le16/15$. Also $\norm{Q_0P_x|_{V_{x_0}}}\le1/16$, hence $\norm{A_x}\le1/15$. Finally
\[
 B_x^{-1}-B_y^{-1}=B_x^{-1}(B_y-B_x)B_y^{-1},\qquad
 \norm{B_x^{-1}-B_y^{-1}}\le(16/15)^2\norm{P_x-P_y}.
\]
Substitution in the formula for $A_x$ proves
$\norm{A_x-A_y}\le[16/15+(1/16)(16/15)^2]\norm{P_x-P_y}$, establishing the stated continuity without assuming it.
The bracket equals $256/225$. After the derivative is identified
below, the $c^{-1}$-Lipschitz bound on $F$ gives
\[
\|u-v\|\le cr_j/4\ \Longrightarrow\
\|Dg(u)-Dg(v)\|\le\frac{256}{225}\,48T_j.
\]
This specifies a continuous-derivative modulus from the tails $T_j$.
Fix $u\in D$ and take $h\in V_{x_0}$ sufficiently small that
$u+h\in D$. Write $x=F(u)$ and $w=F(u+h)-F(u)$. We have
$P_0w=h$, $Q_0w=g(u+h)-g(u)$, and $\norm w\le c^{-1}\norm h$.
Every $v\in V_x$ satisfies $Q_0v=A_xP_0v$, so the map
$Q_0-A_xP_0$ vanishes on $V_x$ and has norm at most
$1+\norm{A_x}\le2$. In particular,
\[
 g(u+h)-g(u)-A_xh
 =(Q_0-A_xP_0)(I-P_x)w.
\]
For small enough $h$, the radius $\norm w$ lies in the range of
\eqref{eq:allradiilimit}; its forward direction at center $x$
gives $\dist(w,V_x)\le\omega(\norm w)\norm w$.
Using also monotonicity of $\omega$ yields
\begin{align*}
\norm{g(u+h)-g(u)-A_xh}
&\le2\dist(w,V_x)\\
&\le2\omega(\norm w)\norm w\\
&\le2c^{-1}\omega(c^{-1}\norm h)\norm h=o(\norm h).
\end{align*}
Thus $Dg(u)=A_{F(u)}$, and this derivative is continuous. Equation~\eqref{eq:localgraphconstructed} is consequently a $C^1$ embedded graph over the open set $D$, without boundary there. Step~6 identifies its tangent explicitly.
\paragraph{Step 6: identify the tangent from the constructed chart.}
For $x=F(u)$, $DF(u)h=h+A_xh$ has rank $d$ because $P_0DF(u)h=h$,
and its range is $V_x$. The proved Taylor expansion and lower bound are
\[
F(u+h)-F(u)=DF(u)h+o(\|h\|),\qquad
\|F(u+h)-F(u)\|\ge\|h\|.
\]
The distance of any secant from $V_x$, divided by its length, tends to
zero; every limiting tangent vector therefore lies in $V_x$.
Conversely, $u+th\in D$ for sufficiently small positive and negative $t$,
and $(F(u+th)-F(u))/t\to DF(u)h$. Thus every vector of $V_x$ is tangent.
In the convention of \citet{federer1959,rataj2017},
$\Tan(E,x)=V_x$ locally. The next covariance proof uses this established
tangent, rather than an assumed chart.

\end{proof}

\paragraph{Scope of the replacement.}
The proof derives local charts, tangents, and their continuity from \eqref{eq:dyadicregularity}. It does not derive \eqref{eq:dyadicregularity} from Ahlfors regularity, a one-scale SVD residual, or finitely many sampled centers. A finite set cannot approximate a positive-dimensional disk to vanishing relative error at all sufficiently small scales. The condition concerns the true closed support, uniformly over its points. A finite certificate for it would require a justified analytic continuation or tail bound and a transfer from observations to all support centers. Neither is supplied by the current experiments. This distinction is consistent with Proposition~\ref{prop:spur} and with the non-$C^1$ examples discussed by \citet{kafer2019}.

\section{Proof of the validation-mass theorem}\label{app:mass}
\begin{proof}
For completeness, we prove Theorem~\ref{thm:tube} from the empirical-distribution inequality of \citet{dkw1956} with the sharp constant of \citet{massart1990}. Let $\mathcal T$ be the entire training information, so $M$ is $\mathcal T$-measurable and the validation latent points are independent of $\mathcal T$. Conditional on $\mathcal T$, the variables $R_i=\dist(Z_i^v,M)$ are iid with cdf $F_M$.
For Borel sets $A_1,\ldots,A_n$, independence of validation from
training gives
\[
\Pr(R_i\in A_i\text{ for all }i\mid\mathcal T)
=\prod_{i=1}^n\mu\{z:\dist(z,M)\in A_i\}.
\]
The right side may depend on the training-chosen $M$. This is the
product law conditional on training, not conditional on the observed
validation distances. Set
\[
 F_{n,M}(w)=\frac1n\sum_{i=1}^n\mathbf1\{R_i\le w\},\qquad
 \mathcal D_M=\{\sup_{w\in\R}|F_{n,M}(w)-F_M(w)|\le u\}.
\]
The imported DKW--Massart theorem states, for every cdf (including atoms),
\[
 \Prob(\mathcal D_M^c\mid\mathcal T)\le2e^{-2nu^2}=\delta_V.
\]
The threshold $u$ is deterministic, so the tower property gives
\[
 \Prob(\mathcal D_M^c)
 =\E\bigl[\Prob(\mathcal D_M^c\mid\mathcal T)\bigr]
 \le\E[\delta_V]=\delta_V.
\] We use this named theorem as an input; we do not attribute its proof or sharp constant to the present work.

For every nonempty set $M$, $\dist(z,M)\le\norm{z-v}+\dist(v,M)$ follows by taking the infimum in $\norm{z-m}\le\norm{z-v}+\norm{v-m}$. Interchanging $z,v$ proves
$|\dist(z,M)-\dist(v,M)|\le\norm{z-v}$. On $\mathcal N=\{\max_i\norm{V_i-Z_i^v}\le a_v\}$, if
$R_i>t+a_v$, the Lipschitz bound gives
\[
 \dist(V_i,M)\ge R_i-\norm{V_i-Z_i^v}
 \ge R_i-a_v>t.
\]
Since $\bar d_i\ge\dist(V_i,M)$ is a verified upper bound,
$\bar d_i>t$ follows as well. This proves the following indicator
inequalities pointwise, for every $t\ge0$:
\[
 \mathbf1\{R_i>t+a_v\}
 \le\mathbf1\{\dist(V_i,M)>t\}
 \le\mathbf1\{\bar d_i>t\}.
\]
On $\mathcal D_M\cap\mathcal N$, the following entire calculation holds simultaneously in $t$:
\begin{align*}
 \mu\{z:\dist(z,M)>t+a_v\}
 &=1-F_M(t+a_v)\\
 &\le1-F_{n,M}(t+a_v)+u\\
 &=\frac1n\sum_i\mathbf1\{R_i>t+a_v\}+u
 \le\widehat p(t)+u.
\end{align*}
Probability is also at most one, giving the clipping in \eqref{eq:tube}. Finally
\[
 \Prob((\mathcal D_M\cap\mathcal N)^c)
 \le\Prob(\mathcal D_M^c)+\Prob(\mathcal N^c)\le\delta_V+\delta_N.
\]
The event is uniform in the real argument $w$, so evaluating it at random $w=t+a_v$ is legitimate. Neither $a_v$ nor the noise event needs to be independent of training or validation. Only the training choice of $M$ needs the specified independence.

If $M_1,\ldots,M_J$ are fixed by training, use
$u_J=\sqrt{\log(2J/\delta_V)/(2n)}$ and apply the bound to each cdf. The conditional failure probability for any candidate is at most
$\sum_{j=1}^J2e^{-2nu_J^2}=\delta_V$. The same $\mathcal N$ works for all candidates, so it is charged once. A measurable selection from this list after validation is allowed; an unrestricted refit is not covered.

For a preassigned latent radius $r_*$, on $a_v\le r_*$ set $t=r_*-a_v$. Then $t\ge0$ and $t+a_v=r_*$ identically, so the same event bounds $\mu\{\dist(z,M)>r_*\}$. If $a_v>r_*$ the rule abstains. Thus no additional probability budget is spent on this random threshold.
For any measurable reported $\widehat t\ge0$ the false-report event
has the pointwise inclusion
\[
\left\{\mu\{\dist(z,M)>\widehat t+a_v\}>
 \min(1,\widehat p(\widehat t)+u)\right\}
\subseteq\mathcal D_M^c\cup\mathcal N^c.
\]
The simultaneous event covers every real argument, including atoms
and the random value $\widehat t+a_v$.
For a training-fixed list use $\bigcap_j\mathcal D_{M_j}$ and $u_J$;
the same inclusion holds after selecting a list index by validation.
The strict tail is $1-F_M(w)$, so the distance law need not be continuous.

\end{proof}


\section{Exact-support obstruction and the dimension distinction}\label{app:obstruction}
\subsection{Proof of Proposition~\ref{prop:spur}}\label{app:spurproof}
\begin{proof}
For completeness, we prove this elementary indistinguishability construction directly. The two-sided regularity convention is that of \citet{davidsemmes1991}; neither a new minimax theorem nor a manifold consequence of that convention is claimed.
Let $Q_\epsilon$ be normalized length on the added segment.
The attachment is a single point of length zero, so
\[
P_\epsilon=w_\epsilon P_0+(1-w_\epsilon)Q_\epsilon,\qquad
w_\epsilon=\frac{2\pi}{2\pi+\epsilon}.
\]
For a probability measure $R_{\epsilon,N}$, expansion gives
\[
P_\epsilon^{\otimes N}
=w_\epsilon^NP_0^{\otimes N}
 +(1-w_\epsilon^N)R_{\epsilon,N}.
\]
Let $K(z_1,\ldots,z_N;A)$ be the conditional probability of the
observation/algorithm event $A$ given the latent sample. The common
noise mechanism and any independent algorithmic randomization give
the same kernel under both laws. Integrating this kernel proves
the observation-level decomposition used next, without requiring
the event to reveal the latent component labels.
Under $P_\epsilon$, an individual latent point belongs to the circle with probability $w=2\pi/(2\pi+\epsilon)$. Conditioned on all $N$ points belonging to the circle, their joint law is $P_0^{\otimes N}$. The repeated-observation noise law remains the same. Consequently $\Prob_{P_\epsilon}(A)\ge w^N\Prob_{P_0}(A)$. Writing $H$ for the event that all latent points lie on the circle, the full decomposition is $\Prob_{P_\epsilon}(A)=w^N\Prob_{P_0}(A)+(1-w^N)\Prob_{P_\epsilon}(A\mid H^c)$; its second term is nonnegative. The attachment point has three local branches, so $E_\epsilon$ is not an embedded one-dimensional manifold there. A sufficiently small punctured neighborhood of the attachment has three connected components, whereas an interval has two and a half-interval has one. It cannot be a positive-dimensional embedded $C^1$ manifold of a different dimension either: its Hausdorff dimension is one, and a $C^1$ chart with full-rank derivative is locally bi-Lipschitz after shrinking, so its Hausdorff dimension equals its chart dimension. It is not zero-dimensional because its points are not isolated. If $A$ asserts exact smooth support and its probability is at most $\delta$ under every $P_\epsilon$, passing to $\epsilon\downarrow0$ gives $\Prob_{P_0}(A)\le\delta$.

The examples also satisfy fixed small-scale mass bounds. For $0<t\le1/4$, from every point a path of length $t$ can be followed in the union without retracing; its Euclidean distance from the starting point is at most its arc length. Thus the ball contains at least length $t$. Total length is at most $2\pi+1/4<8$, giving mass at least $t/8$. For an upper bound, a ball of radius $t\le1/4$ cuts at most $4\arcsin(t)$ length from the unit circle and at most $2t$ from the segment. To justify the circle bound for an arbitrary ball center, choose any circle point in the ball; every other circle point in it is within chord distance $2t$ of that point, an arc of length $4\arcsin(t)$. Also $\arcsin(t)=\int_0^t(1-u^2)^{-1/2}du\le2t$ for $t\le1/4$. This gives total length at most $10t$, and division by at least $2\pi$ gives mass at most $2t$. Hence uniform constants $c=1/8$, $C=2$, $t_0=1/4$ suffice. Ahlfors-type control alone does not eliminate the small-branch alternatives.

If branches are restricted to length at least $\epsilon$ and circle issuance must be at least $1-\beta$, the same argument requires
\[
N\ge\frac{\log((1-\beta)/\delta)}{\log(1+\epsilon/(2\pi))}.
\]
Indeed $\delta\ge w^N(1-\beta)$ implies $N\log(1+\epsilon/(2\pi))\ge\log((1-\beta)/\delta)$; the denominator is positive. If $1-\beta\le\delta$, the lower bound is nonpositive and provides no restriction. This is a necessary condition, not an achieved sample-complexity result. Replicates at the same latent points do not change it.
Subtracting the all-circle probability from the mixture also gives
\[
|\Pr_{P_\epsilon}(A)-\Pr_{P_0}(A)|\le1-w_\epsilon^N.
\]
The remaining conditional probabilities lie in $[0,1]$, which proves
this inequality directly. At $a=(1,0)$,
\[
\Tan(E_\epsilon,a)
=\{t(0,1):t\in\mathbb R\}\cup\{t(1,0):t\ge0\}.
\]
This cone contains $(0,1)$ and $(1,0)$ but not their midpoint.
It is nonconvex. Positive reach requires a convex tangent cone,
as credited to Federer in \citet[Proposition~3.1(ii)]{rataj2017};
hence $\reach(E_\epsilon)=0$.
This explicitly verifies the failure of transferring witness reach
to a Hausdorff-close support.
If $1-\beta>\delta$, $\log(1+\epsilon/(2\pi))\le\epsilon/(2\pi)$
also gives the weaker necessary bound
$N\ge(2\pi/\epsilon)\log((1-\beta)/\delta)$.

\end{proof}

\subsection{Proof of Corollary~\ref{cor:boundary}: verified boundary witnesses}\label{app:boundaryproof}
\begin{proof}
For completeness, we prove Corollary~\ref{cor:boundary} explicitly. We use the reach convention of \citet{federer1959} and the DKW--Massart input \citep{dkw1956,massart1990}; no additional boundary regularity theorem is invoked.
The probability event in Theorem~\ref{thm:main} concerns only $E$ and the recorded set $Y$.
If a deterministic verifier establishes that a proposed $M$ belongs to the enlarged witness class and that
$d_H(Y,M)\le e$, then on the same event
\[
d_H(E,M)\le d_H(E,Y)+d_H(Y,M)\le h+a+e.
\]
This calculation does not use the absence of a boundary. The verifier, rather than the statistical argument,
must establish the candidate's smoothness, embeddedness, and reach. For the mass result,
$|\dist(z,M)-\dist(v,M)|\le\norm{z-v}$ holds for every nonempty closed $M$.
The training conditional DKW argument and its simultaneous threshold statement are therefore unchanged.
No latent manifold premise has been added. This extension does not extend the boundary-free
Reifenberg conclusion at an endpoint or prove the all-scale hypotheses from data.
In event notation, sound geometric verification and $\mathcal G=\{d_H(E,Y)\le h+a\}$ give
\[
 \{\EMIT_{\partial},\ d_H(E,M)>h+a+e\}\subseteq\mathcal G^c,
 \qquad \Prob(\mathcal G^c)\le\delta.
\]
For a training-fixed boundary witness, the distance implication is still
$\{\dist(Z_i^v,M)>t+a_v\}\subseteq\{\bar d_i>t\}$ on the noise event. Substitution in Appendix~\ref{app:mass}'s displayed cdf calculation gives the same $\delta_N+\delta_V$ bound. These two inclusions prove both conclusions of the corollary.
Let $\mathcal W_\partial(M)$ denote membership in the verified
class: compact embedded $C^2$ manifolds, possibly with boundary,
with reach at least the reported positive lower bound.
Soundness means, on every sample path,
$\EMIT_\partial\subseteq\{\mathcal W_\partial(M),d_H(Y,M)\le e\}$.
Thus the full support report satisfies
\[
\{\EMIT_\partial,\ \neg\mathcal W_\partial(M)
 \ \text{or}\ d_H(E,M)>h+a+e\}\subseteq\mathcal G^c.
\]
For a training-fixed boundary witness the full validation report satisfies
\begin{align*}
&\{\EMIT_\partial^{\,v},\ \neg\mathcal W_\partial(M)
 \ \text{or}\ \mu\{\dist(z,M)>t+a_v\}>
      \min(1,\widehat p(t)+u)\}\\
&\hspace{20mm}\subseteq\mathcal D_M^c\cup\mathcal N^c.
\end{align*}
Geometry is supplied by the sound verifier, not inferred from either
probability inequality.

\end{proof}

\subsection{Proof of Proposition~\ref{prop:covariance}: the covariance bridge}\label{app:dimension}
The statement is now in Section~\ref{sec:dimconsistency}. Its proof uses the $C^1$ regularity derived by Theorem~\ref{thm:localexistence}; a $C^2$ premise or a reach lower bound is unnecessary.

\begin{proof}
This is the local-PCA mechanism underlying the analyses of \citet{little2017,aamari2019,lim2024}; we prove it for completeness under the precise hypotheses stated here. Theorem~\ref{thm:localexistence} first establishes the local manifold. For a fixed $x$ in that neighborhood, shrink the chart and rotate the proved tangent to $\R^d\times\{0\}$, translating $x$ to zero. The graph is $\Psi(u)=(u,g(u))$, where $g(0)=0$ and $Dg(0)=0$. The inverse function theorem applied to tangent projection justifies this chart change because its derivative on the tangent is an isomorphism. For all sufficiently small $r$, the entire set $E\cap\overline B(0,r)$ lies in this chart.

\paragraph{Step 1: write the probability and moments as integrals.}
Let $B_d$ be the unit ball in $\R^d$, $v_d=|B_d|$, and set
\[
\begin{aligned}
 D_r&=\{v\in B_d:\norm v^2+\norm{g(rv)/r}^2\le1\},\\
 W_r(v)&=(v,g(rv)/r),\\
 J_r(v)&=\sqrt{\det(I_d+Dg(rv)^{\mathsf T}Dg(rv))}.
\end{aligned}
\]
The area formula for a $C^1$ graph, credited to \citet{federergmt}, uses normalized Hausdorff measure so that a Euclidean $d$-plane has ordinary $d$-dimensional volume.
The Jacobian here is
\[
D\Psi(u)=\begin{pmatrix}I_d\\Dg(u)\end{pmatrix},\qquad
D\Psi(u)^{\mathsf T}D\Psi(u)=I_d+Dg(u)^{\mathsf T}Dg(u).
\]
Under $u=rv$, volume contributes $r^d$, displacement contributes
$rW_r(v)$, and its outer product contributes
$r^2W_r(v)W_r(v)^{\mathsf T}$.
The unnormalized probability, first moment and second moment are
respectively $r^dA_r$, $r^{d+1}b_r$ and $r^{d+2}H_r$.
Dividing by the first quantity and subtracting the mean outer product
gives the scaled covariance below. Its application gives
\begin{align*}
 A_r&:=r^{-d}\mu(\overline B(0,r))
     =\int_{B_d}\mathbf1_{D_r}(v)f(\Psi(rv))J_r(v)\,dv,\\
 b_r&:=\int_{B_d}\mathbf1_{D_r}(v)W_r(v)f(\Psi(rv))J_r(v)\,dv,\\
 H_r&:=\int_{B_d}\mathbf1_{D_r}(v)W_r(v)W_r(v)^{\mathsf T}f(\Psi(rv))J_r(v)\,dv.
\end{align*}
Consequently the scaled conditional mean is $b_r/A_r$ and
\[
 r^{-2}C_x(r)=\frac{H_r}{A_r}-\frac{b_rb_r^{\mathsf T}}{A_r^2}.
\]
The centering term is retained; an uncentered second moment is not silently substituted for covariance.

\paragraph{Step 2: justify the limits and normalization.}
Put $a(r)=\sup_{\norm u\le r}\norm{Dg(u)}$. By continuity, $a(r)\to0$, and
\[
 g(rv)=\int_0^1Dg(trv)\,rv\,dt,
 \qquad \norm{g(rv)/r}\le a(r)\norm v.
\]
Hence $W_r(v)\to(v,0)$ uniformly on $B_d$; also $1\le J_r(v)\le(1+a(r)^2)^{d/2}\to1$ uniformly. Continuity of $f$ gives $f(\Psi(rv))\to f(x)>0$ uniformly. If $\norm v<1$, then $\norm v^2(1+a(r)^2)<1$ for small $r$, so $v\in D_r$ eventually. The exceptional sphere has Lebesgue measure zero. The integrands are bounded by fixed integrable constants (on $D_r$, $\norm{W_r}\le1$). Dominated convergence therefore gives
\[
 A_r\to f(x)v_d>0,\qquad b_r\to f(x)\int_{B_d}(v,0)\,dv=0,
\]
and $H_r\to f(x)\int_{B_d}(v,0)(v,0)^{\mathsf T}\,dv$. In particular $A_r\ge f(x)v_d/2$ for small $r$, which justifies each division.
For the actual ambient-ball domain, the derivative bound gives
\[
(1+a(r)^2)^{-1/2}B_d\subseteq D_r\subseteq B_d,\qquad
|B_d\setminus D_r|\le v_d\{1-(1+a(r)^2)^{-d/2}\}\to0.
\]
The eigenvalues of $Dg^{\mathsf T}Dg$ are between zero and $a(r)^2$;
multiplying its $d$ factors verifies the Jacobian bound.
The boundary sphere has volume zero: for every $0<\varepsilon<1$ it
is contained in an annulus of volume
$v_d[(1+\varepsilon)^d-(1-\varepsilon)^d]\to0$.
These calculations establish the domain limit separately from the
uniform integrable bounds used in dominated convergence.

Reflection symmetry makes every off-diagonal integral zero. Rotation symmetry makes the first $d$ diagonal integrals equal, and polar coordinates give
\[
 \int_{B_d}\norm v^2\,dv
 =dv_d\int_0^1t^{d+1}\,dt=\frac{dv_d}{d+2}.
\]
Each tangent diagonal integral is $v_d/(d+2)$, while normal entries vanish. This disk identity is explicitly recorded in \citet[Lemma~13]{arias2017} and \citet[Lemma~6.1]{lim2024}. Returning to the original coordinates proves
$r^{-2}C_x(r)\to P_x/(d+2)$. Entrywise convergence here implies operator-norm convergence because $q$ is fixed, for example $\norm A_{\rm op}\le q\max_{a,b}|A_{ab}|$.
One can retain an explicit finite-radius remainder without an
assumed convergence rate. In tangent coordinates put
$a_0=f(x)v_d$ and $T_0=\operatorname{diag}(I_d,0)/(d+2)$.
For $A_r\ge a_0/2$,
\begin{align*}
\left\|\frac{H_r}{A_r}-T_0\right\|_{\rm op}
&\le\frac{\|H_r-a_0T_0\|_{\rm op}}{A_r}
  +\frac{|a_0-A_r|}{A_r}\|T_0\|_{\rm op},\\
\theta_x(r)
&\le\frac{2}{a_0}\|H_r-a_0T_0\|_{\rm op}
 +\frac{2|A_r-a_0|}{a_0(d+2)}
 +\frac{4\|b_r\|^2}{a_0^2}.
\end{align*}
Each term tends to zero by the preceding integrals.
The last term is the retained centering correction.
Conjugation by the fixed orthogonal coordinate change preserves the
operator norm and restores the original tangent projection.
The quantities in this display are population proof quantities.

\paragraph{Step 3: quantify the eigenvalue separation.}
For symmetric $A,B$, write $e=\norm{A-B}_{\rm op}$. Every unit $v$ obeys
$|v^{\mathsf T}(A-B)v|\le e$. The spectral variational formula is
\[
 \lambda_j(A)=\max_{\dim L=j}\ \min_{\substack{v\in L\\\norm v=1}}v^{\mathsf T}Av.
\]
To verify it, choose the span of the top $j$ eigenvectors for the lower bound. Every $j$-plane intersects the span of eigenvectors $j,\ldots,q$ nontrivially, since their dimensions sum to $q+1$; a unit vector in that intersection has Rayleigh quotient at most $\lambda_j(A)$. Applying the unit-vector bound inside the maximum and minimum, and then interchanging $A,B$, proves
\[
 |\lambda_j(A)-\lambda_j(B)|\le\norm{A-B}_{\rm op}.
\]
This supplies directly the eigenvalue perturbation step often called Weyl's inequality.

Set $a_d=1/(d+2)$ and
\[
 e_n=\frac{\norm{\widehat C_n-C_x(r_n)}_{\rm op}}{r_n^2}
       +\norm{r_n^{-2}C_x(r_n)-a_dP_x}_{\rm op},
 \qquad \Delta=\tfrac12\min\{\gamma,a_d-\gamma\}>0.
\]
The strict positivity follows from $0<\gamma<1/(q+2)\le a_d$. On $e_n<\Delta$, for $j\le d$,
$\lambda_j(\widehat C_n)/r_n^2\ge a_d-e_n>\gamma$, whereas for $j>d$ this quantity is at most $e_n<\gamma$. Negative estimated eigenvalues cause no difficulty. Therefore
\[
 \{\widehat d_n\ne d\}\subseteq\{e_n\ge\Delta\}.
\]
In probability, $e_n\to0$ implies $\Prob(\widehat d_n\ne d)\to0$. Almost surely, convergence supplies a finite random index after which $e_n<\Delta$ always, so $\widehat d_n=d$ eventually. The assumed estimator error is essential and has not been deduced for noisy-neighborhood selection by this calculation.
\end{proof}


This familiar local-PCA mechanism is related to the finite-sample analyses of \citet{little2017,aamari2019,lim2024}. The logical order is now geometric flatness, proved local manifold existence, a density condition, and the covariance limit. Theorem~\ref{thm:dimconsistent} now proves the estimator-error requirement for the specific observed-repeat estimator in \eqref{eq:actualdimension}. It is not asserted for an arbitrary noisy-neighborhood estimator. The historical circle and torus experiments do not validate that generalization. The annulus example shows why Theorem~\ref{thm:main}'s approximate conclusion cannot replace the all-scale geometric conditions.

\section{Complete calculations for observed intrinsic dimension}
\label{app:dimproof}
The arguments below use the existing geometric existence theorem and the proved covariance bridge. They replace an assumed statistical convergence condition by bounds for the actual statistic in \eqref{eq:actualdimension}. Gaussian concentration is credited to \citet{laurent2000,blm2013}, the Gaussian residual law to \citet{wishart1928}, and matrix Bernstein to \citet[Theorem~6.1]{tropp2012}. For completeness, we prove the specializations and deterministic calculations used here; Tropp's general matrix theorem is an explicitly cited input. We also retain the elementary Hoeffding alternative \citep{hoeffding1963} for comparison. The local-PCA and repeated-averaging antecedents remain those identified in Section~\ref{sec:dimconsistency}.

\subsection{Centered covariance from one matrix Bernstein bound}
\label{app:dimbounded}
Let $T_1,\ldots,T_m$ be iid in the closed unit ball of $\R^q$, with $m\ge1$. Write $b=\E T$, $M=\E TT^{\mathsf T}$, $C=M-bb^{\mathsf T}$, $\bar T=m^{-1}\sum_iT_i$, and $S_m=m^{-1}\sum_iT_iT_i^{\mathsf T}$. The empirical covariance is $\widehat C_T=S_m-\bar T\bar T^{\mathsf T}$. Augment each observation by one coordinate:
\[
w_i=\begin{pmatrix}T_i\\1\end{pmatrix},\quad K_i=w_iw_i^{\mathsf T},\quad
A=\E K_i=\begin{pmatrix}M&b\\b^{\mathsf T}&1\end{pmatrix},\quad D_i=K_i-A.
\]
The $D_i$ are independent, self-adjoint, and mean zero. Since $\norm{w_i}^2\le2$, we have $0\preceq K_i,A\preceq2I_{q+1}$ and hence $\norm{D_i}_{\rm op}\le2$. Furthermore,
\begin{align*}
K_i^2&=\norm{w_i}^2K_i\preceq2K_i,\\
\E D_i^2&=\E K_i^2-A^2
\preceq2A-A^2=I_{q+1}-(A-I_{q+1})^2\preceq I_{q+1}.
\end{align*}
Here $\E(K_iA)=\E(AK_i)=A^2$ justifies the variance identity without assuming that individual $K_i$ commute with $A$.
Indeed,
\[
\E(K_i-A)^2
=\E K_i^2-(\E K_i)A-A(\E K_i)+A^2
=\E K_i^2-A^2.
\]
This expectation is positive semidefinite. Its upper bound by $I_{q+1}$
therefore gives $\|\sum_{i=1}^m\E D_i^2\|_{\rm op}\le m$,
verifying the variance hypothesis separately from the summand bound. Thus the matrix variance parameter is at most $m$. Apply Tropp's theorem to both $\sum_iD_i$ and $-\sum_iD_i$, with summand bound 2 and dimension $q+1$, to obtain
\begin{equation}
\Prob\left\{\left\|\frac1m\sum_iD_i\right\|_{\rm op}>u\right\}
\le2(q+1)\exp\left(-\frac{mu^2}{2+4u/3}\right),\qquad u>0.
\label{eq:dimaugmentedtail}
\end{equation}
For $t=\log(2(q+1)/\delta)$, set $u=\sqrt{2t/m}+4t/(3m)$. Writing $z=t/m$, direct expansion gives
$u^2-z(2+4u/3)=(4/3)z\sqrt{2z}\ge0$; the failure probability is at most $\delta$.
The upper-left and upper-right blocks of $m^{-1}\sum_iD_i$ are $S_m-M$ and $\bar T-b$, respectively. Compression by coordinate projections bounds each block norm by the full matrix norm. On the same event,
\[
\norm{S_m-M}_{\rm op}\le u,\qquad \norm{\bar T-b}\le u.
\]
Also $\norm{\bar T},\norm b\le1$, and
\[
\bar T\bar T^{\mathsf T}-bb^{\mathsf T}
=(\bar T-b)\bar T^{\mathsf T}+b(\bar T-b)^{\mathsf T}.
\]
The centering error is therefore at most $2u$.
The tail inversion can be checked explicitly. For $z=t/m$, the
needed inequality is $u^2-(4z/3)u-2z\ge0$. Its nonnegative root obeys
\[
\frac{2z}{3}+\sqrt{2z+\frac{4z^2}{9}}
\le\sqrt{2z}+\frac{4z}{3}=u,
\]
using $\sqrt{a+b}\le\sqrt a+\sqrt b$. Thus $3u$ is exactly
$3\sqrt{2t/m}+4t/m=\mathcal B_q(m,\delta)$. We have proved
\begin{equation}
\Prob\{\norm{\widehat C_T-C}_{\rm op}\le\mathcal B_q(m,\delta)\}\ge1-\delta,
\label{eq:dimbernstein}
\end{equation}
with $\mathcal B_q$ defined in \eqref{eq:dimbernsteinradius}. Equivalently, without inverting the tail,
\begin{equation}
\Prob\{\norm{\widehat C_T-C}_{\rm op}>\epsilon\}
\le2(q+1)\exp\left(-\frac{m\epsilon^2}{18+4\epsilon}\right).
\label{eq:dimcovariancetail}
\end{equation}
This follows by substituting $u=\epsilon/3$ in \eqref{eq:dimaugmentedtail}. The empirical normalization is $1/m$, and its finite-sample centering bias has been bounded, not assumed absent.
Independence also gives the exact bias:
\[
\E(\bar T\bar T^{\mathsf T})
=\frac{M}{m}+\frac{m-1}{m}bb^{\mathsf T},\qquad
\E\widehat C_T=\left(1-\frac1m\right)C.
\]
Separate the $m$ diagonal terms and $m(m-1)$ off-diagonal terms in
$m^{-2}\sum_{i,j}\E T_iT_j^{\mathsf T}$ to obtain the first identity.
At $m=1$ the covariance is zero as required. Concentration controls
this bias together with fluctuation; the divisor is never changed to $m-1$. The calculation is a specialization of the credited concentration theorem, not a new matrix inequality.

\paragraph{Elementary Hoeffding comparison.}
If a real variable $W$ takes values in $[-1,1]$, define
$\psi(t)=\log\E\exp(t(W-\E W))$. Under exponential tilting, $\psi''(t)$ is the variance of a variable still supported in $[-1,1]$. Such a variance is at most one: $\operatorname{Var}(V)=\min_c\E(V-c)^2\le\E V^2\le1$. Since $\psi(0)=\psi'(0)=0$, twice integrating gives $\psi(t)\le t^2/2$ for positive or negative $t$. Independence, Markov's inequality, and the choice $t=u$ yield
\[
\Prob\left\{\frac1m\sum_{i=1}^m(W_i-\E W_i)>u\right\}
\le\inf_{t>0}\exp(-mtu+mt^2/2)=\exp(-mu^2/2).
\]
Applying this to $-W$ gives the two-sided bound $2e^{-mu^2/2}$. This proves the scalar inequality used in the following comparison; it is not a new concentration result.

Let $T_1,\ldots,T_m$ be iid in the closed unit ball of $\R^q$. Write $\bar T=m^{-1}\sum_iT_i$, $b=\E T$, $H=m^{-1}\sum_iT_iT_i^{\mathsf T}$ and $H_0=\E TT^{\mathsf T}$. Apply the preceding scalar bound to each of $q$ first moments and $q^2$ second moments. With
$u=\sqrt{2\log(2(q^2+q)/\delta)/m}$, a union bound gives, outside probability $\delta$,
\[
\norm{\bar T-b}\le\sqrt q\,u,\qquad
\norm{H-H_0}_{\rm op}\le q u.
\]
For the second inequality, if every entry of $D$ is at most $u$ in absolute value, then $\norm{Dv}^2\le q u^2(\sum_j|v_j|)^2\le q^2u^2\norm v^2$. Also $\norm{\bar T},\norm b\le1$ and
\[
\bar T\bar T^{\mathsf T}-bb^{\mathsf T}
=(\bar T-b)\bar T^{\mathsf T}+b(\bar T-b)^{\mathsf T}.
\]
Consequently
\begin{equation}
\left\|H-\bar T\bar T^{\mathsf T}-(H_0-bb^{\mathsf T})\right\|_{\rm op}
\le(q+2\sqrt q)\sqrt{\frac{2\log(2(q^2+q)/\delta)}m}.
\label{eq:dimhoeffding}
\end{equation}
The matrix bound removes the factor $q+2\sqrt q$ in this coordinatewise argument and introduces a term of order $\log((q+1)/\delta)/m$. This is a dimensional improvement in the displayed sufficient bound, not an optimality claim or a claim of numerical improvement for every small $m$. The algorithm uses \eqref{eq:dimbernsteinradius} alone.

\paragraph{Random local sample sizes.}
Fix $x\in\supp\mu$ and a deterministic $r>0$. For $J=\{1\le i\le n:Z_i\in\overline B(x,r)\}$, condition on its complete membership vector. Independence of the $Z_i$ gives the product law $\mu(\cdot\mid\overline B(x,r))$ for the selected variables: for measurable $A_i\subseteq\overline B(x,r)$, their conditional joint probability is $\prod_{i\in J}\mu(A_i)/\mu(\overline B(x,r))$. Thus \eqref{eq:dimbernstein} applies conditionally with $m=|J|$ to $(Z_i-x)/r$. The failure probability is at most $\delta$ for every nonempty membership vector, hence also after averaging over it. Defining the bound as infinite when $J$ is empty removes any division by zero. For a random reserved anchor, this argument is applied conditional on $Z_0=x$; the other latent groups remain iid. No independence is claimed for the observations selected by their noisy distances.

\subsection{Two deterministic covariance transfer inequalities}
\label{app:dimtransfer}
\paragraph{Moving sample locations.}
Suppose $z_i\in\overline B(x,R)$ and $y_i=z_i+e_i$ with $\norm{e_i}\le a$, over any nonempty fixed index set $I$. Denote empirical covariance with normalization $1/|I|$ by $C_I$. Let $C_{ze}=|I|^{-1}\sum(z_i-\bar z)(e_i-\bar e)^{\mathsf T}$. Expanding centered products gives
\[
C_I(y)-C_I(z)=C_{ze}+C_{ze}^{\mathsf T}+C_I(e).
\]
For unit vectors $u,v$, Cauchy--Schwarz gives
\[
|u^{\mathsf T}C_{ze}v|
\le\sqrt{u^{\mathsf T}C_I(z)u}\sqrt{v^{\mathsf T}C_I(e)v}
\le\sqrt{\operatorname{tr}C_I(z)\operatorname{tr}C_I(e)}\le Ra.
\]
Here $\operatorname{tr}C_I(z)\le |I|^{-1}\sum\norm{z_i-x}^2\le R^2$ and $\operatorname{tr}C_I(e)\le |I|^{-1}\sum\norm{e_i}^2\le a^2$; subtracting the mean only reduces these quantities. Therefore
\begin{equation}
\norm{C_I(y)-C_I(z)}_{\rm op}\le2Ra+a^2.
\label{eq:dimmoving}
\end{equation}
Nothing in this deterministic calculation assumes that $I$ is independent of the errors. It therefore applies to the noise-dependent neighborhood actually used by the estimator.

\paragraph{Changing sample membership.}
Let nonempty $I,J$ index points in $\overline B(x,R)$ and let $c=|I\cap J|$. For the uniform distributions on the indices, total variation is
\[
\operatorname{TV}(I,J)=1-\frac c{\max(|I|,|J|)}
\le\min\left\{1,\frac{|I\mathbin\triangle J|}{\min(|I|,|J|)}\right\}=:T.
\]
Indeed the common probability mass is $\sum_{i\in I\cap J}\min(1/|I|,1/|J|)=c/\max(|I|,|J|)$. Pushing indices to their possibly repeated locations cannot increase total variation. In coordinates centered at $x$, integrating a vector of norm at most $R$ against the signed measure difference gives a mean difference at most $2RT$. Similarly the second-moment difference has operator norm at most $2R^2T$. The two means have norm at most $R$, so their outer-product difference is at most $4R^2T$.
Writing $b_I,b_J$ for means relative to $x$ and $M_I,M_J$ for
second moments, the full subtraction is
\[
C_I-C_J=M_I-M_J-(b_I-b_J)b_I^{\mathsf T}
                         -b_J(b_I-b_J)^{\mathsf T}.
\]
Taking norms gives $2R^2T+(R+R)2RT=6R^2T$.
For location transfer the trace identity is
$\operatorname{tr}C_I(z)=|I|^{-1}\sum_i\|z_i-x\|^2-\|\bar z-x\|^2$;
the same identity for the error vectors yields their $a^2$ bound. It follows that
\begin{equation}
\norm{C_I(z)-C_J(z)}_{\rm op}\le6R^2T.
\label{eq:dimmembership}
\end{equation}
The constant is deliberately conservative. We do not claim the optimal finite-set transfer constant studied in the related work \citep{choi2026}.

\subsection{Proof of Proposition~\ref{prop:dimfinite}}
\label{app:dimfinite}
\begin{proof}
For completeness, we prove the observable bound by composing the preceding calculations with Theorem~\ref{thm:main}'s Gaussian event. Here is its exact random-matrix representation. Let $Q\in\R^{k\times(k-1)}$ have orthonormal columns perpendicular to $\mathbf1_k$, and let $\mathsf V_i=[X_{i1},\ldots,X_{ik}]$. Then
\[
\mathsf V_iQ=[\varepsilon_{i1},\ldots,\varepsilon_{ik}]Q,
\qquad QQ^{\mathsf T}=I_k-\mathbf1_k\mathbf1_k^{\mathsf T}/k.
\]
For $\sigma>0$, concatenate these matrices as $\mathsf R=[\mathsf V_0Q,\ldots,\mathsf V_nQ]=\sigma G$, where $G$ is $q\times s$, $s=(n+1)(k-1)$, with independent standard Gaussian entries. Orthogonal Gaussian coordinates show that these residuals are independent of all group averages and all latent points. By definition of the Wishart law \citep{wishart1928},
\[
\sigma^{-2}\mathsf R\mathsf R^{\mathsf T}\sim W_q(s,I_q),\quad
S_n=\operatorname{tr}(\mathsf R\mathsf R^{\mathsf T}),\quad
S_n/\sigma^2\sim\chi^2_{qs}.
\]
The trace identity follows from $QQ^{\mathsf T}$ above, so $qs=\nu_n$ is exactly the residual degrees of freedom in \eqref{eq:dimcalibration}. This also holds for a fixed anchor since its unknown mean cancels. At $\sigma=0$, $S_n=0$ and all averages equal their latent points. This use of Gaussian random matrices does not assert a Wishart law for the signal covariance or for noisy selected neighborhoods.

Outside probability $\delta_C+\delta_L$, the bound in \eqref{eq:dimcalibration} yields
\[
\max_{0\le i\le n}\norm{Y_i-Z_i}\le a.
\]
On this one event, with $x=Z_0$, the reverse triangle inequality gives
$|\norm{Y_i-Y_0}-\norm{Z_i-x}|\le2a$. Hence
\[
I^-\subseteq J\subseteq I^+,\qquad I^-\subseteq I\subseteq I^+,
\quad I\mathbin\triangle J\subseteq I^+\setminus I^-.
\]
Both $I$ and $J$ have at least $m_-$ elements, and every true point in their union lies in $\overline B(x,R)$: points in $I$ satisfy $\norm{Z_i-x}\le r+2a$, while points in $J$ satisfy the smaller radius $r$. When $m_->0$, apply \eqref{eq:dimmoving} to $I$, then \eqref{eq:dimmembership} to $I,J$, to obtain
\[
\norm{\widehat C_n-C_J(Z)}_{\rm op}
\le2Ra+a^2+6R^2\min\{1,B/m_-\}.
\]
The population-sampling event from \eqref{eq:dimbernstein}, conditional on the latent membership vector and the anchor, has failure at most $\delta_P$. On it,
\[
\norm{C_J(Z)-C_x(r)}_{\rm op}
\le r^2\mathcal B_q(|J|,\delta_P).
\]
Replacing $|J|$ by its lower bound $m_-$ and applying the triangle inequality proves \eqref{eq:dimfinitebound}. A union bound adds the three failure probabilities. We never condition the iid population-sampling argument on the noisy selection event, so independence after noisy selection is not being asserted. If $m_-=0$, the claimed infinite bound holds trivially.
\paragraph{Probability bookkeeping with random observable counts.}
Let $\mathcal G_P$ include $J=\varnothing$ and otherwise assert the
oracle covariance bound with $|J|$. Conditional on $Z_0=x$ and the
latent membership vector, its failure is at most $\delta_P$ for every
nonempty vector. Averaging gives $\Pr(\mathcal G_P^c)\le\delta_P$.
On $\mathcal G_N=\{\max_i\|Y_i-Z_i\|\le a\}$,
\[
d_i^{\rm obs}\le r-2a\Rightarrow d_i^{\rm lat}\le r,\qquad
d_i^{\rm lat}\le r\Rightarrow d_i^{\rm obs}\le r+2a.
\]
These implications include equality; if $r-2a<0$ the declared core
is empty. When $m_->0$, $|J|\ge m_-$ and both sets are nonempty.
Since $\mathcal B_q(m,\delta_P)$ decreases with $m$, replacing
$|J|$ by the random lower bound $m_-$ is a deterministic enlargement
on $\mathcal G_N\cap\mathcal G_P$. Consequently
\[
\{\|\widehat C_n-C_{Z_0}(r)\|_{\rm op}>E_n\}
\subseteq\mathcal G_N^c\cup\mathcal G_P^c.
\]
Its probability is at most $\delta_C+\delta_L+\delta_P$.
No iid calculation has been conditioned on the observed core or on $a$.
\end{proof}

\subsection{Local mass and a shrinking membership shell}
\label{app:dimshell}
Under the geometric and density hypotheses at a fixed $x$, Proposition~\ref{prop:covariance}'s graph calculation already proves
\begin{equation}
p_x(t):=\mu(\overline B(x,t))=f(x)v_dt^d(1+o(1)),\qquad t\downarrow0.
\label{eq:dimmassasymptotic}
\end{equation}
In detail its denominator $A_t=t^{-d}p_x(t)$ converges to $f(x)v_d>0$ by the graph area formula. This is stronger than an arbitrary two-sided power bound: it supplies a limit, which controls thin shells without imposing an unproved quantitative shell constant.
For any deterministic $b_n\ge0$ with $b_n/r_n\to0$, eventually $r_n-2b_n>0$, and
\begin{equation}
\frac{p_x(r_n\pm2b_n)}{p_x(r_n)}
=\left(1\pm\frac{2b_n}{r_n}\right)^d
\frac{1+o(1)\text{ at }r_n\pm2b_n}{1+o(1)\text{ at }r_n}
\longrightarrow1.
\label{eq:dimshellratio}
\end{equation}
Thus the mass between the two balls is $o(p_x(r_n))$. No estimate $O(b_n/r_n)$ is claimed for the unspecified $C^1$ and density moduli.

For clarity we also derive the count bound used below. If $M=\sum_{i=1}^n B_i$ with iid Bernoulli$(p)$ variables, then
\[
\E e^{tM}=(1+p(e^t-1))^n\le\exp(np(e^t-1)).
\]
Markov's inequality with $t=\log(1+\epsilon)$ gives upper-tail exponent
$-np[(1+\epsilon)\log(1+\epsilon)-\epsilon]$.
For $0\le\epsilon\le1$, the bracket is at least $\epsilon^2/3$: its derivative is $\log(1+\epsilon)$, which is at least $2\epsilon/(2+\epsilon)\ge2\epsilon/3$, as follows by differentiating $\log(1+\epsilon)-2\epsilon/(2+\epsilon)$. The lower-tail choice $t=\log(1-\epsilon)$ gives bracket $(1-\epsilon)\log(1-\epsilon)+\epsilon\ge\epsilon^2/2$, since its derivative $-\log(1-\epsilon)\ge\epsilon$. Therefore
\begin{equation}
\Prob\{|M-np|>\epsilon np\}\le2\exp(-\epsilon^2np/3),
\quad0<\epsilon<1.
\label{eq:dimcountbound}
\end{equation}
Applying additive Hoeffding \citep{hoeffding1963} to $M/n$ at error $\epsilon p$ would instead give $2e^{-2n\epsilon^2p^2}$. The multiplicative calculation above retains the $np$ scale needed for shrinking local balls.
Consequently $M/(np)\to1$ in probability if $np\to\infty$. If $np/\log(n+1)\to\infty$, then for each fixed rational $\epsilon>0$ the right side is eventually at most $2(n+1)^{-2}$, a summable series. The union bound gives
$\Prob(\exists n\ge N:\text{a deviation})\le\sum_{n\ge N}\Prob(\text{deviation}_n)\to0$.
Thus only finitely many deviations occur almost surely. Intersecting over countably many rational $\epsilon>0$ proves almost-sure convergence. This is the first Borel--Cantelli argument, and it requires no independence across stages.

Apply this calculation separately to the inner, central, and outer deterministic balls at radii $r_n-2b_n,r_n,r_n+2b_n$. By \eqref{eq:dimshellratio}, their probabilities are asymptotic to $p_n=p_x(r_n)$. If their counts are $M_n^-,M_n,M_n^+$, respectively, then in probability, or almost surely under the stronger condition,
\begin{equation}
\frac{M_n^-}{np_n},\ \frac{M_n}{np_n},\ \frac{M_n^+}{np_n}\longrightarrow1,
\qquad \frac{M_n^+-M_n^-}{M_n^-}\longrightarrow0.
\label{eq:dimcountlimits}
\end{equation}
The last expression is used only when its denominator is positive, which occurs with probability tending to one, and eventually almost surely in the strong case.

\subsection{Proof of Theorem~\ref{thm:dimconsistent}}
\label{app:dimconsistent}
\begin{proof}
We prove it for completeness by conditioning first on the reserved latent anchor $Z_0=x$. The estimator never observes $x$.

\paragraph{1. A simultaneous error bound that vanishes relative to radius.}
For $g\sim N(0,I_q)$, the Gaussian norm bound derived in Appendix~\ref{app:main} from the chi-square moment generating function gives
$\Prob(\norm g>\sqrt q+\sqrt{2t})\le e^{-t}$ \citep{laurent2000}. Set
\[
b_n=\frac{\sigma}{\sqrt{k_n}}
\left(\sqrt q+\sqrt{8\log(n+1)}\right).
\]
The union bound over $n+1$ groups gives
\[
\Prob\left\{\max_{0\le i\le n}\norm{Y_{i,n}-Z_i}>b_n\right\}
\le(n+1)^{-3}.
\]
The exponent is $t=4\log(n+1)$, so the individual failure
probability is $(n+1)^{-4}$ and multiplying by $n+1$ gives the
displayed bound. At $\sigma=0$ the simultaneous error bound holds
surely with $b_n=0$. For $\sigma>0$, $(a+b)^2\le2a^2+2b^2$ gives
\[
 \left(\frac{b_n}{r_n}\right)^2
 \le\frac{2\sigma^2q+16\sigma^2\log(n+1)}{k_nr_n^2}
 =\frac{\log(n+1)}{k_nr_n^2}
   \left(\frac{2\sigma^2q}{\log(n+1)}+16\sigma^2\right)
 \longrightarrow0.
\]
Here fixed $q$ and fixed finite $\sigma$ bound the second factor,
and \eqref{eq:dimweakrates} sends the first to zero. The displayed failures are summable, so the simultaneous event holds eventually almost surely on the infinite observation array. Knowledge of $\sigma$ is used only to prove this deterministic envelope, not to compute the estimator.

\paragraph{2. True and observed neighbors differ by a negligible fraction.}
Let $J_n=\{i:\norm{Z_i-x}\le r_n\}$ and use the three counts from \eqref{eq:dimcountlimits}. On the simultaneous event the anchor error is bounded along
with the other errors. For $1\le i\le n$,
\begin{align*}
 \big|\norm{Y_{i,n}-Y_{0,n}}-\norm{Z_i-x}\big|
 &\le\norm{(Y_{i,n}-Z_i)-(Y_{0,n}-x)}\\
 &\le\norm{Y_{i,n}-Z_i}+\norm{Y_{0,n}-x}\le2b_n.
\end{align*}
Thus a true distance at most $r_n-2b_n$ gives an observed distance
at most $r_n$, and an observed distance at most $r_n$ gives a true
distance at most $r_n+2b_n$. Eventually $r_n-2b_n>0$, and
\[
\{i:\norm{Z_i-x}\le r_n-2b_n\}
\subseteq I_n\subseteq
\{i:\norm{Z_i-x}\le r_n+2b_n\}.
\]
It follows that $|I_n\mathbin\triangle J_n|\le M_n^+-M_n^-$, while both cardinalities are at least $M_n^-$. By \eqref{eq:dimmassasymptotic}, $np_n\sim f(x)v_dnr_n^d$, so \eqref{eq:dimcountlimits} applies under the theorem's respective conditions. In particular $m_n\to\infty$ in probability, and almost surely in the strong case. The empty/singleton fallback eventually disappears in the latter case. All true selected points lie within $R_n=r_n+2b_n$ of $x$.
The negligible fraction can also be read directly from the normalized
counts: whenever $M_n^->0$,
\[
 \frac{M_n^+-M_n^-}{M_n^-}
 =\frac{M_n^+/(np_n)-M_n^-/(np_n)}{M_n^-/(np_n)}
 \longrightarrow\frac{1-1}{1}=0.
\]
The positive density at $x$ is used here through
$p_n/r_n^d\to f(x)v_d>0$; a bare upper mass bound would not give
the required diverging denominator.

\paragraph{3. Explicit transfer to the oracle empirical covariance.}
Write $C_{J_n}(Z)$ for the $1/|J_n|$ covariance of the unobserved
points in the true ball. On the simultaneous error event and
$M_n^-\ge2$, insert the latent covariance on the observed index set:
\begin{align*}
 \norm{\widehat C_n-C_{J_n}(Z)}_{\rm op}
 &\le\norm{C_{I_n}(Y)-C_{I_n}(Z)}_{\rm op}
       +\norm{C_{I_n}(Z)-C_{J_n}(Z)}_{\rm op}\\
 &\le2R_nb_n+b_n^2
       +6R_n^2\min\{1,(M_n^+-M_n^-)/M_n^-\}.
\end{align*}
The first term uses \eqref{eq:dimmoving} with the indices fixed after
realization; the second uses \eqref{eq:dimmembership}. Neither
deterministic comparison assumes iid observations after noisy selection.
Divide by $r_n^2$ and substitute $R_n/r_n=1+2b_n/r_n$ to obtain
\begin{align}
\frac{\norm{\widehat C_n-C_{J_n}(Z)}_{\rm op}}{r_n^2}
\le{}&2\left(1+\frac{2b_n}{r_n}\right)\frac{b_n}{r_n}
+\left(\frac{b_n}{r_n}\right)^2\nonumber\\
&+6\left(1+\frac{2b_n}{r_n}\right)^2
\min\left\{1,\frac{M_n^+-M_n^-}{M_n^-}\right\}.
\label{eq:dimerrortransfer}
\end{align}
Every term tends to zero in the appropriate mode. This is the selected-neighborhood error calculation that was previously an assumption in Proposition~\ref{prop:covariance}.
Put $\chi_n=b_n/r_n$ and
$T_n=\min\{1,(M_n^+-M_n^-)/M_n^-\}$ when $M_n^->0$.
The right side of \eqref{eq:dimerrortransfer} is exactly
\[
\Phi_n=2\chi_n+5\chi_n^2+6(1+2\chi_n)^2T_n.
\]
The deterministic condition gives $\chi_n\to0$; the shell-count
calculation gives $T_n\to0$ in the stated mode. These verify
separately the small-noise and membership components.

\paragraph{4. Oracle empirical covariance converges to population covariance.}
Conditional on $Z_0=x$, let $D_n=\overline B(x,r_n)$ and
$B_{i,n}=\mathbf1\{Z_i\in D_n\}$, $1\le i\le n$.
The latent draws remain independent under this conditioning.
For any realized membership vector of positive probability, and
measurable sets $A_i$ for its selected indices, the product law gives
\[
 \Prob\{Z_i\in A_i\text{ for all }i\in J_n\mid
       (B_{i,n})_{i=1}^n,Z_0=x\}
 =\prod_{i\in J_n}\frac{\mu(A_i\cap D_n)}{p_n}.
\]
The factors for unselected indices cancel between numerator and
denominator. Hence the selected latent points are conditionally iid
with law $\mu(\cdot\mid D_n)$, and $(Z_i-x)/r_n$ lies in the unit
ball. This conditioning is on true membership, not on $I_n$ or the
noise event. With $Z_0=x$ still understood, the matrix Bernstein
specialization \eqref{eq:dimcovariancetail} therefore gives for every
$\epsilon>0$ and every nonempty membership vector
\[
\Prob\left\{r_n^{-2}\norm{C_{J_n}(Z)-C_x(r_n)}_{\rm op}>\epsilon\ \middle|\ J_n\right\}
\le2(q+1)\exp\left(-\frac{|J_n|\epsilon^2}{18+4\epsilon}\right).
\]
To remove membership conditioning, declare the zero-count case a
failure and split at $|J_n|\ge np_n/2$. The count bound
\eqref{eq:dimcountbound} at relative error $1/2$ gives
$\Prob\{|J_n|<np_n/2\mid x\}\le2e^{-np_n/12}$.
On the complementary event, the conditional covariance tail is at most
\[
 2(q+1)\exp\left\{-\frac{(np_n/2)\epsilon^2}{18+4\epsilon}\right\}
 =2(q+1)e^{-np_n\epsilon^2/(36+8\epsilon)}.
\]
Averaging this bound over membership vectors and adding the low-count
probability gives the following bound, still conditional on $Z_0=x$:
\begin{equation}
2e^{-np_n/12}+2(q+1)e^{-np_n\epsilon^2/(36+8\epsilon)}.
\label{eq:dimoracleprob}
\end{equation}
It tends to zero when $np_n\to\infty$ and is summable for every fixed $\epsilon>0$ when $np_n/\log(n+1)\to\infty$. The same countable Borel--Cantelli argument as before proves the almost-sure version. Combining this with Step~3 proves the first assertion in \eqref{eq:dimweakconclusion} and its strong counterpart.
For every fixed $\varepsilon>0$, a complete weak-mode
probability decomposition conditional on $Z_0=x$ is
\begin{align*}
&\Pr\{r_n^{-2}\|\widehat C_n-C_x(r_n)\|_{\rm op}>\varepsilon\mid x\}\\
&\quad\le(n+1)^{-3}+\Pr(M_n^-<2\mid x)\\
&\qquad+\Pr(\Phi_n>\varepsilon/2,\ M_n^-\ge2\mid x)
       +2e^{-np_n/12}\\
&\qquad+2(q+1)e^{-np_n\varepsilon^2/(144+16\varepsilon)}.
\end{align*}
Outside these five events the selected covariance is defined,
the transfer error is at most $\varepsilon/2$, and the oracle error
is at most $\varepsilon/2$. The triangle inequality proves the bound.
In the last exponent the substitution is
\[
 \frac{(\varepsilon/2)^2}{36+8(\varepsilon/2)}
 =\frac{\varepsilon^2/4}{36+4\varepsilon}
 =\frac{\varepsilon^2}{144+16\varepsilon}.
\]
Each term vanishes. Under the strong rate condition, the two oracle
exponentials are summable for each fixed $\varepsilon>0$: for any
constant $c>0$, eventually $c np_n\ge2\log(n+1)$, and thus
$e^{-c np_n}\le(n+1)^{-2}$. Apply Borel--Cantelli at
$\varepsilon=1,1/2,1/3,\ldots$ and intersect the resulting
probability-one events. Together with the almost-sure shell and noise
bounds, this proves the normalized covariance limit almost surely
without assuming independence across stages.
These are proof quantities, not an observed finite-radius geometric
confidence bound. In particular no extra factor $r_n^4$ is lost by treating locally bounded, accurately averaged observations as unbounded single noisy views.

\paragraph{5. The geometric gap forces exact integer recovery.}
Theorem~\ref{thm:localexistence} supplies the local $C^1$ graph at $x$, and Proposition~\ref{prop:covariance} supplies $r_n^{-2}C_x(r_n)\to a_dP_x$ with $a_d=1/(d+2)$. Put
\[
e_n=r_n^{-2}\norm{\widehat C_n-C_x(r_n)}_{\rm op}
+\norm{r_n^{-2}C_x(r_n)-a_dP_x}_{\rm op},\qquad
\gamma=\frac1{2(q+2)}.
\]
Since $d<q$, $a_d>1/(q+2)=2\gamma$. The variational eigenvalue inequality proved in Appendix~\ref{app:dimension} gives, on $e_n<\gamma$,
\[
\lambda_j(\widehat C_n)/r_n^2\ge a_d-e_n>\gamma\quad(j\le d),
\qquad
\lambda_j(\widehat C_n)/r_n^2\le e_n<\gamma\quad(j>d).
\]
Thus $\{\widehat d_n\ne d\}\subseteq\{e_n\ge\gamma\}\cup\{m_n<2\}$. Both probabilities vanish in the weak case. In the strong case both obstructions disappear after a finite random index. Integer recovery implies $\widehat d_n\to d$ almost surely, and is the stronger explicit conclusion \eqref{eq:dimstrongconclusion}.

\paragraph{6. Remove the conditioning and identify the global dimension.}
All preceding statements hold at each $x$ where the stated local hypotheses hold. For convergence in probability, integrate the conditional failure probabilities against $\mu(dx)$; they are bounded by one and converge pointwise to zero, so dominated convergence applies. For the strong result, conditional probability of eventual recovery is one at each such $x$, and integrating gives unconditional probability one. This uses the same reserved anchor through the sequence; it is not an unproved uniform-in-$x$ convergence assertion.
Explicitly, for $A_n=\{\widehat d_n\ne d\}$,
\[
\Pr(A_n)=\int\Pr(A_n\mid Z_0=x)\,\mu(dx)\to0.
\]
Let $A_\infty=\bigcap_{m\ge1}\bigcup_{n\ge m}A_n$. The conditional
strong calculation gives $\Pr(A_\infty\mid Z_0=x)=0$, so
$\Pr(A_\infty)=\int0\,\mu(dx)=0$.
The covariance event is a countable intersection over rational error
tolerances. No supremum over anchors or independence across stages
is introduced in either integration.

Under the global hypotheses, every point of $E$ has a $C^1$ graph neighborhood of dimension $d$. A smaller chart has bounded derivative, so the graph map is Lipschitz there and its inverse, orthogonal projection, is 1-Lipschitz. Such a graph patch has Hausdorff dimension $d$: Lipschitz maps do not increase Hausdorff dimension (apply a Lipschitz constant to every diameter in the defining covers), and the inverse projection contains a nonempty open $d$-ball. A finite subcover of these smaller neighborhoods exists by compactness. The dimension of their finite union is at most $d$, and one nonempty patch gives the opposite inequality. Therefore $\dim_H(E)=d$. The local version uses no conclusion about this finite global cover.
\end{proof}

\subsection{Proof of Corollary~\ref{cor:dimschedule} and shrinking observable bounds}
\label{app:dimschedule}
\begin{proof}
For the logarithmic choices, $r_n\to0$ and, for fixed $d$,
\[
\frac{nr_n^d}{\log(n+1)}
\sim r_*^d\frac n{\log^{d+1}n}\longrightarrow\infty,
\qquad
\frac{k_nr_n^2}{\log(n+1)}\sim r_*^2\log n\longrightarrow\infty.
\]
For the polynomial choices, the corresponding quantities have orders
$n^{1-\alpha d}/\log n$ and $n^{\beta-2\alpha}/\log n$. Both exponents are positive because $\alpha d<d/q<1$ and $\beta>2\alpha$. Ceiling and the finite lower bound two do not change these limits. This proves the corollary without inserting the unknown dimension into the algorithm.
For the elementary growth comparison write $t=\log n$.
Given $A>0$, choose an integer $m>A$; then $e^t\ge t^m/m!$
implies $n/\log^A n\to\infty$. Similarly
$n^\xi/\log n=e^{\xi t}/t\to\infty$ for any $\xi>0$.
The ceiling satisfies $x\le\lceil x\rceil<x+1$ and hence contributes
a relative $o(1)$ for either displayed replication sequence.
\end{proof}

The uncertainty bound is also compatible with the strong result. For example take $\delta_C=\delta_L=\delta_P=(n+1)^{-3}/6$ in \eqref{eq:dimcalibration}--\eqref{eq:dimfinitebound}. Here are the remaining chi-square calculations explicitly. For $W\sim\chi^2_\nu$ and $\theta>0$,
\[
\log\E e^{-\theta(W-\nu)}=\nu\theta-\frac\nu2\log(1+2\theta)\le\nu\theta^2,
\]
because $\log(1+u)\ge u-u^2/2$ for $u\ge0$ (the derivative of the difference is $u^2/(1+u)\ge0$). Applying Markov's inequality at $\theta=\sqrt{t/\nu}$ yields
$\Prob(W\le\nu-2\sqrt{\nu t})\le e^{-t}$. Consequently
\[
q_{\nu_n,\delta_C}\ge\nu_n-2\sqrt{\nu_n\log(1/\delta_C)}.
\]
For $0<\epsilon\le1$, take $t=\nu_n\epsilon^2/16$. The upper-tail bound proved in Appendix~\ref{app:main} has deviation $2\sqrt{\nu_nt}+2t=\nu_n\epsilon/2+\nu_n\epsilon^2/8\le5\nu_n\epsilon/8<\nu_n\epsilon$, while the lower-tail deviation is $2\sqrt{\nu_nt}=\nu_n\epsilon/2$. Thus, when $\sigma>0$,
\[
\Prob\{|S_n/(\sigma^2\nu_n)-1|>\epsilon\}
\le2\exp(-\nu_n\epsilon^2/16).
\]
As $\nu_n\ge(n+1)q$, this is summable. Thus $S_n/(\sigma^2\nu_n)\to1$ almost surely, while the displayed quantile lower bound divided by $\nu_n$ tends to one. In particular $U_n\le2\sigma^2$ eventually almost surely. At $\sigma=0$ the upper bound is zero identically. Consequently $a_n$ is eventually bounded by a deterministic multiple of $\sqrt{\log(n+1)/k_n}=o(r_n)$.
In fact $U_n\to\sigma^2$ almost surely with these budgets.
For $\sigma>0$, exact calibration gives
$\Pr(U_n<\sigma^2)=\delta_{C,n}$, a summable sequence.
Thus $U_n\ge\sigma^2$ eventually. Put
$\epsilon_n=2\sqrt{\log(1/\delta_{C,n})/\nu_n}\to0$.
The quantile lower bound gives, eventually,
\[
U_n\le\frac{S_n/\nu_n}{1-\epsilon_n},\qquad
\limsup_nU_n\le\sigma^2.
\]
The two bounds prove the limit; at $\sigma=0$ it is exact.
For $\sigma>0$, comparison with Step~1's envelope gives
$a_n\le2b_n$ eventually: the asymptotic coefficient ratio is at most
$\sqrt2<2$.

Together with the actual uniform mean-error envelope in Step~1, this sandwiches $I^-$, $I^+$ between true balls with radii $r_n\pm c\sqrt{\log(n+1)/k_n}$ for a fixed model-dependent $c$. Equation~\eqref{eq:dimcountlimits} applies to these deterministic radii. It gives $m_-/(np_n)\to1$, $B/m_-\to0$, and $\log(n+1)/m_-\to0$ almost surely. Substitution in \eqref{eq:dimfinitebound} proves $E_n/r_n^2\to0$ almost surely. The confidence radius can therefore shrink despite using the observed membership shell. It still does not bound the geometric remainder $\theta_x(r_n)$ at any specified finite $n$.
Eventually all mean errors are at most $b_n$ and $a_n\le2b_n$.
Including the anchor, the distance error is at most $2b_n$, whence
\[
\{d_i^{\rm lat}\le r_n-6b_n\}\subseteq I_n^-,
\qquad I_n^+\subseteq\{d_i^{\rm lat}\le r_n+6b_n\}.
\]
Therefore the observed shell is contained in
$\{r_n-6b_n<d_i^{\rm lat}\le r_n+6b_n\}$.
These deterministic radii have $6b_n/r_n\to0$, so the count limits
apply. With $z_n=a_n/r_n$ and $t_n=\log(2(q+1)/\delta_{P,n})$,
\[
\frac{E_n}{r_n^2}
=2z_n+5z_n^2+6(1+2z_n)^2\min(1,B/m_-)
 +3\sqrt{\frac{2t_n}{m_-}}+\frac{4t_n}{m_-}\to0.
\]
With zero noise the observed shell is empty, leaving only the
sampling term.

\subsection{Proof of Proposition~\ref{prop:dimfixednoise}}
\label{app:dimfixed}
\begin{proof}
For fixed $k$ and $\sigma>0$, the observations $Y_i$ have density
\[
g(y)=\int (2\pi\sigma^2/k)^{-q/2}
\exp\left(-\frac{k\norm{y-z}^2}{2\sigma^2}\right)\mu(dz)>0.
\]
It is smooth, by differentiating the Gaussian convolution; compact support of $\mu$ makes local domination immediate. Conditional on the fixed observed anchor $Y_0=y$, the other $Y_i$ are iid with this law.
Every derivative of the Gaussian kernel is a polynomial in $y-z$
times that kernel. On a compact $y$-neighborhood and the compact latent
support it is bounded, so differentiation under $\mu(dz)$ is legitimate
at every order. Strict positivity of the kernel gives $g(y)>0$.
Set
\[
A_y(r)=\int_{B_q}g(y+rv)\,dv,\quad
b_y(r)=\int_{B_q}v\,g(y+rv)\,dv,\quad
H_y(r)=\int_{B_q}vv^{\mathsf T}g(y+rv)\,dv.
\]
Changing variables gives the exact identities
\[
\Pr(Y\in\overline B(y,r))=r^qA_y(r),\qquad
r^{-2}\Cov(Y\mid Y\in\overline B(y,r))
=\frac{H_y(r)}{A_y(r)}-\frac{b_y(r)b_y(r)^{\mathsf T}}{A_y(r)^2}.
\]
Continuity and the computed unit-ball moments give
$A_y(r)\to g(y)v_q$, $b_y(r)\to0$ and
$H_y(r)\to g(y)v_qI_q/(q+2)$.
The normalization and centering are explicit in both limits. The elementary change of variables $Y=y+rv$ in the ambient ball gives
\[
r^{-q}\Prob(Y\in\overline B(y,r))\to g(y)v_q,
\qquad
r^{-2}\Cov(Y\mid Y\in\overline B(y,r))\to I_q/(q+2).
\]
The first limit follows by continuity and integration over the unit ball. For the second, integrate $v$ and $vv^{\mathsf T}$, subtract the squared mean, and use the unit-ball moments already calculated in Appendix~\ref{app:dimension}, now with $d=q$. The count and bounded-moment arguments above, applied directly to the observed $Y_i$, show the same empirical covariance limit in probability when $nr_n^q\to\infty$, and almost surely under the logarithmic strengthening. All $q$ eigenvalues then exceed $\gamma=1/(2(q+2))$ after normalization. Integration over $y$ removes conditioning.

Finally every selected observation lies within $r_n$ of $Y_0$. For any nonempty selected set,
\[
\lambda_{\max}(\widehat C_n)\le\operatorname{tr}\widehat C_n
=\frac1{m_n}\sum_{i\in I_n}\norm{Y_i-Y_0}^2
-\norm{\bar Y_{I_n}-Y_0}^2\le r_n^2.
\]
Hence $\lambda_{\max}(\widehat C_n-\sigma^2I_q/k)\le r_n^2-\sigma^2/k<0$ eventually. Subtracting the unconditional noise covariance inside a ball selected by those same noisy points therefore cannot repair this estimator.
For every $r_n^2<\sigma^2/k$ and every nonempty selected set,
all eigenvalues of the subtracted matrix are strictly negative,
whereas $r_n^2/[2(q+2)]>0$. Its threshold count is exactly zero.
The empty/singleton fallback also returns zero.
This subtraction conclusion is deterministic given the selected points;
the raw count converging to $q$ uses the sampling conditions. This direct calculation explains why increasing replication, or a separately justified deconvolution procedure, matters.
\end{proof}

\subsection{Numerical implementation and reproduced checks}
\label{app:dimimplementation}
The module \texttt{intrinsic\_\allowbreak dimension.py} accepts either grouped raw repeats or their exact-law sufficient statistics, forms the neighborhood around one reserved observed group, and returns the spectrum, point estimate, variance estimate, shell counts, and covariance radius. It does not take reference coordinates, a true dimension, or a true variance as input. The code uses numerical chi-square quantiles and floating-point SVD. It explicitly labels its output as a point estimate with a numerical covariance bound; it does not inherit the separate rational witness verifier's arithmetic status.

\paragraph{Executed calibration check.}

The new code was run on 200 independently generated sufficient-statistic inputs from uniform $S^2\subset\mathbb R^8$, with an independently perturbed fixed north-pole anchor, $\sigma=0.35$, and a fixed orthogonal rotation. The reference covariance is analytic: its two tangent eigenvalues are $r^2/4-r^4/24$, the normal eigenvalue is $r^4/48$, and the other five are zero. Indeed, the cap height $H$ is uniform on $[1-c,1]$, $c=r^2/2$, so $\E H=1-c/2$, $\E H^2=1-c+c^2/3$, $\operatorname{Var}(H)=c^2/12$, and the tangent variance is $(1-\E H^2)/2=c/2-c^2/6$. Rotational symmetry makes cross moments zero. The generator alone uses the true variance and dimension.

Figure~\ref{fig:dimensionchecks} reports all four designs and their point-estimator outcomes.

The four conditions have 0, 0, 0, 0 finite observable covariance bounds: all 200 bounds are vacuous. The 0 detected violations therefore supply no evidence of nontrivial covariance coverage. The experiment tests the point estimator, which does not depend on a finite bound; the nontrivial finite-bound code branch is checked separately at zero noise. These simulations do not prove consistency, provide a uniform success rate, or certify exact manifold support. All individual outputs remain in \texttt{computed\_\allowbreak checks.json}.

Separate executed checks verify the raw-repeat residual by a pairwise identity, centered covariance by pairwise differences, rigid-motion and unit invariance, empty-neighborhood behavior, the finite zero-noise uncertainty branch, 120 deterministic covariance-transfer probes, the augmented-moment variance identity on a nonzero-mean discrete law, and largest-gap recovery under non-diagonal positive semidefinite perturbations. The implemented sampling radius is the matrix Bernstein radius in \eqref{eq:dimbernsteinradius}. The separate \texttt{eigengap\_\allowbreak dimension} output implements \eqref{eq:growinggap}; its algebraic check is not an executed growing-dimension asymptotic experiment. These are numerical checks of implementation, not replacements for the proofs.

\paragraph{Independent checks for this calculation revision.}
The bundled \texttt{code/\allowbreak verify\_\allowbreak
hardcalc\_\allowbreak revision.py} additionally
checks residual/SVD identities, 240 covariance-transfer fixtures,
augmented-moment variance and exact finite-sample centering bias,
the spectral-projector derivative with repeated within-group
eigenvalues, 12 complete objective-gradient fixtures, 96 non-diagonal
positive semidefinite gap fixtures, and rational constants for the
slab bound, Bernstein inversion, disk moments and Julia branch bound.
All seven check groups passed; the largest relative error of the
complete objective gradient over two difference steps was
$1.06\times10^{-10}$. The independent record is
\texttt{audit/\allowbreak hardcalc\_\allowbreak checks.json}. These checks supplement the
proofs and do not rerun or change the historical experiments.


\subsection{Proof of Lemma~\ref{lem:growingspectral}}
\label{app:growingspectral}
\begin{proof}
We prove it for completeness, crediting the tangent-disk identity to \citet[Lemma~13]{arias2017} and \citet[Lemma~6.1]{lim2024}, and the covariance--Wasserstein approach to \citet[Proposition~3.3]{lim2024}. The concentration input is \citet[Theorem~6.1]{tropp2012}, specialized in Appendix~\ref{app:dimbounded}.

\paragraph{1. Geometry identifies the target projection.}
Theorem~\ref{thm:localexistence}, attributed to \citet{kafer2019}, provides the local $p$-dimensional $C^1$ chart and its tangent $T_N$; no chart is inferred merely from a small observed covariance rank. If $V\sim\mathsf U_N$, reflection symmetry gives $\E V=0$ and zero cross moments in tangent coordinates. Since $\Prob(\norm V\le s)=s^p$ for $0\le s\le1$,
\[
\E\norm V^2=\int_0^1s^2p s^{p-1}\,ds=\frac p{p+2}.
\]
Rotational symmetry divides this equally among the $p$ tangent coordinates; all normal coordinates vanish. Therefore $\E VV^{\mathsf T}=P_N/(p+2)$. This is the credited identity \eqref{eq:tangentdiskcredited}.

\paragraph{2. A coupling controls the geometric covariance error.}
Take any coupling $(U,V)$ with $U\sim\nu_N$, $V\sim\mathsf U_N$. Both norms are at most one, and
\[
UU^{\mathsf T}-VV^{\mathsf T}
=(U-V)U^{\mathsf T}+V(U-V)^{\mathsf T}.
\]
Its operator norm is at most $2\norm{U-V}$. Also $\norm{\E U}=\norm{\E(U-V)}\le\E\norm{U-V}$. Thus, writing $a=\E\norm{U-V}$,
\[
\left\|\Cov(U)-\frac{P_N}{p+2}\right\|_{\rm op}
\le2a+a^2.
\]
Take couplings whose costs decrease to the defining infimum $\omega_N$ to obtain $2\omega_N+\omega_N^2$. This specializes the established transport comparison to a common unit ball and a centered uniform reference; no claim of optimal transport constants is needed.

\paragraph{3. Concentration and the gap.}
The uncentered normalized samples $(Z_{Ni}-x_N)/r_N$ are iid in the unit ball. Their centered empirical covariance is exactly $r_N^{-2}S_N$. Equation~\eqref{eq:dimbernstein}, with ambient dimension $d_N$ and sample size $N$, bounds its distance to $\Cov(\nu_N)$ by $\mathcal B_{d_N}(N,\delta)$ outside probability $\delta$. This handles dependence introduced by subtracting the sample mean rather than treating the centered columns as independent. Combining with Step~2 proves \eqref{eq:growingbridge}.

Let $a_p=1/(p+2)$ and $e=\epsilon_N(\delta)$. The eigenvalue perturbation calculation in Appendix~\ref{app:dimension} gives on that event
\[
|\lambda_j(S_N)/r_N^2-a_p|\le e\quad(j\le p),\qquad
0\le\lambda_j(S_N)/r_N^2\le e\quad(j>p).
\]
The gap at $p$ is at least $a_p-2e$, whereas every other gap is at most $2e$, including the appended gap to zero. If $e<a_p/4$, the maximum is unique and occurs at $p$.
Every gap index is included in the comparison:
\[
\begin{array}{ll}
j<p:&(\lambda_j-\lambda_{j+1})/r_N^2
       \le(a_p+e)-(a_p-e)=2e,\\
j=p:&(\lambda_p-\lambda_{p+1})/r_N^2\ge a_p-2e,\\
p<j\le d_N:&(\lambda_j-\lambda_{j+1})/r_N^2\le e\le2e.
\end{array}
\]
The last row includes the appended $\lambda_{d_N+1}=0$.
The strict gap inequality also implies $\lambda_p(S_N)>0$,
excluding the zero-matrix fallback on this event.
No threshold using the unknown answer $p$ is supplied to the statistic. This proves the finite statement without supplying $p$ to the algorithm. Setting $\delta_N=(N+1)^{-2}$ and using $d_N\le C N$ eventually yields
\[
t_N=\log\{2(d_N+1)(N+1)^2\}=O(\log N),\qquad
\mathcal B_{d_N}(N,\delta_N)=O\left(\sqrt{\frac{\log N}N}\right)\to0.
\]
Since $p$ is fixed and $\omega_N\to0$, the gap condition holds deterministically for all sufficiently large $N$. The sum of the failure probabilities $\sum_N\delta_N$ is finite, so the first Borel--Cantelli argument gives eventual exact recovery. No independence across $N$ is required. For the observed version, the triangle inequality adds $\tau_N$, and a union bound adds $\alpha_N$; summability gives its stated strong conclusion.

\paragraph{4. Derive the transport premise from uniform local geometry.}
We verify the sufficient conditions \eqref{eq:growinggeomconditions} rather than replacing them by an unsupported entrywise-to-operator-norm limit as $d_N$ grows. Work in the orthogonal tangent coordinates of the derived chart, and let $B_p$ denote the unit ball with uniform law $\lambda_p$. Suppress $N$ temporarily. Set
\[
D=\{v\in B_p:\norm v^2+\norm{g(rv)/r}^2\le1\},\quad
J(v)=\sqrt{\det(I_p+Dg(rv)^{\mathsf T}Dg(rv))},
\]
and $L=(1+\eta^2)^{p/2}$. The fundamental theorem of calculus gives $\norm{g(rv)/r}\le\eta\norm v$. Consequently
\[
(1+\eta^2)^{-1/2}B_p\subseteq D\subseteq B_p,\quad
1\le J(v)\le L,\quad \lambda_p(B_p\setminus D)\le1-L^{-1}.
\]
By the graph area formula \citep{federergmt}, the projected normalized local law has density $w/\bar w$ with respect to $\lambda_p$, where
\[
w(v)=\mathbf1_D(v)\frac{f(x+(rv,g(rv)))}{f(x)}J(v),\qquad
\bar w=\int w\,d\lambda_p.
\]
For $0\le\zeta<1$, on $D$ the absolute difference $|w-1|$ is at most $(1+\zeta)L-1$; outside $D$ it equals one. Hence
\[
\int|w-1|\,d\lambda_p\le\rho:=(1+\zeta)L-L^{-1},\quad
|\bar w-1|\le\rho.
\]
When $\rho<1$, the total variation distance between $w\lambda_p/\bar w$ and $\lambda_p$ is at most
\[
\frac12\int\left|\frac w{\bar w}-1\right|d\lambda_p
\le\frac{\int|w-1|d\lambda_p+|1-\bar w|}{2\bar w}
\le\frac{\rho}{1-\rho}.
\]
To turn this into a transport bound, couple the common part of the two measures at identical points and couple the residual masses arbitrarily. Their common support $B_p$ has diameter two; the cost is at most twice total variation. Lifting the projected variable to its graph changes it by at most $\eta$. The triangle inequality therefore yields the explicit ambient-independent estimate
\begin{equation}
\omega_N\le\eta_N+\frac{2\rho_N}{1-\rho_N},\quad
\rho_N=(1+\zeta_N)(1+\eta_N^2)^{p/2}-(1+\eta_N^2)^{-p/2}.
\label{eq:growingtransportbound}
\end{equation}
For fixed $p$, $\eta_N,\zeta_N\to0$ imply $\rho_N\to0$ and then $\omega_N\to0$. This proves the announced link from locally established geometry and density to the spectral estimate, including growing ambient dimension.
The transport coupling is explicit at the density level.
If $h=w/\bar w$, the common density $\min(h,1)$ has mass
$1-\operatorname{TV}$. Place it on the diagonal of $B_p\times B_p$.
Couple the remaining two measures of equal mass arbitrarily; every
pair has distance at most two. This proves
$W_1(h\lambda_p,\lambda_p)\le2\operatorname{TV}$.
Lifting $v$ to $(v,g(rv)/r)$ costs at most $\eta$ on $D$.
The determinant has $p$ factors, so neither estimate depends on the
number $d_N-p$ of normal coordinates. At a fixed $C^1$ chart with continuous positive density these two moduli vanish as the radius shrinks; a triangular array of models needs the explicit uniform condition.
\end{proof}

The strong conclusion above concerns a local conditional sample, or a noisy observed covariance with the separately stated transfer bound. In the original repeated-observation setting, a simultaneous Gaussian envelope over $n+1$ groups in dimension $d_N$ is
\[
\frac{\sigma}{\sqrt{k_N}}\left(\sqrt{d_N}+\sqrt{8\log(n+1)}\right).
\]
Making its ratio to $r_N$ vanish, for example by requiring $k_Nr_N^2/(d_N+\log(n+1))\to\infty$, controls the location component. Neighbor counts and membership-shell ratios must also be controlled. This displayed condition alone is not a new full observed-neighborhood consistency theorem. A finite transport bound also does not certify the existence assumptions: the geometric conclusion still comes from Theorem~\ref{thm:localexistence}.




\section{Calculations for the Ahlfors--reach spectral bridge}
\label{app:ahlforsreach}
The geometric structure is imported with credit from Federer and
Rataj--Zaj\'i\v{c}ek, and the concentration theorem from Tropp.
We prove the covering, covariance, and statistical specializations for
completeness. The reach lower bound throughout this appendix concerns the
true support.

\subsection{Proof of Theorem~\ref{ar:thm:geometry}}
\label{app:argeometryproof}
\begin{proof}
We use the structural theory of Federer and
\citet[Theorem~7.5]{rataj2017} as credited inputs, and calculate their
consequences under the stated sampling-measure bounds.

\paragraph{Step 1: absolute continuity and positive Hausdorff measure.}
Temporarily $p$ is the positive real exponent in \eqref{ar:eq:model}.
Write
\[
\mathcal C_\delta^p(A)=
\inf_{\substack{A\subseteq\bigcup_jU_j\\\operatorname{diam}U_j\le\delta}}
       \sum_j(\operatorname{diam}U_j)^p.
\]
Let $c_p^{\rm H}>0$ be the normalizing constant such that
$\mathcal H^p(A)=c_p^{\rm H}\lim_{\delta\downarrow0}\mathcal C_\delta^p(A)$.
For each covering set meeting $E$, choose $x_j\in U_j\cap E$.
Then $U_j\cap E\subseteq\B(x_j,\operatorname{diam}U_j)$, whence,
for $\delta\le r_0$ and measurable $A\subseteq E$,
\[
\mu(A)\le C\sum_j(\operatorname{diam}U_j)^p.
\]
Zero-diameter sets have zero mass because
$\mu(\{x\})\le\lim_{s\downarrow0}Cs^p=0$.
The inequality holds for every countable cover in the definition,
so first taking the infimum gives
$\mu(A)\le C\mathcal C_\delta^p(A)$. Letting $\delta\downarrow0$
then gives
\[
 \mu(A)\le C\lim_{\delta\downarrow0}\mathcal C_\delta^p(A)
 =\frac{C}{c_p^{\rm H}}\mathcal H^p(A).
\]
Thus $\mu\ll\mathcal H^p|_E$, and
$\mathcal H^p(E)\ge c_p^{\rm H}/C>0$.

\paragraph{Step 2: finiteness and dimension.}
For $0<\varepsilon\le r_0$, compactness gives a finite maximal
$\varepsilon$-separated set of $J$ centers in $E$.
Its $\varepsilon$-balls cover $E$ and its $\varepsilon/3$-balls
are pairwise disjoint. The lower bound gives
\[
1\ge Jc(\varepsilon/3)^p,\qquad
J\le\frac{3^p}{c\varepsilon^p},\qquad
\mathcal C_{2\varepsilon}^p(E)\le J(2\varepsilon)^p\le\frac{6^p}{c}.
\]
Consequently
$0<\mathcal H^p(E)\le c_p^{\rm H}6^p/c<\infty$.
The Hausdorff-dimension identity follows directly from the cover
definition: for $\alpha>p$,
$\sum_j\operatorname{diam}(U_j)^\alpha
\le\delta^{\alpha-p}\sum_j\operatorname{diam}(U_j)^p$,
so $\mathcal H^\alpha(E)=0$; for $\alpha<p$, a finite
$\mathcal H^\alpha(E)$ would force $\mathcal H^p(E)=0$ by the same
inequality with $\alpha,p$ interchanged.
Hence $\dim_HE=p$.

\paragraph{Step 3: invoke exactly the needed positive-reach structure.}
For positive-reach sets, the tangent cones are convex,
$\dim_H E=\dim E$, and $\dim\Tan(E,x)\le\dim E$.
These are Federer's facts recorded in
\citet[Proposition~3.1(ii), Eqs.~(5)--(6)]{rataj2017}.
The topological dimension is an integer, so $p$ is an integer.
Since $0<p<d$, their Theorem~7.5 applies:
the regular part $E_{\rm reg}$ is a relatively open embedded
$p$-dimensional $C^{1,1}$ manifold without boundary.
Its complement is locally covered by finitely many $(p-1)$-dimensional
DC surfaces. DC parametrizations are locally Lipschitz: each component
is a difference of convex functions, and finite convex functions are
Lipschitz on compact subsets of their open domains. Exhaust those
domains by countably many such subsets. Second countability of
$E\subset\R^d$ also reduces the local cover to a countable cover.
Each resulting Lipschitz image has Hausdorff dimension at most $p-1$;
the countable-union property gives
$\dim_H(E\setminus E_{\rm reg})\le p-1$ and hence
$\mathcal H^p(E\setminus E_{\rm reg})=0$.
Step~1 then gives $\mu(E\setminus E_{\rm reg})=0$.
The regularity theorem itself is an imported result.

\paragraph{Step 4: exclude a smaller tangent span at any center.}
Fix $x\in E$, put $L=\operatorname{span}\Tan(E,x)$, and $k=\dim L$.
A convex cone has dimension equal to its span: a basis chosen from
the cone generates a relatively open positive simplex in that span.
Thus $k\le p$ by the preceding imported tangent-dimension inequality.
Federer's reach characterization, recorded in
\citet[Proposition~3.3]{rataj2017}, gives
\begin{equation}
\dist(y-x,\Tan(E,x))\le\frac{\|y-x\|^2}{2\tau},\qquad y\in E.
\label{ar:eq:federer}
\end{equation}
Since $\Tan(E,x)\subseteq L$, the support ball lies at normal
distance at most $h=r^2/(2\tau)$ from $x+L$.
Choose $r>0$ sufficiently small that $h\le r$ and $3h\le r_0$.
A maximal $h$-separated net in the radius-$r$ disk of $L$ has
at most $(1+2r/h)^k\le(3r/h)^k$ points. Indeed its disjoint
$h/2$-balls lie in the radius-$r+h/2$ ball, and $k$-dimensional
volume comparison gives this count. For $k=0$ there is one point.
For each $y\in E\cap\B(x,r)$, its projection $P_L(y-x)$
lies in the radius-$r$ disk of $L$. Assign it to a net point $u_j$
within distance $h$. Orthogonality and the normal bound give
\[
 \norm{y-(x+u_j)}^2
 =\norm{P_L(y-x)-u_j}^2+\norm{(I-P_L)(y-x)}^2
 \le h^2+h^2.
\]
Recenter each nonempty portion at a support point $y_j$ in that
portion. Every other point of the portion is within $2\sqrt2h<3h$
of $y_j$. The upper mass bound is now applicable at these genuine
support centers. Taking $r\le r_0$ as well as $3h\le r_0$, we get
\[
 cr^p\le\mu(\B(x,r))
 \le\sum_j\mu(\B(y_j,3h))
 \le C(3r/h)^k(3h)^p.
\]
Expanding the right side and dividing by $r^p$ yields
\begin{align*}
 c&\le C3^{p+k}r^{k-p}h^{p-k}
   =C3^{p+k}(h/r)^{p-k}\\
  &=C3^{p+k}\left(\frac r{2\tau}\right)^{p-k},
\end{align*}
where the last equality uses $h=r^2/(2\tau)$.
If $k<p$, the right side tends to zero as $r\downarrow0$, contrary
to $c>0$. Hence $k=p$ at every center, including singular ones.
This proves the pointwise span claim without identifying a singular
tangent cone with a manifold tangent space.
\end{proof}

\subsection{Proof of Theorem~\ref{ar:thm:bridge}}
\label{app:arbridgeproof}
\begin{proof}
Fix $x,r$ as in the theorem. All expectations in this proof use
$\mu(\cdot\mid\B(x,r))$. Theorem~\ref{ar:thm:geometry} gives
the rank-$p$ space $T=T_x$. Put $P=P_T$ and $h=r^2/(2\tau)$.

\paragraph{Step 1: normal displacement from reach.}
Since $T$ contains the tangent cone, \eqref{ar:eq:federer} gives
\begin{equation}
\|(I-P)(y-x)\|\le h,\qquad y\in E\cap\B(x,r).
\label{ar:eq:slabnormal}
\end{equation}
Write $C=C_x(r)$, $n=(I-P)(Z-x)$ and $\bar n=\E n$. Then
\[
\operatorname{tr}((I-P)C)
=\E\|n-\bar n\|^2
=\E\|n\|^2-\|\bar n\|^2\le h^2.
\]
This bounds the total normal variance, not each coordinate separately.
For completeness, take an orthonormal eigenbasis $(u_j)_{j=1}^d$
of $C$, with eigenvalues in decreasing order, and put
$a_j=\|Pu_j\|^2=u_j^{\mathsf T}Pu_j\in[0,1]$.
Because $P$ has rank $p$,
\[
 \sum_{j=1}^d a_j=\operatorname{tr}P=p,
 \qquad
 \sum_{j\le p}(1-a_j)=p-\sum_{j\le p}a_j
                         =\sum_{j>p}a_j.
\]
For $j\le p$, $\lambda_j\ge\lambda_p$, whereas for $j>p$,
$\lambda_j\le\lambda_p$. These signs give
\begin{align*}
 \operatorname{tr}(PC)
 &=\sum_{j\le p}\lambda_j
   -\sum_{j\le p}\lambda_j(1-a_j)+\sum_{j>p}\lambda_ja_j\\
 &\le\sum_{j\le p}\lambda_j
   -\lambda_p\sum_{j\le p}(1-a_j)+\lambda_p\sum_{j>p}a_j
 =\sum_{j\le p}\lambda_j.
\end{align*}
Subtract from $\operatorname{tr}C=\sum_j\lambda_j$ to obtain
\[
 \sum_{j>p}\lambda_j(C)
 \le\operatorname{tr}C-\operatorname{tr}(PC)
 =\operatorname{tr}((I-P)C)\le h^2=\frac{r^4}{4\tau^2}.
\]
Since a covariance matrix is positive semidefinite, each tail
eigenvalue is nonnegative; in particular $\lambda_{p+1}(C)$ is
bounded by this tail sum.
No small-radius restriction beyond $r\le r_0$ was used for this tail.

\paragraph{Step 2: the upper covariance eigenvalue.}
Centered variance satisfies
\[
\lambda_1(C)\le\operatorname{tr}C
=\E\|Z-x\|^2-\|\E Z-x\|^2\le r^2.
\]
This calculation includes any nonzero conditional mean.

\paragraph{Step 3: a slab contains at most half the mass.}
Fix a unit vector $v\in T$, a real $a$, and $t=t_p$.
Since $\kappa\ge1$ and $K_p\ge4$, $t\le1/8$.
The radius restriction $r\le2t\tau$ gives
$h=r^2/(2\tau)\le tr$. This is the only extra radius restriction
needed for the following slab-mass estimate.
Decompose $y-x=sv+w+n$ with $w\in T\cap v^\perp$ and $n\perp T$.
For
$A_{v,a}=\{y\in E\cap\B(x,r):|s-a|\le tr\}$,
we have $\|w\|\le r$, $|s-a|\le tr$, and $\|n\|\le tr$.
A $tr$-net $(w_j)$ of the radius-$r$ disk of the $(p-1)$-plane
has at most
\[
J\le(1+2/t)^{p-1}\le(3/t)^{p-1}
\]
points by the same disjoint-ball volume calculation; for $p=1$,
it consists of a single point. Each portion indexed by its net point
satisfies
\[
\|y-(x+av+w_j)\|^2
=(s-a)^2+\|w-w_j\|^2+\|n\|^2\le3t^2r^2.
\]
Recenter each nonempty portion at a point of $E$. Its radius is
at most $2\sqrt3tr<4tr$, and $4tr\le r\le r_0$.
The upper bound is therefore used at genuine support centers.
Dividing the covering bound by the lower conditioning mass gives
\begin{align*}
\Pr(Z\in A_{v,a}\mid Z\in\B(x,r))
&\le\frac{J\,C(4tr)^p}{cr^p}\\
&\le\frac{(3/t)^{p-1}\,C\,4^pt^pr^p}{cr^p}\\
&=\frac Cc\,3^{p-1}4^p\,t^{p-(p-1)}
 =\kappa K_pt=\frac12.
\end{align*}

\paragraph{Step 4: variance in every tangential direction.}
Take $a=\E[v^{\mathsf T}(Z-x)]$. At least half the mass has
$|v^{\mathsf T}(Z-x)-a|>tr$. Consequently
\begin{align*}
v^{\mathsf T}Cv
&=\E(v^{\mathsf T}(Z-x)-a)^2\\
&\ge t^2r^2\,\Pr(|v^{\mathsf T}(Z-x)-a|>tr)
\ge t^2r^2/2=a_pr^2.
\end{align*}
This holds for every unit $v\in T$. The variational formula already
proved in Appendix~\ref{app:dimension} gives
\[
\lambda_p(C)=\max_{\dim L=p}\min_{\substack{v\in L\\\|v\|=1}}
v^{\mathsf T}Cv
\ge\min_{\substack{v\in T\\\|v\|=1}}v^{\mathsf T}Cv
\ge a_pr^2.
\]
Together with Steps~1--2 this proves all stated inequalities.
There is no block-diagonal covariance, zero-mean, chart or density
continuity assumption in the calculation.
\end{proof}

\subsection{The matrix Bernstein calculation}
We prove the needed specialization for completeness, with the concentration theorem credited to \citep[Theorem 6.1]{tropp2012}. Let $U_i=(Z_i-x)/r$; then $\norm{U_i}\le1$. Set
\[
 w_i=(U_i,1),\quad K_i=w_iw_i^{\mathsf T},\quad A=\E K_i,\quad D_i=K_i-A.
\]
These $(d+1)\times(d+1)$ matrices satisfy $0\preceq K_i,A\preceq2I$ and $\norm{D_i}_{\rm op}\le2$. Since $K_i^2\preceq2K_i$,
\[
 \E D_i^2=\E K_i^2-A^2\preceq2A-A^2\preceq I;
\]
the last step is the scalar inequality $2a-a^2\le1$ on the spectrum of $A$. Two-sided matrix Bernstein gives
\[
 \Prob\left\{\norm{N^{-1}\sum_iD_i}_{\rm op}>s\right\}
 \le2(d+1)\exp\left\{-\frac{Ns^2}{2+(4/3)s}\right\}.
\]
With $u=\log(2(d+1)/\delta)$, take $s=\sqrt{2u/N}+4u/(3N)$. The upper-left block bounds the raw second-moment error by $s$, and the off-diagonal block bounds the mean error by $s$. Both empirical and population means have norm at most one, so the centering correction is at most $2s$. Therefore, with probability at least $1-\delta$,
\begin{equation}
 \norm{S_N/r^2-C_x(r)/r^2}_{\rm op}\le3s=B_d(N,\delta).
 \label{ar:eq:covbound}
\end{equation}
This argument handles the divisor $N$ and empirical centering explicitly; it does not assume the selected covariance has a Wishart law.
The numerical constant can be traced directly to
Appendix~\ref{app:dimbounded}: the augmented block error is at most
$s$, the second-moment error at most $s$, and the two mean factors
add at most $2s$. Thus
$3s=3\sqrt{2u/N}+4u/N=B_d(N,\delta)$.
Centering the independent $U_i$ is performed after this concentration
argument; the centered columns are not treated as independent.

\subsection{Proof of Corollary~\ref{ar:cor:statistics}}
\label{app:arstatisticsproof}
\begin{proof}
We separate the finite threshold margin from the changing-model limit.
\paragraph{Step 1: compare every eigenvalue with the stated threshold.}
On \eqref{ar:eq:covbound}, put $B=B_d(N,\delta)$,
$\xi_r=(r/(2\tau))^2$, and $b_N=\xi_r+B$.
Theorem~\ref{ar:thm:bridge} and the variational perturbation inequality give
\[
\lambda_j(S_N)/r^2\ge a_p-B\quad(j\le p),\qquad
\lambda_j(S_N)/r^2\le\xi_r+B=b_N\quad(j>p).
\]
If $3b_N<a_p$, the exact remaining margin is
\[
(a_p-B)-2b_N=a_p-2\xi_r-3B
=a_p-3b_N+\xi_r>0.
\]
Since $b_N>0$, all lower-group eigenvalues are strictly below
$2b_Nr^2$ and all $p$ upper-group eigenvalues exceed it.
Thus $\{\widehat p_N\ne p\}$ is contained in the covariance
failure event, proving its probability is at most $\delta$.

\paragraph{Step 2: external bounds and observed-covariance transfer.}
The sequence $K_p=4^p3^{p-1}$ increases by a factor of $12$ at each
integer increment. Hence $p\le p_{\max}$ and
$\kappa\le\kappa_{\max}$ imply
$t_p\ge t_*$ and $a_p=t_p^2/2\ge a_*$.
The supplied-radius and separation checks therefore imply those of
Step~1 without inserting the unknown dimension.
If the total observed-to-population error is bounded by $e$ on an
event of probability $1-\alpha$, repeat Step~1 with $B=e$ and
$b=\xi_r+e$. The same subtraction proves exact recovery there.
If an observed-to-oracle sample error $\tau_{\rm obs}$ instead fails
with probability $\alpha$, the triangle inequality replaces $B$ by
$B+\tau_{\rm obs}$ and the union bound replaces $\delta$ by
$\delta+\alpha$. Neither operation assumes noisy membership is iid.

\paragraph{Step 3: a uniform margin for the changing models.}
Let $\kappa_N\le\bar\kappa<\infty$. Since $p$ is fixed,
\[
t_{p,N}\ge\frac1{2\bar\kappa K_p}=:\underline t>0,\qquad
a_{p,N}\ge\underline t^2/2=:\underline a>0.
\]
The condition $r_N/\tau_N\to0$ gives
$r_N\le2\underline t\,\tau_N$ eventually, so the geometric radius
restriction holds at every sufficiently late stage.
For $\delta_N=(N+1)^{-2}$,
\[
u_N=\log(2(d_N+1))+2\log(N+1).
\]
Under $\log(d_N+1)=o(N)$, $u_N/N\to0$ and therefore
\[
B_{d_N}(N,\delta_N)
=3\sqrt{2u_N/N}+4u_N/N\to0,\qquad
b_N=(r_N/(2\tau_N))^2+B_{d_N}(N,\delta_N)\to0.
\]
In particular $d_N=O(N)$ suffices. Eventually $3b_N<\underline a$.

\paragraph{Step 4: eventual recovery and an independent random anchor.}
For every sufficiently late stage, the incorrect-dimension event
$A_N$ has probability at most $(N+1)^{-2}$. Thus
\[
\Pr\Big(\bigcup_{N\ge m}A_N\Big)
\le\sum_{N\ge m}(N+1)^{-2}\longrightarrow0.
\]
Taking the decreasing limit in $m$ gives
$\Pr(\limsup_NA_N)=0$, which is eventual correctness almost surely.
Independence across stages is unnecessary.
The covariance and slab bounds are uniform over anchors in $E_N$.
For an independent random anchor $\Xi_N$,
$\Pr(A_N)=\int\Pr(A_N\mid\Xi_N=x)\,\Pr_{\Xi_N}(dx)\le\delta_N$,
so the same proof applies.

\paragraph{Step 5: the oracle local count from all latent groups.}
For a fixed or independent anchor, the oracle membership indicators
have success probability $\pi_x(r)\ge cr^p$.
The lower-tail Chernoff calculation of Appendix~\ref{app:dimshell},
with fractional deviation $1/2$, gives
\[
\Pr\{N<\tfrac12ncr^p\}
\le\Pr\{N<\tfrac12n\pi_x(r)\}
\le e^{-n\pi_x(r)/8}\le e^{-ncr^p/8}.
\]
This is a count guarantee for latent-ball selection; it does not
prove the same count or shell control for a noisy observed ball.
\end{proof}

\subsection{Counterexamples and exact moment checks}
At an endpoint of a uniform segment, the conditional interval has length $r$ and centered variance $r^2/12$. At an interior point with a full interval of radius $r$, the variance is $r^2/3$. Thus even positive reach cannot justify the isotropic tangent-disk constant at every support point. The boundary effect is already illustrated in \citep[Section 2.2.3]{arias2017}.

For the disk example, polar area is $\rho\,d\rho\,d\theta$. The total mass of the radius-$r$ disk is $r^2$. Symmetry makes its mean and cross moment zero. Writing $\alpha=9/10$,
\begin{align*}
 \E[Z_1^2\mid\norm Z\le r]
 &=\frac{\int_0^r\rho^3\,d\rho\int_0^{2\pi}\cos^2\theta(1+\alpha\cos2\theta)\,d\theta}{\pi r^2}
 =r^2\left(\frac14+\frac\alpha8\right),\\
 \E[Z_2^2\mid\norm Z\le r]
 &=r^2\left(\frac14-\frac\alpha8\right).
\end{align*}
The density lies between $1/(10\pi)$ and $19/(10\pi)$.
The angular integrals used above are explicit:
\[
\int_0^{2\pi}\cos^2\theta\,d\theta=\pi,\qquad
\int_0^{2\pi}\cos^2\theta\cos(2\theta)\,d\theta=\pi/2,
\]
by $\cos^2\theta=(1+\cos2\theta)/2$ and
$\int_0^{2\pi}\cos^2(2\theta)\,d\theta=\pi$.
The corresponding sine integral is $-\pi/2$.
Also $\int_0^r\rho^3d\rho=r^4/4$ and the conditioning probability
is $r^2$. Thus the diagonal entries are
$r^2(1/4+\alpha/8)$ and $r^2(1/4-\alpha/8)$.
At $\alpha=9/10$ these are $29r^2/80$ and $11r^2/80$;
the gaps $18r^2/80>11r^2/80$ explain the largest-gap failure. For a point $z$ in the unit disk and $s\le1$, the disk centered at $(1-s/2)z$ with radius $s/2$ lies in both the unit disk and $\B(z,s)$. This proves the quoted lower mass constant; the upper follows by area at most $\pi s^2$.

For comparison, the uniform arclength measure on two perpendicular segments crossing at their midpoints is one-Ahlfors regular, but has reach zero. At the crossing and sufficiently small $r$ its local covariance is $\operatorname{diag}(r^2/6,r^2/6)$. Ahlfors regularity alone therefore does not force one dominant local covariance direction at every point.


\section{Boundary witnesses and reproducible fractal controls}\label{app:boundaryfractal}

\subsection{Exact arc and segment geometry}
We prove the following elementary geometry for completeness; all reach claims refer to \citet{federer1959}.
Let $c\in\mathbb Q^2$, $R\in\mathbb Q_{>0}$, and let $D$ be a rational orthogonal matrix with determinant one.
For $-1\le p_0<p_1\le1$, define
\begin{equation}
 F(p)=c+RD\left(\frac{1-p^2}{1+p^2},\frac{2p}{1+p^2}\right),
 \qquad p\in[p_0,p_1].\label{eq:arcparam}
\end{equation}
The angle $2\arctan p$ increases strictly and spans at most $\pi$; hence $F$ is injective.
Direct differentiation gives
\[
 F'(p)=\frac{RD}{(1+p^2)^2}(-4p,2(1-p^2)),\qquad
 \norm{F'(p)}^2=\frac{R^2[16p^2+4(1-p^2)^2]}{(1+p^2)^4}
 =\frac{4R^2}{(1+p^2)^2}>0.
\]
Thus the image is a compact embedded $C^\infty$ one-manifold with two boundary points.
This is verification of an explicitly proposed object, not an assumption about $E$.
The angle derivative is $d(2\arctan p)/dp=2/(1+p^2)>0$.
Since $p_1-p_0>0$, the curve has positive length, and its endpoints
are distinct. On a compact parameter interval a continuous injection
is a homeomorphism onto its image; the nonzero derivative makes it
an embedded smooth arc. This verifies each part of the witness
predicate before any statistical composition.

For $v=D^T(y-c)$, minimizing distance is equivalent to maximizing
$v\cdot(\cos\theta,\sin\theta)$ over the permitted angle interval.
If $v$ lies in its closed angular cone, radial projection is allowed and the distance is
$|\norm v-R|$. Otherwise a maximizing point is an endpoint and the distance is the smaller endpoint distance.
Cone membership is checked with the two exact cross-product signs of the rational endpoint directions.
At $v=0$, the distance is $R$, so the same radial formula remains valid.
Square roots are enclosed outward; all signs and endpoint comparisons are rational.
Indeed $\norm{v-R(\cos\theta,\sin\theta)}^2=\norm v^2+R^2-2R v\cdot(\cos\theta,\sin\theta)$.
For $v\ne0$, an interior maximum of the dot product occurs only in the direction of $v$; the opposite direction is a minimum. If that maximizing direction is excluded, an endpoint must maximize. On an angular interval of width at most $\pi$, the inequalities
$\det(e_0,v)\ge0$ and $\det(v,e_1)\ge0$ describe its closed cone, where $e_0,e_1$ are the ordered endpoint directions. These signs are exact for rational $v,e_0,e_1$.

Every point at distance less than $R$ from the arc has a unique nearest point.
Indeed the radial maximizer is unique when available and $v\ne0$.
If both endpoints tie as maximizers outside the cone, $v$ is on the opposite angle bisector.
Writing the half-span as $\alpha\le\pi/2$ and $v=-t u$ with $t\ge0$ along that bisector,
the squared endpoint distance is $R^2+t^2+2Rt\cos\alpha\ge R^2$.
The center has all arc points as nearest points at distance $R$.
Consequently the arc has reach exactly $R$, including a semicircle.

A parameter grid of $m$ equal intervals includes both endpoints. The derivative bound above implies
\begin{equation}
 \kappa=\frac{R(p_1-p_0)}m\label{eq:arcnet}
\end{equation}
is a net radius: a parameter is within $(p_1-p_0)/(2m)$ of a grid endpoint.
The program uses $m=64$. For two such nets $G_1,G_2$ with radii $\kappa_1,\kappa_2$, let $\ell_{12}$ be the lower endpoint of the rational outward enclosure of
$\sqrt{\min_{g\in G_1,h\in G_2}\norm{g-h}^2}$, computed as in Appendix~\ref{app:arithmetic}. Put
\begin{equation}
 \underline\Delta=\ell_{12}-\kappa_1-\kappa_2.\label{eq:arcseparation}
\end{equation}
Every pair of actual arc points is separated by at least $\underline\Delta$ by the triangle inequality.
If this quantity is not positive, the geometry verifier abstains.
Otherwise the disjoint union is an embedded manifold with boundary, and
\begin{equation}
 \reach(M_1\cup M_2)\ge\min\{R_1,R_2,\underline\Delta/2\}.\label{eq:unionreach}
\end{equation}
To see this directly, a point at distance less than the right side has a unique projection within either
component. If it had nearest points in both components, their mutual distance would be less than
$\underline\Delta$, a contradiction. These are elementary consequences of reach in the sense of
\citet{federer1959}, not a new reach theory.

For a nondegenerate rational segment $[A,B]$, let $v=B-A$ and
$t=\max\{0,\min\{1,(y-A)\cdot v/\norm v^2\}\}$. The exact closest point is $A+tv$.
Strict convexity of squared distance along the parameter gives uniqueness everywhere, so reach is infinite.
In detail, $\norm{y-A-tv}^2=\norm{y-A}^2-2t(y-A)\cdot v+t^2\norm v^2$ has second derivative $2\norm v^2>0$. Its unconstrained minimizer is the displayed quotient, and minimizing on $[0,1]$ clips it to that interval. The affine parametrization has nonzero derivative $v$ and is injective, establishing the segment's smooth structure with boundary.
The program reports the positive lower enclosure of $\norm v$ as a finite, valid reach bound.
A grid of 256 equal subintervals has radius $\overline{\norm v}/512$.

\subsection{Proposals, independent data, and the Two Moons law}
For the proposals in Section~\ref{sec:families}, \texttt{boundary\_\allowbreak geometry.py} uses two-component single linkage to propose the partition. An algebraic circle fit followed by geometric least squares proposes each center and radius.
The complement of the largest angular gap proposes its angular interval, truncated to span at most $\pi$.
Centers, radii and endpoint parameters are rounded to rational values. The direction is produced by a
rational circle parameterization, so orthogonality is checked exactly, not approximately.
Failure of this numerical fitting procedure cannot make an incorrect geometric check pass.
For a segment, the leading training SVD direction and extreme projected coordinates propose its endpoints.
This remains a restricted proposal class, not a general atlas.

The Two Moons shape follows the familiar two-half-circle benchmark, credited to the scikit-learn
documentation \citep{sklearnmoons}. Our probability law is specified separately: choose either arc with
probability $1/2$ and sample its rational parameter $p$ uniformly on $[-1,1]$.
Before a common transformation the arc centers are $(0,0)$ and $(1,1/2)$ with directions $(0,1)$ and
$(0,-1)$ and radius one. Apply scale $11/10$, translation $(2/5,-3/10)$, and the rotation with first
column $(91/109,60/109)$. All these generator constants are withheld from the proposal and verifier.
We do not call the benchmark's finite regular parameter grid an iid latent sample.

Since $\norm{F'(p)}\le2R$ with $R=11/10$, a parameter interval of length $t/(2R)$ extending in at least
one direction from any parameter remains inside its radius-$t$ ball. For $t\le1/4$ this interval fits
inside $[-1,1]$ in at least one direction. The mixture density is $1/4$ per unit parameter on that arc, so
\[
 \mu(\overline B(x,t))\ge\frac{t}{8R}\ge\frac{t}{16},\qquad 0<t\le1/4.
\]
This proves the supplied lower mass constants, including at endpoints. No density or endpoint property
is estimated from validation observations.

\subsection{The selected Julia set and explicit constants}\label{app:juliaproof}
For completeness, we give the contraction, mass and truncation calculations for this particular polynomial; the Julia-set and iterated-function-system theory retains the attribution stated below.
We use the polynomial $f(z)=z^2-3$, and write
\[
 \beta=\frac{1+\sqrt{13}}2,\qquad I=[-\beta,\beta],\qquad T_\pm(z)=\pm\sqrt{3+z}.
\]
Julia-set definitions and the classical Fatou--Julia escape criterion are recalled by
\citet[Introduction, Theorem A]{qiuyin2006}; the special case below is also directly verifiable.
The contraction construction uses the classical iterated-function-system framework of \citet{hutchinson1981}.
We claim no novelty for these facts.

On the complex disk $D=\{|z|\le\beta\}$, the principal square root is analytic because $\beta<3$.
Both branches map $D$ into itself, since $|T_\pm(z)|\le\sqrt{3+\beta}=\beta$.
Moreover
\[
 |T_\pm'(z)|\le\frac1{2\sqrt{3-\beta}}<\frac35,
\]
where $\beta<83/36$ gives $\sqrt{3-\beta}>5/6$.
These constants follow from rational comparisons:
$13=4212/324<4225/324=(65/18)^2$ implies
$\beta<(1+65/18)/2=83/36$.
Therefore $3-\beta>25/36$ and the derivative is less than $3/5$.
Likewise $13<121/9$ gives $\beta<7/3$.
For real $x$, differentiating twice-composed branches gives
\[
|T_t'(x)|=\frac1{2u},\qquad
|T_s'(T_t(x))|=\frac1{2\sqrt{3+tu}},\qquad
|(T_s\circ T_t)'(x)|=\frac1{4u\sqrt{3+tu}},
\]
where $u=\sqrt{3+x}$ and $t$ is interpreted as the sign $+1$ or $-1$.
The bounds on the three intervals below verify the denominator
without treating the two-step contraction as an assumption.
Integrating the derivative on line segments in $D$ proves the Lipschitz bound.
The two branch images are disjoint and separated by the imaginary axis.
More quantitatively, $\Re\sqrt{3+z}=\sqrt{(|3+z|+3+\Re z)/2}\ge\sqrt{3-\beta}>0$ on $D$, whereas the other branch has the negative of this real part.
The nested inverse images of $D$ therefore converge to their unique compact contraction attractor.
Starting on $I$ gives the same attractor, so it lies on the real axis.
If $|z|>\beta$, then $|f(z)|\ge |z|^2-3>|z|$, and iteration escapes to infinity.
To check the last implication, if the increasing orbit moduli were bounded with limit $L>\beta$, their recurrence would give $L\ge L^2-3$, contradicting $L^2-L-3>0$. Thus bounded orbits are exactly $\bigcap_{n\ge0}f^{-n}(D)$. Iterating the two inverse branches describes this intersection. Each length-$n$ cylinder has diameter at most $2\beta(3/5)^n$, so an infinite address specifies one point. Disjoint branch images imply uniqueness of its address; changing a late symbol produces distinct points arbitrarily nearby. Cylinders of shrinking diameter separate any two addresses, proving total disconnectedness and the absence of isolated points.
Thus the attractor is the filled Julia set; it is a Cantor set with empty planar interior, hence equals
its boundary $J=J_f$. In particular $J$ is not a positive-dimensional manifold.

For the Bernoulli measure, each address of length $n$ has probability $2^{-n}$.
To get an elementary lower mass constant, bound a two-step inverse derivative. Put
$u=\sqrt{3+x}\in[\sqrt{3-\beta},\beta]\subset[5/6,7/3]$. Then
\[
 |(T_s\circ T_t)'(x)|=\frac1{4u\sqrt{3+tu}}\le\frac14.
\]
For $t=+$ the denominator without the factor four exceeds one.
For $t=-$, the squared denominator factor is $u^2(3-u)$.
On $[5/6,1]$, $[1,2]$, and $[2,7/3]$, it is respectively at least $25/18$, $1$, and $8/3$.
Therefore a cylinder of length $2m$ has diameter at most $5\cdot4^{-m}$ and mass $4^{-m}$.
Here the mean value estimate applies to the real interval $I$, each two-step branch has Lipschitz constant at most $1/4$, and $|I|=2\beta<5$. Disjoint addresses and the independent fair signs give mass $(1/2)^{2m}=4^{-m}$ exactly.
For $0<t\le1/4$, choose the smallest $m$ for which $5\cdot4^{-m}\le t$.
The cylinder containing $x$ is inside $\overline B(x,t)$, and
\begin{equation}
 \mu(\overline B(x,t))\ge4^{-m}>t/20.\label{eq:juliamass}
\end{equation}
Thus $(c_0,s,t_0)=(1/20,1,1/4)$ is valid. The exponent $s=1$ is a lower-mass exponent, not a claimed
Hausdorff dimension or manifold dimension. The classical theory is not replaced by this loose bound.
For the strict inequality in \eqref{eq:juliamass}, minimality and $m\ge1$ imply $5\cdot4^{-(m-1)}>t$, hence $4^{-m}>t/20$; every point of the selected cylinder is within its diameter of $x$.

Every nondegenerate local restriction of this measure has covariance rank one in $\mathbb R^2$,
because its support lies on a line. Nevertheless $J$ is totally disconnected and has gaps on arbitrarily
small scales. This is a control for the distinction between covariance rank and manifold regularity;
it does not satisfy the bilateral all-scale hypotheses of Theorem~\ref{thm:localexistence}.
Let $m_0=\sqrt{3-\beta}$. Since $I\setminus(T_+(I)\cup T_-(I))=(-m_0,m_0)$, the central gap is the largest gap: all other gaps are its images
under strictly contracting inverse branches. Consequently
\[
d_H(J,I)=m_0\simeq0.835.
\]
Indeed $J\subset I$, so the forward directed distance is zero. The complement in $I$ is the union of open gaps whose endpoints lie in $J$. On a gap of length $l$, the maximal distance to its endpoints is $l/2$. Every noncentral gap is a branch image of the central gap and has strictly smaller length by the derivative bound. The central gap has length $2m_0$ and its midpoint attains distance $m_0$, proving the equality.
Section~\ref{sec:certificateexamples} discusses the implications for coarse support and tube-mass certification.

The code generates a separate independent 64-letter address for every latent point. Truncation error is
at most $5(3/5)^{64}<10^{-12}$, and is absorbed in the declared coordinate recording bound $10^{-10}$.
Gaussian PRNG and floating-point generation accuracy remain external contracts, as in the earlier
experiments; this is not a machine proof of a physical noise law. The simulation draws means and the
independent pooled residual from their exact Gaussian sufficient-statistic law.

\subsection{Population assessment by outward enclosures}
For completeness, the common enclosure argument is as follows. Partition the reference law into cells $C_j$ of known masses $w_j$, and choose reference centers $x_j$ with $\norm{x-x_j}\le\rho_j$ for $x\in C_j$. If $l_j\le\dist(x_j,M)\le u_j$, then
\begin{gather*}
 L_j:=\max(0,l_j-\rho_j)\le\dist(x,M)\le U_j:=u_j+\rho_j,
 \\
 \sum_jw_j\mathbf1\{L_j>r_*\}\le\mu\{\dist(z,M)>r_*\}
 \le\sum_jw_j\mathbf1\{U_j>r_*\}.
\end{gather*}
The lower sum counts only certainly outside cells; the upper sum includes all undecided cells. This is a deterministic enclosure, with no Monte Carlo confidence claim.
For Two Moons, 128 midpoint cells on each true rational parameter interval have equal probability $1/256$.
The derivative bound gives radius $2R/128$ for every cell. Evaluating a candidate distance at the midpoint
and adding/subtracting that radius encloses every distance in the cell, because distance to a set is
1-Lipschitz. A cell is certainly outside the tube only if its lower bound exceeds the radius, and certainly
inside only if its upper bound does not exceed it. Unclassified cells enter the upper probability bound.
Analogous nets on the proposed arcs enclose the reverse support distance. Evaluation uses references
only after the candidate and certificate have been saved.

For Julia, outward square-root intervals construct all $2^{10}$ real inverse cylinders, each of measure
$2^{-10}$. The same distance-Lipschitz argument encloses their outside probability.
Distance to a segment is convex along $I$ and $\pm\beta\in J$, so the forward support supremum occurs at
one of these two endpoints. For the reverse direction, distance to the union of outer cylinder intervals
is a lower bound on distance to $J$. Adding the maximum cylinder width gives an upper bound, because
each interval contains some point of $J$. A 256-cell candidate segment net then encloses the supremum.
Every square root and comparison in these reference assessments uses rational outward arithmetic.
For the forward Julia bound, if $y=(1-t)(-\beta)+t\beta$, convexity of distance to the convex segment gives
$\dist(y,M)\le(1-t)\dist(-\beta,M)+t\dist(\beta,M)\le\max\{\dist(-\beta,M),\dist(\beta,M)\}$; both endpoints belong to $J$. For the reverse bound, writing $O=\bigcup C_j\supset J$ and $w=\max_j\operatorname{diam}C_j$, choose a nearest point of $O$ and then any Julia point in its interval to obtain $\dist(y,O)\le\dist(y,J)\le\dist(y,O)+w$. A candidate net of radius $\kappa$ adds at most $\kappa$ to the maximum of these pointwise bounds.

An emitted assertion is classified as \texttt{HOLDS} when the upper reference bound is at or below the
reported bound, as \texttt{VIOLATION} when the lower reference bound exceeds it, and otherwise as
\texttt{UNRESOLVED}. Undecided evaluation intervals are never silently counted as success.
Requested-target and reported-bound violations are recorded separately.

\subsection{Finite Mandelbrot proxy audit}\label{app:mandelbrot}
Six archived runs (three seeds at each of 128 and 512 iterations) use a $640\times480$ boundary-pixel proxy in $[-2.05,0.65]\times[-1.15,1.15]$. Pixels surviving $z\leftarrow z^2+c$ and adjacent to nonsurvivors form an explicitly finite sampling law. A training-fitted segment receives no mass certificate in any run; whole-support issuance is disabled because no lower-mass contract is supplied. Population outside mass is assessed over all proxy points with rational bounds. Figure~\ref{fig:fractalcontext} shows the saved 512-iteration grid for visual orientation only; these finite runs do not certify membership in or a Hausdorff approximation to the infinite boundary. The dimension theorem for the true Mandelbrot boundary is due to \citet{shishikura1991}; it is not inferred from this proxy.

\subsection{Execution, metrics, and historical replay}
\texttt{run\_\allowbreak moons\_\allowbreak fractals.py} saves the mathematical design before generation: four Two Moons noise
levels $(.02,.06,.12,.24)$ at $(N,k)=(4096,128)$, and three Julia replication levels $(16,64,256)$ at
$N=4096$ and $\sigma=.06$, each with 30 independent trials. The support tolerances are $0.3$ and $1.3$,
respectively; both use reach threshold $0.12$, fixed tube radius $0.12$, and outside-mass tolerance $0.08$.
The first quarter of means is used for fitting. The Julia bounds are additionally evaluated against the
predeclared support tolerances $0.6,0.9,1.3$ on the same inputs; these are paired uses of the same bound,
not extra independent trials. Six Mandelbrot-proxy runs use three seeds at each iteration depth.
Probability budgets remain $(.01,.01,.03,.03)$, respectively for variance, mean noise, support coverage,
and validation. All four methods share the candidate, split, verified geometry, and targets.

The complete 210 main-study input files and six proxy inputs, with their requests and outputs, are preserved in the historical recovery archive identified in the package provenance. This compact package retains the summaries and selected display inputs.
The upper visualization row uses separately saved references; the lower row reads actual recorded means
and fitted witnesses from the public pipeline. Matching point identities share colors; reference colors
never enter estimation. Figure generation is separate from verification.
The earlier circle studies are retained in the appendix, including unfavorable outcomes.
Their old code hashes are preserved in \texttt{code/\allowbreak historical\_\allowbreak v9/\allowbreak }; the historical replay launcher uses
that snapshot. Old AI-review scores concern v9 and are not transferred to this revision.


\subsection{Complete executed comparison counts}
\begin{longtable}{llrrrrrr}
\caption{Each row uses 30 independent trials. S/M: support/mass emissions. SB/MB: certified-reference violations of an emitted reported bound. ST/MT: violations of a requested target. P/O/Q are numerical reference rules; R is the actual verifier.}\label{tab:mffull}\\
\toprule Condition & Rule & S & M & SB & MB & ST & MT\\\midrule\endfirsthead
\toprule Condition & Rule & S & M & SB & MB & ST & MT\\\midrule\endhead
mf00 & P & 30 & 30 & 0 & 0 & 0 & 0\\
mf00 & O & 30 & 30 & 0 & 0 & 0 & 0\\
mf00 & Q & 30 & 30 & 0 & 0 & 0 & 0\\
mf00 & R & 30 & 30 & 0 & 0 & 0 & 0\\
mf01 & P & 30 & 30 & 0 & 0 & 0 & 0\\
mf01 & O & 30 & 30 & 0 & 0 & 0 & 0\\
mf01 & Q & 30 & 30 & 0 & 0 & 0 & 0\\
mf01 & R & 30 & 30 & 0 & 0 & 0 & 0\\
mf02 & P & 30 & 30 & 0 & 0 & 0 & 0\\
mf02 & O & 30 & 30 & 0 & 0 & 0 & 0\\
mf02 & Q & 30 & 30 & 0 & 0 & 0 & 0\\
mf02 & R & 30 & 30 & 0 & 0 & 0 & 0\\
mf03 & P & 30 & 30 & 0 & 0 & 0 & 0\\
mf03 & O & 23 & 0 & 0 & 0 & 0 & 0\\
mf03 & Q & 23 & 0 & 0 & 0 & 0 & 0\\
mf03 & R & 0 & 0 & 0 & 0 & 0 & 0\\
mf04 & P & 30 & 30 & 30 & 0 & 0 & 0\\
mf04 & O & 30 & 0 & 0 & 0 & 0 & 0\\
mf04 & Q & 30 & 0 & 0 & 0 & 0 & 0\\
mf04 & R & 30 & 0 & 0 & 0 & 0 & 0\\
mf05 & P & 30 & 30 & 30 & 0 & 0 & 0\\
mf05 & O & 30 & 30 & 0 & 0 & 0 & 0\\
mf05 & Q & 30 & 30 & 0 & 0 & 0 & 0\\
mf05 & R & 30 & 30 & 0 & 0 & 0 & 0\\
mf06 & P & 30 & 30 & 6 & 0 & 0 & 0\\
mf06 & O & 30 & 30 & 0 & 0 & 0 & 0\\
mf06 & Q & 30 & 30 & 0 & 0 & 0 & 0\\
mf06 & R & 30 & 30 & 0 & 0 & 0 & 0\\
\bottomrule\end{longtable}
The 24 unresolved assessments are all reported-support-bound assessments for P at mf06 (Julia, $k=256$); the other six P support bounds at that condition are certified-reference violations. Unresolved cases are not counted as successes.
A zero count among 30 repetitions has one-sided 95\% upper limit $1-0.05^{1/30}\simeq9.50\%$. This bound is condition-specific. It does not pool conditions or estimate coverage conditional on issuance; rows with no emissions provide no such coverage evidence.


\section{Deterministic geometry and finite nets}\label{app:geometry}

\subsection{Retained circle and torus families}
The historical executable experiments use origin-centered circles in $\R^2$ and ring tori in $\R^3$ with fixed axis and center. They are restricted proposal families, not assumptions about the true support. Parameters are fitted using the first quarter of replicate means and rounded to exact rational numbers. We do not estimate orientation, center, or dimension in these experiments.

For circle radius $R>0$, $\reach(M)=R$. A ring torus with $R>r>0$ is
\begin{equation}
F(\theta,\phi)=((R+r\cos\phi)\cos\theta,
(R+r\cos\phi)\sin\theta,r\sin\phi),
\label{eq:torus}
\end{equation}
and has reach $\min\{r,R-r\}$. Its point distance is
\begin{equation}
\dist((x,y,z),M)=
\left|\sqrt{(\sqrt{x^2+y^2}-R)^2+z^2}-r\right|.
\label{eq:torusdistance}
\end{equation}
The following calculations prove these facts and specify rational nets.

\subsection{Circle and torus distance and reach}
These are elementary verifications using the reach definition of \citet{federer1959}; we prove them for completeness. On a circle $M_R$ in $\R^2$, Cauchy--Schwarz gives
\[
 \min_{\norm u=1}\norm{y-Ru}^2
 =\norm y^2+R^2-2R\max_{\norm u=1}y\cdot u=(\norm y-R)^2.
\]
For $y\ne0$ the maximizer is uniquely $u=y/\norm y$. At $y=0$ every circle point is nearest, at distance $R$. Thus $\dist(y,M_R)=|\norm y-R|$ and $\reach(M_R)=R$.

For the ring torus, write an ambient point as $y=(\rho\cos\psi,\rho\sin\psi,z)$, $\rho\ge0$. With $a(\phi)=R+r\cos\phi>0$,
\[
 \norm{y-F(\theta,\phi)}^2
 =\rho^2+a(\phi)^2-2\rho a(\phi)\cos(\theta-\psi)+(z-r\sin\phi)^2.
\]
For $\rho>0$, minimizing in $\theta$ fixes its azimuth uniquely modulo $2\pi$. The remaining expression is
\[
 (\rho-R)^2+z^2+r^2-2r[(\rho-R)\cos\phi+z\sin\phi].
\]
Set $s_y=\sqrt{(\rho-R)^2+z^2}$. Its minimum is $(s_y-r)^2$, with unique meridian $(\cos\phi,\sin\phi)=((\rho-R)/s_y,z/s_y)$ when $s_y>0$. For $\rho=0$ the azimuth is arbitrary but the same minimum value follows, proving \eqref{eq:torusdistance} in all cases. On the axis,
$\dist(y,M)=\sqrt{R^2+z^2}-r\ge R-r$; on the central circle $(\rho,z)=(R,0)$ the distance is $r$ and every meridian is minimizing. Away from these two loci the nearest point is unique. Therefore the reach is at least $\min(r,R-r)$. The origin attains distance $R-r$ with multiple azimuths, and the central circle attains distance $r$ with multiple meridians, proving equality.

Finally $\norm{F_\theta}=R+r\cos\phi>0$, $\norm{F_\phi}=r>0$, and $F_\theta\cdot F_\phi=0$. The positive cylindrical radius recovers $\theta$, and $((\rho-R)/r,z/r)$ recovers $\phi$ modulo $2\pi$. Thus the map from the compact product of circles is an injective immersion and a homeomorphism onto its image. This verifies a compact smooth embedded surface, rather than assuming one.

\subsection{Rational witness nets and both fit directions}
For completeness, differentiate the rational circle parameterization:
\[
 v(p)=\left(\frac{1-p^2}{1+p^2},\frac{2p}{1+p^2}\right),\quad
 v'(p)=\left(\frac{-4p}{(1+p^2)^2},\frac{2(1-p^2)}{(1+p^2)^2}\right),\quad
 \norm{v'(p)}^2=\frac4{(1+p^2)^2}.
\]
On $[0,1]$, $\norm{v(p)-v(p')}\le2|p-p'|$ by integration. The grid $j/m$, rotated through four quadrants, includes the shared endpoints and has $4m$ distinct points. Every parameter lies within $1/(2m)$ of an endpoint, so the unit-circle net radius is $1/m$; scaling gives $R/m$. The code uses $m=128$.

For torus nets, write $a_\theta=(\cos\theta,\sin\theta)$. Replacing the azimuth with a net direction changes $F$ by at most $(R+r)\norm{a_\theta-a_{\theta'}}\le(R+r)/m_\theta$. At a fixed azimuth, replacing the meridian changes it by $r\norm{v_\phi-v_{\phi'}}\le r/m_\phi$. Thus
\[
 |G|=(4m_\theta)(4m_\phi),\qquad
 \kappa=(R+r)/m_\theta+r/m_\phi.
\]
With $(m_\theta,m_\phi)=(64,32)$ this gives $32768$ points. All coordinates are rational when $R,r$ are rational.

For any $z\in M$, choose $g\in G$ with $\norm{z-g}\le\kappa$. Then
\[
 \dist(z,Y)\le\norm{z-g}+\dist(g,Y)
 \le\kappa+\max_{g\in G}\overline d(g,Y).
\]
The opposite direction is $\sup_{y\in Y}\dist(y,M)\le\max_i\overline d(Y_i,M)$. Taking the maximum proves \eqref{eq:globalfit}. A proposed neighbor index $j(g)$ obeys $\dist(g,Y)\le\norm{g-Y_{j(g)}}$ even if it is not a nearest neighbor. Exact evaluation of that selected norm is consequently sufficient for soundness.

\subsection{Local plane disks and projected grids}
Write $L_y=y+V$, $P$ for its linear orthogonal projection, and $D_y=L_y\cap\overline B(y,r)$. For $p\in\overline B(y,r)$,
\[
 \norm{P(p-y)}\le\norm{p-y}\le r,\qquad
 \dist(p,D_y)=\norm{(I-P)(p-y)}=\dist(p,L_y).
\]
This proves the forward term in \eqref{eq:netcheck}. Its reverse term follows from the preceding 1-Lipschitz distance argument on a $\kappa$-net of $D_y$.

Here is also the required grid calculation. A product grid containing both endpoints of each coordinate interval, with gaps at most $\ell$, has for each point in $[-r,r]^d$ a node at distance at most $\sqrt d\ell/2$. Projection $\pi_C$ to a closed convex set is nonexpansive: if $a=\pi_C u$, minimization along $a+t(c-a)$ gives $\langle u-a,c-a\rangle\le0$ for $c\in C$. Apply this with $c=b=\pi_C v$, and the corresponding inequality with $a,b$ interchanged. Adding yields
\[
 \norm{a-b}^2\le\langle a-b,u-v\rangle\le\norm{a-b}\norm{u-v}.
\]
Thus $\norm{\pi_Cu-\pi_Cv}\le\norm{u-v}$. With $C$ the radius-$r$ ball and a target point already in $C$, projected grid nodes retain the same covering radius. This proves the stated finite bilateral test. Unchecked floating-point orthogonality or omitted centers are not covered by these exact calculations.

\section{Verified arithmetic used in the experiments}\label{app:arithmetic}
The probability inequalities below are classical; see \citet{blm2013} for concentration methods and \citet{massart1990} for the sharp empirical-distribution bound. For completeness, we derive each conservative expression used in the program. Exact arithmetic validates the calculation conditional on the stated statistical and recording contracts; it does not validate those contracts from the data.

\subsection{Square-root enclosures and distance arithmetic}
Finite input coordinates are exact binary rationals via \texttt{Fraction.from\_\allowbreak float}. Rounded proposal parameters define the actual rational witness. For $v=A/B\ge0$, $A,B$ integers with $B>0$, set
\[
 D=2^{48},\qquad j=\left\lfloor\sqrt{\left\lfloor D^2A/B\right\rfloor}\right\rfloor.
\]
Since $j^2\le\lfloor D^2v\rfloor\le D^2v$ and the integer $(j+1)^2$ exceeds $\lfloor D^2v\rfloor$, it also exceeds $D^2v$. Therefore
\[
 j/D\le\sqrt v<(j+1)/D.
\]
If $j^2/D^2=v$ exactly, both returned endpoints may equal $j/D$. Otherwise the upper endpoint is $(j+1)/D$. For an interval $[l,u]$, squaring gives lower endpoint $0$ if $0\in[l,u]$, otherwise $\min(l^2,u^2)$, and upper endpoint $\max(l^2,u^2)$. Applying the monotone square root to nonnegative endpoints and ordinary outward interval operations proves the nested distance enclosures. Every acceptance comparison is between rationals, never an unbounded floating-point approximation.

\subsection{One-sided variance bound and recorded means}
For a centered random variable $W$ of variance $v>0$, $t>0$, and $b\ge0$, on $W\le-t$ one has $(W-b)^2\ge(t+b)^2$. Hence
\[
 \Prob(W\le-t)\le\frac{\E(W-b)^2}{(t+b)^2}
 =\frac{v+b^2}{(t+b)^2}.
\]
Choosing $b=v/t$ gives Cantelli's bound $v/(v+t^2)$. For $Q\sim\chi^2_\nu$, Gaussian moments give $\E Q=\nu$ and $\operatorname{Var}Q=\nu(3-1)=2\nu$. With
\[
 t=\left\lceil\sqrt{2\nu(1-\delta_C)/\delta_C}\right\rceil,
 \qquad\ell_C=\nu-t,
\]
we obtain
\[
 \Prob(Q\le\ell_C)\le\frac{2\nu}{2\nu+t^2}
 \le\frac{2\nu}{2\nu+2\nu(1-\delta_C)/\delta_C}=\delta_C.
\]
If $\ell_C\le0$ there is no usable division and the routine abstains. Otherwise, on $Q\ge\ell_C$ and with $|S-\widetilde S|\le e_S$,
\[
 \sigma^2\le S/\ell_C\le\frac{\max(0,\widetilde S+e_S)}{\ell_C}=U.
\]
A recorded value with $\widetilde S+e_S<0$ is incompatible with $S\ge0$ and is rejected as an input error. For $\sigma=0$ the variance bound is immediate.

The Gaussian integral from Appendix~\ref{app:main}, evaluated at $\lambda=1/4$, gives
\[
 \E e^{\norm G^2/4}=2^{q/2},\qquad
 \Prob\{\norm G^2\ge4m\}\le2^{q/2}e^{-m}
 \le2^{\lceil q/2\rceil-m}.
\]
The last inequality uses $e\ge2$ and $m\ge0$. The integer search chooses $m$ so that $N2^{\lceil q/2\rceil-m}\le\delta_L$. Thus, outside an event of probability at most $\delta_C+\delta_L$, all ideal mean errors are at most $\sqrt{4mU/k}$. If every recorded mean coordinate differs from its ideal mean by at most $e_Y$,
\[
 \norm{\widetilde Y_i-Y_i}^2\le qe_Y^2,
 \quad \norm{\widetilde Y_i-Z_i}\le\sqrt{4mU/k}+\sqrt q\,e_Y
 \le\overline{\sqrt{4mU/k}}+qe_Y.
\]
This proves the implemented radius \eqref{eq:rationalnoise} for the recorded point cloud.

\subsection{Rational coverage search and validation slack}
For $x\ge0$, $e^{-x}\le e^{-\lfloor x\rfloor}\le2^{-\lfloor x\rfloor}$. Therefore the rational sufficient bound is
\[
 B(h)=\frac{4^s}{c_0h^s}\,2^{-\lfloor Nc_0(h/2)^s\rfloor}\le\delta_E.
\]
For positive integer $s$, every operation here is rational. Both $h^{-s}$ and the second factor are nonincreasing for $h>0$, so their positive product is nonincreasing. On the grid $h=j10^{-6}$, $1\le j\le\lfloor2t_0\,10^6\rfloor$, integer bisection finds the first passing point. If the grid is empty or its last point fails, no grid value satisfies the test and the routine abstains. For a fixed design the selected $h$ is deterministic. The proof does not justify choosing external mass constants adaptively from these observations. The theoretical result itself allows real $s>0$.

For validation choose an integer $v$ with $2^{1-v}\le\delta_V$ and an upper enclosure $u\ge\sqrt{v/(2n)}$. Directly,
\[
 2e^{-2nu^2}\le2e^{-v}\le2^{1-v}\le\delta_V.
\]
The rational frequency $n^{-1}\sum_i\mathbf1\{\bar d_i>t\}$ plus $u$, clipped at one, is thus an admissible DKW-based mass bound. The experimental choice is $v=7$. The same simultaneous mean-noise event can cover training and validation, so use of a global residual statistic does not require an independence assumption on that statistic. The witness and its training preprocessing must still be independent of the validation sample. The per-run support and mass error bounds are $\delta_C+\delta_L+\delta_E$ and $\delta_C+\delta_L+\delta_V$; their union is bounded by $\delta_C+\delta_L+\delta_E+\delta_V$, charging their shared noise event once.

\subsection{Exact aggregation of recorded repeats}\label{app:rawprocessing}
Let $\widetilde X$ be the recorded binary vector and assume an external coordinate error bound $|X_j-\widetilde X_j|\le b$. Write $D_0=Nkq$ and let $A$ be the blockwise orthogonal projection subtracting replicate-group means. If one coordinate's $k$ recorded values are $a_\ell/D$ with a common power-of-two denominator, then
\begin{align*}
 \sum_{\ell=1}^k\left(\frac{a_\ell}{D}-\frac{\sum_m a_m}{kD}\right)^2
 &=\frac{\sum_\ell a_\ell^2}{D^2}
   -\frac{2(\sum_\ell a_\ell)^2}{kD^2}
   +\frac{(\sum_\ell a_\ell)^2}{kD^2}\\
 &=\frac{k\sum_\ell a_\ell^2-(\sum_\ell a_\ell)^2}{kD^2}\ge0.
\end{align*}
The last inequality follows from $(\sum a_\ell)^2\le k\sum a_\ell^2$. Summing gives the exact rational $S_r=\norm{A\widetilde X}^2$.

Set $h=A(X-\widetilde X)$. Orthogonal projection is a contraction, so $\norm h\le\sqrt{D_0}b$. Expanding rather than replacing an ideal residual by a recorded one,
\begin{align*}
 S-S_r&=\norm{A\widetilde X+h}^2-\norm{A\widetilde X}^2
       =2\langle A\widetilde X,h\rangle+\norm h^2,\\
 |S-S_r|&\le2\sqrt{S_r}\norm h+\norm h^2
       \le2b\sqrt{D_0S_r}+D_0b^2.
\end{align*}
If $\rho_S=|\widetilde S-S_r|$ is the exact rational-to-stored conversion error and $\rho_Y$ is the largest mean-coordinate conversion error, the triangle inequality yields
\[
 e_S=\rho_S+2b\,\overline{\sqrt{D_0S_r}}+D_0b^2,
 \qquad e_Y=b+\rho_Y.
\]
Indeed each ideal-to-recorded arithmetic mean differs by at most $k^{-1}\sum_\ell b=b$. These are the implementation's declared error bounds. Aggregation uses $O(Nkq)$ arithmetic operations; rational bit costs additionally depend on input precision. The external input-error and Gaussian contracts, including any pseudorandom-generation accuracy premise, remain assumptions.


\section{Matched-target calibration and boundary experiments}\label{app:round2}
This study was added after the first internal review. Its design file is written before generating the data; it is not an externally time-stamped preregistration. The script fixes 17 conditions, 60 repetitions each, and the seed vector $(20260919,902,j,t)$ for condition $j$ and trial $t$. All inputs contain only recorded means, a pooled residual, and the repeat count. The public verifier is run on each saved file; latent geometry and generator variance are used only in the separate evaluation code. These 1020 runs supplement, rather than regenerate, the 143 historical runs. Conditions are independent, while comparison methods within a condition/trial share exactly the same input and candidate.

\subsection{Population model and exact evaluation formulas}
For completeness, we spell out the exact geometry and probability calculations; the statistical guarantee is the application of Theorems~\ref{thm:main} and~\ref{thm:tube}, with their cited inputs.
Let $C$ be the unit circle, $I_L=[1,1+L]\times\{0\}$, and
\[
\mu_{p,L}=(1-p)\operatorname{Unif}(C)+p\operatorname{Unif}(I_L),\qquad 0\le p<1.
\]
For $p>0$ the support is a circle with an attached branch, so it is not a boundaryless manifold. From a circle point, a ball of radius $t\le1/4$ contains circle probability at least $(1-p)t/\pi\ge(1-p)t/4$. From a segment point, the intersection with the segment has length at least $t$ when $t\le L$. Thus valid external constants are
\[
c_0=\min\{(1-p)/4,p/L\},\quad s=1,\quad t_0=\min\{1/4,L\}.
\]
For $p=0$ use $(c_0,s,t_0)=(1/4,1,1/4)$. These are population contracts supplied from the declared generator, not inferred from data. They remain valid at endpoints and the attachment point.

Let the training fit produce the rational circle radius $\widehat R$. Both directed support distances give, for $p>0$,
\[
d_H(C\cup I_L,M_{\widehat R})=\max\{|1-\widehat R|,|1+L-\widehat R|\}.
\]
Indeed the segment-to-circle distance is $|r-\widehat R|$, maximized at an endpoint; the reverse distance is at most $|1-\widehat R|$ because $C$ is included in the support. For $p=0$ the distance is $|1-\widehat R|$. The reach is $\widehat R$. The exact outside probability at requested radius $r_*$ is
\[
m_*=(1-p)\mathbf1\{|1-\widehat R|>r_*\}
 +p\left(1-\frac{[\min(1+L,\widehat R+r_*)-\max(1,\widehat R-r_*)]_+}{L}\right).
\]
For the mass formula, the segment coordinate is uniform on $[1,1+L]$, and $|r-\widehat R|\le r_*$ holds precisely on its intersection with $[\widehat R-r_*,\widehat R+r_*]$. The intersection length is the positive-part numerator displayed above. Dividing by $L$ gives inside probability, and subtracting from one gives the segment contribution. The circle contribution is constant in angle. All parameters in this evaluation, including $\widehat R$, are rational; the evaluator uses Fraction arithmetic. The formula evaluates the population law, not a finite reference-cloud approximation.

\subsection{Fixed radius and comparison rules}
The new public option sets $t=r_*-a$ exactly. If $a>r_*$ it returns $\ABSTAIN$ with a distinct reason and the trivial outside-mass upper bound one. Otherwise Theorem~\ref{thm:tube}'s event holds simultaneously at this random nonnegative threshold, since both the DKW event and the noise event are uniform over the relevant distances. Hence the claimed radius is exactly $r_*$, with unchanged error budget. Requests containing both a fixed latent radius and a fixed observed threshold are rejected. Training recording and preprocessing must not depend on validation data; a bounded error magnitude by itself does not establish that condition.

The numerical oracle uses $U=\sigma^2$, and the unknown-variance reference uses $U=(\widetilde S+e_S)/q_{\nu,\delta_C}$. Both use the Gaussian bound in \eqref{eq:noise}, a root of the exponential coverage inequality, and $u=\sqrt{\log(2/\delta_V)/(2n)}$. The rational method instead uses Section~\ref{sec:arithmetic}. Every comparison uses the same training-fitted candidate and the same rationally computed bilateral fit bound. Thus geometry is not refitted to favor one calibration. The plug-in diagnostic uses $a=h=u=0$ and makes uncalibrated decisions from the observed distances; it carries no error guarantee. The two quantile references also do not issue machine-verified certificates because their special functions and roots use ordinary floating point. The oracle is privileged by knowing $\sigma$ and is identified as such.

The budgets are $(\delta_C,\delta_L,\delta_E,\delta_V)=(0.01,0.01,0.03,0.03)$ for all comparisons; the oracle leaves the calibration budget unused. For support and mass separately, the rational method's per-run false-emission bound is 0.05, and their union bound is 0.08. No new simultaneous guarantee over experimental repetitions is claimed.

The mass study fixes $N=8192$, $k=128$, $\sigma=0.04$, $L=0.3$, $r_*=0.15$ and varies the branch mass over
\[p\in\{0,0.06,0.10,0.14,0.16,0.165,0.20\}.\]
 The support study retains these sampling settings, uses $p=0.2$ and $L\in\{0.08,0.20,0.24,0.28,0.32\}$. The budget study uses the circle, $\sigma=0.06$, $r_*=0.08$ and uses the fixed-budget pairs
\[(N,k)\in\{(1024,256),(2048,128),(4096,64),(8192,32),(16384,16)\}.\]
 Requested support error, reach floor, and mass are always $0.3,0.75,0.08$.

\subsection{Error events, intervals, and compute accounting}
For every repetition we record emission, emission with a false reported numerical bound, and emission violating the weaker requested tolerance, separately for support and mass. Each frequency uses all 60 repetitions as denominator. This is not an error probability conditional on emission. Pointwise two-sided 95\% binomial intervals invert the binomial tails as in \citet{clopper1934}. For $x=0$ errors out of $n$ the one-sided 95\% upper limit is $1-0.05^{1/n}$, which is about 0.04870 for $n=60$. For completeness, with $X\sim\operatorname{Bin}(n,p)$, the one-sided zero-count upper endpoint solves $\Prob_p(X=0)=(1-p)^n=0.05$. More generally the equal-tailed Clopper--Pearson endpoints solve
\[
 \sum_{j=x}^n\binom njp_L^j(1-p_L)^{n-j}=0.025\quad(x>0),\qquad
 \sum_{j=0}^x\binom njp_U^j(1-p_U)^{n-j}=0.025\quad(x<n),
\]
with $p_L=0$ at $x=0$ and $p_U=1$ at $x=n$ \citep{clopper1934}. The first tail is increasing in $p$ and the second decreasing, as is seen by coupling Bernoulli trials with common uniform variables. To see coverage directly, let $T_p^+(x)=\Prob_p(X\ge x)$ and $T_p^-(x)=\Prob_p(X\le x)$. Tail monotonicity in $p$ gives
\[
 \{p<p_L(X)\}\subseteq\{T_p^+(X)\le0.025\},\qquad
 \{p>p_U(X)\}\subseteq\{T_p^-(X)\le0.025\}.
\]
Each right-hand event has probability at most $0.025$: the first is an upper-tail set $\{X\ge k\}$ with that tail probability, and the second is a lower-tail set, with an empty set allowed. Adding proves coverage at least $0.95$, with conservatism from discreteness. These reported intervals use numerical quantiles, not the verifier's exact rational decision arithmetic. Neither these intervals nor the observed zero counts establish uniform validity over the model class. No across-condition pooling or multiplicity-adjusted claim is made.

Timing rows below report only calibration for all four methods, excluding shared fitting/geometry and file I/O. A separate full-pipeline timing is reported for the rational method. Comparing its complete runtime with another method's calibration-only time would be invalid. Runs use the recorded CPU/Python/NumPy/SciPy environment; timings are machine-dependent and are not acceptance criteria. Simulation uses the exact Gaussian sufficient-statistic distribution, with declared recording tolerances $e_Y=10^{-12},e_S=10^{-8}$; end-to-end PRNG accuracy is not formally proved.

\begin{center}\small\begin{tabular}{lrrrrrr}\toprule
Method ($p=0.10$) & $a$ & $h$ & $u$ & Calib. ms & Full ms & Mass emits\\\midrule
P & 0.00000 & 0.00000 & 0.00000 & 0.004 & -- & 60/60\\
O & 0.02345 & 0.01175 & 0.01849 & 0.169 & -- & 55/60\\
Q & 0.02347 & 0.01175 & 0.01849 & 0.465 & -- & 55/60\\
R & 0.03256 & 0.01736 & 0.02387 & 0.513 & 292.51 & 3/60\\
\bottomrule\end{tabular}\end{center}

At fixed $Nk$, increasing $N$ improves coverage and validation while
reducing replication and worsening the noise radius. The component
values remain in the tables below; Figure~\ref{fig:matched} reports the
resulting matched-target decisions.

\subsection{All conditions, including abstentions and diagnostic failures}
P/O/Q/R denote the four rules above. S and M are support and mass emission counts. $F_S,F_M$ count emissions with a false reported bound, not merely a false target tolerance. CIs are available for every frequency in the saved JSON/CSV; the table reports the mass issuance interval for compactness. Values are pointwise, not simultaneous across rows.
{\small\setlength{\tabcolsep}{4pt}
\begin{longtable}{llrrrrlrr}
\toprule
Condition & Rule & S & M & $F_S$ & $F_M$ & M rate 95\% CI & $a$ & $h$\\\midrule
\endfirsthead
\toprule Condition & Rule & S & M & $F_S$ & $F_M$ & M rate 95\% CI & $a$ & $h$\\\midrule
\endhead
$p=0.000$ & P & 60 & 60 & 0 & 0 & [0.940, 1.000] & 0.0000 & 0.0000\\
$p=0.000$ & O & 60 & 60 & 0 & 0 & [0.940, 1.000] & 0.0234 & 0.0106\\
$p=0.000$ & Q & 60 & 60 & 0 & 0 & [0.940, 1.000] & 0.0235 & 0.0106\\
$p=0.000$ & R & 60 & 60 & 0 & 0 & [0.940, 1.000] & 0.0326 & 0.0156\\
$p=0.060$ & P & 9 & 60 & 8 & 29 & [0.940, 1.000] & 0.0000 & 0.0000\\
$p=0.060$ & O & 0 & 60 & 0 & 0 & [0.940, 1.000] & 0.0234 & 0.0132\\
$p=0.060$ & Q & 0 & 60 & 0 & 0 & [0.940, 1.000] & 0.0235 & 0.0132\\
$p=0.060$ & R & 0 & 60 & 0 & 0 & [0.940, 1.000] & 0.0326 & 0.0195\\
$p=0.100$ & P & 3 & 60 & 2 & 39 & [0.940, 1.000] & 0.0000 & 0.0000\\
$p=0.100$ & O & 0 & 55 & 0 & 0 & [0.816, 0.972] & 0.0234 & 0.0117\\
$p=0.100$ & Q & 0 & 55 & 0 & 0 & [0.816, 0.972] & 0.0235 & 0.0117\\
$p=0.100$ & R & 0 & 3 & 0 & 0 & [0.010, 0.139] & 0.0326 & 0.0174\\
$p=0.140$ & P & 2 & 59 & 0 & 29 & [0.911, 1.000] & 0.0000 & 0.0000\\
$p=0.140$ & O & 0 & 0 & 0 & 0 & [0.000, 0.060] & 0.0234 & 0.0123\\
$p=0.140$ & Q & 0 & 0 & 0 & 0 & [0.000, 0.060] & 0.0235 & 0.0123\\
$p=0.140$ & R & 0 & 0 & 0 & 0 & [0.000, 0.060] & 0.0326 & 0.0182\\
$p=0.160$ & P & 0 & 29 & 0 & 27 & [0.352, 0.616] & 0.0000 & 0.0000\\
$p=0.160$ & O & 0 & 0 & 0 & 0 & [0.000, 0.060] & 0.0234 & 0.0126\\
$p=0.160$ & Q & 0 & 0 & 0 & 0 & [0.000, 0.060] & 0.0235 & 0.0126\\
$p=0.160$ & R & 0 & 0 & 0 & 0 & [0.000, 0.060] & 0.0326 & 0.0186\\
$p=0.165$ & P & 0 & 15 & 0 & 15 & [0.147, 0.379] & 0.0000 & 0.0000\\
$p=0.165$ & O & 0 & 0 & 0 & 0 & [0.000, 0.060] & 0.0234 & 0.0127\\
$p=0.165$ & Q & 0 & 0 & 0 & 0 & [0.000, 0.060] & 0.0235 & 0.0127\\
$p=0.165$ & R & 0 & 0 & 0 & 0 & [0.000, 0.060] & 0.0326 & 0.0187\\
$p=0.200$ & P & 1 & 0 & 0 & 0 & [0.000, 0.060] & 0.0000 & 0.0000\\
$p=0.200$ & O & 0 & 0 & 0 & 0 & [0.000, 0.060] & 0.0234 & 0.0132\\
$p=0.200$ & Q & 0 & 0 & 0 & 0 & [0.000, 0.060] & 0.0235 & 0.0132\\
$p=0.200$ & R & 0 & 0 & 0 & 0 & [0.000, 0.060] & 0.0326 & 0.0195\\
$L=0.08$ & P & 60 & 60 & 0 & 0 & [0.940, 1.000] & 0.0000 & 0.0000\\
$L=0.08$ & O & 60 & 60 & 0 & 0 & [0.940, 1.000] & 0.0234 & 0.0132\\
$L=0.08$ & Q & 60 & 60 & 0 & 0 & [0.940, 1.000] & 0.0235 & 0.0132\\
$L=0.08$ & R & 60 & 60 & 0 & 0 & [0.940, 1.000] & 0.0326 & 0.0195\\
$L=0.20$ & P & 60 & 60 & 0 & 33 & [0.940, 1.000] & 0.0000 & 0.0000\\
$L=0.20$ & O & 60 & 0 & 0 & 0 & [0.000, 0.060] & 0.0234 & 0.0132\\
$L=0.20$ & Q & 60 & 0 & 0 & 0 & [0.000, 0.060] & 0.0235 & 0.0132\\
$L=0.20$ & R & 60 & 0 & 0 & 0 & [0.000, 0.060] & 0.0326 & 0.0195\\
$L=0.24$ & P & 60 & 57 & 0 & 35 & [0.861, 0.990] & 0.0000 & 0.0000\\
$L=0.24$ & O & 60 & 0 & 0 & 0 & [0.000, 0.060] & 0.0234 & 0.0132\\
$L=0.24$ & Q & 60 & 0 & 0 & 0 & [0.000, 0.060] & 0.0235 & 0.0132\\
$L=0.24$ & R & 55 & 0 & 0 & 0 & [0.000, 0.060] & 0.0326 & 0.0195\\
$L=0.28$ & P & 60 & 0 & 0 & 0 & [0.000, 0.060] & 0.0000 & 0.0000\\
$L=0.28$ & O & 0 & 0 & 0 & 0 & [0.000, 0.060] & 0.0234 & 0.0132\\
$L=0.28$ & Q & 0 & 0 & 0 & 0 & [0.000, 0.060] & 0.0235 & 0.0132\\
$L=0.28$ & R & 0 & 0 & 0 & 0 & [0.000, 0.060] & 0.0326 & 0.0195\\
$L=0.32$ & P & 0 & 0 & 0 & 0 & [0.000, 0.060] & 0.0000 & 0.0000\\
$L=0.32$ & O & 0 & 0 & 0 & 0 & [0.000, 0.060] & 0.0234 & 0.0132\\
$L=0.32$ & Q & 0 & 0 & 0 & 0 & [0.000, 0.060] & 0.0235 & 0.0132\\
$L=0.32$ & R & 0 & 0 & 0 & 0 & [0.000, 0.060] & 0.0326 & 0.0195\\
$N=1024$ & P & 60 & 60 & 0 & 0 & [0.940, 1.000] & 0.0000 & 0.0000\\
$N=1024$ & O & 60 & 60 & 0 & 0 & [0.940, 1.000] & 0.0233 & 0.0698\\
$N=1024$ & Q & 60 & 60 & 0 & 0 & [0.940, 1.000] & 0.0234 & 0.0698\\
$N=1024$ & R & 60 & 60 & 0 & 0 & [0.940, 1.000] & 0.0321 & 0.1016\\
$N=2048$ & P & 60 & 60 & 0 & 0 & [0.940, 1.000] & 0.0000 & 0.0000\\
$N=2048$ & O & 60 & 60 & 0 & 0 & [0.940, 1.000] & 0.0337 & 0.0374\\
$N=2048$ & Q & 60 & 60 & 0 & 0 & [0.940, 1.000] & 0.0338 & 0.0374\\
$N=2048$ & R & 60 & 60 & 0 & 0 & [0.940, 1.000] & 0.0467 & 0.0547\\
$N=4096$ & P & 60 & 60 & 0 & 0 & [0.940, 1.000] & 0.0000 & 0.0000\\
$N=4096$ & O & 60 & 60 & 0 & 0 & [0.940, 1.000] & 0.0487 & 0.0199\\
$N=4096$ & Q & 60 & 60 & 0 & 0 & [0.940, 1.000] & 0.0489 & 0.0199\\
$N=4096$ & R & 60 & 0 & 0 & 0 & [0.000, 0.060] & 0.0678 & 0.0293\\
$N=8192$ & P & 60 & 60 & 0 & 0 & [0.940, 1.000] & 0.0000 & 0.0000\\
$N=8192$ & O & 60 & 0 & 0 & 0 & [0.000, 0.060] & 0.0703 & 0.0106\\
$N=8192$ & Q & 60 & 0 & 0 & 0 & [0.000, 0.060] & 0.0705 & 0.0106\\
$N=8192$ & R & 60 & 0 & 0 & 0 & [0.000, 0.060] & 0.0982 & 0.0156\\
$N=16384$ & P & 60 & 60 & 0 & 0 & [0.940, 1.000] & 0.0000 & 0.0000\\
$N=16384$ & O & 60 & 0 & 0 & 0 & [0.000, 0.060] & 0.1015 & 0.0056\\
$N=16384$ & Q & 60 & 0 & 0 & 0 & [0.000, 0.060] & 0.1017 & 0.0056\\
$N=16384$ & R & 60 & 0 & 0 & 0 & [0.000, 0.060] & 0.1422 & 0.0081\\
\bottomrule\end{longtable}}

\subsection{Historical variable-radius results}
The original observed-distance threshold was $b=0.08$. Its latent radius is $b+a$; reporting a radius of 0.08 for those runs would overstate the result. The following ranges are read from the unchanged historical records, not recomputed from ideal formulas. In particular the few-repeat mass emissions refer to a wider tube than the nominal emissions. Historical records retain the generating code hashes; after the API extension those hashes are not expected to equal current source hashes. Decision-field replay remains separately meaningful.
\begin{center}\small\begin{tabular}{lrrr}\toprule
Historical group & Runs & Min radius & Max radius\\\midrule
Original main (56) & 56 & 0.111788 & 0.114022\\
Original sweep (36) & 36 & 0.109722 & 0.212575\\
Raw demonstrations & 3 & 0.144141 & 0.148277\\
SL/CGL nominal & 6 & 0.147608 & 0.147791\\
SL/CGL few repeats & 6 & 0.359802 & 0.362554\\
SL/CGL large noise & 6 & 0.621514 & 0.622330\\
\bottomrule\end{tabular}\end{center}

The 48 SL/CGL cases have three seeds per condition and are retained as implementation examples. At matched equilibrium radius and noise, a uniform initial phase plus any deterministic phase shift remains uniform, so their two-coordinate sensor observation laws coincide. Their different physical equations are not independent statistical geometries. The archived CGL field continuation uses known wave parameters and adds no independent sensor-certification task.

\subsection{Replay and boundaries of the evidence}
Run \texttt{python3 code/\allowbreak run\_\allowbreak round2\_\allowbreak experiments.py -{}-replay} for all saved rational certificates, and \texttt{python3 code/\allowbreak plot\_\allowbreak round2\_\allowbreak experiments.py} for the retained matched-target plot and the archived component plot. A smaller \texttt{-{}-replay -{}-limit 10} is a quick check only. Cold generation uses \texttt{-{}-trials 60}; it rejects a changed design in an existing output directory. The four new fixed-radius contract tests complement the 12 historical tests. Observation files contain no latent references; evaluation and plotting use the separately declared generator.

These experiments quantify a restricted useful decision: when a circular explanation is support- or mass-compatible with singular alternatives at a requested tolerance, and how much power calibration sacrifices. They do not provide an external state-of-the-art benchmark, validate a general atlas, establish the population lower-mass model from data, or certify true support dimension. The visible conservatism motivates sharper validated bounds rather than an optimality claim.


\subsection{Matched-target summary and retained adverse outcomes}
\paragraph{Matched targets and prior ingredients.}
The P/O/Q/R rules and their shared targets are specified above. Only R issues machine-verified certificates; the numerical references and plug-in rule remain diagnostics.

\paragraph{Design and records.}
The 17 predeclared conditions use 60 independent repetitions each: seven circle--segment mixtures, five segment lengths, and five allocations at fixed $Nk=262144$. All 1020 sufficient-statistic inputs, requests and exact certificates, and all 4080 comparison records are saved. Simulation draws the joint Gaussian means/residual law rather than individual repeats. The population formulas, error denominators and pointwise intervals are given above.

\begin{table}[!t]
\centering\small
\caption{Emissions/decisions out of 60 at matched targets. P: uncalibrated plug-in; O: known-variance numerical reference; Q: unknown-variance numerical reference; R: rational verifier. Only R is machine-verified. The last row has $r_*=0.08$; the first two have $r_*=0.15$.}
\label{tab:round2main}
\begin{tabular}{lrrrr}\toprule
Condition & P & O & Q & R\\\midrule
Mass, $p=0.10$ & 60 & 55 & 55 & 3\\
Mass, $p=0.165$ & 15 & 0 & 0 & 0\\
Support, $L=0.24$ & 60 & 60 & 60 & 55\\
Mass, $N=4096$ & 60 & 60 & 60 & 0\\
\bottomrule\end{tabular}
\end{table}
\paragraph{Observed tradeoffs.}
At $p=0.10$, the rational rule emits 3/60 tube certificates, compared with 55/60 for the unknown-variance numerical rule. At $p=0.165$, the uncalibrated plug-in rule makes 15/60 false positive decisions against the requested mass target, while the rational rule abstains. The rational verifier has 0 observed false emitted bounds across the study; each fixed condition has a one-sided 95\% upper error limit of 4.87\%, not a pooled or uniform empirical guarantee. These observations support the implementation contract and expose its power cost; they do not replace its proof.
Known-variance and unknown-variance reference decisions are close at these large residual degrees of freedom. The larger measured loss is due to the combined conservative arithmetic bounds. Appendix~\ref{app:round2} separates noise, coverage, and DKW terms and reports calibration-only and full-pipeline timings. No claim of optimal power is made.


\paragraph{Historical controls.}
The 95 original runs and 48 SL/CGL runs retain their variable-radius interpretation above and their records in Appendices~\ref{app:experiments}--\ref{app:landau}. Historical local-PCA and MNIST diagnostics remain in Appendix~\ref{app:restored}; none is a rational dimension certificate.


\section{Experimental details and additional figures}\label{app:experiments}
\subsection{Distributions and lower mass constants}
For completeness, we verify each supplied generator condition explicitly. These are applications of the lower-mass coverage principle \citep{chazal2015}, not estimates of the constants from data.
The circle is uniform in angle. For spurs, sample normalized length on the union: choose the spur with probability $L/(2\pi+L)$ and a uniform position on it; otherwise choose a uniform circle angle. The short and long lengths are $L=0.08$ and $0.5$. The missing-arc distribution has angle uniform on $[-3\pi/4,3\pi/4]$. The cross is uniform length on $([-1,1]\times\{0\})\cup(\{0\}\times[-1,1])$. For these five distributions, $(c_0,s,t_0)=(1/8,1,1/4)$ is valid: at least length $t$ can be followed from each point inside its radius-$t$ ball, and total length is less than eight. This argument also covers endpoints and the crossing.

For the annulus, radius is uniform on $[0.96,1.04]$ independently of a uniform angle. Put $e=0.04$. From any reference point, radial displacements of magnitude at most $t/2$ within the annulus and angular displacements at most $t/(2(1+e))$ give Euclidean displacement at most $t$. The probability of this parameter rectangle is at least
\[
\frac{t}{2\pi(1+e)}\min\left\{1,\frac{t}{4e}\right\}\ge\frac12t^2
\quad(0<t\le1/4).
\]
For $t\le4e$, the probability divided by $t^2$ is at least $[8\pi(1+e)e]^{-1}>1/2$ at $e=1/25$. For $4e<t\le1/4$, it is at least $[2\pi(1+e)t]^{-1}\ge2/[\pi(1+e)]>1/2$. These numerical comparisons follow already from $\pi<22/7$. We therefore supply $(c_0,s,t_0)=(1/2,2,1/4)$, without treating this set as a one-dimensional support.

For the torus, both angles are independent uniform variables and $(R,r)=(2,0.6)$. The distance between points with angle differences $(u,v)$ obeys
\[
\norm{F(\theta+u,\phi+v)-F(\theta,\phi)}^2
\le(R+r)^2u^2+r^2v^2.
\]
To verify this without discarding a cross term, put $a_0=R+r\cos\phi$ and $a_1=R+r\cos(\phi+v)$. Direct expansion gives
\[
 \norm{F(\theta+u,\phi+v)-F(\theta,\phi)}^2
 =4a_0a_1\sin^2(u/2)+4r^2\sin^2(v/2)
 \le(R+r)^2u^2+r^2v^2.
\]
The ellipse on the right of radius $t\le1$ has semiaxes $t/(R+r)\le5/13<\pi$ and $t/r\le5/3<\pi$, so it lies within a fundamental angular square. Its normalized angular area is $t^2/[4\pi r(R+r)]\ge t^2/20$, proving $(c_0,s,t_0)=(1/20,2,1)$. Uniform angles are not uniform surface area, and this calculation uses the actual generator distribution.

\subsection{Raw-input results and exact replay}\label{app:rawresults}
The three raw-input runs use seeds $927100$, $927101$, and $927102$ for circle, missing arc, and torus, respectively. Each stores the full $N\times32\times q$ array separately from latent references. The saved sufficient statistics and their derived recording-error bounds reproduce every mathematical certificate field exactly; only the processing and provenance descriptions differ between raw-input and sufficient-statistic replay. The table reports actual output bounds and elapsed end-to-end certification time, including aggregation and writing, before plotting.
\begin{center}

\begin{tabular}{lrrrrr}
\toprule
Setting & $N$ & Support & Mass & $\widehat\epsilon$ & Seconds\\
\midrule
Circle & 4096 & EMIT & EMIT & 0.149 & 0.32 \\
Missing arc & 4096 & ABSTAIN & EMIT & 0.886 & 0.34 \\
2-torus & 16384 & EMIT & EMIT & 0.528 & 3.44 \\
\bottomrule
\end{tabular}


\end{center}

\subsection{Parameter fitting and reference error}
Circle radius is the median norm of the training means. For tori, write $\rho_i=\sqrt{Y_{i1}^2+Y_{i2}^2}$ and solve least squares for $R,b$ in $\rho_i^2+Y_{i3}^2=2R\rho_i+b$; set $r=\sqrt{\max\{b+R^2,0\}}$. Rational rounding is performed before verification, and $R>r>0$ is required. The origin and axis are fixed for the proposal family. These fits do not assert that the population belongs to that family.

For a radius-$R$ circle witness the reference errors are $|R-1|$ for the circle, $\max\{|R-1|,|1+L-R|\}$ for a spur, $\sqrt{1+R^2-\sqrt2R}$ for the missing arc, $\max\{R,|1-R|\}$ for the cross, and $\max\{|R-0.96|,|1.04-R|\}$ for the annulus. For a torus witness, matching angles gives the upper bound $|R-2|+|r-0.6|$. These reference calculations are solely for evaluation; the verifier does not use them.

\begin{figure}[ht]
\centering
\includegraphics[width=.98\textwidth]{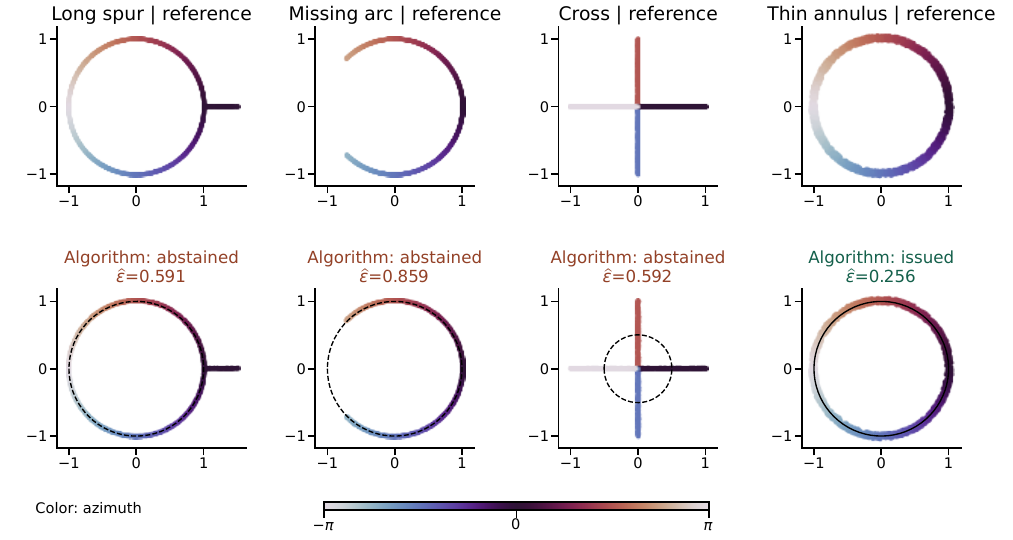}
\caption{Additional controls. Top: generated latent reference points. Bottom: means and the algorithm's circle candidate. Dashed curves indicate abstention from the support certificate. The annulus receives an approximate one-dimensional witness despite having a two-dimensional support. The missing arc passes the separate mass criterion, which does not demand that the entire witness be supported by data. Azimuth colors are computed independently from the plotted coordinates in each row.}
\label{fig:controls}
\end{figure}

\subsection{Additional certificate diagnostics}\label{app:diagnostics}
The individual error-budget components are $\delta_C=\delta_L=0.01$ and $\delta_E=\delta_V=0.03$. For the sufficient-statistic inputs, $e_Y=10^{-12}$ and $e_S=10^{-8}$. Raw-repeat processing uses the external coordinate-error bound $10^{-12}$. Candidate fitting uses the first quarter of groups; the last three quarters validate tube mass. The 36 sweep runs use $N\in\{512,2048,8192\}$ and $k\in\{8,32,128\}$, with four seeds per pair. None of the reference support-error calculations is passed to either algorithm.
\begin{table}[ht]
\centering
\caption{Original certificate summaries: eight seeds per geometry. The reference-error column is computed outside the verifier. Runtime here measures verification after fitting. Counts are not a common-task accuracy score.}
\label{tab:results}

\begin{tabular}{lrrrrr}
\toprule
Setting & Support & Mass & $\widehat\epsilon$ & $e_{\rm ref}$ & Seconds\\
\midrule
Circle & 8/8 & 8/8 & 0.105 & 0.00024 & 0.08 \\
Short spur & 8/8 & 8/8 & 0.173 & 0.08017 & 0.09 \\
Long spur & 0/8 & 0/8 & 0.588 & 0.49989 & 0.09 \\
Missing arc & 0/8 & 8/8 & 0.860 & 0.76542 & 0.08 \\
Cross & 0/8 & 0/8 & 0.603 & 0.52291 & 0.09 \\
Thin annulus & 8/8 & 8/8 & 0.258 & 0.04259 & 0.08 \\
2-torus & 8/8 & 8/8 & 0.485 & 0.00016 & 1.79 \\
\bottomrule
\end{tabular}


\end{table}

\section{Two physical realizations of one sensor geometry}\label{app:landau}

These experiments instantiate the repeated-observation certificate with independent phase ensembles from two classical dynamical systems. The amplitude equation and its history belong to the work of Landau and Stuart; we cite Stuart's primary analysis \citep{stuart1960} and the review and plane-wave formula of \citet[Sec.~II.D, Eq.~(17)]{aranson2002}. The integration method for the PDE is classical symmetric splitting \citep{strang1968}. Our contribution here is the observation-to-certificate implementation and the executed evaluation under its declared contract. Neither these equations nor their circular phase orbits are new results.

\subsection{Equations, observations, and the exact comparison geometry}
We use the Stuart--Landau normal form
\begin{equation}
\dot z=(\mu+i\omega-|z|^2)z,\qquad \omega=2,\quad
\mu\in\{0.25,1,2.25\}.
\label{eq:slstudy}
\end{equation}
Writing $z=re^{i\theta}$ gives $\dot r=\mu r-r^3$ and $\dot\theta=\omega$. The phase orbit $z(t)=\sqrt\mu\,e^{i(\omega t+\theta_0)}$ is therefore an exact solution, with radius $R=\sqrt\mu$. For $r_0>0$, its finite-time radial relaxation is
\begin{equation}
r(t)=\frac{r_0e^{gt}}{\sqrt{D(t)}},\qquad
D(t)=1+\frac{r_0^2}{g}(e^{2gt}-1),\qquad
\int_0^t r(s)^2\,ds=\tfrac12\log D(t),
\label{eq:landauradius}
\end{equation}
with $g=\mu$. For completeness, $g>0$ in every listed setting and
\[
 D'(t)=2r_0^2e^{2gt}=2r(t)^2D(t),\qquad
 \frac{r'(t)}{r(t)}=g-\frac{D'(t)}{2D(t)}=g-r(t)^2.
\]
Also $D(0)=1$, so integrating $D'/D=2r^2$ proves the second identity with its integration constant.
The solution formula can also be derived from the equation.
For $r>0$ set $u=r^{-2}$. Then
$u'=-2r^{-3}r'=-2gu+2$ and
$(e^{2gt}u)'=2e^{2gt}$. Integrating from zero to $t$ gives
\[
u(t)=r_0^{-2}e^{-2gt}+\frac{1-e^{-2gt}}{g}.
\]
Taking the positive reciprocal square root gives
\eqref{eq:landauradius}. Since $g>0$, its denominator is positive
for all $t\ge0$, and $r(t)\to\sqrt g$.
This derives the amplitude and checks the stated initial condition. Substituting $z=re^{i\theta}$ into the Stuart--Landau equation and dividing by $e^{i\theta}$ gives $r'+ir\theta'=(\mu-r^2+i\omega)r$, which proves the amplitude and phase equations used above \citep{stuart1960}.

The complex Ginzburg--Landau experiment is a \emph{spatially one-dimensional} periodic PDE with a complex-valued field:
\begin{equation}
\partial_t A=\mu A+(1+i\alpha)\partial_{xx}A-(1+i\beta)|A|^2A,
\quad x\in[0,4\pi),\quad \alpha=1,\quad\beta=\tfrac12,
\quad \mu\in\{0.5,1.25,2.5\}.
\label{eq:cglstudy}
\end{equation}
For $\kappa=1/2$ it has the plane-wave family
\begin{equation}
A(x,t;\theta_0)=R e^{i(\kappa x-\Omega t+\theta_0)},\qquad
R^2=\mu-\kappa^2>0,\quad
\Omega=\alpha\kappa^2+\beta R^2.
\label{eq:planewave}
\end{equation}
Indeed, $A_{xx}=-\kappa^2A$ and $|A|^2=R^2$, so substitution gives
$-i\Omega=\mu-\kappa^2-R^2-i(\alpha\kappa^2+\beta R^2)$.
For non-equilibrium initial amplitude $r_0$, the same single spatial mode evolves with \eqref{eq:landauradius}, $g=\mu-\kappa^2$, and phase
$\theta_0-\alpha\kappa^2t-(\beta/2)\log D(t)$.
For completeness, the ansatz $A=r(t)e^{i(\kappa x+\varphi(t))}$ in \eqref{eq:cglstudy} gives
\[
 r'+ir\varphi'=(\mu-\kappa^2-r^2)r-i(\alpha\kappa^2+\beta r^2)r.
\]
Thus $r'=gr-r^3$ and $\varphi'=-\alpha\kappa^2-\beta r^2$; integrating with the preceding identity yields the stated phase, consistent with the credited plane-wave family \citep{aranson2002}. The period condition holds because $\kappa(4\pi)=2\pi$.
No stability claim for arbitrary perturbations is needed to use these explicit solutions. In particular, we do not infer an attracting manifold for the full PDE.

For each latent group we draw a fresh $\theta_i\sim\operatorname{Unif}[-\pi,\pi)$ and observe two quadratures at the common time $T=8$: $(\Re z(T),\Im z(T))$ for SL and $(\Re A(0,T),\Im A(0,T))$ for CGL. A deterministic phase shift preserves normalized angular measure: for every periodic integrable $h$, $(2\pi)^{-1}\int_{-\pi}^{\pi}h(\theta+c)d\theta=(2\pi)^{-1}\int_{-\pi}^{\pi}h(\theta)d\theta$. Their reference supports are therefore circles of matched radii $R=0.5,1,1.5$ in $\R^2$, each of dimension one and reach $R$. Thus the PDE grid dimension and spatial dimension are different from the dimension and ambient dimension of the \emph{sensor law}. The two examples deliberately have identical sensor-space geometry and different dynamical generators. They test application of one implemented family, not general high-dimensional manifold reconstruction.

Each experiment generates every raw repeat
$X_{i\ell}=Z_i+\varepsilon_{i\ell}$ with independent $N(0,\sigma^2I_2)$ errors. Noise is added directly to these two sensor coordinates; there is no nonlinear preprocessing that changes the Gaussian model. Different groups use independent phases; adjacent time samples of one trajectory are not treated as independent groups. Analytic solutions generate the sensor data, while separate numerical integrations check the differential equations. Neither the true radius nor $\sigma$ is supplied to the fitting/verifying pipeline.

\subsection{Protocol fixed before execution and all 48 decisions}
The default is $N=4096$, $k=64$, $\sigma=0.06$, a first-quarter training split, $\epsilon_*=0.3$, and $\tau_*=0.75$. The verifier's tube threshold is $b=0.08$, so its reported tube radius is $b+a$, not $0.08$; its outside-mass target is $\alpha_*=0.08$. The external scalar recording bound is $10^{-12}$, propagated by exact aggregation as in Appendix~\ref{app:rawprocessing}. Ideal Gaussian draws and the recording contract remain model assumptions; a pseudorandom generator is not formally validated by rational arithmetic.

For the equilibrium circle cases, a direct arc calculation gives, for $0<t\le1/4$,
\[
\mu(B(x,t))=\frac{2}{\pi}\arcsin\frac{t}{2R}
\ge\frac{t}{\pi R}\ge\frac{t}{8}\qquad(R\le1.5).
\]
We therefore supply the conservative contract $(c_0,s,t_0)=(1/8,1,1/4)$. This calculation justifies a population condition for these generators; it does not claim that the verifier estimated that condition from data. Choosing a known smooth positive example to evaluate a test is not an additional manifold premise in Theorem~\ref{thm:main}. Controls without the mass contract explicitly use the independent mass theorem instead.

Eight settings are fixed for each equation: the three radii; $N=512$ at the nominal radius; $k=4$; $\sigma=0.48$; a missing lower-mass contract; and a transient-amplitude control. Unspecified parameters keep their default values. The last control uses $r_0\sim\operatorname{Unif}[0.25,2]$, an independent uniform phase, and $T=0.1$ in the nominal equation. For fixed $t$, differentiation in the positive initial radius gives $\partial r(t)/\partial r_0=e^{gt}D(t)^{-3/2}>0$. Hence the interval of initial radii maps continuously and strictly increasingly to a nondegenerate interval of final radii. The conditional phase remains uniform even though its shift depends on $r_0$. Its latent sensor support is consequently an annulus, not the proposed circle. We do not supply a lower-mass contract for this control. Its support abstention is automatic from that omission; its additional tube-mass abstention is a substantive validation result. It is not a proof that no other manifold could explain the data at another tolerance.

There are three seeds per setting, indexed by $j=0,1,2$. With equation index $m=0,1$ and the zero-based setting index $s$ in the order above, the seed is $610000+10000m+100s+j$. The complete protocol is written to \path{results/landau/design.json} before the first draw. No threshold, seed, or displayed query is selected using validation performance.

\begin{table}[ht]
\centering

\begin{tabular}{llrrrr}
\toprule
Model & Setting & Support & Mass & Max. $\epsilon$ & Max. $\widehat\alpha$\\
\midrule
SL & small radius & 0/3 & 3/3 & 0.157 & 0.034\\
SL & nominal & 3/3 & 3/3 & 0.161 & 0.034\\
SL & large radius & 3/3 & 3/3 & 0.158 & 0.034\\
SL & few groups & 0/3 & 0/3 & 0.492 & 0.095\\
SL & few repeats & 0/3 & 3/3 & 0.461 & 0.041\\
SL & large noise & 0/3 & 0/3 & 0.847 & 0.219\\
SL & no mass contract & 0/3 & 3/3 & -- & 0.034\\
SL & transient annulus & 0/3 & 0/3 & -- & 0.927\\
\midrule
CGL & small radius & 0/3 & 3/3 & 0.154 & 0.034\\
CGL & nominal & 3/3 & 3/3 & 0.155 & 0.034\\
CGL & large radius & 3/3 & 3/3 & 0.155 & 0.034\\
CGL & few groups & 0/3 & 0/3 & 0.497 & 0.095\\
CGL & few repeats & 0/3 & 3/3 & 0.458 & 0.043\\
CGL & large noise & 0/3 & 0/3 & 0.824 & 0.226\\
CGL & no mass contract & 0/3 & 3/3 & -- & 0.034\\
CGL & transient annulus & 0/3 & 0/3 & -- & 0.927\\
\bottomrule
\end{tabular}


\caption{All executed SL/CGL settings. Counts are emissions among three seeds. $\epsilon$ is the computed support bound even when reach fails; a dash means no support bound because the mass contract is absent. $\widehat\alpha$ is the outside-mass upper bound for the run-specific tube of radius $0.08+a$. Maxima are over all three seeds. The table is generated from the executed JSON records.}
\label{tab:landau}
\end{table}

Overall, 12 of 48 support certificates and 30 of 48 mass certificates emit. All nominal- and large-radius runs emit both certificates. Small-radius runs fail the requested reach floor even though the reference support is itself a smooth circle. Fewer groups cause both bounds to be insufficient; fewer repeats fail the support tolerance but still pass the mass test for its enlarged reported tube. Large noise and transient annuli fail the mass target. Missing-mass runs show that tube-mass certification remains possible without whole-support certification. Per-run false-issuance bounds remain $0.05$ for either certificate and $0.08$ jointly. These are not simultaneous guarantees over the 48-run study, and three seeds are insufficient to empirically establish a rare-error rate.

\subsection{Compressed reconstruction and implementation audit}
At the nominal radius, across three seeds, the true circle-to-candidate Hausdorff errors are $1.33\times10^{-4}$--$4.22\times10^{-4}$ (SL) and $4.7\times10^{-5}$--$2.43\times10^{-4}$ (CGL); certified support bounds are $0.1517$--$0.1607$ and $0.1538$--$0.1549$. Projecting held-out means onto the circle reduces Euclidean RMSE from about $0.0105$--$0.0107$ to $0.0074$--$0.0076$ in those cases. The transient annulus instead incurs output error about $0.39$. All settings and both successes and failures remain in the saved CSV. The six former plots are omitted because matched equilibrium sensor laws coincide and known-wave continuation does not test a new manifold-certificate target.

Separate RK4 and Fourier-splitting diagnostics are archived. The latter uses the classical symmetric splitting of \citet{strang1968}; these numerical solver comparisons do not furnish certified continuum errors. Exact phase solutions generate the inferential data, so unproved PDE discretization bounds are not inserted into the certificate.

\paragraph{Classical covariance check, retained without duplicate plots.}
For a uniform circle of radius $R$, at $x=(R,0)$ and $0<\rho<2R$, set $a_\rho=2\arcsin(\rho/(2R))$. The conditional angle in $B(x,\rho)$ is uniform on $[-a_\rho,a_\rho]$. Odd moments vanish. Direct integration of $\sin^2\theta$, $\cos\theta$, and $\cos^2\theta$ gives
\begin{align*}
\lambda_T(\rho)&=R^2\left(\frac12-\frac{\sin(2a_\rho)}{4a_\rho}\right),\\
\lambda_N(\rho)&=R^2\left(\frac12+\frac{\sin(2a_\rho)}{4a_\rho}
                       -\left(\frac{\sin a_\rho}{a_\rho}\right)^2\right).
\end{align*}
Consequently, $\lambda_T(\rho)/\rho^2\to1/3$ and
$\lambda_N(\rho)/\rho^2\sim\rho^2/(45R^2)\to0$.
These are the $d=1$ limits in Proposition~\ref{prop:covariance}. This elementary calculation and its implementation illustrate the established local-covariance mechanism \citep{little2017,aamari2019,lim2024}; they do not constitute a new existence result.

Observed small-scale spectra from fixed noisy repeats remain archived as diagnostics. They do not establish all-scale regularity or asymptotic consistency; the observed-data point estimator is evaluated separately in Figure~\ref{fig:dimensionchecks}.

\paragraph{Reproduction.}
All 48 runs retain sufficient statistics, recording bounds, requests and certificate records; the complete raw arrays for seed zero in each of the 16 equation/setting pairs are also bundled. The original recorded replay passed for those 48 moment records and 16 raw arrays. Run \texttt{python3 code/\allowbreak run\_\allowbreak landau\_\allowbreak experiments.py -{}-replay}. The generator and complete numerical records remain supplied. The tabulated certificate experiments retain their original runs and interpretation.


\clearpage
\section{MNIST reconstruction audit}\label{app:restored}
The MNIST outputs are local affine PCA reconstructions. Their mean-crossing rank rule, \eqref{eq:termsdiagnosticrank}, is not the proved consistency estimator \eqref{eq:actualdimension}. The complete image grid remains in the reproducibility archive as a reconstruction diagnostic. That historical run makes no matched Yao--Xia claim. The distinct corrected synthetic-noise follow-up is reported in Appendix~\ref{app:matchedmnist}.

\paragraph{Other reconstruction studies.}
The transferred Swiss-roll, torus, ODE and independent KS diagnostics, including adverse outcomes, now appear in Section~\ref{sec:arxivvisuals} and Appendix~\ref{app:arxivbench}, with observed-input replay in the full source/data archive. The historical complete summary remains in the earlier reproducibility archive (\texttt{reproduction/\allowbreak point\_ode\_ks\_table8.csv}). These local-PCA runs do not invoke the manifold-witness verifier or validate the consistency estimator; some projections increase reconstruction error.

\subsection{MNIST reference and algorithm comparison}
The fixed 8000 official-training images and 40 official-test images from \citet{lecun1998} yield 39/40 reconstructions at the historically fixed radius $0.28$, with ranks 5--89. Over the 39 outputs, mean unclipped MSE is $0.008363933$ and arithmetic mean per-image PSNR is $22.62793$ dB. The first predeclared digit~5 query abstains. No extra noise is injected and no repeated natural observations are available. These are compression errors against observed images, not denoising errors against clean latent targets or evidence for the repeated-Gaussian model. All 40 predeclared outputs, including the failed output, remain in the archived record.

\paragraph{Reproduction.}
The full experimental archive supplies the MNIST input data, fixed training/test indices, all 40 expected outputs and the observation-only projection code. From that archive, run \texttt{python3 code/\allowbreak reproduce\_\allowbreak mnist.py -{}-recompute} to recompute every image and compare the result with the stored output. The other reconstruction studies and the exported table remain in the separate reproducibility archive; their former image galleries are omitted.

\clearpage
\subsection{Corrected synthetic-noise follow-up}\label{app:matchedmnist}
The historical reconstruction above is distinct from the following new
synthetic-corruption experiment. An earlier auxiliary benchmark was withdrawn:
it used test targets for parameter selection, mislabeled noise of standard
deviation $0.20/\sqrt{255}$ as $0.20$, gave the Yao--Xia branch only 500 of
8000 training images, omitted terms from the objective derivative, and computed
SSIM incorrectly. Those auxiliary numbers support no comparative conclusion.
The present results replace them; they do not alter the verified historical
39/40 reconstruction record.

\paragraph{Observed-data contract.}
For each original pixel vector $z\in[0,1]^{784}$, generate
$x_\ell=z+\sigma\epsilon_\ell$, with independent standard Gaussian errors,
and average $k$ observations. The original MNIST image is the synthetic
denoising target, not an assumed noise-free natural object. Training and
queries use the same observation model. Both comparison methods use all
8000 training averages. A common 64-dimensional PCA map is fitted only to
the observed training averages divided by 28. Outputs are decoded into
pixel units. A further 40 official-training images, disjoint from the 8000,
are used only for validation. The grid is
$r\in\{0.24,0.32,0.40\}$ and $d\in\{8,16,32\}$.
Each method first maximizes validation numerical output coverage and then
minimizes all-query validation MSE. All settings and file hashes are frozen
before test evaluation. The 40 historical test images are evaluated at
three noise seeds. The $k=1$ observation is the first of the corresponding
$k=4$ observations. This panel was already viewed historically; it is not
a pristine new population test.

\paragraph{Methods and credit.}
The comparison local PCA projects onto a rank-$d$ affine space centered at
the local mean. The Yao--Xia branch independently implements
\citet[Equations (2.4)--(2.6)]{yaoxia2025}, using sample-centered second
moments for tangent estimates, $r'=2r$, $\beta=d+2$, and the complete
objective derivative below. Armijo descent is limited to 50 iterations;
finite iterates are evaluated even if the stopping tolerance is not met.
This is not execution of the authors' MATLAB solver or a reproduction of
their face experiment, which uses different numerical settings.
The native 784-dimensional mean-crossing algorithm is also rerun, with
radius selected from $\{0.24,0.28,0.32,0.36,0.40\}$ on the same validation
set. Its preprocessing and rank rule differ from the common-PCA64 comparison;
it is shown as a separate diagnostic, not as the consistency estimator.

\paragraph{Scoring and failures.}
All 120 image--seed outputs are scored, without clipping. An undefined
comparison output uses the common PCA64 reconstruction as an explicitly
flagged fallback; the native path uses the observed repeat mean.
PSNR is averaged per image. SSIM uses all 484 valid $7\times7$ uniform
windows, covariance denominator 48, data range one and constants
$0.01^2,0.03^2$. Common-success scores, paired image bootstrap summaries,
individual predictions, failures and objective traces are supplied in the
new benchmark archive. These are conditional fixed-panel diagnostics,
not manifold certificates or population performance guarantees.

\begin{figure}[!t]\centering
\includegraphics[width=.98\textwidth]{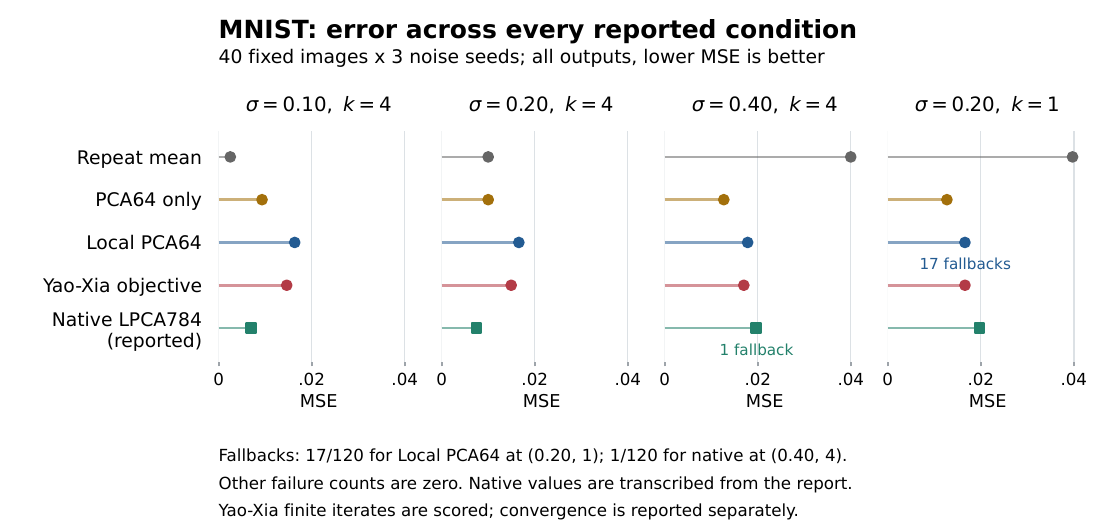}
\caption{Every corrected MNIST condition and method, using the exact reported MSEs and failures across the 40 frozen images and three noise seeds. The observed repeat mean and PCA-only baseline remain visible. Undefined comparison outputs use the flagged PCA64 fallback; Native LPCA784 uses the observed repeat mean. Failure counts are numerical availability outcomes, not certification abstentions. Native LPCA784 uses reported aggregate values transcribed from the existing table and is distinct from the local-PCA64 image row in Figure~\ref{fig:mnistapplication}. The exact MSE, PSNR, SSIM and failure table remains in numerical supplement Table~S6. The restricted validation grid and reused historical panel do not establish population rankings or general superiority.}
\label{fig:mnistscores}
\end{figure}
The PCA-only baseline makes the effect of preprocessing visible. The fixed
search grid, small validation set and historical test selection limit any
ranking; these results establish neither general superiority nor failure
of a published method.

\begin{center}\small
\begin{tabular}{@{}ccccc@{}}\toprule
$(\sigma,k)$ & Local PCA $(r,d)$ & Yao--Xia $(r,d)$ & Native $r$ & YX converged\\\midrule
$(0.10,4)$ & $(0.32,32)$ & $(0.40,32)$ & 0.32 & 119/120 \\
$(0.20,4)$ & $(0.32,32)$ & $(0.40,32)$ & 0.36 & 115/120 \\
$(0.40,4)$ & $(0.32,32)$ & $(0.40,32)$ & 0.40 & 115/120 \\
$(0.20,1)$ & $(0.24,16)$ & $(0.40,32)$ & 0.40 & 118/120 \\
\bottomrule\end{tabular}\end{center}
Numerically defined output and convergence are separate statuses. The
stopping rule is $F\le10^{-10}$ or $\|\nabla F\|\le10^{-8}$; all remaining
finite iterates retain their termination status. Timings separately record
the common PCA calculation, sample tangent construction and query work.

\subsection{Complete derivative of the comparison objective}
This calculation implements the credited Yao--Xia objective; it is not
a new fitting construction. Let $U_i$ contain the sample tangent vectors,
$S=\sum_i\alpha_iU_iU_i^T$, $Q=I-\Pi_{\mathrm{top}\,d}(S)$,
$b=\sum_i\alpha_i y_i$, $v=y-b$, and $f=Qv$. Then
$F=\|f\|^2=v^TQv$. Within a kernel support,
\[
g_i=\nabla\log w_i=-\frac{2\beta(y-y_i)}{r^2-\|y-y_i\|^2},\qquad
\nabla\alpha_i=\alpha_i\left(g_i-\sum_j\alpha_jg_j\right).
\]
Write $\lambda_1\ge\cdots\ge\lambda_q$ for the eigenvalues of $S$ and
$u_a$ for corresponding orthonormal eigenvectors. Assume the strict group gap $\lambda_d>\lambda_{d+1}$ and hold
the sample matrices $U_i$ fixed. A derivative of the group projector
does not require the individual eigenvalues inside either group to
be simple. Put $P=\Pi_{\mathrm{top}\,d}(S)$.
A contour separating the two spectral groups gives
$P=(2\pi i)^{-1}\int_\Gamma(zI-S)^{-1}dz$.
The resolvent is differentiable while the group gap persists, with
$D(zI-S)^{-1}[E]=(zI-S)^{-1}E(zI-S)^{-1}$.
To see both statements directly, diagonalize $S$:
$(zI-S)^{-1}=\sum_j u_ju_j^{\mathsf T}/(z-\lambda_j)$.
The contour integral equals one on the enclosed eigenvalues and zero
on the others, yielding $P$ even when an enclosed eigenvalue is repeated.
For $R_t(z)=(zI-S-tE)^{-1}$, the resolvent identity is
$R_t-R_0=tR_tER_0$. The contour is a positive distance from the
spectrum, so the inverses are uniformly bounded for sufficiently
small $t$ by the Neumann series. Dividing the identity by $t$ and
letting $t\to0$ gives the displayed derivative uniformly on the
contour, justifying differentiation of its integral.
Write $\dot P=DP(S)[E]$. Differentiating $P^2=P$ gives
$P\dot P P=(I-P)\dot P(I-P)=0$.
Differentiating $SP=PS$ gives
$S\dot P-\dot P S=PE-EP$.
In an orthonormal eigenbasis of $S$, for $a\le d<b$ this is
\[
(\lambda_b-\lambda_a)u_b^{\mathsf T}\dot P u_a
=-u_b^{\mathsf T}Eu_a,\qquad
u_b^{\mathsf T}\dot P u_a
=\frac{u_b^{\mathsf T}Eu_a}{\lambda_a-\lambda_b}.
\]
Only these off-diagonal blocks remain, and symmetry gives the
transposed block. Consequently,
\[
D\Pi_{\mathrm{top}\,d}(S)[E]
=\sum_{a\le d<b}\frac{u_bu_b^TEu_au_a^T+u_au_a^TEu_bu_b^T}
 {\lambda_a-\lambda_b}.
\]
Within-group rotations cancel; individual simple eigenvalues are not required.
Define
\[
H=-\sum_{a\le d<b}\frac{(u_a^Tv)(u_b^Tv)}{\lambda_a-\lambda_b}
(u_bu_a^T+u_au_b^T).
\]
Since $dv=dy-\sum_i y_i\,d\alpha_i$ and
$dS=\sum_i U_iU_i^T d\alpha_i$, direct differentiation yields
\[
dF=2f^Tdy+\sum_i\bigl[\langle H,U_iU_i^T\rangle-2y_i^Tf\bigr]d\alpha_i.
\]
Thus, with $c_i=\langle H,U_iU_i^T\rangle-2y_i^Tf$ and
$\bar c=\sum_i\alpha_i c_i$, the evaluated gradient is
\[
\nabla F=2f+\sum_i\alpha_i(c_i-\bar c)g_i.
\]
Undefined supports or unresolved spectral gaps are recorded as numerical
failures. The kernel mean and spectral projection derivatives are both present.
For the kernel derivative, if
$w_i=(1-\|y-y_i\|^2/r^2)_+^\beta$ and $W=\sum_jw_j>0$,
then inside active supports
\[
dw_i=w_i g_i^{\mathsf T}dy,\qquad
d\alpha_i=\frac{dw_i}{W}-\frac{w_i}{W^2}\sum_jdw_j
=\alpha_i(g_i-\bar g)^{\mathsf T}dy.
\]
The calculation requires differentiable weights; at support boundaries
the implemented smooth kernel exponent exceeds one and the
derivative is zero by the limiting formula. The singular expression
$g_i$ is only evaluated for strictly positive weights.
For the objective,
$dF=2v^{\mathsf T}Q\,dv-v^{\mathsf T}dP\,v$;
inserting the group-projector derivative gives
$-v^{\mathsf T}dP\,v=\langle H,dS\rangle$ with the displayed $H$.
Since $Q=Q^{\mathsf T}=Q^2$ and $f=Qv$, substitution of
$dv=dy-\sum_i y_i\,d\alpha_i$ and
$dS=\sum_iU_iU_i^{\mathsf T}d\alpha_i$ produces the stated gradient.
This is a derivative with the locally fitted sample tangents held
fixed, as required by this comparison objective.
Eight independently constructed fixtures and two central-difference steps give
maximum relative error $5.24\times10^{-9}$. Direct 484-window SSIM agrees
to $3.33\times10^{-16}$. Five checks against the authors' public Python
objective give maximum value difference $10^{-14}$ or less; its reported
gradient differs from the independently finite-difference-checked derivative
on those fixtures. This observation does not audit the authors' MATLAB code
or determine the validity of a theorem. Repository commit and source hashes
are retained in the benchmark audit record.

\subsection{Finite-cloud sphere and torus controls}
The same independently implemented methods are evaluated in native
$\mathbb R^3$, with known dimension two supplied to both. The unit sphere
and torus of major/minor radii $1/0.35$ are sampled uniformly in surface
area. Each condition uses 800 training points, 64 independent validation
points, 512 test points, a separate 4096-point reference and three seeds.
Noise has $\sigma=0.06$, with $k=1$ or $4$. Radii in
$\{0.35,0.55,0.80\}$ are selected by validation coverage and then pointwise
MSE before test evaluation. No true dimension is estimated in this control.

For a point $x$, the exact forward distances for the sphere and torus are
\[
d_{S^2}(x)=|\|x\|-1|,\qquad
d_{T^2}(x)=\left|\sqrt{(\sqrt{x_1^2+x_2^2}-1)^2+x_3^2}-0.35\right|.
\]
Reverse distances run from the independent reference to the finite
512-point output. The final column reports the percentage of reference
points with reverse distance exceeding $0.15$.

\begin{center}\small
\begin{tabular}{@{}lc l rrr@{}}\toprule
Shape & $k$ & Method & Forward mean & Reverse mean & Uncovered (\%)\\\midrule
Sphere & 1 & Repeat mean & 0.04735 & 0.09385 & 8.85 \\
Sphere & 1 & Local PCA & 0.02791 & 0.08468 & 6.26 \\
Sphere & 1 & Yao--Xia objective & 0.02698 & 0.08444 & 6.23 \\
Sphere & 4 & Repeat mean & 0.02331 & 0.08509 & 7.37 \\
Sphere & 4 & Local PCA & 0.02975 & 0.08560 & 6.84 \\
Sphere & 4 & Yao--Xia objective & 0.01231 & 0.08091 & 6.56 \\
Torus & 1 & Repeat mean & 0.04751 & 0.09766 & 10.60 \\
Torus & 1 & Local PCA & 0.04269 & 0.09418 & 8.31 \\
Torus & 1 & Yao--Xia objective & 0.04843 & 0.10236 & 13.86 \\
Torus & 4 & Repeat mean & 0.02333 & 0.08752 & 8.08 \\
Torus & 4 & Local PCA & 0.04241 & 0.09342 & 7.81 \\
Torus & 4 & Yao--Xia objective & 0.03318 & 0.09538 & 9.92 \\
\bottomrule\end{tabular}\end{center}
The results depend on shape and repeat count. The Yao--Xia sphere runs all
reach objective tolerance; for the torus, 462/1536 outputs at $k=1$ and
437/1536 at $k=4$ reach the iteration limit. Their finite iterates remain
in the evaluation. Reverse coverage includes finite sampling holes and is
not a certified continuum Hausdorff distance. Full clouds and empirical
bidirectional maxima are retained. The existing 200-run intrinsic-dimension
check was also rerun: the correct counts remain 47/50, 50/50, 50/50 and
50/50. All 200 covariance bounds are vacuous, so zero reported bound
violations is not evidence of useful finite-sample coverage.

\paragraph{New reproduction archive.}
The corrected MNIST benchmark documentation describes the fixed image
split, outputs, protocols, selection records, figures and exact code.
The following full-inference commands require a complete archive, beyond
the partial copy recovered for this visual revision: run \texttt{code/\allowbreak verify\_\allowbreak numerics.py}, then
\texttt{code/\allowbreak run\_\allowbreak mnist.py tune} and \texttt{code/\allowbreak run\_\allowbreak mnist.py evaluate}.
The native and geometry launchers are \texttt{code/\allowbreak run\_\allowbreak native\_\allowbreak lpca.py}
and \texttt{code/\allowbreak run\_\allowbreak geometry.py}; \texttt{code/\allowbreak summarize\_\allowbreak and\_\allowbreak plot.py}
rechecks scores from the saved arrays and draws the supplemental figures.
The historical MNIST grid remains archived. Figure~\ref{fig:mnistapplication} uses only CRC-verified saved arrays from the corrected benchmark, with frozen seed and query selection; no denoising or fitting run was repeated for this revision. The available corrected archive copy is truncated after part of its entries, so the compact display package includes the verified common-PCA64 arrays needed by the figure and records their hashes. The recovered protocol and per-run settings agree, but the frozen-selection manifest and tuning records are absent from this copy; this revision therefore cannot independently reverify the reported pre-test freezing chronology. Native LPCA784 images are not reconstructed from a substitute source; its reported aggregate values, transcribed from the existing table, remain in Figure~\ref{fig:mnistscores} and numerical supplement Table~S6. Geometry arrays and the original inference code remain in their source archives. The face protocol follows.

\subsection{Auxiliary image diagnostics: published face settings}\label{app:faceaudit}
These follow-up experiments address two distinct questions: reproduction of
\citet[Section~5.2]{yaoxia2025}'s particular facial-image setting, and the
behavior of the observation-only numerical procedures on separated image
subsets. They assess pixel reconstruction; no image score certifies
manifold structure, intrinsic dimension, or the hypotheses of our theorems.
The face data are grayscale video frames from one individual.
Image denoising is distinct from state-space reconstruction through scalar
delays. All unfavorable scores and undefined computations remain reported.
The frozen protocol, exact inputs, validation selections, output arrays,
termination records and all images are preserved in the historical image-experiment archive, identified in the accompanying reproducibility package, and its companion
\texttt{Yao\_\allowbreak Xia\_\allowbreak Face\_\allowbreak SourceData\_\allowbreak 20260922.zip} source-data archive.

\subsubsection{Published face setting and numerical port}\label{sec:faceoriginal}
The pinned author repository (commit
\texttt{c87cf4db53\allowbreak 8c4b0567f0\allowbreak 396903650f\allowbreak 39339688fd}) provides 1000 grayscale
$80\times80$ frames from the same individual and published Gaussian-corrupted
arrays. The face video is credited to \citet{happy2012faces}.
The original query indices are $50,300,740,850,970$ (MATLAB indexing);
all other 995 frames are training points. We use the supplied noisy arrays,
not new simulated noise, and retain $d=10$, $\beta=2$, ten iterations,
initial step one, ten backtracks and tolerance $10^{-8}$. The radius is the
mean distance to the tenth other observation among all 1000 noisy frames;
this preserves the original transductive radius calculation. The parameter
$\rho\in\{0.2,0.3,0.4\}$ means
$\sigma=\rho\,\operatorname{mean}(X_0)$, where the supplied mean intensity
is $0.2423256373$. It does not mean a pixel standard deviation of $\rho$,
and $d=10$ is a chosen rank, not a known intrinsic dimension.

MATLAB/Octave was unavailable. We implemented the published MATLAB
\texttt{xy} update direction and line search in NumPy/SciPy, and independently
cross-checked the direction against the unchanged public Python objective
routine. The maximum relative difference on four full-rank fixtures is
$7.58\times10^{-15}$. This is not execution of unchanged MATLAB or a bitwise
comparison to an author-supplied output array. One edge case is explicit:
when no training point lies outside radius $2r$, every other training point
is used. The source's search for the first outside point otherwise has no
result (the Python version raises \texttt{StopIteration}). This occurs at
145/995 tangent centers for $\rho=0.4$ and at zero centers for the other rates.

For the original-dimensional calculation, an orthonormal matrix $E$ spans
all noisy sample and initial-query differences. If $y_i=c+Ez_i$ and
$x=c+Ez$, distances and weights are preserved, every sample tangent lies
in this span, and $F_D(c+Ez)=F_q(z)$. Both the full derivative and the
published numerical direction transform by $E$, so all updates remain
in the span. We retain all 1000 QR columns in $\mathbb R^{6400}$ (including
one extra null direction), with no explained-variance truncation. The
maximum noisy-pixel discrepancy after the coordinate round trip is
$3.70\times10^{-15}$. This complete-span coordinate change is distinct
from the rank-64 preprocessing below. Numerical zero eigenvalues in the
thin factorization use relative tolerance $10^{-13}$ and absolute
$10^{-15}$; the omitted nullspace term is included analytically.

Table~\ref{tab:facehistoricalmain} gives the scores for all three original-setting noise conditions. The saved image arrays remain in the experiment archive.

\begin{table}[H]\centering\small
\caption{Published face settings: mean unclipped pixel MSE over the five original queries. The native local-PCA control uses the mean-crossing rank in $D=6400$.}\label{tab:facehistoricalmain}
\begin{tabular}{@{}rrrrr@{}}\toprule
$\rho$ & Actual $\sigma$ & Noisy MSE & Native LPCA mean rank & Reference $G$ MSE\\\midrule
0.2 & 0.048465 & 0.002344 & 0.000147 & 0.000128 \\
0.3 & 0.072698 & 0.005337 & 0.000372 & 0.000242 \\
0.4 & 0.096930 & 0.009490 & 0.000811 & 0.000487 \\
\bottomrule\end{tabular}\end{table}
The native local-PCA controls were added after inspecting the reference
run, without retuning the published radius. They use identical noisy
training points and queries; the diagnostic mean-crossing rule uses the
original $D=6400$. These are descriptive five-image controls, not an
independently designed superiority test. Fixed-rank-$10$ local-PCA outputs
and failures are also archived. Fourteen of fifteen reference-direction
queries reach the objective tolerance; frame 740 at $\rho=0.4$ stalls in
line search and remains in the scores. Training contains nearby video
frames, so these results do not establish performance on unseen people.

\subsubsection{Separated face comparison}\label{sec:imagematched}
Before validation or test scoring, we froze five length-200 face blocks.
Within each block, zero-based offsets $0{:}109$ are training,
$120,122,\ldots,138$ are validation and $150,152,\ldots,188$ are test.
Thus there are 550 training, 50 validation and 100 test frames, with at
least eleven frames of index separation between different roles. This
is still one individual. The independently generated noise uses
$\sigma=\rho$ times the clean training mean for $\rho=0.2,0.4$.
The face comparison uses $k=1,4$
and three fixed noise seeds, with the first observation nested across
repeat counts. Each training and query image receives independent
Gaussian observations. Divide pixels by $\sqrt D$ and fit centered PCA64
only to noisy training means; all methods receive this identical
preprocessing. Radius candidates are $0.9,1.2,1.6$ times the mean training
tenth-neighbor distance in these coordinates. Rank candidates are
$5,10,20$. Each method minimizes clean
validation pixel MSE separately, then its selected record and SHA256 are
saved before test inference. No clean test score participates in selection.

The comparison includes observation mean, PCA only, fixed-rank local PCA,
mean-crossing local PCA in $q=64$, the published reference direction
($G$), and the full analytic derivative ($\nabla F$). Both iterative
methods use $\beta=2$ and at most twenty iterations. The reference monotone
line search and full-derivative Armijo search retain different tolerances
and step rules; this is not an isolated derivative ablation. The
mean-crossing rule in $q=64$ is a preprocessing variant of the historical
native-$D$ rule, not the estimator of our consistency theorem.

Table~\ref{tab:facematchedmain} below reports all four face noise/replication conditions, each averaged over three noise seeds. Complete imagewise metrics and termination counts are retained in the accompanying data files; Table~\ref{tab:imagefailures} below displays the fallback counts explicitly.
Entries are unclipped pixel MSE over 100 face images, averaged
across three noise seeds. Undefined computations fall back to PCA-only,
with every failure retained in the denominator. In particular, apparent
low MSE for fixed-rank local PCA can reflect this fallback rather than a
successful local projection. All failures, stalls, iteration limits,
imagewise PSNR, valid-window SSIM, paired differences and joint-success
subsets are recorded. Solvers and tuning use float64; archived image
outputs use float32 and are scored in float64 from those exact stored
values. Historical outputs remain float64.

\begin{table}[H]\centering\scriptsize
\caption{Separated image comparison: unclipped pixel MSE averaged over three noise seeds. Face noise is the relative rate $\rho$. Undefined computations return PCA-only and remain in the scores; Table~\ref{tab:imagefailures} reports these counts.}\label{tab:facematchedmain}
\begin{tabular}{@{}lrrrrrrr@{}}\toprule
Data & $\rho$ & $k$ & Mean & PCA64 & LPCA mean rank & Ref. $G$ & Full $\nabla F$\\\midrule
Face & 0.2 & 1 & 0.002358 & 0.001883 & 0.001992 & 0.001966 & 0.001929 \\
Face & 0.2 & 4 & 0.000590 & 0.001640 & 0.001664 & 0.001643 & 0.001641 \\
Face & 0.4 & 1 & 0.009433 & 0.002372 & 0.002696 & 0.002441 & 0.002461 \\
Face & 0.4 & 4 & 0.002361 & 0.001886 & 0.002004 & 0.001969 & 0.001934 \\
\bottomrule\end{tabular}\end{table}

\subsubsection{Added NLM and BM3D baselines}\label{app:newimagebaselines}
The two added spatial denoisers operate on raw image means, not on the
rank-64 PCA coordinates used by the geometric fitting tracks. For each
job, the scalar noise level is the median of wavelet-MAD estimates
computed from noisy training-image means alone. We freeze this value and
the filter options in \texttt{selection.json} before reading test arrays.
Non-local means uses patch size 5, search distance 6, fast mode and
$h=0.8\widehat\sigma$; the implementation is \texttt{skimage.restoration}
in scikit-image 0.26.0 \citep{buades2005,buades2011}. BM3D uses all
stages and the \texttt{np} profile in package \texttt{bm3d} 4.0.3,
which implements the correlated-noise collaborative filter
\citep{dabov2007,makinen2020}. Inputs and outputs are not clipped during fitting or scoring.
Non-finite results fall back to the input and are counted; none occurred.
The original PCA64, reference-direction and local-PCA output arrays are
rescored without modification. The code, software versions, per-image
MSE/PSNR/SSIM, failures, fixed representative indices and SHA256
provenance are included with the job records. Replay with
\texttt{python3 code/run\_face\_baselines.py --dataset face --workers 4}.
The runs cover 12 face noise/repeat/seed jobs, or 1,200
test-image/filter evaluations per new filter. Face rows are
within five separated blocks of one person and do not provide iid-frame
or unseen-identity inference.

\input{image_extensions/main_table.tex}
Table~\ref{tab:fourimagebaselines} scores every fixed test image and all
three seeds. Display clipping in the archived galleries is separate from
unclipped MSE, imagewise PSNR and SSIM. The spatial filters and the
training-fitted PCA64 tracks have intentionally different preprocessing.

\subsubsection{PCA error and interpretation}\label{sec:imagepca}
Conditional on the noisy training data, let $P$ be the fixed rank-$q$
orthogonal projector and $c$ its center. For an independent query-average
noise vector $e\sim N(0,\sigma^2 I_D/k)$, the reconstruction is
$\widehat x=c+P(x+e-c)$. Orthogonality gives
\begin{align*}
 \|\widehat x-x\|^2
 &=\|(I-P)(x-c)\|^2+\|Pe\|^2,\\
 \E[\mathrm{MSE}\mid\text{training},x]
 &=\frac{\|(I-P)(x-c)\|^2}{D}+\frac{q\sigma^2}{Dk}.
\end{align*}
The first term is lost image information; the second is remaining
query noise. Our post-hoc decomposition uses the already fixed
training basis for evaluation only; clean test targets do not train or
tune it. In several conditions projection loss exceeds the noise left
in an unprojected observation mean. The initially frozen $q=128$
sensitivity uses seed zero at the higher single-view face noise
and its own validation-only selection; it does not establish
an optimal rank. Tables~\ref{tab:imagepcavalues} and~\ref{tab:imageqsensitivity} retain the
decomposition values and the predeclared sensitivity. Increasing $q$
changes both projection loss and neighborhoods, without improving every
method. The separated face results do not demonstrate
uniform improvement over both the observation mean and PCA-only.

Face summaries use five time-block means and
leave-one-block-out ranges, not iid-frame confidence intervals. The data
are not fresh training samples for every seed and do not establish
unseen-identity or general population guarantees. These image diagnostics
do not replace the known-geometry sphere and torus controls.

\subsubsection{Full derivative versus the reference numerical direction}
Write $\delta_i=x-y_i$, $\rho_i=r^2-\|\delta_i\|^2>0$,
$g_i=-2\beta\delta_i/\rho_i$ and
$\nabla\alpha_i=\alpha_i(g_i-\sum_l\alpha_lg_l)$.
For $S=\sum_i\alpha_iU_iU_i^{\mathsf T}$, let $t_j,s_j$ be its leading
$d$ eigenpairs and $n_l,\lambda_l$ its remaining eigenpairs. With
$v=x-\sum_i\alpha_i y_i$, $b=(I-P_{\rm top\,d}(S))v$, set
\[
 R_j=\sum_l\frac{n_l n_l^{\mathsf T}b}{s_j-\lambda_l},\qquad
 c_i=\sum_j(t_j^{\mathsf T}v)
        (U_i^{\mathsf T}R_j)^{\mathsf T}(U_i^{\mathsf T}t_j).
\]
On a separated eigenspace and fixed open neighborhood region,
\[
 \nabla F=2b+2\sum_i(\delta_i^{\mathsf T}b-c_i)\nabla\alpha_i.
\]
Direct algebra from the archived \texttt{get\_obj\_grad.m} instead yields
\[
 G_{\rm ref}=b+\sum_i\{\alpha_i\delta_i^{\mathsf T}b+c_i\}
                  \nabla\alpha_i.
\]
We preserve the latter expression in the reference track. The maximum
relative centered-difference error of the full derivative on eight
fixtures is $4.72\times10^{-10}$. The tests also cover a nontrivial
normal nullspace and orthogonal embedding. These are checks of numerical
expressions on finite fixtures, not an audit of the paper's theorems.
Every saved image score, selection hash, nested first observation and
nonincreasing objective trace is replayed. Uniform $7\times7$ SSIM with
sample covariance factor $49/48$ is separately
checked against direct window enumeration. Tables~\ref{tab:facehistoricalmain}--\ref{tab:facematchedmain} above retain the original and separated image-comparison scores. The PCA calculation is in Appendix~\ref{sec:imagepca}; detailed fallback counts and decomposition values follow below. Individual outputs remain in the experiment archive.

\subsubsection{Fallback counts and PCA sensitivity}
Table~\ref{tab:facematchedmain} includes every query. Undefined neighborhoods,
tangents or projectors return PCA-only; Table~\ref{tab:imagefailures} reports
these fallbacks. Finite stalled and iteration-limited outputs are retained
and separately recorded in \texttt{data/summary.json}. Totals cross fixed
images with three noise seeds, not independently redrawn training samples.

\begin{table}[H]\centering\scriptsize\setlength{\tabcolsep}{4pt}
\caption{Undefined-computation counts with PCA-only fallback. The observation mean and PCA-only baselines have zero such failures.}
\label{tab:imagefailures}
\begin{tabular}{@{}lrrrrrrr@{}}\toprule
Data & $\rho$ & $k$ & Total & LPCA fixed & LPCA mean & Ref. $G$ & Full $\nabla F$\\\midrule
Face & 0.2 & 1 & 300 & 292 & 148 & 189 & 215 \\
Face & 0.2 & 4 & 300 & 300 & 241 & 293 & 293 \\
Face & 0.4 & 1 & 300 & 150 & 22 & 21 & 21 \\
Face & 0.4 & 4 & 300 & 292 & 142 & 186 & 212 \\

\bottomrule\end{tabular}
\end{table}

\begin{table}[H]\centering\scriptsize\setlength{\tabcolsep}{4pt}
\caption{PCA error decomposition: mean over three seeds. All entries are pixel MSE; the projection loss is an evaluation-only quantity computed after selection.}
\label{tab:imagepcavalues}
\begin{tabular}{@{}lrrrrrr@{}}\toprule
Data & $\rho$ & $k$ & Projection loss & Noise expectation & Actual PCA & Mean\\\midrule
Face & 0.2 & 1 & 0.001860 & 0.000024 & 0.001883 & 0.002358 \\
Face & 0.2 & 4 & 0.001634 & 0.000006 & 0.001640 & 0.000590 \\
Face & 0.4 & 1 & 0.002279 & 0.000094 & 0.002372 & 0.009433 \\
Face & 0.4 & 4 & 0.001862 & 0.000024 & 0.001886 & 0.002361 \\

\bottomrule\end{tabular}
\end{table}

\begin{table}[H]\centering\scriptsize\setlength{\tabcolsep}{4pt}
\caption{Predeclared $q=128$ sensitivity: high noise, $k=1$, seed zero, identical fixed images and separate validation-only selection. Failure counts are retained for both dimensions.}
\label{tab:imageqsensitivity}
\begin{tabular}{@{}llrrrr@{}}\toprule
Data & Method & $q=64$ MSE & $q=128$ MSE & Fallback $q=64$ & Fallback $q=128$\\\midrule
Face & PCA only & 0.002364 & 0.002425 & 0/100 & 0/100 \\
Face & LPCA mean & 0.002663 & 0.002541 & 8/100 & 0/100 \\
Face & Ref. $G$ & 0.002430 & 0.002381 & 7/100 & 0/100 \\
Face & Full $\nabla F$ & 0.002523 & 0.002381 & 7/100 & 0/100 \\

\bottomrule\end{tabular}
\end{table}
This fixed seed-zero comparison does not determine an optimal $q$;
validation choices and fallbacks can also change. Saved arrays and
reproduction code retain the original visual diagnostics for replay.
\FloatBarrier
\section{Imported benchmarks: provenance, protocols and replay}\label{app:arxivbench}
\subsection{Role and provenance}
The point/ODE scores and independent KS endpoint failures below are
auxiliary local-PCA diagnostics. They do not certify latent dimension,
a manifold witness, Whitney coordinates or a scalar-delay map. The
separate certified route is summarized in Figure~\ref{fig:masterconcept}.
The six original figures from \citet{choi2026}, arXiv 2609.21311v1,
are retained byte-for-byte in the separate full source/data archive,
with programs, seeds, observed inputs, returned arrays and complete
score records under \texttt{arxiv\_reuse/}. The compact document/figure
package omits that complete replay tree; Section~\ref{sec:arxivreplay}
identifies the required archive. No new simulation or favorable subset
replaces the archived results. We retain the independently initialized
KS endpoints rather than duplicate dependent-trajectory and slice views.

\subsection{Point and ODE protocols and complete displayed scores}
Swiss roll and the area-uniform two-torus use 20,000 independent pilot
points, 2,000 test points, $q=3$, $k=4096$, and 512 independent calibration
pairs. Raw coordinate noise is $0.02,0.06,0.12$; fixed radii are $0.16$
and $0.25$. Swiss-roll boundary points remain in the test set.
For each Lorenz parameter $\rho=24,28,40$ and each R\"ossler parameter
$c=4,5.7,8.5$, 10,000 pilot endpoints, 2,000 test endpoints and 256 design
endpoints have independent initializations. Vectorized RK4 uses
$\Delta t=0.02$ and final time $40.02$. The pilot prefixes are nested;
test points stay fixed. Coordinate noise has standard deviation $0.02$
per raw view in normalized coordinates, after division by 30 for Lorenz
and 15 for R\"ossler. Here $k=1024$, with 512 calibration pairs; fixed
radii in these same coordinates are $0.10$ and $0.12$, respectively.

Independence concerns the finite-time endpoint sampling design. It does
not establish an invariant measure or a low-dimensional flow support:
a full-dimensional initial law can retain full-dimensional support.
Moreover, the mean-crossing count
$\#\{j:\lambda_j(C)>\operatorname{tr}(C)/q\}$ is at most $q-1$:
all eigenvalues cannot strictly exceed their own average. For $q=3$,
a rank-two count alone cannot distinguish a surface from a
full-dimensional endpoint law or establish attractor dimension.
The algorithm takes observed pilot and query averages, not clean vectors.
The supplied code simulates the exact Gaussian average laws; it does not
store $k$ individual measurements of every endpoint.

\begin{figure}[!t]\centering
\includegraphics[width=.98\textwidth]{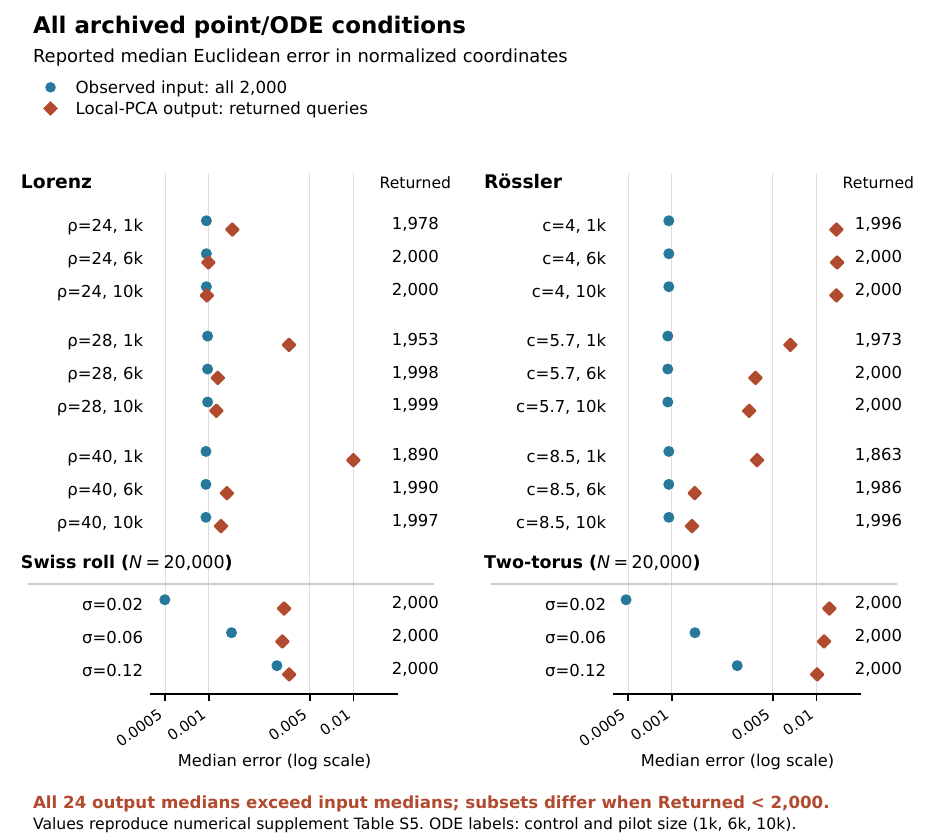}
\caption{All 24 archived point/ODE conditions, replacing a long numerical table. Input medians use all 2,000 queries; output medians condition on available projections, with returned counts shown for each row. Thus this is a comparison of reported medians, not paired per-query improvement. Errors are Euclidean in the stated normalized coordinates; the error axis is logarithmic. Every displayed output median exceeds its input median, including at larger pilot sizes. Rank-two counts and the exact six-decimal values are preserved verbatim in numerical supplement Table~S5. Population-spectral certificate counts are zero throughout, and no geometric certificate is available.}
\label{fig:pointodeerrors}
\end{figure}

\clearpage
\subsection{Independent KS fields and their limitations}
The periodic mean-zero scalar potential KS equation \citep{larios2022} is
\[
 u_t=-\Delta u-\Delta^2u
      -\tfrac12\bigl(|\nabla u|^2-\overline{|\nabla u|^2}\bigr).
\]
The finite-grid ETDRK4 integrations \citep{kassam2005} use
$\Delta t=0.25$, $T=160$, grids $32^2$ and $16^3$, and sensor grids
$16^2$ and $8^3$. Six cases each use 10,000 pilot, 64 test and 128 design
endpoints (61,152 independently initialized rollouts in total);
the $T=80$ checkpoint is not an independent sample. Observed vectors
follow the average law of 1,024 Gaussian views with raw sensor noise
$0.03$, in RMS coordinates $u_{\rm sensors}/\sqrt q$, $q=256$ or $512$.

Initial fields use 10 independent real sine/cosine coefficients in two
spatial dimensions and 16 in three. The fixed finite-grid solver and
sensor map are smooth functions of these bounded coefficients in exact
arithmetic, with no additional latent random forcing. Thus the noiseless
endpoint support has Hausdorff dimension at most 10 or 16, not necessarily
equal to these bounds. Gaussian observation noise is added afterward.
Fifteen of 18 median mean-crossing ranks exceed these bounds,
reflecting finite-scale covariance rather than latent dimension recovery.

The radius is 1.4 times the independent design scale; $L$ is the periodic
domain side length. Reused nested pilots and the same 64 tests across
sample sizes make rows dependent. Archived field views use test index
zero. Neither spectral nor geometric certificates are issued. Order-one
long-time refinement differences remain; these finite-grid results certify
neither continuum trajectories nor an invariant attractor.

\begin{table}[H]\centering\small
\setlength{\tabcolsep}{4pt}
\begin{tabular}{rrrrrrr}\toprule
Spatial dim. & $L$ & $N$ & Returned & Median rank & Input RMSE & Output RMSE\\\midrule
2 & 12 & 1,000 & 64/64 & 12 & 0.000937 & 0.310699 \\
2 & 12 & 6,000 & 64/64 & 12 & 0.000937 & 0.310456 \\
2 & 12 & 10,000 & 64/64 & 12 & 0.000937 & 0.311323 \\
2 & 18 & 1,000 & 64/64 & 31 & 0.000929 & 0.662572 \\
2 & 18 & 6,000 & 64/64 & 31 & 0.000929 & 0.669425 \\
2 & 18 & 10,000 & 64/64 & 31 & 0.000929 & 0.653899 \\
2 & 24 & 1,000 & 64/64 & 44 & 0.000940 & 0.964652 \\
2 & 24 & 6,000 & 64/64 & 44 & 0.000940 & 0.943657 \\
2 & 24 & 10,000 & 64/64 & 44 & 0.000940 & 0.952001 \\
3 & 8 & 1,000 & 64/64 & 12 & 0.000941 & 0.033145 \\
3 & 8 & 6,000 & 64/64 & 12 & 0.000941 & 0.033446 \\
3 & 8 & 10,000 & 64/64 & 12 & 0.000941 & 0.033269 \\
3 & 12 & 1,000 & 64/64 & 24 & 0.000943 & 0.820826 \\
3 & 12 & 6,000 & 64/64 & 24 & 0.000943 & 0.813902 \\
3 & 12 & 10,000 & 64/64 & 24 & 0.000943 & 0.814013 \\
3 & 16 & 1,000 & 64/64 & 80 & 0.000941 & 1.426808 \\
3 & 16 & 6,000 & 64/64 & 80 & 0.000941 & 1.382894 \\
3 & 16 & 10,000 & 64/64 & 80 & 0.000941 & 1.381539 \\
\bottomrule\end{tabular}
\caption{All 18 configurations of the independently initialized KS endpoint study at radius multiplier 1.4. Errors are medians over the same 64 tests of per-endpoint sensor RMSE, equal to Euclidean error in $u_{\rm sensors}/\sqrt q$ coordinates. Rank is the mean-crossing diagnostic in $q=256$ or $512$. The full archive retains all three radii, abstentions, certificate flags and discretization checks. Every output-error median exceeds its input-error median.}
\label{tab:arxivks}
\end{table}

\clearpage

\subsection{Code and observation replay}\label{sec:arxivreplay}
The replay commands require the separate \texttt{intrinsic\_dimension\_source.zip}
full source/data archive, Library version 13 (517,214,187 bytes), whose
SHA-256 identity is recorded in the accompanying README. The compact
package and the figure-recovery overlay alone are insufficient. The
full archive retains the six original PDFs, unchanged experiment programs,
and complete observed and returned arrays. Its bridge launcher
\texttt{code/reproduce\_arxiv\_visuals.py} selects these experiments.
\begin{center}\small
\begin{tabular}{@{}p{.27\linewidth}p{.67\linewidth}@{}}\toprule
Archived diagnostic & Original program under \texttt{arxiv\_reuse/}\\\midrule
Point and ODE & \texttt{benchmark\_comparison/code/\allowbreak audited\_point\_figures.py}, calling \texttt{reconstruct\_visuals.py:local\_project} and \texttt{id\_estimation/dimension.py}.\\
KS fields & \texttt{benchmark\_comparison/code/\allowbreak plot\_ks\_rollouts.py}; integration and estimation: \texttt{ks\_batch.py}, \texttt{ks\_rollout\_study.py}, \texttt{ks\_rollout\_evaluate.py}.\\\bottomrule
\end{tabular}
\end{center}
Run from the extracted full archive's
\texttt{Intrinsic\_Dimension\_Direct\_Whitney\_20261003/} root:
\begin{verbatim}
python3 code/reproduce_arxiv_visuals.py --stage quick
python3 code/reproduce_arxiv_visuals.py --stage figures
python3 code/reproduce_arxiv_visuals.py --stage full
\end{verbatim}
The quick stage checks two fixed observed queries in every point/ODE
configuration and two KS queries at all saved radii, retaining the original
full-query confidence budgets. It also reaggregates all saved scores.
The figure stage redraws the six archived figures using the original
programs and checks rendered-pixel equality with the arXiv PDFs. The full
stage refits all 48,000 point/ODE and 3,456 KS projections; it does not
repeat the PDE integrations. Clean references enter scoring and display
only. Original random seeds and integration code remain available for
fresh data generation, which is a separate, more expensive operation.
The present audit reaggregated all displayed rows and performed 48 point/ODE
and 108 KS fixed-query inference checks (156 total), all with zero
output difference. It did not rerun the full inference study or the
ODE/PDE integrations.

\section{List of codes and immediate replay instructions}\label{app:codes}

\paragraph{Core routines and their mathematical roles.}
For Theorem~\ref{thm:main}, \texttt{observations.py} propagates recording errors, \texttt{noise\_\allowbreak radius} computes $a$, and \texttt{coverage\_\allowbreak radius} computes $h$ from the supplied lower-mass model. The routines \texttt{witness\_\allowbreak net}, \texttt{point\_\allowbreak distance\_\allowbreak upper}, and \texttt{verify} provide the deterministic bounds in \eqref{eq:verify} and \eqref{eq:globalfit}. The public pipeline is \texttt{mh\_\allowbreak certificate.py}; \texttt{certify.py} performs rational checks. Training-only fitting and independent validation implement Theorem~\ref{thm:tube}. The all-scale hypotheses of Theorem~\ref{thm:localexistence} have no implemented finite-data verifier.

\paragraph{Dimension estimators and numerical checks.}
\texttt{code/\allowbreak reach\_\allowbreak dimension.py} implements \eqref{ar:eq:estimator} and the observed-covariance transfer using a supplied normalized error bound. Its check program evaluates the anisotropic-disk counterexample and spectral separation, without verifying the geometric premises.
\texttt{code/\allowbreak intrinsic\_\allowbreak dimension.py} implements \eqref{eq:actualdimension} and Proposition~\ref{prop:dimfinite}'s numerical covariance radius. It accepts keys \texttt{means}, \texttt{residual\_\allowbreak sum}, \texttt{k}, or raw \texttt{observations}, but no true dimension or noise variance. The stored replay input has $n=4096$ and $r=0.8(4096/256)^{-0.15}$; the replay command uses $0.5278031643091577$. The calibration program generates 200 sphere runs and algebraic checks, as detailed in Appendix~\ref{app:dimimplementation}. These numerical outputs are separate from rational witness certification. The MNIST projector instead uses the mean-crossing diagnostic \eqref{eq:termsdiagnosticrank}.

\paragraph{Boundary, fractal and historical certificate replay.}
\texttt{boundary\_\allowbreak geometry.py} implements the restricted witnesses of Appendix~\ref{app:boundaryfractal}; \texttt{run\_\allowbreak moons\_\allowbreak fractals.py} generates 210 main-study and six proxy inputs or replays saved certificates. Its plotter only draws saved outputs. Geometry and Julia-reference tests and \texttt{audit\_\allowbreak moons\_\allowbreak fractals.py} check provenance, population assessments and summaries. The frozen historical-v9 launcher preserves its original modules and replays 1020 matched-target records; the Landau launcher replays 48 moment records and 16 retained raw arrays. Unfavorable and unresolved results remain in the records.

\paragraph{Direct coordinates and observation interfaces.}
\texttt{code/\allowbreak verify\_direct\_embedding.py} implements Appendix~\ref{app:directembedding}'s exact affine-span calculation; its fixtures are algebraic, not new observational certificates.
\texttt{code/\allowbreak prepare\_\allowbreak observation\_\allowbreak route.py} implements Algorithm~\ref{alg:observationroute}'s fixed vector and independent-group scalar selection, preserving observations and dtype. Its records mark unverified physical assumptions and no certificate issuance. Its \texttt{-{}-bank} option checks a predeclared finite protocol bank and its rational budget sum without issuing protocol-specific statistical certificates.
\texttt{code/\allowbreak reconstruction\_bounds.py} computes the supplied-bound consequences of Theorem~\ref{thm:coordinatecomposition} and Proposition~\ref{prop:transitionbounds} with exact fractions; its nearest-net option minimizes rational squared residuals exactly. These array, algebra and recurrence fixtures establish neither state observability, Takens assumptions, independent sampling, unknown dynamics, global transition coverage nor an invariant prediction domain.

\paragraph{Replay instructions and package scope.}
Compile with \texttt{python3 build.py} and check preservation with \texttt{python3 code/\allowbreak verify\_\allowbreak application\_\allowbreak revision.py}.
The accompanying \texttt{README\_\allowbreak APPENDIX\_\allowbreak REPLAY.txt} contains the command sequences and operational notes transferred from this appendix; their original source is preserved verbatim in \texttt{archive/\allowbreak appendix\_\allowbreak trim\_\allowbreak 20261015/}. The current 15-figure asset, program, input and hash map is \texttt{review/figure\_code\_map.json}. The compact package compiles the manuscript and numerical supplement and regenerates the included saved-array and table-derived application plots without new inference. Tables~S1--S6 retain the exact values, including adverse outcomes and numerical failures. The corrected MNIST archive is only partially recovered, as recorded with its application evidence.

Full numerical replay requires the separate \texttt{intrinsic\_dimension\_source.zip}, Library version 13 (517,214,187 bytes), with SHA-256 identity in the README; the compact package and recovery overlay alone are insufficient. The overlay restores legacy figure dependencies where identified. MNIST replay uses the original archive and fixed train/test indices, checking outputs, ranks, neighborhood counts and abstentions. Ordinary figure replay redraws saved results rather than rerunning all inference. The transferred instructions preserve these distinctions and the audit scope of excluded diagnostics.

\clearpage
\section{Mathematical terminology, SVD conventions, and credits}\label{app:terms}
This appendix fixes the mathematical meaning of the terms used in the paper. It distinguishes a population hypothesis, a numerical proposal or diagnostic, and a verified certificate. These definitions do not add a new implemented estimator or change any experimental decision.

\subsection{Support, distances, reach, and the certificate event}\label{app:termsgeometry}
For a Borel probability measure $\mu$ and a nonempty closed set $A\subset\R^q$, the conventions are
\begin{equation}
\supp\mu=\{x:\mu(B(x,r))>0\text{ for every }r>0\},\qquad
\dist(z,A)=\inf_{a\in A}\norm{z-a}.
\label{eq:termssupport}
\end{equation}
The directed distance and the Hausdorff distance in \eqref{eq:hausdorff} are
\[\overrightarrow d(A,B)=\sup_{a\in A}\dist(a,B),\qquad d_H(A,B)=\max\{\overrightarrow d(A,B),\overrightarrow d(B,A)\}.\]
 Following \citet{federer1959},
\begin{equation}
\reach(A)=\sup\{\rho\ge0:\text{every }z\text{ with }\dist(z,A)<\rho
\text{ has a unique nearest point in }A\}.
\label{eq:termsreach}
\end{equation}
A finite $\kappa$-net $G\subset M$ satisfies $\sup_{z\in M}\dist(z,G)\le\kappa$. Consequently,
\begin{equation}
\sup_{z\in M}\dist(z,Y)\le\max_{g\in G}\dist(g,Y)+\kappa.
\label{eq:termsnet}
\end{equation}
This reverse-direction check is a geometric coverage calculation, not a reconstruction residual. The tubular neighborhood and its outside mass are
\begin{equation}
T_\epsilon(M)=\{z:\dist(z,M)\le\epsilon\},\qquad
R_\mu(M,\epsilon)=\mu(T_\epsilon(M)^c).
\label{eq:termstube}
\end{equation}
Thus a support claim $d_H(E,M)\le\epsilon$ and a mass claim $R_\mu(M,\epsilon)\le\alpha$ are different targets. $\EMIT$ means that the specified certificate is issued; $\ABSTAIN$ means that it is withheld. If $\mathcal C$ is the reported claim, the controlled quantity is $\Prob(\EMIT\cap\mathcal C^c)\le\delta$. It is not a bound on $\Prob(\mathcal C^c\mid\EMIT)$ when issuance is rare.

\subsection{Ahlfors--David regularity and the lower mass hypothesis}\label{app:termsmass}
The two-sided, local measure version of $s$-regularity is
\begin{equation}
0<c_{\rm A}\le\frac{\mu(\overline B(x,r))}{r^s}\le C_{\rm A}<\infty,
\qquad x\in E,\quad 0<r\le r_0.
\label{eq:termsahlfors}
\end{equation}
The set version replaces $\mu(\overline B(x,r))$ by $\mathcal H^s(E\cap\overline B(x,r))$. For compact sets the scale range must be bounded: a finite measure cannot have a positive $r^s$ lower bound for arbitrarily large $r$. Normalizing Hausdorff measure to a probability measure changes the constants. More generally, if $\mu=f\mathcal H^s|_E$ with positive upper and lower bounds on $f$, the measure and set versions imply one another with adjusted constants.

\paragraph{Attribution.}
The name Ahlfors--David regularity is retained explicitly. We acknowledge \citet{ahlfors1935} as historical background in geometric function theory. For the modern two-sided definition, the precise source used here is \citet[Section~1, Eq.~(1.5), p.~8]{davidsemmes1991}, localized to our bounded scale range. The historical Ahlfors citation is not being used as a proof of the modern higher-dimensional statement or of manifold existence. The statistical coverage argument uses only the lower inequality, via the support-estimation reasoning credited to \citet{chazal2015}.

In particular, the exponent $s$ in \eqref{eq:mass} is an externally supplied coverage exponent. It is not estimated by assuming $s=d$, and it does not certify a tangent plane. The circle-with-branch construction in Appendix~\ref{app:obstruction} satisfies uniform two-sided length bounds while having a nonmanifold branch point. Local existence instead uses the all-scale bilateral hypotheses of Theorem~\ref{thm:localexistence}, credited to \citet{kafer2019}.

\subsection{Local SVD, covariance PCA, and affine-plane fitting}\label{app:termssvd}
Fix a center $y$ and a radius $r>0$, and suppose the neighborhood contains $m\ge2$ observations. With observations stored as columns, define
\begin{align}
I(y,r)&=\{i:\norm{Y_i-y}\le r\},\qquad m=|I(y,r)|,
&\bar Y_{y,r}&=\frac1m\sum_{i\in I(y,r)}Y_i,\nonumber\\
A_{y,r}&=\frac1{\sqrt m}[Y_i-\bar Y_{y,r}]_{i\in I(y,r)}\in\R^{q\times m},
& A_{y,r}&=U\Sigma V^{\mathsf T}.
\label{eq:termslocalsvd}
\end{align}
Let $s_1\ge\cdots\ge s_q\ge0$ denote its singular values, padded with zeros when necessary. The covariance uses the empirical-measure divisor $m$, not the unbiased-sampling divisor $m-1$:
\begin{equation}
\widehat C_y(r)=A_{y,r}A_{y,r}^{\mathsf T}
=U\operatorname{diag}(s_1^2,\ldots,s_q^2)U^{\mathsf T},
\qquad \lambda_\ell(\widehat C_y(r))=s_\ell^2.
\label{eq:termscovsvd}
\end{equation}
Thus SVD acts on a rectangular data matrix, whereas PCA diagonalizes its centered covariance; their leading subspaces agree. Singular values and covariance eigenvalues differ by a square. The centered matrix has rank at most $\min(q,m-1)$.

For $1\le d<q$, write $U_d=(u_1,\ldots,u_d)$ and $P_{d,y,r}=U_dU_d^{\mathsf T}$. The least-squares affine plane and the plane anchored at the test center are, respectively,
\begin{equation}
L^{\rm LS}_{y,r}=\bar Y_{y,r}+\operatorname{im}P_{d,y,r},\qquad
L^{\rm anchor}_{y,r}=y+\operatorname{im}P_{d,y,r}.
\label{eq:termsplanes}
\end{equation}
The second is the form needed by \eqref{eq:bilateral}. They need not coincide. Orthogonal projection and centering give the exact identities
\begin{align}
\frac1m\sum_{i\in I(y,r)}\dist(Y_i,L^{\rm LS}_{y,r})^2
&=\norm{(I_q-P_{d,y,r})A_{y,r}}_F^2
=\sum_{\ell>d}\lambda_\ell(\widehat C_y(r)),\label{eq:termsresidual}\\
\frac1m\sum_{i\in I(y,r)}\dist(Y_i,L^{\rm anchor}_{y,r})^2
&=\sum_{\ell>d}\lambda_\ell(\widehat C_y(r))
+\norm{(I_q-P_{d,y,r})(\bar Y_{y,r}-y)}^2.\nonumber
\end{align}
We prove these classical identities for completeness and credit low-rank least-squares approximation to \citet{eckart1936}. Put $w_i=Y_i-\bar Y$ and $b=\bar Y-y$. Since $\sum_iw_i=0$,
\[
 \frac1m\sum_i\norm{(I-P)(w_i+b)}^2
 =\frac1m\sum_i\norm{(I-P)w_i}^2+\norm{(I-P)b}^2,
\]
with cross term $2m^{-1}\langle\sum_i(I-P)w_i,(I-P)b\rangle=0$. Also
\[
 \frac1m\sum_i\norm{(I-P)w_i}^2
 =\operatorname{tr}((I-P)\widehat C)=\operatorname{tr}\widehat C-\operatorname{tr}(P\widehat C).
\]
For any orthogonal rank-$d$ projection, write $a_\ell=\norm{Pu_\ell}^2\in[0,1]$, so $\sum_\ell a_\ell=d$ and $\operatorname{tr}(P\widehat C)=\sum_\ell\lambda_\ell a_\ell$. Since $\lambda_\ell$ decrease,
\[
 \sum_\ell\lambda_\ell a_\ell
 \le\sum_{\ell\le d}\lambda_\ell-\lambda_d\sum_{\ell\le d}(1-a_\ell)
       +\lambda_d\sum_{\ell>d}a_\ell
 =\sum_{\ell\le d}\lambda_\ell.
\]
Equality holds for the leading eigenspace. Thus the displayed plane is indeed a least-squares minimizer, with residual $\sum_{\ell>d}\lambda_\ell$, including ties. These are average one-sided errors. They do not check missing regions of a plane disk. A leading $d$-dimensional subspace is unambiguous when $s_d>s_{d+1}$; at a tie, any selected subspace must be specified, and stability must not be presumed.

\subsection{Multiscale SVD and bilateral flatness are different checks}\label{app:termsmultiscale}
The multiscale SVD framework of \citet{little2017} repeats the preceding decomposition over a chosen set of radii, for example $r_j=2^{-j}r_0$, $0\le j\le J$. It analyzes a family, not a different matrix factorization:
\begin{equation}
\mathcal S_J(y)=\{(r_j,m_j,\Sigma_{y,r_j},P_{d,y,r_j}):0\le j\le J\}.
\label{eq:termsmultiscale}
\end{equation}
Our notation for illustrative summaries of this family is
\begin{equation}
\vartheta_{\ell,j}=\frac{s_\ell(y,r_j)^2}{r_j^2},\qquad
T_{d,j}=\sum_{\ell>d}\vartheta_{\ell,j},\qquad
\Delta_{d,j}=\norm{P_{d,y,r_{j+1}}-P_{d,y,r_j}}_{\rm op}.
\label{eq:termsscaleprofiles}
\end{equation}
Here $T_{d,j}$ is the normalized mean squared fit residual and $\Delta_{d,j}$ measures subspace change at adjacent scales, where the selected subspaces are defined. These formulas describe scale diagnostics; they are not a claim to reproduce every selection rule in \citet{little2017}. Neither a small $T_{d,j}$ nor stable $\Delta_{d,j}$ proves a reverse covering bound.

For comparison, the normalized bilateral flatness of a closed set is
\begin{equation}
\theta_{E,d}(x,r)=\inf_{\substack{V\text{ linear}\\\dim V=d}}
\frac{d_H\bigl(E\cap\overline B(x,r),(x+V)\cap\overline B(x,r)\bigr)}{r}.
\label{eq:termsbilateral}
\end{equation}
Theorem~\ref{thm:localexistence} requires uniform control by a summable sequence at every support center in its specified neighborhood:
\begin{equation}
\sup_{x\in E\cap B(x_0,2R)}\theta_{E,d}(x,2^{-j}R)\le\epsilon_j,
\qquad \sum_{j=0}^{\infty}\epsilon_j<\infty.
\label{eq:termsdini}
\end{equation}
A finite sequence $\mathcal S_J$ neither establishes this infinite sum nor tests unobserved centers. We credit the geometric regularity result to \citet{kafer2019}; the executable witness verifier does not check \eqref{eq:termsdini}.

There is also a noise-selection distinction. Under the repeated Gaussian model, for any fixed latent selection set $B$ of positive probability,
\begin{equation}
\Cov(Y_i\mid Z_i\in B)=\Cov(Z_i\mid Z_i\in B)+\frac{\sigma^2}{k}I_q.
\label{eq:termsnoiseidentity}
\end{equation}
For completeness, write $Y_i=Z_i+\bar\varepsilon_i$ with $\E\bar\varepsilon_i=0$ and $\Cov(\bar\varepsilon_i)=\sigma^2I_q/k$. Independence gives $\E(\bar\varepsilon_i\mid Z_i\in B)=0$ and both conditional cross-covariance matrices zero. Expanding the centered outer product therefore gives the displayed sum. Selection by $Y_i\in B$ generally invalidates this identity. Subtracting an isotropic noise estimate from a covariance computed in a noisy selected neighborhood is therefore not, by itself, an unbiased local covariance correction.

\subsection{Dimension rules and the implemented projection}\label{app:termsdimension}
The conditional population-to-covariance bridge in Proposition~\ref{prop:covariance} gives
\begin{equation}
\frac{C_x(r)}{r^2}\longrightarrow\frac{P_x}{d+2},\qquad
\widehat d_n^{\rm bridge}=\sum_{\ell=1}^q
\mathbf1\{\lambda_\ell(\widehat C_n)>\gamma r_n^2\},\quad
0<\gamma<\frac1{q+2}.
\label{eq:termsbridgerank}
\end{equation}
It implies dimension consistency only under the stated local-existence and density hypotheses and $r_n^{-2}\norm{\widehat C_n-C_x(r_n)}_{\rm op}\to0$ in the claimed mode of convergence. Theorem~\ref{thm:dimconsistent} now proves this last condition for the observed-repeat statistic \eqref{eq:dimcov}, with the explicit threshold $\gamma=1/[2(q+2)]$. It is not established for every observed-data covariance or for the different historical diagnostic below.

Lemma~\ref{lem:growingspectral} instead uses the largest adjacent gap in \eqref{eq:growinggap}. In that subsection's notation $d_N$ is ambient dimension and $p$ is geometric dimension. Its $d_N/N=O(1)$ conclusion requires a uniform tangent-disk approximation and a controlled local sample; observed noise requires the additional transfer bound. The code's \texttt{eigengap\_\allowbreak dimension} field records this separate spectral rule.

The archived reconstruction study uses a different, mean-crossing diagnostic:
\begin{equation}
\bar\lambda=\frac1q\operatorname{tr}\widehat C_y(r),\qquad
\widehat d^{\rm diag}=\#\{\ell:\lambda_\ell(\widehat C_y(r))>\bar\lambda\}.
\label{eq:termsdiagnosticrank}
\end{equation}
In \texttt{local\_\allowbreak project}, the actual index set is a guarded core
$I_g(x,r)=\{i:\norm{Y_i-x}\le r-h_{\rm loc}-g_{\rm num}\}$, provided $r>h_{\rm loc}+g_{\rm num}$. Here $h_{\rm loc}$ is the supplied localization allowance and $g_{\rm num}$ is the floating-point distance guard. With $\widehat C$ computed on this core and machine precision $\varepsilon_{\rm mach}$, the code uses
\[
t_{\rm num}=64\varepsilon_{\rm mach}q\max\{r^2,\max_\ell|\lambda_\ell|,\mathrm{tiny}\}
\]
and withholds a projection when $m<2$, $\widehat d^{\rm diag}\notin\{1,\ldots,q-1\}$, or $\min_\ell|\lambda_\ell-\bar\lambda|\le t_{\rm num}$. Otherwise it returns
\begin{equation}
\widehat x=\bar Y_{x,r}+P_{\widehat d^{\rm diag},x,r}(x-\bar Y_{x,r}).
\label{eq:termsprojection}
\end{equation}
For $m<q$, the code uses the dual Gram matrix $A^{\mathsf T}A$, which has the same nonzero eigenvalues as $AA^{\mathsf T}$; the ambient spectrum is padded to $q$ entries before averaging. Algebraically, if $A^{\mathsf T}Av=\lambda v$ with $\lambda>0$, then $u=Av/\sqrt\lambda$ is a unit eigenvector of $AA^{\mathsf T}$ with eigenvalue $\lambda$; conversely use $v=A^{\mathsf T}u/\sqrt\lambda$. The null eigenvalues, including those due to centering $A\mathbf1=0$, must still be counted when forming the ambient mean. These finite-precision guards do not prove the consistency premise in \eqref{eq:termsbridgerank} and are not interval verification.

The current segment proposal in \texttt{code/\allowbreak boundary\_\allowbreak geometry.py} instead performs one \emph{global training SVD}. If $u_1$ is the leading training direction, it proposes the endpoints
\[
a=\bar Y_{\rm tr}+u_1\min_{i\in\mathrm{tr}}u_1^{\mathsf T}(Y_i-\bar Y_{\rm tr}),\qquad
b=\bar Y_{\rm tr}+u_1\max_{i\in\mathrm{tr}}u_1^{\mathsf T}(Y_i-\bar Y_{\rm tr}).
\]
Rounded endpoints define the actual rational candidate $[a,b]$. Its dimension is one by construction; acceptance separately verifies distances and the certificate bounds. Neither this proposal nor the historical local projection is advertised as a general multiscale-SVD certificate implementation.

Finally, covariance rank is not heat-kernel spectral dimension. Given a specified nonnegative generator $\mathcal L$ and reference measure, suppose $e^{-t\mathcal L}$ has heat kernel $p_t(x,z)$. One small-time convention, when $0<p_t(x,x)<\infty$ and the limit exists, is
\begin{equation}
d_{\rm heat}(x)=-2\lim_{t\downarrow0}\frac{\log p_t(x,x)}{\log t}.
\label{eq:termsheatdim}
\end{equation}
For $\mathcal L=-\Delta$ on $\R^d$, $p_t(x,x)=(4\pi t)^{-d/2}$ gives $d_{\rm heat}=d$ directly. Our algorithm constructs no heat kernel or Laplacian and estimates no quantity in \eqref{eq:termsheatdim}. Proposition~\ref{prop:covariance} concerns a covariance spectrum.

\subsection{HADES, MMD, Wasserstein distance, and our target}\label{app:termshades}
HADES is the \emph{Hypothesis-testing Algorithm for the Detection and Exploration of Singularities} of \citet{hades2025}. We use its open preprint, Definitions~5--6, for the following notation. For $\operatorname{tr}C>0$, its residual-variance convention gives
\begin{equation}
\widehat d_{\rm H}(C)=\min\left\{d\in\{1,\ldots,q\}:\frac{\sum_{\ell>d}\lambda_\ell(C)}{\operatorname{tr}C}\le\zeta\right\},
\quad 0<\zeta<1,\qquad
S_{\rm H}(y,r)=\operatorname{MMD}_K(\widehat\nu_{y,r}^{\perp},\mathfrak u_{\widehat d_{\rm H}}).
\label{eq:termshades}
\end{equation}
Here $\widehat\nu_{y,r}$ is the empirical neighborhood, excluding its anchor, translated by $-y$ and rescaled by $1/r$; $\perp$ denotes projection onto its leading PCA subspace, and $\mathfrak u_d$ is uniform on the unit $d$-ball. We use $\zeta$ to avoid confusion with our support radius $\eta$.

For a positive-definite kernel $K$, the MMD between probability measures $\nu,\rho$ is \citep{gretton2012}
\begin{equation}
\operatorname{MMD}_K^2(\nu,\rho)=
\iint K(u,v)\,d\nu(u)d\nu(v)+\iint K(u,v)\,d\rho(u)d\rho(v)
-2\iint K(u,v)\,d\nu(u)d\rho(v).
\label{eq:termsmmd}
\end{equation}
All displayed integrals are assumed finite.
The HADES comparison measures departure from a projected uniform disk. Its geometric recovery theorem assumes a specified union of smooth manifolds. This is a legitimate conditional detection guarantee. Our issued statement instead combines $\reach(M)\ge\widehat\tau$, $d_H(E,M)\le\widehat\epsilon$, and the joint error control in \eqref{eq:honesty}, or the separate tube-mass target. These are different claims, not an experimental ranking.

The Wasserstein distance used in local geometric analysis \citep{lim2024,hades2025} is, for finite $p$th moments,
\begin{equation}
W_p(\nu,\rho)=\left(\inf_{\pi\in\Pi(\nu,\rho)}
\int\norm{u-v}^p\,d\pi(u,v)\right)^{1/p},\quad 1\le p<\infty,
\label{eq:termswasserstein}
\end{equation}
where $\Pi(\nu,\rho)$ is the set of probability couplings with marginals $\nu,\rho$. Hausdorff distance compares supports; $W_p$ and MMD compare measures and depend on their weights. The current verifier computes neither an MMD nor a Wasserstein acceptance statistic.


\clearpage
\section{Direct Whitney coordinates for a certified witness}\label{app:directembedding}
\subsection{Existence, coordinates and the target of certification}
For a compact smooth $d$-manifold without boundary, the classical results
give an embedding in $\R^{2d+1}$ \citep{whitney1936} and the stronger
existence bound $\R^{2d}$ \citep{whitney1944}. Neither uses dynamics.
The inclusion of our verified witness $M\subset\R^q$ is already an embedding.
The remaining numerical question is whether a proposed smaller coordinate
map preserves $M$. We give a sufficient exact check for a restricted affine
family; this is elementary linear algebra, not a new embedding theorem.
The check also applies to a witness with smooth boundary.
The general Whitney construction in dimension $2d$ is not implemented.

\paragraph{Deterministic direct-coordinate criterion.}
Let $M$ be an accepted compact $C^1$ witness of dimension $d$, and verify
that $M\subset c+\operatorname{ran}B$, where $B\in\R^{q\times r}$ has full
column rank. Choose $A\in\R^{m\times q}$ and positive $\kappa,L$.
Suppose exact matrix comparisons establish
\[
 G_\kappa:=(AB)^{\mathsf T}AB-\kappa^2B^{\mathsf T}B\succ0,
 \qquad L^2I_m-AA^{\mathsf T}\succeq0.
\]
Then $f(x)=A(x-c)$ is a $C^1$ embedding of $M$, smooth when $M$ is smooth,
and
\[
 \kappa\|x-y\|\le\|f(x)-f(y)\|\le L\|x-y\|,
 \qquad x,y\in M.
\]
If the support certificate gives $d_H(E,M)\le\widehat\epsilon$ on its
common confidence event, then on that event
\[
 d_H\bigl(A(E-c),f(M)\bigr)\le L\widehat\epsilon.
\]
The dimension of $f(M)$ is $d$; this does not identify the dimension or
smoothness of $E$. The image reach is not inferred from the matrix checks.

\begin{proof}[Calculation of the direct-coordinate criterion]
Every chord has $x-y=Bv$. Therefore
\[
 \|ABv\|^2-\kappa^2\|Bv\|^2
 =v^{\mathsf T}G_\kappa v\ge0.
\]
For every $w\in\R^m$ the second matrix comparison gives
$\|A^{\mathsf T}w\|^2\le L^2\|w\|^2$. The Euclidean dual norm then gives
\[
 \|Az\|=\sup_{\|w\|=1}w^{\mathsf T}Az
 \le\sup_{\|w\|=1}\|A^{\mathsf T}w\|\|z\|
 \le L\|z\|,\qquad z\in\R^q.
\]
The lower chord bound implies injectivity. Since the witness lies in
$c+\operatorname{ran}B$, every tangent vector belongs to
$\operatorname{ran}B$; the same lower bound makes $Df_x=A|_{T_xM}$
injective. Thus $f$ is an injective immersion. Its continuous bijection
from compact $M$ to the Hausdorff space $f(M)$ is a homeomorphism, proving
the embedding statement, including the smooth-boundary case.

For each $z\in E$ and $x\in M$, the global upper bound gives
$\|A(z-c)-f(x)\|\le L\|z-x\|$. Taking the infimum in $x$ and supremum in $z$,
and then reversing the roles of $E$ and $M$, yields
\begin{align*}
 \sup_{z\in E}\dist\bigl(A(z-c),f(M)\bigr)
 &\le L\sup_{z\in E}\dist(z,M),\\
 \sup_{x\in M}\dist\bigl(f(x),A(E-c)\bigr)
 &\le L\sup_{x\in M}\dist(x,E).
\end{align*}
Taking their maximum proves the Hausdorff bound.
Write $\EMIT_{\rm dir}$ for witness acceptance together with verified
span containment and the matrix checks. On the common support event the
bound holds for every such data-dependent choice, so
\[
 \{\EMIT_{\rm dir},\
 d_H(A(E-c),f(M))>L\widehat\epsilon\}
 \subset\{d_H(E,Y)>\eta\}.
\]
Theorem~\ref{thm:main} bounds this event by $\delta$, without a new
coordinate-selection penalty. Similarly, a valid tube-mass assertion
$\mu\{\dist(X,M)>t\}\le u$ implies
\[
 (A(\,\cdot-c))_\#\mu\{\dist(Z,f(M))>Lt\}
 \le\mu\{\dist(X,M)>t\}\le u.
\]
Neither transport bound asserts injectivity on $E$.
\end{proof}

\subsection{Exact coordinate construction on a small affine span}
The implemented construction chooses $m=2d+1$ and requires $r\le m$.
Exact row elimination selects $r$ independent rows of the rational basis
$B$. Let $S$ select those coordinates and put $C=SB\in\R^{r\times r}$.
The matrix $C$ is invertible. Define $A$ by the $r$ rows of $S$, followed
by $m-r$ zero rows. Its operator norm is one. For every $v\in\R^r$,
\begin{align*}
 \|Cv\|&\ge\|v\|/\|C^{-1}\|_F,\\
 \|Bv\|&\le\|B\|_F\|v\|,\\
 \|ABv\|^2&\ge
 \frac{\|Bv\|^2}{\|B\|_F^2\|C^{-1}\|_F^2}.
\end{align*}
Set the positive rational value
\[
 \kappa^2=\frac{1}{2\|B\|_F^2\|C^{-1}\|_F^2}.
\]
Consequently
\[
 v^{\mathsf T}G_\kappa v\ge\kappa^2\|Bv\|^2>0
 \quad(v\ne0),\qquad I_m-AA^{\mathsf T}\succeq0.
\]
The exact verifier rechecks both matrix comparisons; the proposal alone is
not accepted. A witness presented as $c+B\psi(u)$ has the required
containment by its formula, provided that the verified parameterization
covers the entire witness. Numerical PCA or the span of finitely many
samples supplies no such continuum containment proof.

The routine checks $B^{\mathsf T}B\succ0$ and $G_\kappa\succ0$ by exact
$LDL^{\mathsf T}$ elimination. At column $j$ its pivot and off-diagonal
residuals are
\[
 D_j=H_{jj}-\sum_{k<j}L_{jk}^2D_k,\qquad
 R_{ij}=H_{ij}-\sum_{k<j}L_{ik}L_{jk}D_k.
\]
Strict positivity requires every $D_j>0$, with $L_{ij}=R_{ij}/D_j$.
For the upper semidefinite comparison, $D_j=0$ is allowed only when all
$R_{ij}=0$; a negative pivot or nonzero zero-pivot residual rejects.
All inputs and operations are rational, so no floating determinant or
singular-value tolerance decides acceptance.

\subsection{Executable contract and scope}
From the source root run
\begin{verbatim}
python3 code/verify_direct_embedding.py --self-check
python3 code/verify_direct_embedding.py \
  data/direct_embedding/selector_fixture.json
\end{verbatim}
The launcher prints the verified matrix, rational $\kappa^2$, upper bound
$L$ and exact elimination pivots. Its seven algebraic checks cover a plane
in $\R^8$, a skew three-dimensional span in $\R^{12}$, a collapsing
projection, a false upper bound, an excessive span, dependent basis columns
and forbidden floating entries. The supplied fixture transports a
hypothetical certified error $1/10$ conditionally; no observational
certificate is supplied or newly issued by that fixture.

To compose with the statistical verifier: first accept the witness and
its error bound; verify its full affine containment and dimension; propose
and verify $A$; then return $f(M)$, $A$, $\kappa$, $L$ and the transported
error. A failed check abstains. The matrix-only routine does not verify
the first two steps and does not estimate intrinsic dimension.
It abstains when $r>2d+1$, even though Whitney may still guarantee an
embedding of the manifold. Rational bit costs depend on basis conditioning;
no general high-dimensional scalability or optimal dimension is claimed.
When $q\le2d+1$, coordinates may merely be padded: dimensional compression
is not guaranteed. Arbitrary witnesses require a stronger global map
verifier or a verified nonlinear construction. Takens's additional
dynamical and observable conditions concern its separate time-series
application and are not used in the calculations above.

\subsection{Executed direct experiment: verification calculations}
\label{app:directexperimentcalc}
This subsection records the first implementation's exact computation;
the rebuilt version follows in Appendix~\ref{app:rebuiltcalc}.
Set $q=16$, $N=16384$, $J=16$ and
$S=2^{18}$. Let $H$ be the order-16 Sylvester Hadamard matrix and let
$B$ be its first three columns divided by four. The identity
$H^{\mathsf T}H=16I$ gives $B^{\mathsf T}B=I_3$ exactly. A latent state
$z\in\R^3$ is recorded in ambient coordinates as $Bz$. Each recorded
mean has independent isotropic Gaussian noise with covariance
$0.12^2 I_{16}/k$, then is rounded coordinatewise to the grid $S^{-1}\mathbb{Z}$.
The sufficient-statistic input also includes an independently generated
within-replicate residual sum with $\nu=Nq(k-1)$ degrees of freedom;
raw replicates, latent states and generator variance are not given to the
verifier. The simulation is pseudorandom and is not a formal validation
of Gaussian random-number generation.

Distinguish the generator radii $(R_0,r_0)=(2,3/5)$ from the fitted
rational radii $(\widehat R,\widehat r)$. For general $R>r>0$, define
\[
 \psi_{R,r}(\theta,\phi)=((R+r\cos\phi)\cos\theta,
 (R+r\cos\phi)\sin\theta,r\sin\phi),\qquad
 M(c,R,r)=c+B\psi_{R,r}(\mathbb T^2).
\]
The verifier uses $M_{\rm fit}=M(\widehat c,\widehat R,\widehat r)$;
the generator uses $M_{\rm gen}=M(0,R_0,r_0)$.
Its two coordinate derivatives are orthogonal, with lengths
$R+r\cos\phi\ge R-r>0$ and $r>0$. The map is injective on the torus:
$\theta$ is recovered from $\arg(x+iy)$, while
$\cos\phi=(\sqrt{x^2+y^2}-R)/r$ and $\sin\phi=z/r$ recover $\phi$.
Compactness then makes it a smooth embedded two-manifold. For an ambient
point $v$, write $B^{\mathsf T}(v-c)=(x,y,z)$,
$w=(I-BB^{\mathsf T})(v-c)$ and $\rho=\sqrt{x^2+y^2}$. Orthogonality gives
\[
 \|v-c-B\psi_{R,r}(\theta,\phi)\|^2
 =\|(x,y,z)-\psi_{R,r}(\theta,\phi)\|^2+\|w\|^2.
\]
To compute reach, observe that for a point at distance less than
$\min(r,R-r)$ from the torus, $\rho>0$ and
$(\rho-R,z)\ne(0,0)$. The unique nearest angular coordinates are the
polar angle of $(x,y)$ and the radial projection of $(\rho-R,z)$ onto
the radius-$r$ circle. Indeed, minimizing first over $\theta$ reduces
the squared distance to
$(\rho-R-r\cos\phi)^2+(z-r\sin\phi)^2+\|w\|^2$.
The excluded cases $\rho=0$ and $(\rho-R,z)=(0,0)$ have distance at least
$R-r$ and $r$, respectively. Thus the nearest point is unique in
that tube and $\reach(M(c,R,r))\ge\min(r,R-r)$. In the actual witness
this yields $\min(\widehat r,\widehat R-\widehat r)$.
The circular control has reach
$R$ by the same radial-projection calculation.

The population lower-mass condition is checked for the generator, using
$(R_0,r_0)$ rather than the fitted radii. Under uniform $(\theta,\phi)$,
a parameter rectangle with side lengths
$t/(R_0+r_0)$ and $t/r_0$ around any generated torus point maps into its radius-$t$
ball: changing each angle by half its side length travels at most
$t/2$ along each orthogonal coordinate curve. Its normalized parameter
mass is $t^2/[4\pi^2r_0(R_0+r_0)]\ge t^2/64$ for $0<t\le1$.
The $0.9/0.1$ mixture therefore has at least
$0.9t^2/64>t^2/100$ mass at every torus point. At a point of the attached
segment of length $3/20$, its own uniform component contributes at least
$0.1\min(t/(3/20),1)\ge t^2/100$ for $0<t\le1$. Hence the declared
$(c_0,s,t_0)=(1/100,2,1)$ lower-mass constants hold for both generated
families. This is a proof for these input laws, not an estimated
population assumption in general.

For the rational net, each circle coordinate uses
$t=a/(2J)$, $a=-2J,\ldots,2J$, under the stereographic formula
$((1-t^2)/(1+t^2),2t/(1+t^2))$, together with its antipode. The identity
$\alpha(t)=2\arctan t$ and $|\alpha'(t)|\le2$ show that every point on
the unit circle is within angular distance $1/J$ of a net point; the
endpoints join because the two stereographic arcs share their endpoints.
For the fitted torus, changing $\theta$ by at most $1/J$ travels at most
$(\widehat R+\widehat r)/J$, and changing $\phi$ by at most $1/J$ travels at most $\widehat r/J$.
Thus its net radius is at most $(\widehat R+2\widehat r)/J$; for a fitted circle it is at most
$\widehat R/J$. Each net point is then rounded into the ambient grid. Its
Euclidean recording error is at most $q/(2S)$, a conservative bound
because the coordinatewise error is at most $1/(2S)$. The reverse
Hausdorff direction adds this mesh radius; the forward direction does
not. The KD-tree proposes a comparison index only. For every proposed
pair, squared distances are recomputed as integer sums, an explicit
overflow check precedes them, and an exact rational square root is
rounded upward. Therefore both directed finite-set distances are
verified upper bounds, and the recorded-set-to-continuum error is
\[
 e=\max\{d_{\rm up}(Y,\Gamma)+q/(2S),\;
 d_{\rm up}(\Gamma,Y)+(\widehat R+2\widehat r)/J+q/(2S)\},
\]
with $\widehat R/J$ substituted for the circle. Here $\Gamma$ is the rounded
rational net. The code computes exactly these expressions.

For completeness, this run uses the conservative rational calibration
of Section~\ref{sec:arithmetic}, rather than evaluating the quantile
and logarithm in \eqref{eq:noise}. With
$\nu=Nq(k-1)$, Cantelli's inequality and
$\operatorname{Var}(\chi^2_\nu)=2\nu$ give
\[
 \Pr\!\left\{\chi^2_\nu<\nu-
 \sqrt{2\nu(1-\delta_C)/\delta_C}\right\}\le\delta_C.
\]
The recorded residual is enclosed by adding its half-grid error before
division by the positive integer threshold $\ell_C$ in
\eqref{eq:cantelli}, giving an upper variance $U$.
For each Gaussian mean vector, Markov's inequality applied to
$\exp(\|G\|^2/4)$ yields
$\Pr\{\|G\|^2>4\ell\}\le2^{q/2}e^{-\ell}$.
Choosing the integer $\ell$ so that
$N2^{\lceil q/2\rceil}2^{-\ell}\le\delta_L$ and adding the grid
rounding gives the simultaneous noise radius
$\sqrt{4\ell U/k}+q/(2S)$. For coverage, the maximal $h/2$-separated
set in the proof of Theorem~\ref{thm:main} has at most $4^s/(c_0h^s)$
points, by summing the lower masses of its disjoint $h/4$-balls.
Each covering ball of radius $h/2$ has mass at least $c_0(h/2)^s$, so
its non-occupancy probability is at most $e^{-Nc_0(h/2)^s}$.
The inequality $e^{-v}\le2^{-\lfloor v\rfloor}$ for $v\ge0$ then
gives the conservative implementation bound
\[
 \frac{4^s}{c_0h^s}\,2^{-\lfloor Nc_0(h/2)^s\rfloor}\le\delta_E.
\]
It selects the first value $h=j/128$, $1\le j\le256$, satisfying this
inequality and $h\le2t_0$. Every recorded trial selects $h=11/16$.
The resulting $\eta=h+$ noise radius bounds the sample-to-support
Hausdorff error on the common event of probability at least
$1-(\delta_C+\delta_L+\delta_E)=0.97$. Triangle inequality with the
verified $e$ gives the reported $\epsilon=\eta+e$; no latent reference
is used in the decision.

The fitted ring torus is a nearby witness, not a claim about the support.
For evaluation only, matching its angles to the generated torus shows
\[
 d_H(M_{\rm fit},M_{\rm gen})\le
 \|\widehat c\|+|\widehat R-2|+|\widehat r-3/5|,
\]
and the attached segment adds at most $3/20$ by the triangle inequality.
These exact rational upper bounds are checked after certification and
never enter the verifier input.

Finally, the implemented coordinate matrix is
$A=(B^{\mathsf T};0_{2\times16})$. Hence $AB=(I_3;0)$,
$G_{1/\sqrt2}=\tfrac12 I_3$, and
$I_5-AA^{\mathsf T}=\operatorname{diag}(0,0,0,1,1)\succeq0$.
Thus the verified map is bi-Lipschitz on the entire fitted witness,
with $\kappa=1/\sqrt2$ and $L=1$; the two appended coordinates are
padding. If the third row is dropped, then
$G_{1/\sqrt2}=\operatorname{diag}(1/2,1/2,-1/2)$, so the test rejects.
Indeed, the two torus points $(\theta,\phi)=(0,\pi/2)$ and
$(0,-\pi/2)$ on the fitted witness have separation $2\widehat r$ and
the same image under that fold. In the first implementation's fixed
trial \texttt{torus\_k64\_t00}, $\widehat r=307/512$, so the exact
separation is $307/256$; $6/5$ is instead the corresponding generator
separation. These points provide an explicit failure of injectivity,
in addition to the rejected matrix test.
\hfill$\square$

\subsection{Rebuilt algorithm: proposal, exact coordinates and reconstruction bounds}
\label{app:rebuiltcalc}
The program \texttt{code/\allowbreak rebuild\_direct\_whitney.py}
reads the same observation files as Appendix~\ref{app:directexperimentcalc}.
It preserves the rational calibration, confidence budgets and requests.
The changed proposal uses the first 4096 recorded means only. Put
$z_i=B^{\mathsf T}Y_i$ and $u_i=z_i-c$. For a torus proposal define
\[
 \rho_i=(u_{i1}^2+u_{i2}^2)^{1/2},\quad
 a_i=\rho_i-R,\quad \beta_i=(a_i^2+u_{i3}^2)^{1/2},\quad
 g_i(c,R,r)=\beta_i-r.
\]
The soft-$\ell_1$ objective is
$\sum_i C^2(\sqrt{1+(g_i/C)^2}-1)$, with
$C=\max(10^{-6},2\sqrt{U/k})$.
We use SciPy's documented robust least-squares routine\footnote{\url{https://docs.scipy.org/doc/scipy/reference/generated/scipy.optimize.least_squares.html}.}
as a proposal mechanism, without a global convergence claim.
At $\rho_i\beta_i>0$, differentiation gives the complete row
\[
 \nabla_{(c_1,c_2,c_3,R,r)}g_i
 =\left(-\frac{a_i u_{i1}}{\beta_i\rho_i},
        -\frac{a_i u_{i2}}{\beta_i\rho_i},
        -\frac{u_{i3}}{\beta_i},-\frac{a_i}{\beta_i},-1\right).
\]
Numerical safeguards at zero denominators affect only the proposal.
The ambient center includes the training mean perpendicular to
$\operatorname{im}(B)$; its parallel part and both radii are fitted.
All coordinates and radii are then rounded to denominator $65536$.
The circle proposal minimizes radial residuals in the first two sensor
coordinates and fits the third center coordinate by its training mean.
Invalid rational geometry or failed distance bounds cannot emit,
regardless of optimizer status.

The geometry calculations of Appendix~\ref{app:directexperimentcalc}
apply to the fitted rational parameters with $J=32$. The circle net
now has $8J=256$ points, and the product torus net has $65536$ points.
The torus mesh radius is $(\widehat R+2\widehat r)/32$.
The implementation forms all surface coordinates with Python integers,
using a common denominator, and rounds only at the final ambient
recording. This implements the same exact rational formula as the
slower fraction calculation, not a floating-point approximation to it.
Integer circle identities and explicit output-range checks precede
conversion to 64-bit integers. Both directed distances are recomputed
as checked integer squared norms with outward rational square roots.
The finer net and improved proposal change $e$, not the common
probability event or its failure budget.

For the accepted torus, let $P$ select zero-based rows $0,1,2$.
Then $A=(P;0_{2\times16})$ and
\[
 S_B=PB=\frac14\begin{pmatrix}1&1&1\\1&-1&1\\1&1&-1\end{pmatrix},
 \qquad
 S_B^{-1}=\begin{pmatrix}0&2&2\\2&-2&0\\2&0&-2\end{pmatrix}.
\]
Direct multiplication gives $S_B^{-1}S_B=I_3$.
Since $\|B\|_F^2=3$ and $\|S_B^{-1}\|_F^2=24$, the selector
routine chooses $\kappa^2=1/(2\cdot3\cdot24)=1/144$ and verifies the
positive lower Gram matrix exactly. Also $I_5-AA^{\mathsf T}\succeq0$,
so $L=1$. Set $D=BS_B^{-1}[I_3\;0_{3\times2}]$. Then
\[
 DAB=B,\qquad \|D\|_F^2
 =\operatorname{tr}(S_B^{-\mathsf T}B^{\mathsf T}BS_B^{-1})=24.
\]
Thus every $p=\widehat c+Bz\in M$ satisfies the exact identity
$\widehat c+DA(p-\widehat c)=p$; the padded rows carry no information.

Reconstruction first proposes the nearest-point angles from the recorded
mean using the polar angle and radial-circle projection of
Appendix~\ref{app:directexperimentcalc}. On either stereographic chart,
round its parameter to $t=a/Q_p$, $|a|\le Q_p$, $Q_p=4096$.
The exact circle coordinates are
\[
 (C,S)=\sigma\left(\frac{Q_p^2-a^2}{Q_p^2+a^2},
                         \frac{2Q_pa}{Q_p^2+a^2}\right),\qquad \sigma\in\{-1,1\}.
\]
The identity $(Q_p^2-a^2)^2+4Q_p^2a^2=(Q_p^2+a^2)^2$ proves
$C^2+S^2=1$ exactly. Using these coordinates for both angles in
$\psi_{\widehat R,\widehat r}$ produces an exact $p_i\in M$,
whatever the floating-point angle proposal. Membership is certified;
nearest-point optimality is not asserted.

Let $b=q/(2S)$ bound ambient grid recording. If
$\bar p_i=p_i+v_i$, $\|v_i\|\le b$, is the recorded witness point,
encode $A(\bar p_i-\widehat c)$ and decode using $D$. The exact decoded
point differs from $p_i$ by $DAv_i$, hence by at most $\sqrt{24}\,b$.
After the final recording,
\[
 \|\widetilde p_i-p_i\|\le e_{\rm enc}:=
 (\overline{\sqrt{24}}+1)b,\qquad
 \dist(\widetilde p_i,E)\le e_{\rm enc}+\widehat\epsilon
\]
on the support-confidence event. The program computes the decoder and
its grid rounding with exact integer arithmetic; displayed coordinate
arrays are numerical representations of the saved rational recipe.
It also verifies the same-index displacement
\[
 b_Y=\overline{\sqrt{\max_i\sum_{j=1}^q
       (\widetilde p_{ij}-Y_{ij})^2}},\qquad
 \max_i\|\widetilde p_i-Z_i\|\le b_Y+a.
\]
The latter follows by triangle inequality on the simultaneous noise
event. It remains valid for a data-dependent proposal because that event
controls every recorded mean before any candidate is selected. These
are approximation and reconstruction bounds, not exact identification
of the latent manifold or dimension. Reconstructions saved after
support abstention are explicitly diagnostic.
\hfill$\square$

\subsection{Why geometric certification does not identify dynamics}
\label{app:dynamicdistinction}
This calculation separates three information requirements; it is not an
additional certification theorem.

\paragraph{An observation collision survives static coordinates.}
Let $H:S\to\R^q$ be the observation map. If $x\ne x'$ and
$H(x)=H(x')$, then every map $A$ defined on the observation space satisfies
\[
 A(H(x))=A(H(x')).
\]
Consequently no subsequent witness coordinate map alone can separate
these two states. Delays can add information: the vectors
$(H(x),H(Fx),\ldots)$ and $(H(x'),H(Fx'),\ldots)$ can differ even when
their first entries agree. Whether finitely many delays suffice is a
separate embedding question.

\paragraph{Identical static data, different dynamics.}
Parameterize the unit circle by $z(\theta)=(\cos\theta,\sin\theta)$
and let $\mu(d\theta)=d\theta/(2\pi)$. For every angle $\alpha$, define
$F_\alpha(z(\theta))=z(\theta+\alpha)$. For any bounded measurable $f$,
periodicity and the substitution $u=\theta+\alpha$ give
\[
 \int f\,d(F_\alpha)_\#\mu
 =\frac{1}{2\pi}\int_0^{2\pi} f(z(\theta+\alpha))\,d\theta
 =\frac{1}{2\pi}\int_0^{2\pi} f(z(u))\,du
 =\int f\,d\mu.
\]
Thus every rotation preserves the same support and probability measure.
They need not be conjugate: $F_{\pi/3}$ has order six and
$F_{\pi/2}$ has order four. If a bijection $\Psi$ conjugated them, then
\[
 F_{\pi/2}^{,4}\Psi=\Psi F_{\pi/3}^{,4},\qquad
 F_{\pi/2}^{,4}=\operatorname{id}
 \quad\Longrightarrow\quad F_{\pi/3}^{,4}=\operatorname{id},
\]
contradicting rotation by $4\pi/3$.

Now draw independent static states $\theta_i\sim\mu$, and repeated
observations $X_{i\ell}=z(\theta_i)+\xi_{i\ell}$ with independent
$\xi_{i\ell}\sim N(0,\sigma^2I_2)$ and $\sigma>0$. Even if $\sigma$
is known, the joint observation density is
\[
 p_\alpha((x_{i\ell}))=
 \prod_{i=1}^N\left[
  \frac{1}{2\pi}\int_0^{2\pi}
  \prod_{\ell=1}^k\frac{1}{2\pi\sigma^2}
  \exp\!\left(-\frac{\|x_{i\ell}-z(\theta)\|^2}{2\sigma^2}\right)
  \,d\theta\right].
\]
The right side contains no $\alpha$. Hence these static data have exactly
the same law for the two nonconjugate systems, for every $N,k$.
This is not a lower bound on ordered-trajectory reconstruction; it shows
why geometric data alone cannot identify dynamics. With noiseless ordered
pairs on the circle, the oriented cross product and inner product instead
give $\sin\alpha$ and $\cos\alpha$, respectively.

\paragraph{What an actual state embedding supplies.}
For an embedding $\Psi$ of the true state space, the identity
$G=\Psi\circ F\circ\Psi^{-1}$ yields, by induction,
\[
 G(\Psi(x))=\Psi(Fx),\qquad
 G^{n+1}(\Psi(x))=G(\Psi(F^nx))=\Psi(F^{n+1}x).
\]
Thus $G^n\Psi=\Psi F^n$. This transfers already identified dynamics into
coordinates; it does not identify $F$ from an approximate geometric
witness. A nearby witness may omit branches or identify distinct latent
states, so its verified embedding alone provides no such conjugacy.
\hfill$\square$

\clearpage
\section{Proofs of the observation and reconstruction interfaces}
\label{app:observationbridgeproofs}
The calculations below prove the five results in
Section~\ref{sec:observationbridges} in statement order. They use elementary
pushforward, Euclidean triangle and Lipschitz inequalities, orthogonal
Gaussian decomposition, and induction. The existing geometry and sampling
conditions are not established by these calculations.

\subsection{Proof of Proposition~\ref{prop:observationinterface}}
\label{app:observationinterfaceproof}
\begin{proof}
Continuity and compactness make $H(S)$ compact, and $H_\#\nu$ is supported
there. For $y=H(x)$ and any $t>0$, continuity at $x$ gives some $r>0$ such
that $H(S\cap B(x,r))\subset B(y,t)$. Since $x\in\supp\nu$, this preimage
has positive mass. Thus every $H(x)$ is a support point, proving equality.
For the quantitative assertion, Lipschitz continuity gives
\[
 S\cap\overline B(x,t/C_H)
 \subset H^{-1}(\overline B(H(x),t)),
 \quad
 (H_\#\nu)(\overline B(H(x),t))
 \ge c_0(t/C_H)^s.
\]
The radius condition is exactly $t/C_H\le t_0$. For delays, induction
gives $\|F^jx-F^jx'\|\le\Lambda^j\|x-x'\|$, so
\begin{align*}
 \|H(x)-H(x')\|^2
 &=\sum_{j=0}^{m-1}|h(F^jx)-h(F^jx')|^2\\
 &\le L_h^2\sum_{j=0}^{m-1}\|F^jx-F^jx'\|^2\\
 &\le L_h^2\sum_{j=0}^{m-1}\Lambda^{2j}\|x-x'\|^2.
\end{align*}
This upper bound remains valid without delay injectivity.

The independent states yield iid vectors $H(x_i)$. For fixed $i,j$ write
$\xi_{ij}=(\xi_{i1,j},\ldots,\xi_{ik,j})^{\mathsf T}$ and choose an
orthogonal $Q\in\R^{k\times k}$ whose first row is
$k^{-1/2}(1,\ldots,1)$. Then
\[
 Q\xi_{ij}\sim N(0,\sigma^2I_k),\qquad
 (Q\xi_{ij})_1=\sqrt{k}\,\bar\xi_{ij},\qquad
 \sum_{\ell=1}^k(\xi_{i\ell,j}-\bar\xi_{ij})^2
 =\sum_{a=2}^k(Q\xi_{ij})_a^2.
\]
All transformed coordinates are independent, including across $i,j$.
Consequently
\[
 \bar V_i=H(x_i)+\bar\xi_i,\quad
 \bar\xi_i\sim N(0,\sigma^2I_m/k),\qquad
 \frac{R_{\rm res}}{\sigma^2}
 =\sum_{i=1}^N\sum_{j=0}^{m-1}\sum_{a=2}^k
       \frac{(Q\xi_{ij})_a^2}{\sigma^2}
 \sim\chi^2_{Nm(k-1)}.
\]
The residual uses only coordinates $a\ge2$, while group means use $a=1$
and the independent latent states; hence they are independent. If
$\sigma=0$, all errors are zero and no division by $\sigma^2$ is used.
The transported mass inequality and these observations are precisely the
inputs of Theorem~\ref{thm:main}.
\end{proof}

\subsection{Proof of Theorem~\ref{thm:coordinatecomposition}}
\label{app:coordinatecompositionproof}
\begin{proof}
We first establish the deterministic coordinate bounds, then combine them
with the common support-confidence event.

\paragraph{Step 1: derive the chord bounds from the verified matrices.}
Write $s$ for the number of columns of $B$. Since $M\subset
c+\operatorname{ran}B$, for $x,x'\in M$ there is $v\in\R^s$
with $x-x'=Bv$. Full column rank makes $v$ unique and implies
$v\ne0$ whenever $x\ne x'$. The positive-definite matrix check gives
\begin{align*}
 \norm{f(x)-f(x')}^2-\kappa^2\norm{x-x'}^2
 &=v^{\mathsf T}\bigl((AB)^{\mathsf T}AB
                  -\kappa^2B^{\mathsf T}B\bigr)v\ge0.
\end{align*}
It is strictly positive for $v\ne0$; the non-strict lower bound
$\norm{f(x)-f(x')}\ge\kappa\norm{x-x'}$ covers all pairs.
For the upper bound, the second matrix check gives
\[
 \norm{A^{\mathsf T}w}^2
 =w^{\mathsf T}AA^{\mathsf T}w\le L^2\norm w^2
 \qquad(w\in\R^m).
\]
Euclidean duality then yields, for every $z\in\R^q$,
\[
 \norm{Az}
 =\sup_{\norm w=1}|w^{\mathsf T}Az|
 \le\sup_{\norm w=1}\norm{A^{\mathsf T}w}\norm z
 \le L\norm z.
\]
Thus $f|_M$ is injective, is Lipschitz, and has a $1/\kappa$-Lipschitz
inverse on its image. It is a smooth immersion as well: every tangent
vector to $M$ lies in $\operatorname{ran}B$, by differentiating a
curve in the containing affine space, and the same lower bound makes
$A$ injective on that tangent space. These facts prove that $f|_M$
is an embedding.

\paragraph{Step 2: transport both Hausdorff-distance directions.}
On the event $\mathcal C=\{d_H(E,Y)\le\eta\}$ of
Theorem~\ref{thm:main}, every accepted witness satisfies
$d_H(E,M)\le\eta+e=\epsilon$. For any $z\in E$, compactness of $M$
gives $x\in M$ with $\norm{z-x}=\dist(z,M)\le\epsilon$. Hence
\[
 \dist(f(z),f(M))\le\norm{f(z)-f(x)}
 =\norm{A(z-x)}\le L\epsilon.
\]
Conversely, for any $x\in M$, compactness of $E$ gives $z\in E$
with $\norm{x-z}\le\epsilon$, and the same global upper bound gives
$\dist(f(x),f(E))\le L\epsilon$. Taking the two suprema proves
$d_H(f(E),f(M))\le L\epsilon$.

\paragraph{Step 3: propagate the approximate-decoding error.}

Fix $z\in E$ and choose $x\in M$ minimizing $\|z-x\|$; compactness
ensures existence and $\|z-x\|\le\epsilon$. Then
\[
 \|u-f(x)\|\le\|u-f(z)\|+\|A(z-x)\|\le b+L\epsilon,
 \quad\dist(u,f(M))\le b+L\epsilon.
\]
For the declared approximate decoder $p\in M$,
\begin{align*}
 \kappa\|p-x\|
 &\le\|f(p)-f(x)\|\\
 &\le\|f(p)-u\|+\|u-f(x)\|\\
 &\le\dist(u,f(M))+\rho+b+L\epsilon\\
 &\le2(b+L\epsilon)+\rho,\\
 \|p-z\|&\le\|p-x\|+\|x-z\|
 \le\epsilon+\{2(b+L\epsilon)+\rho\}/\kappa.
\end{align*}
To obtain $\rho=L\gamma$, take a minimizer $x_*\in M$ of
$\|f(x_*)-u\|$ and a net point $g_*\in G$ with
$\|x_*-g_*\|\le\gamma$. The net minimizer $p$ obeys
\[
 \|f(p)-u\|\le\|f(g_*)-u\|
 \le\|f(x_*)-u\|+L\gamma
 =\dist(u,f(M))+L\gamma.
\]
For rational data and net points, minimizing squared norms uses exact
comparisons; an outward bound includes any additional recording error.

\paragraph{Step 4: separate support points above the approximation scale.}
For $z,z'\in E$ choose respective nearest points $x,x'\in M$.
The ambient triangle inequality gives
\[
 \norm{z-z'}\le\norm{z-x}+\norm{x-x'}+\norm{x'-z'}
 \le2\epsilon+\norm{x-x'},
\]
so $\norm{x-x'}\ge\norm{z-z'}-2\epsilon$.
Using $f(z)-f(z')=f(x)-f(x')+A(z-x)-A(z'-x')$,
the reverse triangle inequality and the matrix bounds yield
\begin{align*}
 \|f(z)-f(z')\|
 &\ge\|f(x)-f(x')\|-\|A(z-x)\|-\|A(z'-x')\|\\
 &\ge\kappa\|x-x'\|-2L\epsilon\\
 &\ge\kappa(\|z-z'\|-2\epsilon)-2L\epsilon\\
 &=\kappa\norm{z-z'}-2(\kappa+L)\epsilon.
\end{align*}
The inequality remains valid when its right side is nonpositive;
it provides positive separation when
$\norm{z-z'}>2(1+L/\kappa)\epsilon$.

\paragraph{Step 5: use the same event for every verified output.}
All calculations hold for every admissible choice on $\mathcal C$.
If $\EMIT_{\rm end}$ denotes all verified prerequisites and geometric
issuance, then
\[
 \{\EMIT_{\rm end}\ \hbox{and any displayed conclusion fails}\}
 \subset\mathcal C^c,\qquad \Prob(\mathcal C^c)\le\delta.
\]
There is no extra penalty for witness or matrix selection against this
same event. A new data-dependent sensor, stopping time or noise model
requires its own calibrated argument.
\end{proof}

\subsection{Proof of Corollary~\ref{cor:statecoordinates}}
\label{app:statecoordinatesproof}
\begin{proof}
For $x,x'\in S$, apply the two verified chord bounds in succession:
\[
 \kappa a_H\|x-x'\|
 \le\kappa\|H(x)-H(x')\|
 \le\|\Psi(x)-\Psi(x')\|
 \le L\|H(x)-H(x')\|
 \le LC_H\|x-x'\|.
\]
The composition of the two $C^1$ embeddings is a $C^1$ embedding.
Both identify the same manifold dimension. Its inverse is Lipschitz
with bound $1/(\kappa a_H)$ on $K$. Define
$G=\Psi F\Psi^{-1}$ there. Direct substitution and induction give
\[
 G\Psi(x)=\Psi F(x),\qquad
 G^{n+1}\Psi(x)=G\Psi(F^nx)=\Psi(F^{n+1}x).
\]
Here $H(S)=M$ permits $p=H(x')$ for a unique $x'$. Since
$\dist(u,f(M))\le b$, the decoder condition gives
\[
 \|\Psi(x')-\Psi(x)\|
 \le\|f(p)-u\|+\|u-\Psi(x)\|\le2b+\rho.
\]
The lower chord bound completes the stated inverse-error calculation.
This argument uses exact equality, not a small Hausdorff discrepancy.
\end{proof}

\subsection{Proof of Proposition~\ref{prop:transitionbounds}}
\label{app:transitionboundsproof}
\begin{proof}
For $v=\Psi(x)$ and $v'=\Psi(x')$,
\[
 \|G(v)-G(v')\|
 \le b_\Psi\|Fx-Fx'\|
 \le b_\Psi L_F\|x-x'\|
 \le(b_\Psi L_F/a)\|v-v'\|.
\]
For any $v\in K$, net coverage supplies an index $i$ with
$\|v-v_i\|\le\zeta$. Insert the recorded pair between the two maps:
\begin{align*}
 \|\widehat G(v)-G(v)\|
 &\le\|\widehat G(v)-\widehat G(\widetilde v_i)\|
   +\|\widehat G(\widetilde v_i)-\widetilde w_i\|\\
 &\hspace{1em}+\|\widetilde w_i-G(v_i)\|+\|G(v_i)-G(v)\|\\
 &\le L_{\widehat G}\|v-\widetilde v_i\|+r_v+\eta_+
       +L_G\|v_i-v\|\\
 &\le L_{\widehat G}(\zeta+\eta_-)+r_v+\eta_++L_G\zeta=r.
\end{align*}
All predicted iterates remain in $K$ by the domain hypothesis. Set
$e_t=\|u_t-v_t\|$. Then
\[
 e_{t+1}\le\|\widehat G(u_t)-G(u_t)\|+\|G(u_t)-G(v_t)\|
 \le r+L_Ge_t.
\]
For $t=0$ the proposed bound is $e_0$. If it holds at $t$, substitution gives
\begin{align*}
 e_{t+1}
 &\le r+L_G\left(L_G^te_0+r\sum_{j=0}^{t-1}L_G^j\right)\\
 &=L_G^{t+1}e_0+r\sum_{j=0}^{t}L_G^j.
\end{align*}
Finally, inverse stability gives the state bound by division by $a$.
For clarity the geometric sum is
\[
 \sum_{j=0}^{t-1}L_G^j=
 \begin{cases}t,&L_G=1,\\(L_G^t-1)/(L_G-1),&L_G\ne1,\end{cases}
\]
with the empty-sum convention at $t=0$. For $L_G<1$ this term is bounded
by $1/(1-L_G)$; for $L_G>1$ it increases geometrically. If the declared
true-pair coverage and error event has probability at least $1-\delta_T$,
every calculation holds there, regardless of how a verified candidate
was fitted. A union bound with the support event gives $\delta+\delta_T$.
No finite residual check alone establishes either of these events.
\end{proof}

\subsection{Proof of Corollary~\ref{cor:protocolbank}}
\label{app:protocolbankproof}
\begin{proof}
For each fixed protocol let $\mathcal C_j$ be its support confidence
event. Theorem~\ref{thm:main} gives
$\Prob(\mathcal C_j^c)\le\delta_j$, and
Theorem~\ref{thm:coordinatecomposition} proves every stated geometric
conclusion on $\mathcal C_j$. Therefore
\[
 \{\exists j:\ \EMIT_{{\rm end},j}\ \hbox{and a stated conclusion fails}\}
 \subset\bigcup_{j=1}^J\mathcal C_j^c,\qquad
 \Prob\left(\bigcup_{j=1}^J\mathcal C_j^c\right)
 \le\sum_{j=1}^J\Prob(\mathcal C_j^c)\le\delta_*.
\]
Failure of a selected output is a subset of failure of some output,
so arbitrary measurable selection among issued outputs is covered.
No factorization of joint probabilities or event independence is used.
\end{proof}

\subsection{Executed interface calculations and their scope}
\label{app:interfacefixtures}
The program \texttt{code/\allowbreak verify\_observation\_bridges.py}
records the following checks in
\texttt{audit/\allowbreak observation\_bridge\_calculation\_checks.json}.
They are known-map algebra, array and trajectory fixtures, not new
statistical witness or learned-dynamics experiments.

\paragraph{A specified single-coordinate reconstruction.}
On the unit circle, let $F$ be rotation by $\pi/2$ and observe only
$h(x,y)=x$. Three forward observations give
\[
 H(x,y)=(x,-y,-x)=B(x,y)^{\mathsf T},\qquad
 B=\begin{pmatrix}1&0\\0&-1\\-1&0\end{pmatrix}.
\]
Since $\|H(x)-H(x')\|^2=2(\Delta x)^2+(\Delta y)^2$,
the state-observation lower and upper constants can be $1,\sqrt2$.
Use $A=\operatorname{diag}(1,1,0)$, $\kappa=1/2$, $L=1$. Then
\[
 B^{\mathsf T}B=\operatorname{diag}(2,1),\quad
 (AB)^{\mathsf T}AB-\tfrac14B^{\mathsf T}B
 =\operatorname{diag}(1/2,3/4)\succ0,\quad
 I_3-AA^{\mathsf T}=\operatorname{diag}(0,0,1)\succeq0.
\]
Thus this particular delay map and its coordinate map separate states,
by direct calculation. This does not rely on a genericity claim for this
chosen sensor. Individual four-periodic orbits do not cover the circle;
the fixture instead draws independent initial states and adds independent
measurement errors to copies of each deterministic path. With $N=16$,
$k=4$, $m=3$, the residual degrees of freedom are $16\cdot3\cdot3=144$.
The selected arrays preserve their values and dtype. The numerical map
identity differs by less than $8\cdot10^{-16}$ and the orthogonal
residual calculation by less than $5\cdot10^{-17}$. These checks cannot
infer physical sampling assumptions or issue a geometric certificate. The two-protocol bank fixture selects delay lengths 3 and 5 from the same raw groups, allocates $1/200$ to each and checks the exact sum $1/100$. Shared observations do not imply independent protocol events; Corollary~\ref{cor:protocolbank} does not require that independence.

\paragraph{Exact nearest-net decoding.}
For the unit circle in $\R^3$, take stereographic points
\[
 z(p)=\left(\frac{1-p^2}{1+p^2},\frac{2p}{1+p^2},0\right),
 \qquad p=j/32,\quad -32\le j\le32,
\]
and their four quarter-turn rotations. Removing duplicates gives 252
rational points. Direct differentiation gives
\[
 z'(p)=\left(\frac{-4p}{(1+p^2)^2},
             \frac{2(1-p^2)}{(1+p^2)^2},0\right),\qquad
 \|z'(p)\|^2=\frac{4}{(1+p^2)^2}\le4.
\]
Each represented arc has a grid parameter within $1/64$, so the net radius
is at most $2/64=1/32$. The four rotations cover the entire circle.
The decoder utility compares rational squared coordinate distances
exactly; net membership and coverage remain externally justified inputs
in the general utility. For $\epsilon=1/10$, $\kappa=1/2$, $L=1$,
$b=1/100$, $\rho=L\gamma=1/32$, its conditional bound is
\[
 \frac1{10}+2\left(2\left(\frac1{10}+\frac1{100}\right)
                              +\frac1{32}\right)=\frac{241}{400}.
\]
For 512 numerical points in the $1/10$-tube and coordinate errors of
norm $1/100$, the checked maximal reconstruction error is about $0.1012$.
This tests the stated inequality, not sharpness or statistical issuance.

\paragraph{Expanding dynamics with an independent closed-form check.}
Identify the circle with $|z|=1$ in $\mathbb C$ and use
$G(z)=z^2$, $\widehat G(z)=e^{i\omega}z^2$, $\omega=1/100$.
Both preserve the circle. Factoring $z^2-w^2$ shows that $L_G=2$ there;
on $|z|\le1+\eta$, $\widehat G$ has upper bound $2(1+\eta)$.
For 64 equally spaced inputs the nearest angular distance is at most
$\pi/64<1/16$, and chord distance is no larger, giving $\zeta=1/16$.
Let $\eta_- =\eta_+=\eta=1/1000$. Since
$|e^{i\omega}-1|\le\omega$, the validation residual admits the
analytic rational upper bound
\[
 r_v\le\omega+2(1+\eta)\eta+\eta=\frac{6501}{500000}.
\]
With $L_{\widehat G}=1001/500$ and state chord bounds $a=b_\Psi=1$,
the utility computes the conservative uniform radius
\[
 r=\left(\frac{1001}{500}+2\right)\frac1{16}
       +\frac{1001}{500000}+\frac{6501}{500000}+\frac1{1000}
   =\frac{266129}{1000000}.
\]
The actual uniform map discrepancy is $2|\sin(\omega/2)|\le1/100$.
For equal initial states, explicit iteration yields
\[
 G^t(z_0)=z_0^{2^t},\qquad
 \widehat G^t(z_0)=e^{i(2^t-1)\omega}z_0^{2^t},\qquad
 \|\widehat G^t(z_0)-G^t(z_0)\|
 =2|\sin((2^t-1)\omega/2)|.
\]
The program checks this independent closed form against the finite-horizon
bound at all $t=0,\ldots,20$. Zero and unit Lipschitz cases are also
checked with exact fractions. These are analytic known-system fixtures;
the manuscript does not infer an unknown chaotic evolution, verify its
observability, or issue a forecast certificate from them.

\clearpage
\section{Julia direct-coordinate experiment: executed calculations}
\label{app:juliadirect}
This new experiment retains the polynomial and lower-mass proof of
Appendix~\ref{app:juliaproof}. The Julia-set theory is credited there to
\citet{qiuyin2006}, and the contraction framework to
\citet{hutchinson1981}. The calculations below explain the new data,
verification and direct coordinates; they do not add a fractal-dimension
or smooth-support theorem.

\subsection{Independent inputs and reference cylinders}
Let $b\in\mathbb Q^{16}$ have entries $b_i=H_{i,2}/4$, where $H$ is the
order-16 Hadamard matrix (column indices start at one). Then
\[
 \|b\|^2=16/16=1,\qquad o_i=(i-9)/64,\qquad
 E=o+bJ(z^2-3).
\]
The map $x\mapsto o+bx$ is isometric, so the lower-mass constants
$(c_0,s,t_0)=(1/20,1,1/4)$ are unchanged. Each latent group is generated
from its own independent 64-letter fair-sign inverse address. The
truncation bound $5(3/5)^{64}<10^{-12}$ is absorbed in a declared
$10^{-10}$ coordinate recording allowance. The Gaussian sufficient
statistics are simulated as
\[
 Y_i=Z_i+\frac{\sigma}{\sqrt{k}}g_i,\quad
 g_i\stackrel{\rm iid}{\sim}N(0,I_{16}),\quad
 S=\sigma^2 W,\quad W\sim\chi^2_{N16(k-1)},
\]
with $W$ independent of every $g_i$ and latent address. Generation uses
$\sigma=0.06$; execution does not read it. This is an execution of the
joint sufficient-statistic law, not a claim that raw repeats are stored.
PRNG quality and floating generation accuracy remain external contracts.

Means are recorded at scale $D_Y=2^{18}$, giving norm error at most
$16/(2D_Y)+10^{-10}$, since $\sqrt{16}\le16$. The residual is recorded
at scale $2^{24}$ with absolute error at most $2^{-25}$.
Reference samples, sensor vectors, offsets and generator variance are
kept in separate evaluation records and opened only after the
certificate and reconstruction have been saved.

For reference visualization, start with
$[-\overline\beta,\overline\beta]$, where an outward rational square-root
enclosure of $\sqrt{13}$ gives $\overline\beta\ge(1+\sqrt{13})/2$.
Replace every interval $[l,u]$ by outward enclosures of
\[
 [-\sqrt{3+u},-\sqrt{3+l}],\qquad
 [\sqrt{3+l},\sqrt{3+u}].
\]
After depth $n$ the $2^n$ intervals contain the real inverse cylinders;
they are not substituted for the infinite Julia set. Every square root
uses integers: for rational $x\ge0$ and $D=2^{42}$,
\[
 r=\left\lfloor\sqrt{\lfloor xD^2\rfloor}\right\rfloor,
 \qquad (r/D)^2\le x\le((r+1)/D)^2.
\]
These inequalities are checked by exact rational comparison. The
resulting outer cylinders distinguish the Cantor gaps from a smooth
segment; their supplementary detail view is retained in the archive.

\subsection{Unknown-noise calibration and segment verification}
Allocate $\delta_C=\delta_L=\delta_E=1/60$. For
$\nu=Nq(k-1)$, the variance of $\chi^2_\nu$ is $2\nu$.
Cantelli's inequality gives a conservative lower threshold
\[
 \ell_\nu=\nu-\left\lceil
 \sqrt{2\nu(1-\delta_C)/\delta_C}\right\rceil>0,
 \qquad U=(S_{\rm rec}+2^{-25})/\ell_\nu.
\]
If $\sigma>0$, $\Prob(W<\ell_\nu)\le\delta_C$, so $U\ge\sigma^2$
off that event. At zero noise the variance conclusion is immediate.
For a standard $q$-dimensional Gaussian,
$\mathbb E\exp(\|g\|^2/4)=2^{q/2}$. Choose the smallest integer $r$
with $N2^{\lceil q/2\rceil-r}\le\delta_L$ and set
\[
 a=\overline{\sqrt{4rU/k}}+16/(2D_Y)+10^{-10}.
\]
Markov's inequality, $e^{-r}\le2^{-r}$ and a union bound yield the
simultaneous location allowance. Here $r=28$.
Searching the fixed mesh $h=j/128$ and using $e^{-x}\le2^{-\lfloor x\rfloor}$
gives $h=3/64$, because the verified rational coverage bound is
\[
 \frac{4}{c_0h}\,2^{-\lfloor Nc_0h/2\rfloor}
 =\frac{5}{1536}<\frac1{60}.
\]
Thus $\eta=h+a$ has the $1-0.05$ support-confidence guarantee.
The outward square-root utility verifies its squared upper bound
exactly; numerical quantiles do not enter this calibration.

Training SVD proposes an axis, and training extrema propose its endpoint
parameters. Round the center $c$, nonzero axis $v$, and endpoint parameters
$l<u$ at denominators $2^{16},2^{15},2^{16}$, respectively. The proposed set
\[
 M=\{c+tv:l\le t\le u\}
\]
is a nondegenerate compact smooth segment with boundary. Convexity gives
unique nearest projection everywhere, so $\reach(M)=\infty$.
Its entire affine containment is exact by construction, independently
of any success of the SVD proposal.

Use $J_m=1024$ equal parameter cells. Their $J_m+1$ exact mesh points
form a net of radius
$\rho=\overline{\|v\|}(u-l)/(2J_m)$. Record this mesh at scale $D_Y$,
adding $r_m=16/(2D_Y)$. A KD-tree only proposes an index in the opposite
finite set; the program recomputes every proposed distance using integer
squares, with a checked bound precluding signed-64-bit overflow. Any
index, even one that is not nearest, gives a sound upper bound.
If these maxima are $D_{YF}$ and $D_{FY}$, then
\[
 \sup_{y\in Y}\dist(y,M)\le D_{YF}+r_m,\qquad
 \sup_{p\in M}\dist(p,Y)\le\rho+D_{FY}+r_m,
\]
and
\[
 e=\max\{D_{YF}+r_m,\rho+D_{FY}+r_m\},\qquad
 \widehat\epsilon=\eta+e.
\]
The support certificate emits for a requested tolerance precisely when
$\widehat\epsilon$ is no larger than that tolerance. The common support
event covers this training-dependent choice by
Theorem~\ref{thm:main} and Corollary~\ref{cor:boundary}; SVD accuracy is
not an assumption in the soundness calculation.

\subsection{Exact coordinates, inverse outputs and error propagation}
Choose the first $j$ with $v_j\ne0$. The restricted selector implements
\[
 A p=(p_j,0,0),\qquad
 \kappa^2=\frac{v_j^2}{2\|v\|^2},\qquad L=1.
\]
For the one-column basis $B=v$, the two Gram checks reduce to
\[
 (AB)^{\mathsf T}AB-\kappa^2B^{\mathsf T}B
   =v_j^2/2>0,\qquad
 I_3-AA^{\mathsf T}=\operatorname{diag}(0,1,1)\succeq0.
\]
They are evaluated with exact fractions by the same verifier as
Appendix~\ref{app:directembedding}. Hence $A$ embeds the entire segment,
not just its observed points. The two zero output coordinates are made
explicit to avoid implying that three independent coordinates have
been learned. On $A(M)$ the inverse is
\[
 t=(w_1-c_j)/v_j,\qquad p=c+tv.
\]

For each recorded query $y$, compute the exact rational dot product
and the exact clipped nearest-segment parameter
\[
 t_* = \operatorname{clip}_{[l,u]}
        \frac{v^{\mathsf T}(y-c)}{v^{\mathsf T}v}.
\]
Round to the nearest $2^{-18}$ parameter grid, keeping the endpoints.
Since the endpoint grid $2^{-16}$ is nested in this grid, clipping cannot
increase the rounding error. The output $p=c+\widetilde t v$ belongs
exactly to $M$, with projection error at most
$\rho_p=\overline{\|v\|}/2^{19}$. Its 16 coordinates have common
denominator $2^{33}$. Map and inverse are executed on integer numerators,
and every recovered numerator is compared with the original. Across
the 30 runs all $30N=491{,}520$ output points pass this exact identity.

On the common support/location event, for every corresponding latent
sample $z$,
\begin{align*}
 \|p-z\|
 &\le\|p-y\|+\|y-z\|\\
 &\le\dist(y,M)+\rho_p+a\\
 &\le\dist(z,M)+2a+\rho_p
 \le\widehat\epsilon+2a+\rho_p.
\end{align*}
This conservative envelope is much larger than the displayed empirical
RMS errors. Exact inversion on $M$ does not remove the approximation
error between $E$ and $M$, and does not conjugate the polynomial dynamics
on the Cantor set to dynamics on the full segment.

\subsection{Deterministic assessment and complete reproducibility}
Write $I_E=o+b[-\beta,\beta]$. Appendix~\ref{app:juliaproof} proves
\[
 d_H(E,I_E)=m_0=\sqrt{3-\beta}.
\]
For two segments, convexity of distance to the other segment puts each
directed maximum at an endpoint. We evaluate the four endpoint distances
with exact rational projection and outward square roots, using an outward
enclosure of $\beta$. This gives $d_H(M,I_E)\le\Delta$ and therefore
\[
 \max(0,\underline m_0-\Delta)\le d_H(E,M)
 \le\overline m_0+\Delta.
\]
These deterministic enclosures, calculated only after saving the outputs,
verify that every emitted reported bound in Figure~\ref{fig:juliadirect}
holds. They do not validate a probability model or estimate a failure
rate. The three requested tolerances reuse the same inputs and are not
additional trials.

The source package includes every new input, rational certificate,
coordinate/inverse array, separate reference-coordinate array, assessment and summary
under \texttt{julia\_whitney/}. Clean 16-dimensional references are reconstructed from the stored Julia coordinates and the separately recorded rational sensor vector and offset; this avoids storing 16 redundant copies of a one-dimensional reference. Execute from the package root:
\begin{verbatim}
OPENBLAS_NUM_THREADS=1 python3 code/run_julia_whitney.py
OPENBLAS_NUM_THREADS=1 python3 code/run_julia_whitney.py --replay
python3 code/plot_julia_whitney.py
python3 build.py
\end{verbatim}
Replay requires exact equality of every rational decision record and
checks all reconstructed integer identities. NumPy, SciPy and Matplotlib
versions are recorded in the revision audit. The figure audit in \texttt{code/\allowbreak audit\_julia\_torus\_figures.py} checks regenerated inputs, all Julia output arrays, and the torus outputs against archived array hashes. The primary torus display archive is included, so \texttt{plot\_rebuilt\_whitney.py} can run without a separate result ZIP. The original Two Moons/Julia
aggregate results remain historical records; their external full-archive
commands are separate from this self-contained new implementation.

\end{document}

%% file: image_extensions/main_table.tex
\begin{table}[H]\centering\scriptsize
\setlength{\tabcolsep}{3.5pt}
\caption{Four image baselines on the separated face tests, with the observed mean and local-PCA diagnostic. Entries are unclipped pixel MSE averaged over three fixed noise seeds and all test images. Face noise is relative $\rho$.}\label{tab:fourimagebaselines}
\begin{tabular}{@{}llcrrrrrr@{}}\toprule
Data & Noise & $k$ & Mean & PCA64 & Y--X $G$ & NLM & BM3D & LPCA diag.\\\midrule
Face & 0.2 & 1 & 0.002358 & 0.001883 & 0.001966 & 0.000359 & 0.000261 & 0.001992\\
Face & 0.2 & 4 & 0.000590 & 0.001640 & 0.001643 & 0.000145 & 0.000107 & 0.001664\\
Face & 0.4 & 1 & 0.009433 & 0.002372 & 0.002441 & 0.000915 & 0.000625 & 0.002696\\
Face & 0.4 & 4 & 0.002361 & 0.001886 & 0.001969 & 0.000359 & 0.000262 & 0.002004\\
\bottomrule\end{tabular}\end{table}